\documentclass[10pt,twoside,openright]{report}
\pdfoutput=1

\title{On the concordance orders of knots}
\author{Julia Collins}
\date{\today}

\usepackage{fancyhdr, amsmath,amssymb,amsfonts,dsfont,latexsym,mathrsfs,amsthm, amscd, url, enumitem}
\usepackage[phd]{edmaths}
\usepackage{graphicx,color,rotating,subfigure, wrapfig}
\usepackage[all]{xy}
\usepackage{tikz}

\def\R{\mathds{R}}
\def\Z{\mathds{Z}}
\def\Zp{\widetilde{\mathds{Z}}_p}
\def\N{{\mathds{N}}}
\def\Q{{\mathds{Q}}}
\def\F{{\mathds{F}}}
\def\CC{{\mathds{C}}}
\def\C{{\mathcal{C}}}
\def\A{{\mathcal{A}}}
\def\G{{\mathcal{G}}}
\def\CF{{\mathcal{F}}}
\def\I{{\mathcal{I}}}

\theoremstyle{definition}
\newtheorem{theorem}{Theorem}[section]
\newtheorem*{theoremNoNumber}{Theorem}
\newtheorem{proposition}[theorem]{Proposition}
\newtheorem{lemma}[theorem]{Lemma}
\newtheorem{corollary}[theorem]{Corollary}
\newtheorem{definition}[theorem]{Definition}
\newtheorem{conjecture}[theorem]{Conjecture}
\newtheorem{remark}[theorem]{Remark}
\newtheorem*{notation}{Notation}
\newtheorem{example}[theorem]{Example}
\newtheorem*{note}{Note}

\newtheorem*{thmtorussig}{Theorem \ref{Thm:L2torussig}}
\newtheorem*{cortwistslice}{Corollary \ref{Cor:twistslice}}
\newtheorem*{cortwisteddouble}{Corollary \ref{Cor:twisteddouble}}
\newtheorem*{thmtwistedpoly}{Theorem \ref{Thm:twistedpoly}}
\newtheorem*{thmtwistedpolyAlt}{Theorem \ref{Thm:twistedpolyAlt}}
\newtheorem*{thminfiniteordersum}{Theorem \ref{Thm:infiniteordersum}}
\newtheorem*{thmtwistedpolyq^n}{Theorem \ref{Thm:twistedpolyq^n}}
\newtheorem*{thmHighercover}{Theorem \ref{Thm:Highercover}}
\newtheorem*{thmreverseInfiniteOrder}{Theorem \ref{Thm:reverseInfiniteOrder}}
\newtheorem*{thmhighercoverConnSum}{Theorem \ref{Thm:highercoverConnSum}}

\newtheorem*{coralgorder4}{Corollary \ref{Cor:AlgOrder4}}

\DeclareMathOperator{\lk}{\emph{lk}}
\DeclareMathOperator{\GL}{GL}
\DeclareMathOperator{\sign}{sign}
\DeclareMathOperator{\disc}{Disc}
\DeclareMathOperator{\interior}{int}
\DeclareMathOperator{\Hom}{Hom}
\newcommand{\eps}{\varepsilon}

\setdescription{noitemsep,topsep=0pt}

\begin{document}
%

\maketitle
\newpage



\pagestyle{fancy}
\fancyhead{} 
\renewcommand{\headheight}{28pt}
\cfoot{}
\fancyfoot[LE,RO]{\thepage}

\renewcommand{\chaptermark}[1]{\markboth{\chaptername \ \thechapter.\ #1}{}}
\renewcommand{\sectionmark}[1]{\markright{\thesection.\ #1} {}}

\fancyhead[LE]{\small \slshape \leftmark}      
\fancyhead[RO]{\small \slshape \rightmark}     
\renewcommand{\headrulewidth}{0.3pt}    


\pagenumbering{roman}
\setcounter{page}{1}


\addcontentsline{toc}{chapter}{Abstract}
\chapter*{Abstract}
\noindent

This thesis develops some general calculational techniques for finding the orders of knots in the topological concordance group $\C$ . The techniques currently available in the literature are either too theoretical, applying to only a small number of knots, or are designed to only deal with a specific knot. The thesis builds on the results of Herald, Kirk and Livingston \cite{HeraldKirkLivingston10} and Tamulis \cite{Tamulis02} to give a series of criteria, using twisted Alexander polynomials, for determining whether a knot is of infinite order in $\C$.

 \medskip

There are two immediate applications of these theorems.  The first is to give the structure of the subgroups of the concordance group $\C$ and the algebraic concordance group $\G$ generated by the prime knots of $9$ or fewer crossings.  This should be of practical value to the knot-theoretic community, but more importantly it provides interesting examples of phenomena both in the algebraic and geometric concordance groups.  The second application is to find the concordance orders of all prime knots with up to $12$ crossings.  At the time of writing of this thesis, there are $325$ such knots listed as having unknown concordance order. The thesis includes the computation of the orders of all except two of these.

\medskip

In addition to using twisted Alexander polynomials to determine the concordance order of a knot, a theorem of Cochran, Orr and Teichner \cite{COT03} is applied to prove that the $n$-twisted doubles of the unknot are not slice for $n\neq 0,2$.  This technique involves analysing the `second-order' invariants of a knot; that is, slice invariants (in this case, signatures) of a set of metabolising curves on a Seifert surface for the knot.  The thesis extends the result to provide a set of criteria for the $n$-twisted double of a general knot $K$ to be slice; that is, of order 0 in $\C$.

 \medskip

The structure of the knot concordance group continues to remain a mystery, but the thesis provides a new angle for attacking problems in this field and it provides new evidence for long-standing conjectures.

\vspace{10mm}
\normalsize 

\newpage
\addcontentsline{toc}{chapter}{Declaration}
\chapter*{Declaration}

\normalsize
I do hereby declare that this thesis was composed by myself and that
the work described within is my own, except where explicitly stated otherwise.

\vspace{20mm}
\hfill {\it Julia Collins}

\hfill May 2011

\newpage
\addcontentsline{toc}{chapter}{Acknowledgements}
\chapter*{Acknowledgements}

\noindent

\normalsize

My biggest thanks and gratitude must go to Charles Livingston, without whom this thesis would not have been possible. For the past four years he has given selflessly his time, his ideas, his moral support and his patience, and I cannot describe how much these things have meant to me.

I am also incredibly grateful to my supervisor, Andrew Ranicki, for supporting me through the whole of my studies (and indeed, my career), even though I veered off from the path he originally envisaged for me. Through his encouragement I have become a better public speaker, better expositor and better graphic designer. I have had the privilege of meeting some of the world's greatest mathematicians at Andrew and Ida's wonderful dinner parties and at the high-level conferences he has organised. No graduate student could have had a supervisor who did more for them than mine did for me.

On the academic side I must also thank Paul Kirk for many illuminating conversations on the ideas in this thesis, and Chris Smyth for helping me out with the number theory side of things. Thanks go to my examiners Brendan Owens and Mark Grant for carefully reading the thesis and offering many valuable suggestions for improvement.

My mathematical brothers Mark Powell and Spiros Adams-Florou have been a constant source of inspiration and help to me throughout my thesis. Mark has given much of his time in helping me to understand the finer points of knot concordance, whilst Spiros deserves the credit for finding the slice movie for $11a_5$ and for helping me out when I was stuck in Section \ref{Sec:combiningprimepowers}. Their unfailing enthusiasm for mathematics and topology, and their readiness to help out a friend in need, have been much appreciated.

Many thanks go to my officemates Florian Pokorny, Thomas K\"oppe and (more recently!) Patrick Orson for keeping me cheerful, joining me for lunch every day and for putting up with the ever-increasing numbers of sheep-based objects in the room. A huge extra thanks goes to Thomas for always being there to help with computer and LaTeX-related queries and for helping me to make the classification website described in Chapter \ref{chapter7}. I feel privileged to have known such a generous and talented person.

My time in Edinburgh would not have been so wonderful without the presence of the friends who have been here (and elsewhere), who have kept me smiling and cheered along my research efforts. Graeme Taylor (the best housemate and friend I'll ever have), Shona Yu, Erik Jan de Vries, Berian James, Nathan Ryder, and the incomparable Paul Reynolds.

Many thanks to my parents and sister, who have always supported me in everything I chose to do, and who have done well in putting up with me talking about maths whenever I came home. Without their insistence on booking tickets to my graduation before I had even handed in, this thesis might never have been finished.

Last, but not least, my love and thanks goes to Michael Wharmby for his unfailing support in my thesis-writing, even when it meant many months of boring weekends spent working. Not only that, but he has the honour of being the one person (other than my examiners) who has read my thesis fully from cover to cover. I hope that one day I will be able to repay the time he spent in meticulously proofreading this document.

Even more lastly, I could not end these acknowledgements without a thanks to my sheep Haggis for being my constant companion and mastermind behind all my ideas. And, of course, EPSRC and the University of Edinburgh for financially supporting me through the long years of my PhD. 

\singlespacing

\ \newpage
\addcontentsline{toc}{chapter}{Contents}
\tableofcontents

\ \newpage
\addcontentsline{toc}{chapter}{List of symbols}
\chapter*{List of symbols}

\noindent

\normalsize

\begin{tabular}{cl}
Symbol & Meaning \\ \hline
\\
$S^n$ & The $n$-dimensional sphere, homeomorphic to $\{ x \in \R^n, \; |x| = 1\; \}$\\
$D^n$ & The $n$-dimensional disc, homeomorphic to $\{ x \in \R^n, \; |x| \leq 1\; \}$\\
$\#$ & Connected sum of manifolds\\
$M^T$ & The transpose of a matrix $M$\\
$-K$ & The mirror image of the knot $K$ with reversed orientation\\
$K^{\text{r}}$ & The knot $K$ with reversed orientation \\
$X$ & The knot complement $S^3 \backslash K$ (or the knot exterior $S^3 \backslash n(K)$)\\
$X_n$ & The $n$-fold cyclic cover of the knot complement $X$\\
$\Sigma_n$ & The $n$-fold branched cover of a knot \\
$\lk$ & The linking form of a knot\\
$\C_n$ & The $n$-dimensional knot concordance group\\
$\G, \G^{\Z}$ & The (integral) algebraic concordance group\\
$\G^{\Q}$ & The rational algebraic concordance group \\
$\G_{\F}$ & Cobordism classes of isometric structures over a field $\F$ \\
$\A$ & The group of algebraically slice knots in $\C$\\
$\F_n$ & The field with $n$ elements\\
$\F_n^*$ & The group of units of the field $\F_n$ \\
$\Z_n$ & The integers modulo $n$, otherwise known as $\Z/n\Z$ \\
$\Q_p$ & The $p$-adic rationals\\
$\widetilde{\Z_p}$ & The $p$-adic integers\\
$\Z_{(2)}$ & The integers $\Z$ localised at $2$, i.e. rationals with odd denominator \\
$\zeta_q$ & A complex number which is a $q^{\text{th}}$ root of unity \\
$W(R)$ & The Witt group over a ring $R$ \\
$\disc(p(t))$ & The discriminant of a polynomial $p(t)$\\
$\Delta_K(t)$ & The Alexander polynomial of a knot $K$ \\[0.5ex]
$\Delta_{\chi}^K$ & The twisted Alexander polynomial of a knot $K$ corresponding to \\
 & the map $\chi$ \\
$\doteq$ & Equal up to norms, in the context of twisted Alexander polynomials\\
$\sigma_n(p)$ & Galois conjugation of a polynomial $p(t) \in \Q(\zeta_q)[t,t^{-1}]$.\\
$\tau(M,\rho)$ & The Casson-Gordon invariant of a $3$-manifold $M$ and \\
 & map $\rho \colon M \to \Z \oplus \Z_d$ \\
$\sigma$ & The signature of a knot \\
$\sigma_{\omega}(K)$ & The $\omega$- or Tristram-Levine signature of a knot $K$\\
$\sigma_{\omega}(I)$ & The Casson-Gordon signature of a class $I \in W(\CC(t),\I)$ at the unit \\
& complex number $\omega$ \\
$\disc(I)$ & The Casson-Gordon discriminant of a class $I \in W(\CC(t),\I)$\\
$j_K(e^{2\pi i x})$ or $j_K(x)$ & The jump at $\omega= e^{2\pi i x}$ of the signature function $\sigma_{\omega}(K)$\\
$E$ & The set of prime $9$-crossing knots, including any distinct reverses\\
$\C_E$ & The subgroup of $\C$ generated by $E$\\
$\CF_E$ & The free abelian group generated by $E$\\
$\A_E$ & The kernel of the map $\CF_E \to \G$ \\
\end{tabular} 
\ \newpage
\clearpage




\onehalfspacing
\pagenumbering{arabic}

\setcounter{page}{1}
\chapter{\label{chapter1} Introduction}

In this chapter we give an overview of the subject of knot concordance: what it is, why it is interesting, and how mathematicians have gone about studying it.  We highlight the difficulties of working in this field and how the results of this thesis fit into the general picture we have of the structure of the knot concordance group.

Other good surveys of knot concordance may be found in Livingston \cite{Livingston05} and in the unpublished lecture notes of Peter Teichner \cite{Teichner01}.  Surveys of smooth knot concordance can be found in the introduction to the thesis of Adam Levine \cite{ALevine10} and in Jabuka \cite{Jabuka07}.

\section{What is knot concordance?}

Every knot $S^1 \hookrightarrow S^3$ bounds a surface in $S^3$, but only the unknot bounds a disc; that is, a surface without any holes in (see Figure \ref{Fig:knotsurface}).  Is it possible that knots could bound discs if the discs were allowed to sit in the 4th dimension?  This is the motivation behind the definition of a slice knot.

\begin{figure}
  \centering
  \includegraphics[height=4cm]{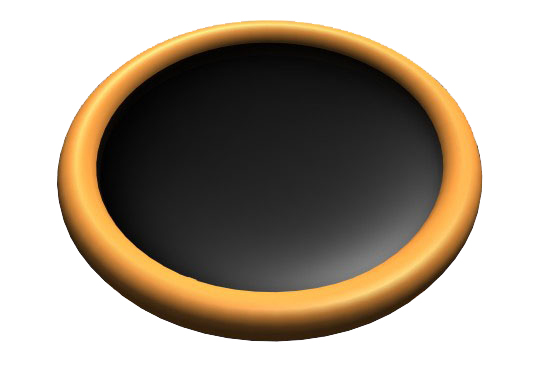}
  \includegraphics[height=4cm]{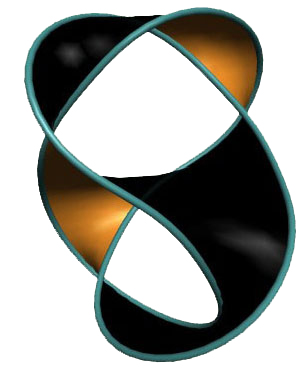}
  \caption{The unknot and figure-8 knot bounding surfaces}
  \label{Fig:knotsurface}
\end{figure}

A slice knot is a knot which bounds a locally flat disc $D^2$ in the 4-ball $D^4$.  Such a disc is called a \emph{slice disc}.  Locally flat means that every point of the disc has a neighbourhood around it which looks like the standard embedding of a disc $D^2$ into $D^4$.  This restriction is necessary in order to make the definition non-trivial.  For without the requirement that the disc be locally flat, \emph{every} knot would bound a disc in $D^4$.

To see this, simply take the cone over the knot, as in Figure \ref{Fig:cone}.  The cone over the knot is homeomorphic to a disc $D^2$ and, because of the embedding into 4 dimensions, there are no intersections of the surface with itself (as are visible in the 3-dimensional picture).  However, the vertex of the cone is not locally flat, and this is where the knot gets `squashed' to a point.  Slice knots are special kinds of knots where it is possible to find discs that they bound which avoid these kinds of singularities.

\begin{figure}
  \begin{center}
    \includegraphics[width=3cm]{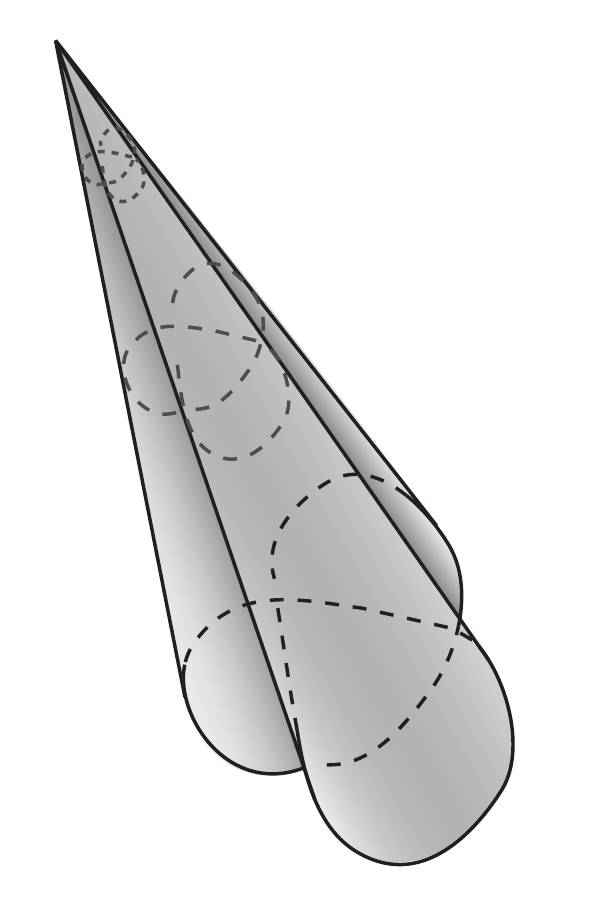}
  \end{center}
  \caption{Every knot bounds a non-locally flat disc: the cone over the knot}
  \label{Fig:cone}
\end{figure}

The definition of a slice knot was first made by Fox \cite{Fox62}, who used the word `slice' because he thought of slice knots as being those which were the cross-sections of locally flat $2$-spheres in $\R^4$, sliced by 3-dimensional hyperplanes.  Indeed, the original purpose of slice knots was to help study knotted $2$-spheres in $4$-space, non-trivial examples of which had been found by Artin in 1925~\cite{Artin25}.  Slice knots remain important in the study of embeddings of surfaces in 4 dimensions and in the classification of 4-manifolds.

Another reason why slice knots are interesting is that we can use them to make the set of knots into a group.  The problem with doing this usually is that, under the operation $\#$ of connected sum, a knot has no inverse.  Two knots can never be tied into a single piece of string so that the first cancels out the second.  To remedy this fact, we change our definition of a `trivial' knot. Instead of the unknot being the only trivial knot, we consider all slice knots to be trivial.  Formally, we say that two knots $K_1$ and $K_2$ are \emph{concordant} if $K_1 \# -K_2$ is slice, where $-K$ is the mirror image of $K$ with reversed orientation.  Then we make a group, called the \emph{concordance group} and denoted by $\C$, which is defined to be the set of knots under connected sum, modulo the equivalence class of slice knots.  One of the reasons this works is because $K \# -K$ is always slice, so knots do now have inverses.

Now that we have a group, the natural question to ask is ``What is the structure of this group?''.  Is it trivial: is every knot slice?  Is \emph{any} knot slice other than the unknot?  Are there elements of finite order?  Is it infinitely generated? Slowly we are coming to understand the answers to these questions but the structure of $\C$ remains a mystery to the mathematical community.

\section{An algebraic approach to slicing}

If a knot is slice then we can show this by simply exhibiting a slice disc for the knot.  But if it isn't slice then we need a topological obstruction to prove it.

An idea proposed by Michel Kervaire~\cite{Kervaire65} was that instead of looking for a slice disc straightaway, we should first find an arbitrary surface that the knot bounds (such surfaces are called \emph{Seifert surfaces} and there is an algorithm for constructing them) and then perform surgery to reduce the genus of that surface so that it becomes a disc in the $4^\text{th}$ dimension.  To be able to perform surgery to reduce the genus of a surface $F$, one needs to be able to find $n= \frac{1}{2} \dim H_1(F)$ closed curves on the surface which represent different elements in $H_1(F)$, which do not link with each other and which are each slice.  If such a set of curves exists, we can remove an annulus around each curve (i.e. remove $S^1 \times D^1$) and glue in a double set of slice discs (a $D^2 \times S^0$) to remove the homology class represented by that curve. (See Figure~\ref{Fig:surgerysurface}.)

\begin{figure}[h]
\begin{center}
  \includegraphics[width=7cm]{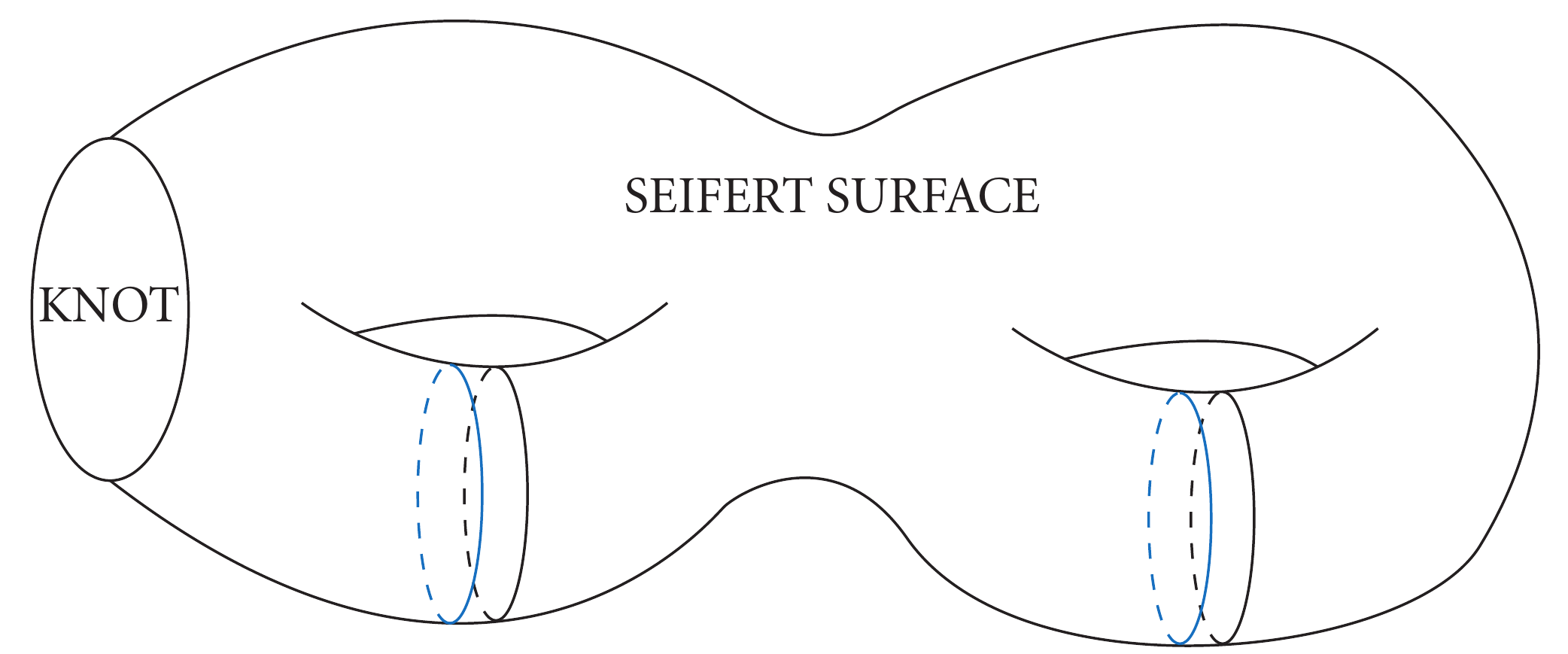}
  \includegraphics[width=7cm]{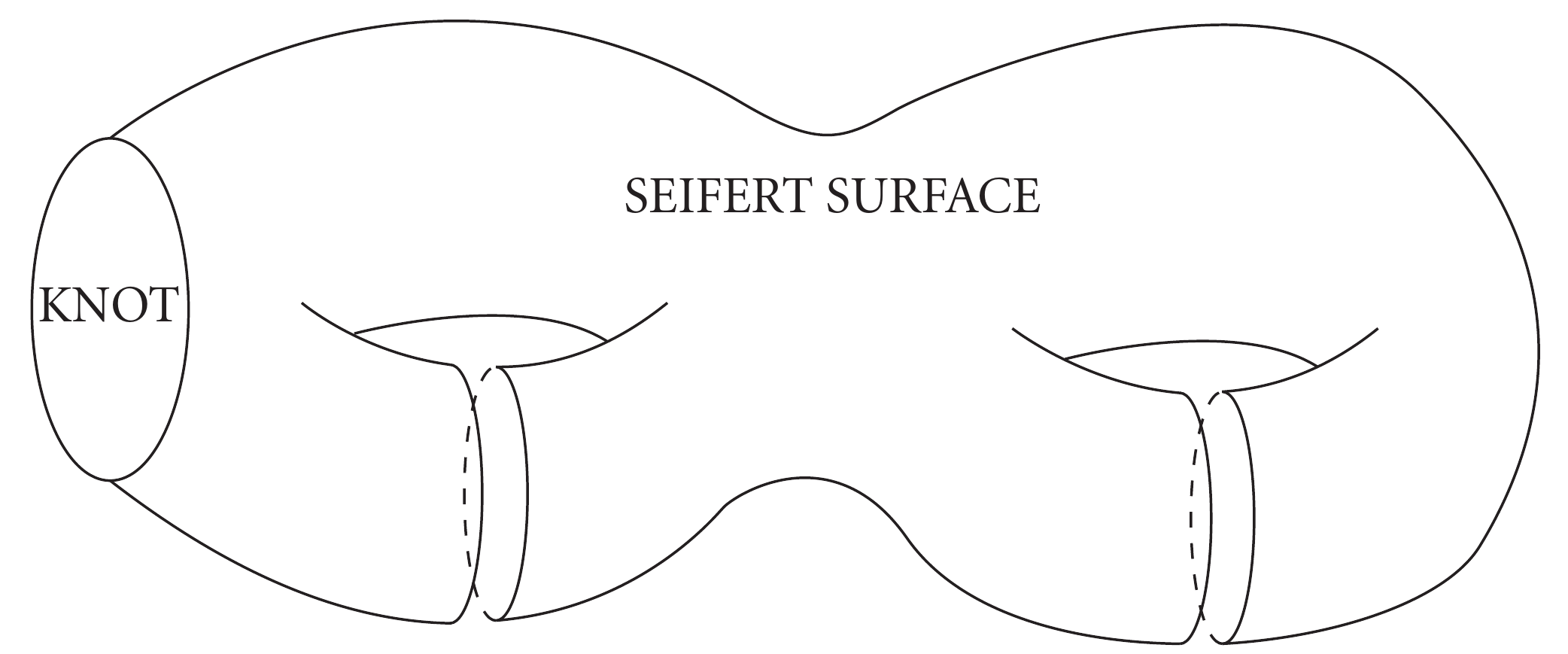}
\end{center}
\caption{Doing surgery to reduce the genus of a surface}
\label{Fig:surgerysurface}
\end{figure}

To implement this programme, Jerome Levine~\cite{Levine69} defined an object called the \emph{algebraic concordance group} $\G$.  The elements of this group are square integral matrices $A$ with the property that $\det(A + \eps A^T) = \pm 1$ where $\eps = \pm 1$.  A matrix is zero, or \emph{null cobordant}, in this group if it is congruent to a matrix of the form
  \[ \left(\begin{array}{cc} 0 & B \\ C & D \end{array}\right) \]
where $B$, $C$ and $D$ are square matrices of the same size.
Addition in the group is by block sum, and two matrices $A_1, A_2$ are considered equivalent, or \emph{cobordant}, if $A_1 \oplus (-A_2)$ is null-cobordant.  By analysing this group with the theory of quadratic forms, Levine~\cite{Levine69-1} proved that
  \[ \G \cong (\Z)^\infty \oplus (\Z_2)^\infty \oplus (\Z_4)^\infty \text{ .} \]
Every knot has an associated matrix in $\G$, called a Seifert matrix, and it can be proved (see Chapter 2, Theorem \ref{Thm:lagrangian}) that slice knots have null-cobordant Seifert matrices.  This means there is a group epimorphism from $\C$ to $\G$.  A knot whose image in $\G$ is zero is called \emph{algebraically slice}.

\subsection{Problems with the algebraic method}
\label{subsec:problemAlgebra}

Is the information contained in the algebraic concordance group enough to classify knots in $\C$?  In other words, if a knot is algebraically slice, does that mean it is geometrically slice?  The evidence points towards a negative answer:
\begin{itemize}
  \item The zeros in a Seifert matrix of an algebraically slice knot correspond to curves with zero linking number.  However, two knots with zero linking number may still be non-trivially linked (for example, the Whitehead link).
  \item Even if the zeros in the matrix \emph{do} correspond to knots which are unlinked, there is no guarantee that the knots in question are slice knots.  This means that surgery along these curves may not reduce the genus of the surface.
\end{itemize}

On the other hand, for slice knots in higher dimensions (i.e. $(S^n,S^{n+2}) = (\partial D^{n+1}, \partial D^{n+3})$, where $n\geq 2$) Levine and Kervaire proved that the algebra \emph{was} equivalent to the geometry (see Section \ref{Sec:higherdim}).   Their surgery-theoretic techniques worked perfectly in dimensions of 5 and above, but somehow there was always a problem when people tried applying it to 4 dimensions.  This problem was called the Whitney trick, which in high dimensions allows cancellation of opposite pairs of singularities (meaning that manifolds with `linking number' zero really have no intersections), but which fails to work in 4 dimensions.  Could there be a way of making it work for slice knots?

In the 1970s the first example was found of a knot which was algebraically slice but not geometrically slice.  Andrew Casson and Cameron Gordon~\cite{CassonGordon86} devised a concordance invariant which looked at intersection forms on $4$-manifolds whose boundaries were appropriate $3$-manifolds associated to the knot.  It sounds complicated, and it is: their invariant is almost incalculable for knots of genus higher than $1$, but it sufficed to prove that the map $\C \to \G$ has a non-trivial kernel.  The knots which were the first elements to be found in this kernel were the $n$-twisted doubles of the unknot (otherwise known simply as the twist knots), shown in Figure~\ref{Fig:twistknot}.

\begin{figure}
  \centering
  \includegraphics[width=6cm]{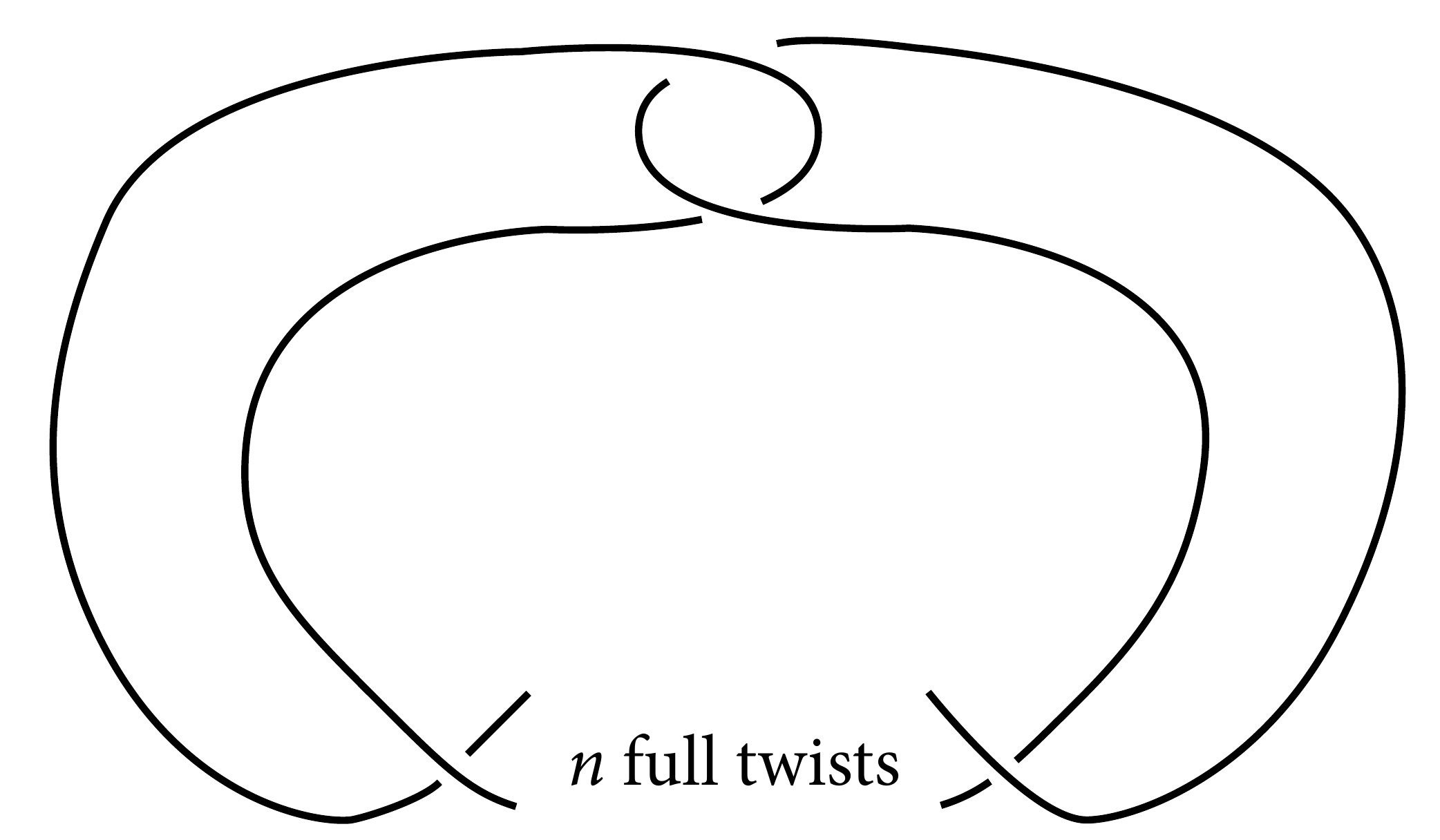}
  \caption{The $n$-twisted double of the unknot}
  \label{Fig:twistknot}
\end{figure}

A natural question following this discovery was: how much of the structure of the algebraic concordance group is present in the geometric concordance group?  For example, there are knots of algebraic order $2$ and algebraic order $4$; does this mean there are knots of geometric order $2$ and geometric order $4$?

The existence of order $2$ knots turned out to be easy to prove: any non-slice negative amphicheiral knot (that is, a knot $K$ where $K = -K$) is of order $2$, because $K \# -K$ is always slice.  For example, the Figure-8 knot $4_1$ is not slice but is of order $2$ because it is negative amphicheiral.  The orientation of the knot is important: positive amphicheiral knots (where $K=-K^r$) are not necessarily of order $2$, as Charles Livingston found \cite{Livingston01}.

A further question concerning $2$-torsion in $\C$ is whether all of it is generated by amphicheiral knots.  It was an interesting discovery that not all knots of order $2$ are themselves amphicheiral (and there are examples of this in Chapter \ref{chapter8}), but all of these examples have been found to be concordant to (negative) amphicheiral knots.

The question of whether there is $4$-torsion in $\C$ continues to elude mathematicians.  The prevailing conjecture, put forward by Livingston and Naik \cite{LivingstonNaik01}, is that all knots of algebraic order $4$ have infinite order in $\C$.  This conjecture has been proven for infinite families of algebraic order $4$ knots, but a general proof for all knots is yet to appear.  Interestingly, the techniques for proving that such knots are of infinite order have usually hinged on number-theoretic arguments that make use of primes which are equivalent to $3$ modulo $4$.  The number fields $\F_p$ of such primes have the property that $-1$ is not a square; it is fascinating that such a simple and esoteric fact can have such far-reaching consequences in the theory of knot concordance.

\section{The structure of $\C$}

As a group, we know little more than that $\C \supset \Z^{\infty} \oplus (\Z_2)^{\infty}$.  Rather than analysing the structure of $\C$, a more interesting project is to analyse the structure of the kernel of the map $\C \to \G$ to the algebraic concordance group.  This group also has a factor $\Z^{\infty} \oplus (\Z_2)^{\infty}$ because there are algebraically slice knots of infinite order \cite{Jiang81} and algebraically slice knots of order $2$ in $\C$ \cite{Livingston98}. But are all the knots in this kernel detected by Casson-Gordon invariants?  Could there be knots which are non-trivial in $\C$, but which are algebraically slice and have vanishing Casson-Gordon invariants?

The answer to this latter question is `yes', and indeed, there are infinitely many of them.  What about the invariant that detects such knots?  Could that be zero and yet the knot still not be slice?  Again, `yes'.  Three mathematicians, Tim Cochran, Kent Orr and Peter Teichner \cite{COT03}, constructed an infinite tower of such invariants, creating what is known as the \emph{filtration} of the knot concordance group:
\[ \dots \subset \CF_{n.5} \subset \CF_n \subset \dots \subset \CF_{0.5} \subset \CF_0 \subset \C \]
The first few stages of the filtration are relatively well understood: $\CF_{0}$ consists of those knots with vanishing Arf invariant; $\CF_{0.5}$ contains those knots which are algebraically slice; $\CF_{1.5}$ consists (roughly) of those knots with vanishing Casson-Gordon invariants.  The higher levels of the filtration are poorly understood by all except a few people.  Let us discuss some geometric intuition for what this filtration is measuring.

The idea behind algebraic sliceness is to say ``We don't know how to find a slice disc for a knot, so let's start with a higher genus surface and try to perform surgery to turn it into a disc.''  To be able to perform successful surgery on a genus $g$ surface, we have to find $g$ curves representing different homology classes on the surface, and each of those curves needs to be slice and unlinked with the other curves.  The algebraic concordance group measures the ability to find these $g$ unlinked curves; a better invariant might also try to measure whether the curves are slice.

But deciding whether these curves are slice is the same problem as we had with the initial knot!  We have simply pushed the problem down a level, as it were, in the hopes that these new curves might be simpler to deal with than the first knot.

So here is an iterative procedure for deciding if a knot $K$ is slice: find a genus $g$ surface $F$ that $K$ bounds, find $g$ unlinked curves of different homology classes on $F$ (if you can't, $K$ is not slice, so stop), then for each of these $g$ curves, repeat the procedure as for $K$.  The resulting `surface' is (amusingly) known as a \emph{grope} since it appears to have `multitudinous fingers' (see Figure \ref{Fig:grope} and \cite{Teichner04}).

\begin{figure}
  \centering
  \includegraphics[width=4cm]{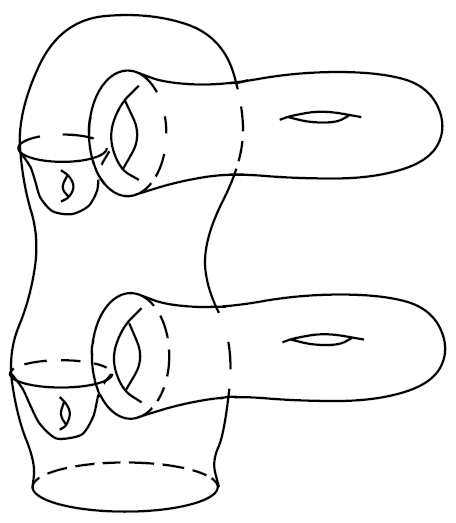}
  \caption{A grope}
  \label{Fig:grope}
\end{figure}

If the procedure stops after $n$ iterations, then the knot should intuitively be in level $n-1$ of the filtration, i.e. $\CF_{n-1}$ (although the Cochran-Orr-Teichner filtration is more complicated than this).  If a knot is slice, it should intuitively be in $\displaystyle \bigcap_{i=0}^\infty \CF_i$ because eventually the curves in some level will be slice and bound genus $0$ surfaces, which makes finding the requisite $g$ curves a triviality.  However, one of the open questions with the filtration is whether $\displaystyle \bigcap_{i=0}^\infty \CF_i$ consists \emph{exactly} of the slice knots.  That is, is it possible that there should be knots which are not slice, and yet for which this procedure never terminates?

Cochran, Orr and Teichner have proved that the number of knots in each level of the filtration does not decrease, as one might hope/expect.  Indeed, each quotient $\CF_{n}/\CF_{n.5}$ is not only infinite but is infinitely generated \cite{CochranTeichner07, CochranHarveyLeidy09}.  There are very few examples of knots which occur high up in the filtration (i.e. above the Casson-Gordon $1.5$ level), and those which we know of have been constructed artificially rather than found naturally in a table of knots.  This is because proving that a knot lies in a particular level of the filtration is conceptually and computationally very difficult. One would have to analyse all possible metabolising curves (i.e. sets of $g$ unlinked curves) on a surface at every level and show that every one of them was obstructed from being slice.  At the moment this is only possible for genus $1$ knots at the first level of curves \cite{COT04}.

\section{Smooth vs topological?}

So far we have discussed knot concordance in the context of the topological locally flat category.  We could instead work in the smooth category, where slice discs and embeddings are smooth (i.e. $C^\infty$).  It is one of the interesting things about $1$-dimensional slice knots (with slice discs embedded in $4$ dimensions) that the two categories are not equivalent.

One of the cornerstone theorems of algebraic topology is the $h$-cobordism theorem.  Given a compact cobordism $W$ between two $n$-dimensional manifolds $M$ and $N$ where the inclusion maps $M \hookrightarrow W$ and $N \hookrightarrow W$ are homotopy equivalences, the $h$-cobordism theorem states that if $M$ and $N$ are simply connected then $W$ is diffeomorphic to $M \times [0,1]$.  The Generalized Poincar\'{e} conjecture turned out to be a special case of the $h$-cobordism theorem; it said that objects which are homotopy equivalent to $S^n$ are homeomorphic to $S^n$.

Of course, given that the proof of the Poincar\'e Conjecture was worth $\$1$ million in 2003 and the $h$-cobordism theorem was proved by Smale in $1961$~\cite{Smale61}, there has to be a caveat.  The caveat is that $n$ must be greater than or equal to $5$.  The proof relies on the Whitney trick (mentioned in Section \ref{subsec:problemAlgebra} and further explained in Section \ref{Sec:higherdim}), and the Whitney trick fails in dimension $4$.  Despite this, in $1982$ Michael Freedman was able to prove that the $h$-cobordism theorem is true \emph{topologically} when $n=4$ \cite{FreedmanQuinn}.  Donaldson's work in 1987~\cite{Donaldson87} then showed that the theorem was \emph{not} true smoothly; a result which has implications for knot concordance.

A consequence of Freedman's work was that any knot whose Alexander polynomial $\Delta(t)$ (an algebraic invariant derived from the Seifert matrix - see Definition \ref{Def:Alexpoly}) is equal to $1$ is slice.  For example, the Whitehead double of any knot must be topologically slice (see Figure \ref{fig:Whiteheaddouble}).

\begin{figure}
	\centering
  \includegraphics[width=4cm]{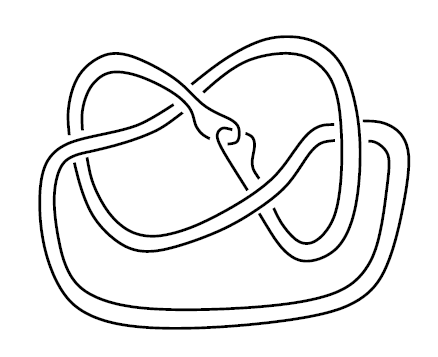}
  \caption{The Whitehead double of the Figure-8 knot}
  \label{fig:Whiteheaddouble}
\end{figure}

However, using invariants derived from Donaldson's work, Akbulut [unpublished] proved that the Whitehead double of the right-handed trefoil was not smoothly slice, thus providing the first example of a difference between the two categories of knot concordance.  This was followed by Rudolph \cite{Rudolph93} who showed that any strongly quasipositive knot could not be smoothly slice. This provided examples of topologically slice knots that are of infinite order in the smooth concordance category.

Since then there have been many invariants developed to detect the non-triviality of knots in the smooth knot concordance group.  Ozsv\'{a}th and Szab\'{o} developed Heegaard-Floer homology~\cite{OszvathSzabo04}, which is a $3$-manifold invariant and a powerful knot invariant which in some sense categorifies the Alexander polynomial.  Two integer-valued concordance invariants derived from it are the Oszv\'{a}th-Szab\'{o} $\tau$ invariant~\cite{OzsvathSzabo03} and the Manolescu-Owens $\delta$ invariant~\cite{ManolescuOwens07}.  There is also Khovanov homology, which categorifies the Jones polynomial, and from it we obtain the integer-valued Rasmussen $s$-invariant~\cite{Rasmussen04}.  Each of these invariants will vanish on torsion elements in the smooth concordance group, so their non-triviality will detect elements of infinite order.  However, it means that they are not powerful enough to distinguish, for example, between knots which are slice and knots which are of order 2.

Yet Heegaard-Floer homology \emph{is} strong enough to detect elements of finite order.  By using correction terms coming from Heegaard-Floer homology, Jabuka and Naik~\cite{JabukaNaik07} were able to find elements of (smooth) finite order.  Their argument was somewhat similar to the Casson-Gordon one, in the sense that both were trying to obstruct intersection forms of $4$-manifolds which are bounded by cyclic branched covers of the knot.  The Casson-Gordon invariants themselves were developed for the smooth category of knot concordance, but with Freedman's work were shown to also obstruct sliceness in the topological category.

The relationship between finite order elements in the topological concordance group $\C_{top}$ and the smooth concordance group $\C_{diff}$ remains unknown.  For example, what is the kernel of the map $\C_{diff} \to \C_{top}$?  All that is currently known is that it contains a subgroup isomorphic to $\Z^3$, via the three maps $\tau$, $\delta$ and $s$ above.  It is possible that there could exist examples of knots which have different finite orders in the different categories, but none have yet been found.

\section{Outline and main results of the thesis}

In this thesis we set out to probe the structure of the knot concordance group by developing computational techniques to decide whether a knot is of infinite order in $\C$.  The obstructions -- twisted Alexander polynomials -- are special cases of Casson-Gordon invariants, which means they find knots in level $1.5$ of the Cochran-Orr-Teichner concordance filtration.  It also means that they are both smooth and topological obstructions.  Out of $325$ prime knots of unknown concordance order, these techniques are powerful enough to find the slice status of all but two of them.

All calculations for this thesis were done in Maple 13 and in Sage; copies of the programs are available on request from the University of Edinburgh library or from the author. In all the theorems and results given here, we need to assume that Maple's polynomial factorisation algorithm (over $\Q(\zeta_q)[t,t^{-1}]$) is accurate.

\begin{theoremNoNumber}(Chapter \ref{chapter8})
  Of the prime knots of 12 or fewer crossings listed as having unknown concordance order, they are all of infinite order with the exception of the following:
    \begin{itemize}
      \item $11n_{34}$ is slice because it has Alexander polynomial equal to 1.
      \item $12a_{1288}$ is of order 2 because it is fully amphicheiral.
      \item $11a_5$, $11a_{104}$, $11a_{112}$, $11a_{168}$, $11n_{85}$, $11n_{100}$, $12a_{309}$, $12a_{310}$, $12a_{387}$, $12a_{388}$, $12n_{286}$ and $12n_{388}$ are all of order $2$, concordant to the Figure Eight knot $4_1$.
      \item $11a_{44}$, $11a_{47}$ and $11a_{109}$ are all of order $2$, concordant to the knot $6_3$.
      \item $12a_{631}$ remains of unknown order, but is suspected to be finite order and possibly slice.
      \item $12n_{846}$ remains of unknown order, and there are no suspicions as to whether it is of finite or infinite order.
    \end{itemize}
\end{theoremNoNumber}

The remaining two knots look to be strong candidates for examples of knots lying in the higher levels of the Cochran-Orr-Teichner concordance filtration, and as such will be interesting objects of further study.

\subsection{How to prove that a knot has infinite order in $\C$}

What are the techniques that we developed to determine the concordance orders of these $325$ prime knots? We summarise here the main results from Chapter \ref{chapter5}, which give a series of criteria for deciding if a knot has infinite order in the concordance group $\C$.  The idea is to look at metabolisers of the first homology of the $p$-fold branched cover $\Sigma_p$ of a knot, for a prime $p$, and to show that under certain circumstances there is always a twisted Alexander polynomial $\Delta_{\chi}$ which obstructs the knot from being slice. In what follows $\zeta_q$ is always a complex $q^{\text{th}}$ root of unity, a norm is an element of $\Q(\zeta_q)[t,t^{-1}]$ of the form $g(t)\overline{g(t^{-1})}$ and $\doteq$ means `up to norms'.

 In the first case we concentrate on the $2$-fold branched cover.

 \begin{thmtwistedpoly}
Suppose that we have a knot $K$ where $H_1(\Sigma_2;\Z) \cong \Z_{q} \oplus T$ for some prime $q \equiv 1$ mod $4$, where the order of $T$ is coprime to $q$.  Let $\chi_0 \colon H_1(\Sigma_2;\Z) \to \Z_q$ be the trivial map and $\chi_i \colon H_1(\Sigma_2;\Z) \to \Z_q$ be $\lk(-,i)$.  Then $K$ is of infinite order if it satisfies the following conditions:
  \begin{enumerate}
    \item $\Delta_{\chi_0}(t)$ is not a norm in $\Q(\zeta_q)[t,t^{-1}]$.
    \item There is a non-trivial irreducible factor $f(t)$ of $\Delta_{\chi_0}(t)$ for which $\overline{f(t^{-1})}$ is not a factor of $\Delta_{\chi_i}(t)$ for any $i$.
    \item $\Delta_{\chi_a}(t) \not\doteq \Delta_{\chi_1}(t)$, where $1+a^2 \equiv 0$ (mod $q$).
  \end{enumerate}
\end{thmtwistedpoly}

This theorem is extended to give criteria for a sum of knots to be of infinite order.  This is necessary for proving that knots are independent within $\C$.

\begin{thminfiniteordersum}
Let $K=K_1 + \dots + K_n$ with $K_{i_1},\dots,K_{i_{n'}}$ having $H_1(\Sigma_2(K_{i_j});\Z) \cong \Z_q \oplus T_j$ with the order of the $T_j$ coprime to $q$, and the remainder of the $K_i$ having $H_1(\Sigma_2(K_i);\Z) \cong T_i$ with the order of the $T_i$ coprime to $q$.  Then $K$ has infinite order in $\C$ if the twisted Alexander polynomials (as defined in Theorem~\ref{Thm:twistedpoly}) satisfy:
\begin{enumerate}
  \item $\Delta_{\chi_0}^{K_{i_j}}(t)$ does not factorise over $\Q(\zeta_q)[t,t^{-1}]$ as a norm for any $j=1,\dots,n'$.
  \item There is a non-trivial irreducible factor $f_j(t)$ of $\Delta_{\chi_0}^{K_{i_j}}(t)$ for which $\overline{f_j(t^{-1})}$ is not a factor of $\Delta_{\chi_\alpha}^{K_{i_k}}(t)$ for all $j,k=1,\dots,n'$ and all $\alpha \neq 0$.
  \item $\Delta_{\chi_1}^{K_{i_j}}(t)\not\doteq \Delta_{\chi_\gamma}^{K_{i_k}}(t)$ where $\gamma = \sqrt{-\alpha\beta^{-1}}$ with $\alpha = \lk_{K_{i_j}}(1,1)$ and $\beta = \lk_{K_{i_k}}(1,1)$, for all $j,k$ where $\gamma$ is defined. This means that if $q\equiv 1$ mod $4$ then $\alpha$ and $\beta$ must be the same modulo squares and if $q \equiv 3$ mod 4 then $\alpha$ and $\beta$ must be different modulo squares.
\end{enumerate}
\end{thminfiniteordersum}

Next we look at criteria for finding infinite order knots using branched covers whose homology has a more complicated structure.

\begin{thmtwistedpolyq^n}
Suppose that we have a knot $K$ where $H_1(\Sigma_2;\Z) \cong \Z_{q^n} \oplus T$ for some prime $q$, where the order of $T$ is coprime to $q$.  Let $\chi_0 \colon H_1(\Sigma_2;\Z) \to \Z_{q}$ be the trivial map and $\chi_i \colon H_1(\Sigma_2;\Z) \to \Z_{q}$ be $\lk(-,i)$ (mod $q$).  Then $K$ is of infinite order if it satisfies the following conditions:
  \begin{enumerate}
    \item If $n>1$, $\Delta_{\chi_0}(t)\Delta_{\chi_1}(t)$ is not a norm.
    \item $\Delta_{\chi_0}(t)$ is not a norm.
    \item $\Delta_{\chi_0}(t)$ is coprime, up to norms in $\Q(\zeta_q)[t,t^{-1}]$, to $\Delta_{\chi_i}(t)$ for all $i \neq 0$.
    \item If $q \equiv 1$ (mod $4$) then $\Delta_{\chi_a}(t) \not\doteq \Delta_{\chi_1}(t)$, where $1+a^2 \equiv 0$ (mod $q$).
  \end{enumerate}
\end{thmtwistedpolyq^n}

Next we look at twisted Alexander polynomials associated to branched covers $\Sigma_p$ of a knot $K$ with $p>2$.

\begin{thmHighercover}
Suppose $H_1(\Sigma_p;\Z)\cong \Z_q \oplus \Z_q \cong E_a \oplus E_b$ where $E_a$ and $E_b$ are the eigenspaces of the deck transformation $T$.  Let $e_a$ be an $a$-eigenvector (i.e. $a e_a = Te_a$) and $e_b$ be a $b$-eigenvector. Now define $\chi_a \colon H_1(\Sigma_p) \to \Z_q$ by $\chi_a(e_a) = 0$ and $\chi_a(e_b) = 1$.  Similarly, $\chi_b \colon H_1(\Sigma_p) \to \Z_q$ is defined by $\chi_b(e_a) = 1$ and $\chi_b(e_b) = 0$.
The knot $K$ is of infinite order in $\C$ if the following conditions on the twisted Alexander polynomial of $K$ are satisfied:
\begin{enumerate}
  \item $\Delta_{\chi_0}$ is coprime, up to norms, to both $\Delta_{\chi_a}$ and $\Delta_{\chi_b}$, and $\Delta_{\chi_0}$ is not a norm.
  \item $\Delta_{\chi_a + \chi_b} \not\doteq \Delta_{d \chi_a -d^{-1}\chi_b}$ for any $d \in \Z_q$.
\end{enumerate}
\end{thmHighercover}

An extension of this theorem gives a procedure to decide whether a knot is concordant to its reverse -- a very subtle and difficult problem in knot theory which has previously been solved for only a few special cases of knots by Kirk, Livingston and Naik \cite{Livingston83}, \cite{Naik96}, \cite{KirkLivingston99b}. In this next theorem we give criteria not only to tell if a knot is distinct (in $\C$) from its reverse, but whether the difference is of infinite order in $\C$. The notation $\sigma_d$ means the map taking $\zeta_q$ to $\zeta_q^d$ in the coefficients of $\Delta_{\chi}$.

\begin{thmreverseInfiniteOrder}
The knot $K-K^r$ is of infinite order in $\C$ if the following conditions on the twisted Alexander polynomial of $K$ are satisfied:
\begin{enumerate}
  \item $\Delta_{\chi_0}$ is coprime, up to norms, to both $\Delta_{\chi_a}$ and $\Delta_{\chi_b}$, and none of these polynomials are themselves norms.
  \item $\Delta_{\chi_a} \not\doteq \sigma_d(\Delta_{\chi_b})$ for any $d \in \Z_q$
  \item $\Delta_{\chi_a + \chi_b} \not\doteq \Delta_{d \chi_a -d^{-1}\chi_b}$ for any $d \in \Z_q$.
\end{enumerate}
\end{thmreverseInfiniteOrder}

Finally, we have a theorem which gives criteria for a connected sum of knots to be of infinite order, using higher-order branched covers. The set-up is again the same as in Theorem \ref{Thm:Highercover}.

\begin{thmhighercoverConnSum}
The knot $K=K_1 + \dots + K_n$ is of infinite order in $\C$ if the following conditions on the twisted Alexander polynomial of the $K_i$ are satisfied:
\begin{enumerate}
  \item $\Delta_{\chi_0}^{K_i}(s)$ is not a norm for any $i=1,\dots,n$.
  \item $\Delta_{\chi_0}^{K_i}$ is coprime, up to norms, to $\Delta_{\chi_a}^{K_j}$ and $\Delta_{\chi_b}^{K_j}$ for every $i$ and $j$.
  \item $\Delta_{\chi_a}^{K_i} \not\doteq \sigma_d(\Delta_{\chi_a}^{K_j})$ (or $\Delta_{\chi_b}^{K_i} \not\doteq \sigma_d(\Delta_{\chi_b}^{K_j})$) for any $d \in \Z_q$ and any $i\neq j$.
  \item $\Delta_{\chi_a + \chi_b}^{K_i} \not\doteq \Delta_{d \chi_a -d^{-1}\chi_b}^{K_i}$ for any $d \in \Z_q$ and any $i=1,\dots,n$.
\end{enumerate}
\end{thmhighercoverConnSum}

The combination of all these theorems allows us to attempt a full classification of all the prime knots with 9 or fewer crossings, which we will describe in the next section.

\subsection{A concordance classification of 9-crossing prime knots}
\label{Subsec:9crossingClass}

We wish to find the structure of the subgroups of $\C$ and $\G$ generated by all prime knots with 9 or fewer crossings. Such a result would allow us, given any linear combination of $9$-crossing prime knots, to instantly be able to decide the algebraic and geometric concordance orders of that knot.  The reason that $9$-crossing knots were chosen is that there are sufficiently many to throw up interesting challenges and phenomena, whilst being a small enough set to allow us to perform many of the calculations by hand.  Now that the groundwork and algorithms have been laid out, it should not be too difficult, in future work, to automate the process and produce a classification for $10$-, $11$- and even $12$-crossing knots.

\medskip

Let $E = \left\{ 3_1,4_1,5_1,5_2,6_1,6_2,6_3,7_1, \dots,7_7,8_1,\dots,8_{21}, 9_1, \dots, 9_{49} \right\}$, where $|E|=87$ since the list includes the distinct reverses $8_{17}^r$, $9_{32}^r$ and $9_{33}^r$.

\begin{notation}
Let $\C_E$ denote the subgroup of $\C$ generated by $E$.  Denote by $\CF_E$ the free abelian group generated by $E$.
\end{notation}

There are natural maps $\CF_E \xrightarrow{\psi} \C_E \xrightarrow{\phi} \G$.  We use the term `concordance classification' of $E$ to mean finding the kernel of both $\psi$ and $\phi \circ \psi$, since knowing these kernels would enable us to identify whether any linear combination of knots in $E$ were slice (algebraically and geometrically).

The following result, which is labelled as a conjecture because part of it remains unproved, writes $\CF_E$ in terms of a basis from which it is possible to read off the orders of any linear combination of knots.  It is an amalgamation of the results of Chapters \ref{chapter4} and \ref{chapter6}.

\begin{conjecture}
  A basis of $\CF_E$ consists of the union of the following independent sets.  In the notation, the superscript gives the order of the elements in the $A$=`algebraic' or $T$=`topological' concordance groups. For example, $\C_{A}^2$ means knots which represent elements of order $2$ in $\G$.
  \begin{itemize}
   \item $\C^\infty = \{ 3_1, 5_1, 5_2, 6_2, 7_1, \dots, 7_6, 8_2, 8_4, \dots 8_7, 8_{14}, 8_{16}, 8_{19}, 9_1, 9_3, \dots, 9_7, 9_9, 9_{10}, 9_{11}, 9_{13}, \\ \quad 9_{15},9_{17}, 9_{18}, 9_{20}, 9_{21}, 9_{22}, 9_{25}, 9_{26}, 9_{31}, 9_{32}, 9_{35}, 9_{36}, 9_{38}, 9_{43}, 9_{45}, 9_{47}, 9_{48}, 9_{49}\}$
   \item $\C_A^4  =  \left\{ 7_7, 9_{34} \right\}$
   \item $\C_A^2  =  \{8_1, 8_{13},(8_{15}-7_2-3_1), (9_2-7_4), (9_{12}-5_2), 9_{14}, (9_{16}-7_3-3_1), 9_{19},\\
       (9_{28}-3_1), 9_{30}, 9_{33},(9_{42}+8_5-3_1),(9_{44}-4_1)\}$
   \item $\C_A^1 = \{ (8_{21}-8_{18}-3_1), (9_8-8_{14}), (9_{23}-9_2-3_1), (9_{29}-9_{28}+2(3_1)),\\
             (9_{32}-9_{32}^r), (9_{33}-9_{33}^r), (9_{39}+7_2-4_1), (9_{40}-8_{18}-4_1-3_1) \}$
   \item $\C_A^{1'} = \{ (8_{17}-8_{17}^r) \}$
   \item $\C_T^2 = \{4_1, 6_3, 8_3, 8_{12}, 8_{17}, 8_{18}\}$
   \item $\C_T^1 = \{6_1, 8_8, 8_9, (8_{10}+3_1), (8_{11}-3_1), 8_{20}, (9_{24}-4_1), 9_{27}, (9_{37}-4_1), 9_{41}, 9_{46}\}$
  \end{itemize}

  A basis for the kernel $\A_E$ of $\phi \circ \psi \colon \CF_E \to \G$ is the union of $4\C_A^4$, $2\C_A^2$, $2\C_T^2$, $\C_A^1$, $\C_A^{1'}$ and $\C_T^1$.

  A basis for the kernel of $\psi \colon \CF_E \to \C_E$ is the union of $2\C_A^{1'}$, $2\C_T^2$ and $\C_T^1$.
\end{conjecture}

\begin{corollary}
  The image of $\CF_E$ in the algebraic concordance group $\G$ is $\G_E \cong \Z^{46} \oplus \Z_2^{19} \oplus \Z_4^2$.

  The image of $\A_E$ (the kernel of $\psi \circ \phi$) in the geometric concordance group $\C$ is $\Z^{23} \oplus \Z_2$.
\end{corollary}

The only part of this conjecture which remains unproven is whether the knots $(9_2-7_4)$, $(8_{21}-8_{18}-3_1)$, $(9_{23}-9_2-3_1)$, $(9_{40}-8_{18}-4_1-3_1)$ and $(9_{32}-9_{32}^r)$ are linearly independent in $\C$.  That is, could there be some linear combination of these knots which is slice?  We have not yet been able to extend the theorems of Chapter \ref{chapter5} to deal with this case, which is where the homology of the $2$-fold branched cover of the knot is isomorphic to $a\Z_q \oplus b\Z_{q^2} \oplus T$, with the order of $T$ coprime to $q$. (Alternatively, where the homology of the $p$-fold branched cover, for $p>2$, is isomorphic to $\Z_q \oplus \Z_q \oplus \Z_q \oplus \Z_q$.)

Although algebraic concordance is well-understood, and a complete set of invariants exists to classify knots in $\G$, this investigation of $9$-crossings knots has raised more subtle questions about the structure of $\G$ and these invariants.  For example, is knowledge of the image of a knot in $W(\Q_2)$ necessary for the classification of the knot in $\G$, and can we find a prime knot which represents an element of order $2$ in $\G$ but which is twice another prime knot of order $4$?

The nature of the invariants which are developed in this investigation also provide more evidence that knots which are of algebraic order $4$ will be of infinite order in $\C$.

\subsection{Second-order slice obstructions}

 We end the thesis with a look at the computationally feasible second-order invariants provided by Cochran, Orr and Teichner. The metabolising curves for the twist knots have been shown to be torus knots, so by calculating signatures for the torus knots we provide another proof that the twist knots are not slice.

 \begin{thmtorussig}
  Let $p$ and $q$ be coprime positive integers.  Then the integral of the $\omega$-signatures of the $(p,q)$ torus knot is
  \[ \displaystyle \int_{S^1} \sigma_{\omega} = -\frac{(p-1)(p+1)(q-1)(q+1)}{3pq} \text{ .} \]
 \end{thmtorussig}

 \begin{cortwistslice}
  The twist knots $K_n$ are not slice unless $n=0$ or $n=2$.
 \end{cortwistslice}

This result was already known to Casson and Gordon in the 1970s but the new proof given in this thesis is much shorter and generalises to larger classes of knots, such as the $n$-twisted doubles of an arbitrary knot $K$.

\begin{figure}
  \centering
  \includegraphics[width=9cm]{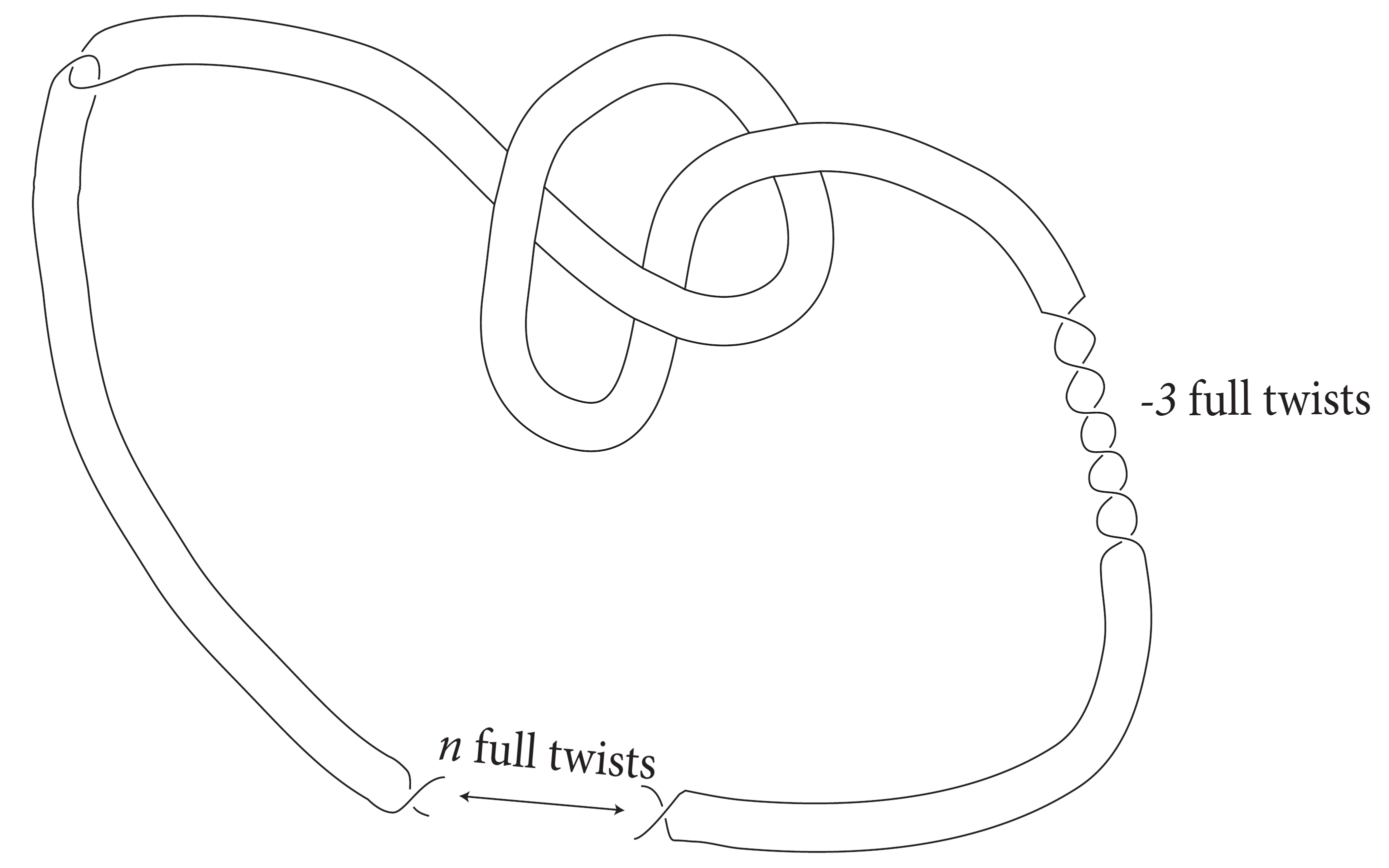}
  \caption{The $n$-twisted double of the right-handed trefoil.}
  \label{fig:twisteddoubleCh1}
\end{figure}

  \begin{cortwisteddouble}
  Let $K$ be a knot and $D_n(K)$ the $n$-twisted double ($n \neq 0$) of $K$ as shown in Figure \ref{fig:twisteddoubleCh1}.
      \begin{itemize}
        \item[(a)] $D_n(K)$ cannot be slice unless $n=m(m+1)$ for some $m \in \Z$ and $\int_{S^1} \sigma_\omega(K) = \frac{(m-1)(m+2)}{3}$.  In particular, $D_2(K)$ can only be slice if $\int_{S^1} \sigma_\omega(K)=0$.
        \item[(b)] For any given $K$ with $\int_{S^1} \sigma_\omega(K) \neq 0$, there is at most one $D_n(K)$ which can be slice.
      \end{itemize}
  \end{cortwisteddouble}

\subsection{Outline}

The structure of this thesis is summarised below:

\begin{itemize}
  \item \textbf{Chapter 2:} A rigorous introduction to the required background knowledge for the thesis; including a definition of the knot concordance group, how to prove when a knot is slice, details of the algebraic concordance group, a discussion of results in higher-dimensional knot concordance, and a description of Casson-Gordon invariants.

  \item \textbf{Chapter 3:} A full and detailed description of the algebraic concordance group, focusing on the work of Levine and the invariants needed to find the image of a knot in $\G$.

  \item \textbf{Chapter 4:} The classification of the prime $9$-crossing knots in $\G$, using the results from Chapter 3.

  \item \textbf{Chapter 5:} A detailed description of twisted Alexander polynomials as slicing obstructions, followed by the main results which use these polynomials to provide criteria for a knot to be of infinite order in $\C$.

  \item \textbf{Chapter 6:} The classification of the prime $9$-crossing knots in $\C$, using the results from Chapters 4 and 5.

  \item \textbf{Chapter 7:} Detailed examples of how to use the classifications provided in
Chapters 4 and 6 to find the concordance order of any linear combination of $9$-crossing knots.  Discussion of how this procedure might fail when dealing with larger classes of knots.

  \item \textbf{Chapter 8:} A description of how the theorems in Chapter 5 have been applied to the prime knots of $12$ or fewer crossings of unknown concordance order, along with slice diagrams for those knots which have been shown to be of finite order.

  \item \textbf{Chapter 9:} How solving an elementary but difficult number-theoretic problem is found to be equivalent to looking at signatures of torus knots, and how these signatures are successful in obstructing the twist knots from being slice. Extensions of this result to obstructing the $n$-twisted doubles of an arbitrary knot from being slice.

  \item \textbf{Chapter 10:} A look at the open problems related to work in this thesis and suggestions on how these might be attacked.
\end{itemize} 
\chapter{\label{chapter2} Background}

In this chapter we give rigorous definitions of what it means for a knot to be (geometrically) slice and algebraically slice. We discuss the map from the geometric to the algebraic concordance group, and see that this is an isomorphism in higher dimensions. It is the failure of the Whitney trick in dimension 4 which causes the map $\C_1 \to \G$ to have a non-trivial kernel, and we look at the work of Casson and Gordon who first exhibited non-trivial elements in this kernel.

\section{The knot concordance group}

\begin{definition}
 A \emph{knot} is a locally flat embedding of a circle $S^1$ into the $3$-sphere $S^3$, under the equivalence relation of ambient isotopy.  That is, two knots $K_1$ and $K_2$ are considered equivalent if there is a homotopy of (orientation-preserving) homeomorphisms $f_t \colon S^3 \to S^3$ such that $f_0$ is the identity and $f_1$ carries $K_1$ to $K_2$.
\end{definition}

We will abuse the terminology in the standard way, with the word `knot' sometimes referring to the embedding and sometimes to the image of the embedding.

\begin{definition}
\label{Def:locallyflat}
  An embedding of manifolds $M^m \stackrel{q}{\hookrightarrow} N^n$ is called \emph{locally flat} if for each point $x \in M$ there is a neighbourhood $U$ of $x$ and a neighbourhood $V$ of $q(x)$ such that the pair $(V,qU)$ can be mapped homeomorphically onto $(D^n, D^m)$.
\end{definition}

\begin{definition}
 A knot $K$ is \emph{topologically (smoothly) slice} if it is the boundary of a locally flat (smooth) disc $D^2$ embedded into the 4-ball $D^4$.
\end{definition}

\begin{definition}
\label{Def:concordant}
 Two knots $K_1$ and $K_2$ are called \emph{concordant} if $K_1 \# -K_2$ is slice, where $-K_2$ means the mirror image of $K_2$ with reversed orientation.  An alternative and equivalent definition is that two knots are concordant if $K_1$ and $K_2$ cobound a locally flat annulus $S^1 \times [0,1]$ embedded in $S^3 \times [0,1]$.
\end{definition}

\begin{remark}
 To see that the two different definitions in \ref{Def:concordant} are equivalent, suppose that $K_1$ and $K_2$ cobound an annulus $A= S^1 \times I \subset S^3 \times I$.  Remove $(D^3 \times I, D^1 \times I)$ from $(S^3 \times I,A)$.  This turns the annulus into a slice disc in $D^4$, bounded by the connected sum $K_1 \# -K_2$ -- see Figure \ref{Fig:slicediscannulus}.
 \begin{figure}
   \begin{center}
     \includegraphics[width=8cm]{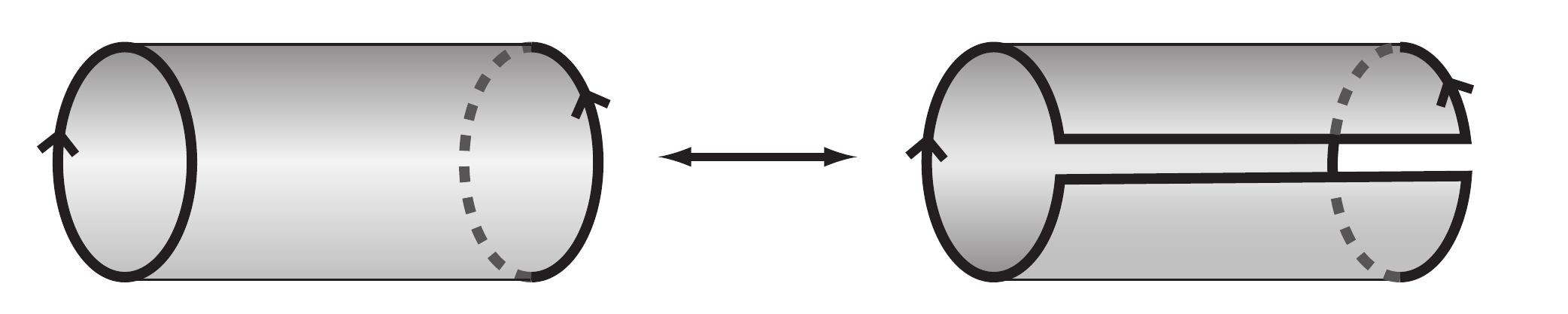}
   \end{center}
   \caption{Turning an annulus into a slice disc}
   \label{Fig:slicediscannulus}
 \end{figure}
\end{remark}

\begin{proposition}
 Concordance is an equivalence relation on the set of all knots.
\end{proposition}

\begin{proof}
 A knot is concordant to itself; to see the slice disc, construct the connected sum of the knot and its reversed mirror image, then join with lines the points which would be identified across a mirror.  These lines together trace out a smoothly immersed disc in $S^3$, which can be turned into an embedded disc by pushing the interior into $D^4$.

\begin{example}
 Here is the connected sum of the trefoil and its mirror image, followed by the corresponding  slice disc.
   \begin{figure}[h]
     \begin{center}
     \includegraphics[width=5cm]{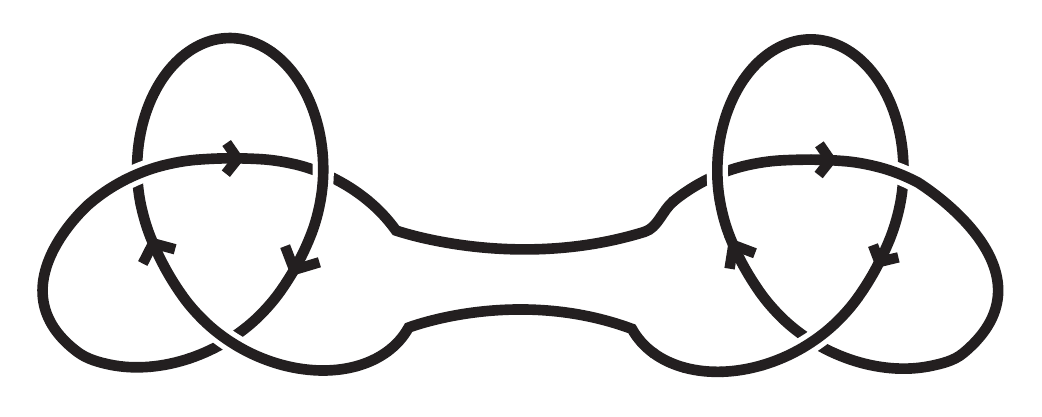}
     \includegraphics[width=5cm]{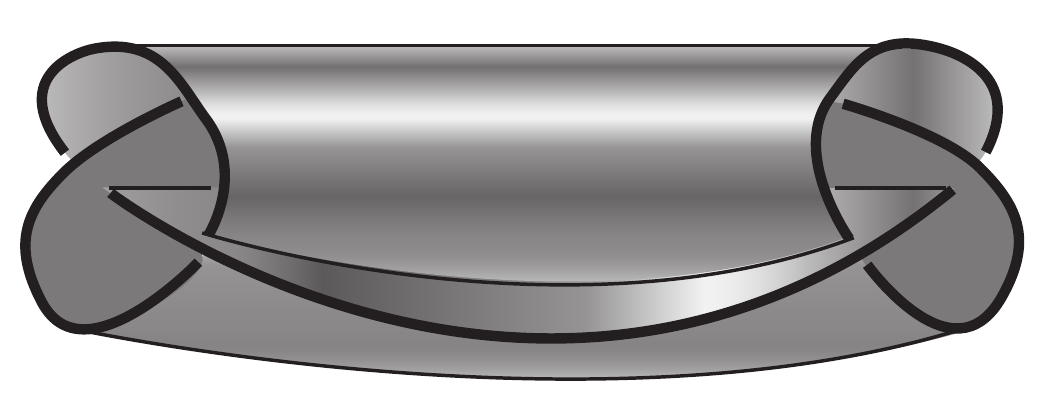}
     \end{center}
     \label{Fig:mirrorimagetrefoildisc}
   \end{figure}
 \end{example}

If $K_1$ is concordant to $K_2$, then $K_1 \# -K_2$ bounds a slice disc.  By reversing the orientation of the slice disc we get a slice disc for $-K_1 \# K_2$, so $K_2$ is concordant to $K_1$.

Finally, if $K_1$ is concordant to $K_2$ and $K_2$ is concordant to $K_3$, then $K_1$ and $K_2$ cobound an annulus $A_1$ whilst $K_2$ and $K_3$ bound another annulus $A_2$. We may glue the annuli $A_1$ and $A_2$ together to create one large annulus with boundary $K_1 \sqcup K_3$.  Thus $K_1$ is concordant to $K_3$.
\end{proof}

\begin{definition}
 The \emph{knot concordance group} $\C$ is the set of knots with the operation of connected sum modulo the equivalence relation of concordance.  The identity of the group is the set of all slice knots, and the inverse of a knot $K$ is $-K$.
\end{definition}

\subsection{How to tell if a knot is slice}
\label{Subsec:sliceMovie}

If a knot is smoothly slice, then its slice disc may be put into general position so that concentric $3$-spheres move through and intersect it to produce knots and links, together with a finite number of singularities.  These singularities correspond to
  \begin{enumerate}
	 \item a simple maximum or minimum
	 \item a saddle point
  \end{enumerate}
  \begin{figure}[h]
    \label{Fig:criticalpoints}
    \begin{center}
      \includegraphics[width=12cm]{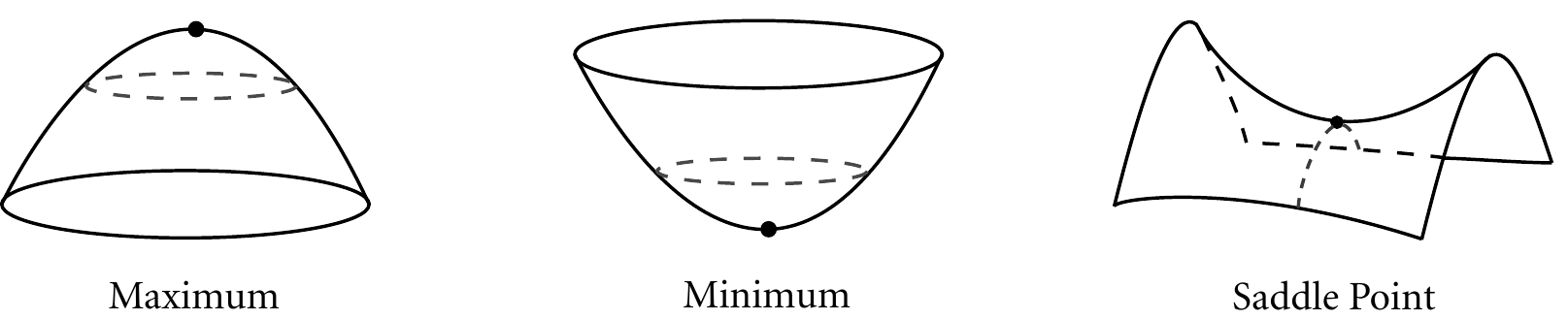}
    \end{center}
  \end{figure}

  To implement a move through a saddle point on the diagram of a knot, one must perform a surgery: remove $S^0 \times D^1$ and glue in $D^1 \times S^0$. Find two arcs of the knot with opposite orientation, remove them and re-glue as shown in the following figure:
    \begin{figure}[h]
    \begin{center}
      \includegraphics[width=5cm]{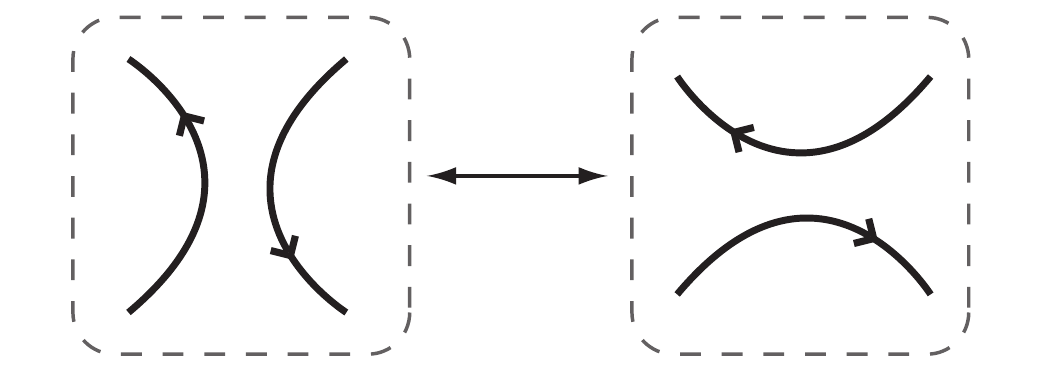}
    \end{center}
    \end{figure}

\begin{example}
  Stevedore's knot, otherwise known as $6_1$, is the simplest slice knot (other than the unknot).  The following `movie' shows how $3$-spheres move through the slice disc:
  \begin{center}
    \includegraphics[width=14cm]{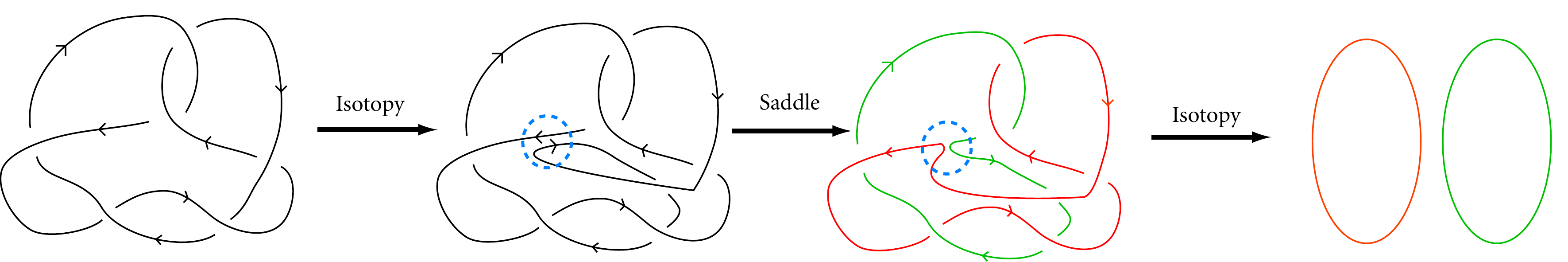}
  \end{center}

  The slice disc is shown schematically below:
  \begin{center}
    \includegraphics[width=6cm]{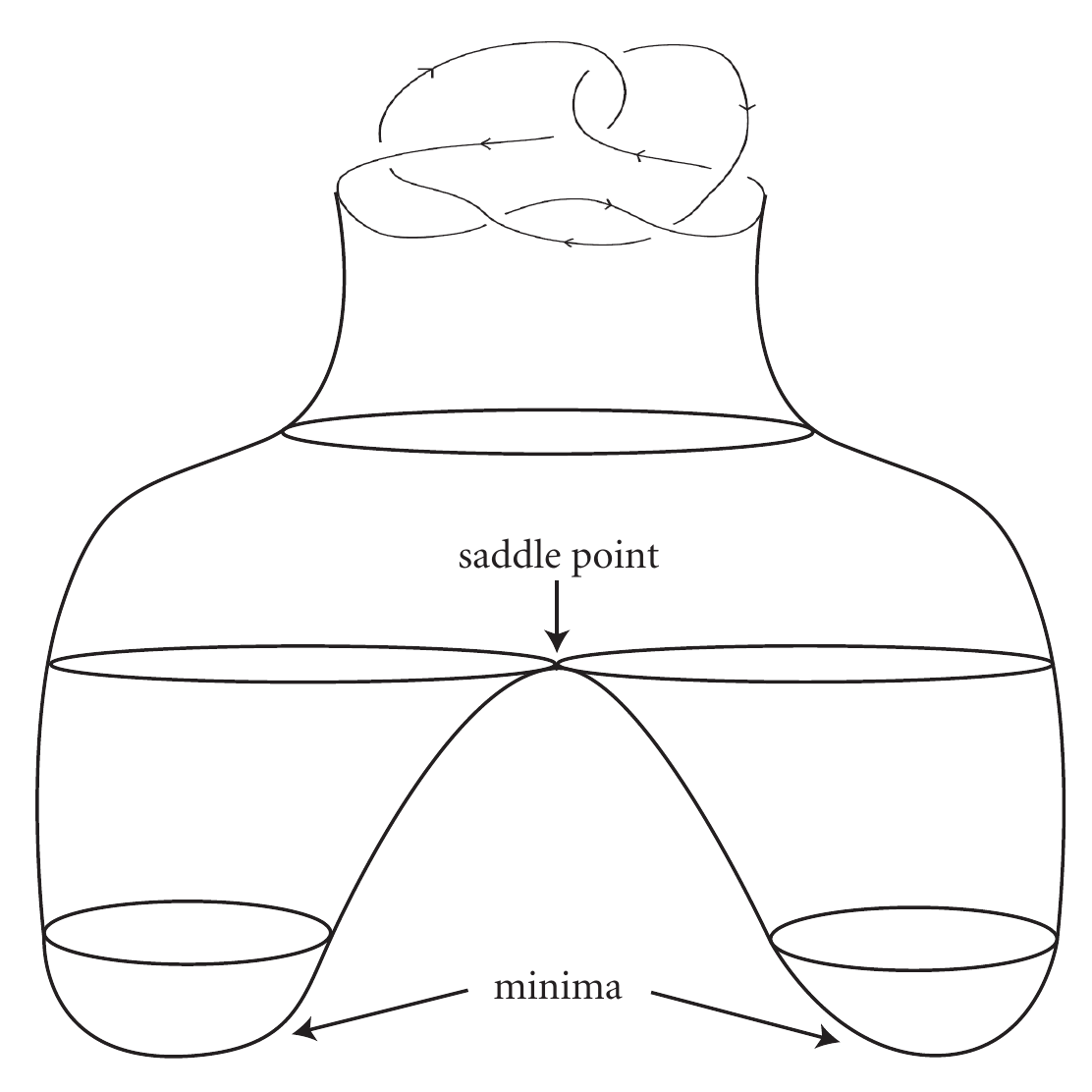}
  \end{center}
\end{example}

\begin{example}
  Another example of a slice knot is the 8-crossing knot $8_8$.  Here is the corresponding slice movie:
  \begin{center}
    \includegraphics[width=14cm]{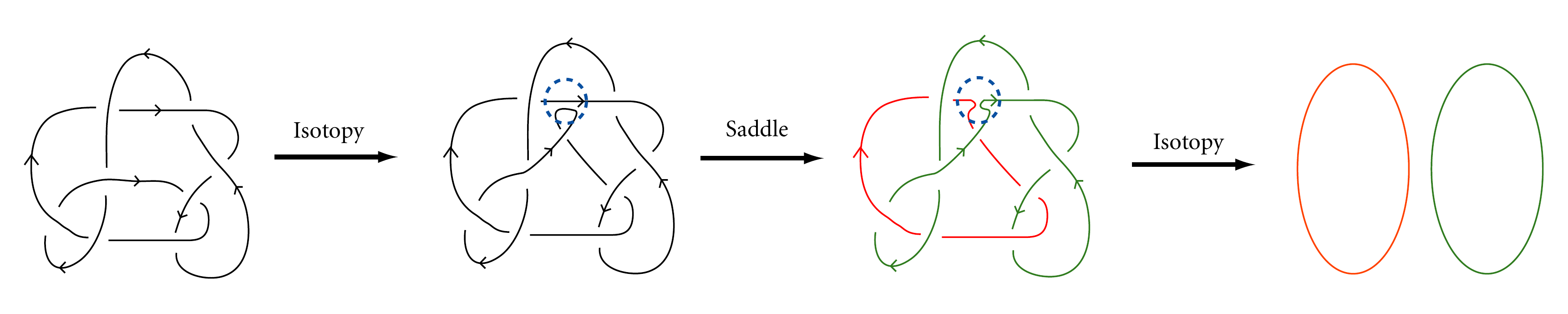}
  \end{center}
\end{example}

The important thing to notice in these examples is that after the saddle move has been made, the resulting knots are unknotted and unlinked.  If the knots after the saddle move were linked (which is what always happens when you try making a slice movie with a trefoil, for example) then we would not be able to cap them off with discs to finish making the slice disc.

\medskip

It is, in general, quite difficult to find a slice movie for a knot, even when one knows that the knot is slice.  The problem is not one which can be algorithmically implemented, since there are infinitely many places to try doing a surgery, including adding in more crossings to the knot and even tying in other slice knots. (For an example of this latter phenomenon see Example 10.1 in \cite{HeraldKirkLivingston10}, where the authors prove that $12a_{990}$ is slice by first tying in the connected sum of the trefoil with its reverse mirror image.)

\begin{remark}
 A \emph{ribbon knot} is a knot which bounds a smooth disc $D^2$ in $D^4$ such that the singularities of the ribbon disc are either minima or saddle points.  Clearly every ribbon knot is (smoothly) slice, but it is an open conjecture whether all smoothly slice knots are ribbon.  If it were so, this would make the process of looking for slice discs easier because we would not have to worry about introducing maxima into the slice movie.
\end{remark}

To prove that a knot is not slice is often much easier.  In the next section we will look at obstructions that are derived from Seifert matrices.

\section{Algebraic concordance}
\label{Sec:algConcordance}

Although only the unknot can bound a disc in $S^3$, all knots can bound some higher-genus surfaces in $3$ dimensions.  One approach to deciding whether a knot is slice is to take these surfaces and see if we can do surgery on them to reduce them to discs embedded in the $4^\text{th}$ dimension.

\subsection{Seifert surfaces and Seifert matrices}

\begin{definition}
A \emph{Seifert surface} of an oriented knot $K$ is a compact connected oriented surface whose boundary is $K$.
\end{definition}

\begin{theorem}
Every oriented link has a Seifert surface.
\end{theorem}

\begin{proof}(Seifert, 1934)
(The following method is known as \emph{Seifert's algorithm}.)  Fix an oriented projection of the link. At each crossing of the projection there are two incoming strands and two outgoing strands.  Eliminate the crossings by swapping which incoming strand is connected to which outgoing strand (see diagram below).  The result is a set of non-intersecting oriented topological circles called \emph{Seifert circles}, which, if they are nested, we imagine being at different heights perpendicular to the plane with the z-coordinate changing linearly with the nesting.  Fill in these circles, giving
discs, and connect the discs together by attaching twisted bands where the crossings used to be.  The direction of the twist corresponds to the direction of the crossing in the link.
  \begin{center}
    \includegraphics[width=8cm]{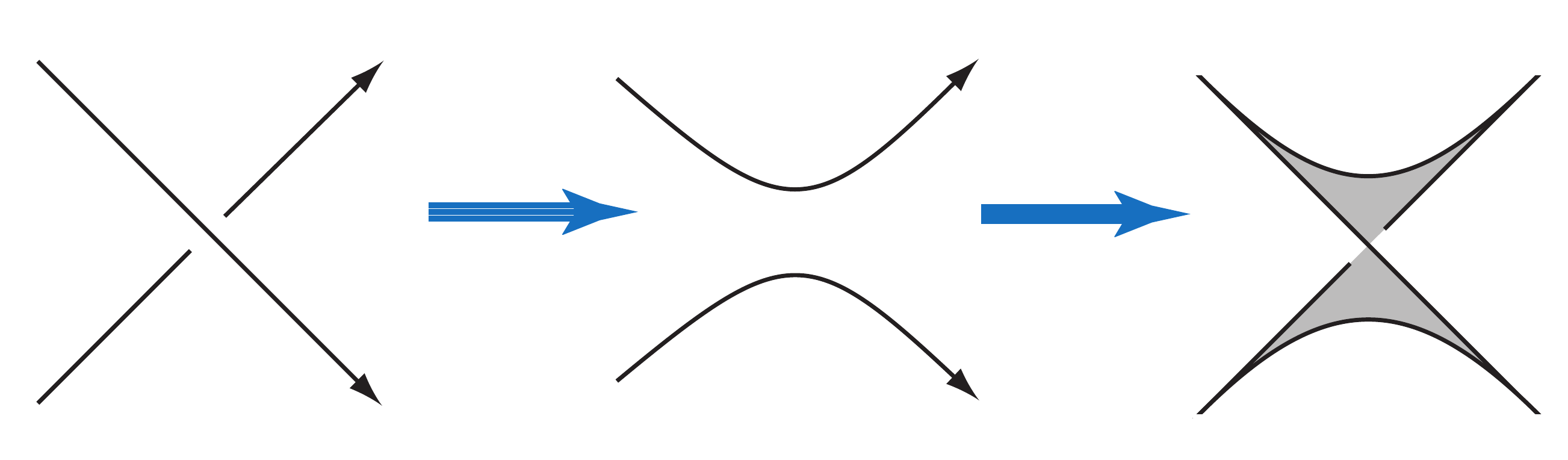}
  \end{center}
This procedure forms a surface which has the link as its boundary, and it is not hard to see that it is orientable.  If we colour the Seifert circles according to their orientation, e.g. the upward face blue for clockwise and the upward face red for anticlockwise, then the twists will consistently continue the colouring to the whole surface. According to the convention in Rolfsen~\cite{Rolfsen} we consider the red face as the `positive' side.
\end{proof}

\begin{remark}
  An excellent program called Seifertview~\cite{Wijk06} constructs a $3$-dimensional visualisation of a Seifert surface for any given knot.  The surfaces are constructed using exactly the algorithm described above. The Seifert surfaces shown in Figure \ref{Fig:knotsurface} are courtesy of Seifertview.
\end{remark}

Given a Seifert surface for a knot, we want to analyse the `holes' in this surface.  This means looking at generators of the first homology group of the surface and seeing how they interact with each other.  This interaction will be captured by the linking form.

\begin{definition}
Suppose that $D$ is a regular oriented projection of a two-component link with components $J$ and $K$.  Assign each crossing a sign:
  \begin{center}
    \includegraphics[width=4cm]{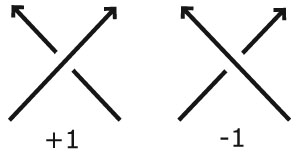}
  \end{center}
The \emph{linking number} $\lk(J,K) \in \Z$ of $J$ and $K$ is half the sum of the signs of the crossings at which one strand is from $J$ and the other is from $K$.
\end{definition}

\begin{definition}
\label{Def:linkingform}
Let $F$ be a Seifert surface for an oriented knot $K$.  We define the \emph{linking pairing} (also known as the \emph{linking form} or \emph{Seifert form})
\[ \lk \colon H_1(F) \times H_1(F) \to \Z \]
by $\lk(a,b) := \lk(a,b^+)$ where $b^+$ denotes the translation of $b$ in the positive normal direction into $S^3$.  A \emph{Seifert matrix} for $K$ is a matrix representing $\lk$ in some basis of $H_1(F)$.
\end{definition}

Since the Seifert matrix depends on the choice of Seifert surface and on the choice of a homology basis, it is not an invariant for the knot.  However, we can say how two different Seifert matrices for the same knot must be related.  Two Seifert surfaces for a knot are related by a sequence of ambient isotopies and by handle additions/removals (see \cite[Theorem 8.2]{Lickorish} for a proof).  If $F_2$ is obtained from $F_1$ by a handle addition, and if the respective Seifert matrices are $M_1$ and $M_2$ then we have
\[ M_2 = \left(\begin{array}{ccc} M_1 & v & 0 \\ 0 & 0 & 1 \\ 0 & 0 & 0  \end{array}\right)  \text{   or   }
         \left(\begin{array}{ccc} M_1 & 0 & 0 \\ w^T & 0 & 0 \\ 0 & 1 & 0  \end{array}\right) \]
for some vectors $v$ or $w$.  $M_2$ is called an \emph{elementary enlargement} of $M_1$ and $M_1$ is called an \emph{elementary reduction} of $M_2$.

\begin{definition}
  Two matrices $A$ and $B$ are called \emph{S-equivalent} if they are related by a finite sequence of elementary enlargements, reductions and by unimodular congruences (i.e. relations of the form $B=P^TAP$ with $\det P = \pm 1$).  Two knots are called \emph{S-equivalent} if they have S-equivalent Seifert matrices.
\end{definition}

\begin{lemma}\cite[Theorem 8.4]{Lickorish}
  Two Seifert matrices for a knot are S-equivalent.
\end{lemma}

Two inequivalent knots may be S-equivalent, and thus all of the invariants derived from their Seifert matrices will be identical.  This is the first clue that algebraic invariants will not be sufficient to tell us the whole story of which knots are slice.  Some of the invariants which will be important are contained in the following definition.

\begin{definition}
\label{Def:Alexpoly}
  The \emph{Alexander polynomial} of a knot $K$ with Seifert matrix $V$ is
    \[ \Delta_K(t):= \det(V-tV^T)\]
  (defined up to multiples of $\pm t^n$).  The \emph{signature} $\sigma$ is the number of positive eigenvalues minus the number of negative eigenvalues in $V+V^T$.  The \emph{$\omega$-signature} $\sigma_\omega(K)$ for a unit complex number $\omega$ is the signature of the hermitian matrix
    \[ (1-\omega)V + (1-\overline{\omega})V^T \text{ .}\]
\end{definition}

\subsection{Seifert matrices for slice knots}
\label{subsec:SeifertMatricesSlice}

To be able to use these invariants as slicing obstructions, we need to know what the Seifert matrix of a slice knot looks like.  This was determined in $1969$ by Levine~\cite[Lemma 2]{Levine69}.

\begin{theorem}
\label{Thm:lagrangian}
  If $K$ is slice, then for any Seifert surface $F$ of $K$ there exists a half-rank direct summand $L$ in $H_1(F)$ such that $V|_{L} = 0$ for $V$ a Seifert form for $F$.
\end{theorem}

The proof of this theorem is long, but it reveals a lot about the topology of the slice disc so we have included it in full.

\begin{proof}
  \textbf{Step 1:} For a slice knot $K$ with Seifert surface $F$ and slice disc $\Delta$ there exists an oriented submanifold $M^3 \subset D^4$ with boundary $F \cup \Delta$.

  \medskip

  \emph{Proof of Step 1.} [This section of the proof is taken from \cite[Lemma 8.14]{Lickorish}.]
  Let $X$ be the exterior of $K$.  We want to define a map $\phi \colon X \rightarrow S^1$ so that $\phi_* \colon H_1(X) \rightarrow H_1(S^1)$ is an isomorphism and $\phi^{-1}(\mbox{pt}) = F$.  On a product neighbourhood of $F$ in $X$, define $\phi$ to be the projection $F \times [-1,1] \rightarrow [-1,1]$ followed by the map $t \mapsto e^{i \pi t} \in S^1$.  Let $\phi$ map the remainder of $X$ to $-1 \in S^1$.

  Let $N = \Delta \times I^2$, a neighbourhood of $\Delta$.  We extend $\phi$ to the rest of $\partial (\overline{D^4 - N})$ so that the inverse image of $1 \in S^1$ is $F \cup (\Delta \times \{*\})$ for some point $* \in \partial I^2$ (note: $\partial \Delta \times \{*\}$ is a longitude of $K$).  We now need to extend the map over all of $\overline{D^4 - N}$.

  Consider the simplices of some triangulation of $\overline{D^4 - N}$.  Let $T$ be a tree in the $1$-skeleton containing all the vertices of this triangulation, that contains a similar maximal tree of $\partial (\overline{D^4 - N})$.  Extend $\phi$ over all of $T$ in an arbitrary way.  Then on a $1$-simplex $\sigma$ not in $T$, define $\phi$ so that if $c$ is a $1$-cycle consisting of $\sigma$ summed with a $1$-chain in $T$ (joining up the ends of $\sigma$), then $[\phi c] \in H_1(S^1)$ is the image of $[c]$ under the isomorphism
  \[ H_1(\overline{D^4 - N}) \xleftarrow{\cong} H_1(X) \xrightarrow{\phi_*} H_1(S^1). \]

  Trivially, the boundary of a $2$-simplex $\tau$ of $\overline{D^4 - N}$ represents zero in $H_1(\overline{D^4 - N})$, so $[\phi(\partial \tau)]=0 \in H_1(S^1)$.  Hence $\phi$ is null-homotopic on $\partial \tau$ and so extends over $\tau$.  Finally, $\phi$ extends over the $3$ \& $4$-simplices, as any map from the boundary of an $n$-simplex to $S^1$ is null-homotopic when $n \geq 3$.

  Now regard $\phi \colon \overline{D^4 - N} \rightarrow S^1$ as a simplicial map to some triangulation of $S^1$ in which $1$ is not a vertex.  Then $\phi^{-1}(1)$ is a $3$-manifold $M^3$, and $\phi$ was constructed so that $\partial M^3 = F \cup (\Delta \times *)$.

  \medskip

  \textbf{Step 2:} $P:=\ker[H_1(\partial M; \Q) \to H_1(M,\Q)]$ is a Lagrangian subspace of dimension $g$ (where $\partial M$ has genus $g$).

  \begin{definition}
  \label{Def:metaboliser}
    A \emph{Lagrangian subspace} or \emph{metaboliser} of $H_1(F)$ with respect to the linking form $\lk$  is a vector subspace $P \in H_1(F)$ which satisfies $P=P^\perp$, where
    \[ P^\perp := \left\{ x \in H_1(F) \: | \: \lk(x,y^+) = 0 \quad \forall \, y \in P \right\}\]
    This implies that $P$ has half the rank of $H_1(F)$.
  \end{definition}

  \emph{Proof of Step 2.} [Taken from the unpublished lecture notes of a course given by Peter Teichner~\cite{Teichner01}.]  Look at the homology exact sequence of $(M, \partial M)$ (over $\Q$):
  \[
    \xymatrix{
      0 \ar[r] & H_3(M,\partial M) \ar[r] & H_2(\partial M)
      \ar[r] & H_2(M) \ar[r] & H_2(M,\partial M) \ar[r] &\\
      H_1(\partial M) \ar[r] & H_1(M) \ar[r] & H_1(M,\partial
      M) \ar[r] & H_0(\partial M) \ar[r] & H_0(M) \ar[r] & 0 }
  \]

  Lefschetz duality says $H_k(M,\partial M) \cong H_{3-k}(M)$ (since $H^k(M,\partial M) \cong H_{3-k}(M)$ and $H^k(M,\partial M) \cong H_k(M,\partial M)$ by the Universal Coefficient Theorem, as we are working over a field).  Poincar\'{e} duality says $H^k(\partial M) \cong H_{2-k}(\partial M)$, which again implies $\dim(H_k(\partial M)) = \dim(H_{2-k}(\partial M))$.

  So let
  \begin{description}
    \item[\phantom{bob}] $a = \dim(H_3(M,\partial M)) = \dim(H_0(M))$,
    \item[\phantom{bob}] $b = \dim(H_2(\partial M)) = \dim(H_0(\partial M))$,
    \item[\phantom{bob}] $c = \dim(H_2(M)) = \dim(H_1(M,\partial M))$,
    \item[\phantom{bob}] $d = \dim(H_2(M,\partial M)) = \dim(H_1(M))$,
    \item[\phantom{bob}] $e = \dim(H_1(\partial M))$.
  \end{description}

  Then the exactness of the sequence implies that
\[ a - b + c - d + e - d + c - b + a = 0 \Rightarrow 2(a - b + c - d) +e = 0\]
  We also have the exact sequence
  \[ 0 \to P \to H_1(\partial M) \to H_1(M) \to H_1(M,\partial M) \to H_0(\partial M) \to H_0(M) \to 0\]
  so $\dim P = e - d + c - b + a = e + (a - b + c - d)$.

  Therefore $2(\dim P -e) + e = 0$, so  $2 \dim P = e$ and thus $\dim P = \frac{1}{2} \dim(H_1(\partial M))$.

  Now note that $H_1(\partial M) = H_1(F)$ (recall $\partial M = F \cup \Delta$).  Suppose we have $\alpha$, $\beta \in
  \ker[H_1(\partial M) \rightarrow H_1(M)]$.  There exist surfaces $A, B \subset M$ with $\partial A = \alpha$, $\partial B = \beta$.  When $\alpha$ is moved to $\alpha^+$, the surface $A$ can also be moved off $M$ to $M \times \left\{1\right\}$ (so $\partial A^+ = \alpha^+$), and then the intersection of $A^+$ and $B$ is empty. Thus $\lk(\alpha^+,\, \beta) = A^+ \cdot \, B = 0$.

   Moving from $\Q$ to $\Z$ coefficients is not a problem, although the $\Z$-kernel of $H_1(F) \rightarrow H_1(M)$ might not be a direct summand.  We use instead $L = \left\{ a \in H_1(F) \colon \exists \, n \in \Z \backslash \left\{0 \right\} \mbox{s.t. } na \in P \right\}$.  The rank of $L$ is the same as that of the kernel and it is a direct summand because $H_1(F)/L$ is torsion-free.  Note also that $V(na,b)=0$ implies $V(a,b)=0$ by linearity, so $V|_L = 0$, as required.
\end{proof}

\begin{definition}
A square matrix congruent to one of the form
  \[ \left(\begin{array}{cc} 0 & A\\ B & C\end{array}\right)\]
for square matrices $A$, $B$ and $C$ of the same size is called \emph{metabolic} or \emph{Witt trivial}.
\end{definition}

Thus our above proof has shown that any Seifert matrix for a slice knot must be unimodular congruent to a metabolic matrix.  What do the signatures and Alexander polynomials of a metabolic form look like?

\begin{corollary}
\label{Cor:sliceinvts}
  For a slice knot $K$ we have
  \begin{description}
    \item[(i)] the signature $\sigma_{\omega}(K)$ is zero for every unit complex number $\omega$ except those which are roots of the Alexander polynomial $\Delta_K(t)$. (Notice that the regular signature is $\sigma_{-1}$ so the same corollary holds for it too.)
    \item[(ii)] the Alexander polynomial $\Delta_K(t)$ is of the form $f(t)f(t^{-1})$.
  \end{description}
\end{corollary}

\begin{proof}
  Let $V$ be a Seifert matrix for $K$.  By Theorem \ref{Thm:lagrangian} we may assume
      \[ V = \left(\begin{array}{cc} 0 & A\\ B & C\end{array}\right)\]
  for square matrices $A$, $B$ and $C$ of equal size ($g \times g$ where $g$ is the genus of the associated Seifert surface) with entries in $\Z$.
  \begin{description}
    \item[(i)] Let $\omega$ be a unit complex number such that $\Delta_K(\omega) \neq 0$ and let $M= (1-\omega)V+(1-\overline{\omega})V^T$.  Notice that $\omega^{-1}= \overline{\omega}$ and that $\Delta(\omega) = \Delta(\omega^{-1})$ up to a power of $\omega$. Now
    \begin{eqnarray*}
    M & = & (1-\omega)\left(\begin{array}{cc} 0 & A\\ B & C \end{array}\right) + (1-\overline{\omega})\left(\begin{array}{cc} 0 & B^T\\ A^T & C^T \end{array}\right)\\
      & = & \left(\begin{array}{cc} 0 & (1-\omega)A + (1-\overline{\omega})B^T\\ (1-\omega)B+ (1-\overline{\omega})A^T & (1-\omega)C+(1-\overline{\omega})C^T \end{array}\right)\\
      & =: & \left(\begin{array}{cc} 0 & L\\ \overline{L}^T & N \end{array}\right)
    \end{eqnarray*}
    We can rewrite $M$ as $(1-\omega)(V-\omega^{-1}V^T)$ so $\det M = (1-\omega)^{2g}\Delta_K(\omega^{-1})$.  Since $\Delta_K(\omega) \neq 0$ we know that $M$ is non-singular.  This implies that $L$ is non-singular and therefore invertible over $\Q$.  Define
    \[ P = \left(\begin{array}{cc} L^{-1} & 0\\ (1-\omega)CL^{-1} & -I \end{array}\right) \text{ .}\]
    Then $PM\overline{P}^T = \left(\begin{array}{cc} 0 & -I\\ -I & 0 \end{array}\right)$, so $\sigma_\omega(K) = \sigma(M) = \sigma(PM\overline{P}^T) = 0$.

    \item[(ii)] We have
    \begin{eqnarray*} \Delta_K(t) & = & \det(V - tV^T)\\
    & = & \det \left(\begin{array}{cc} 0 & A - tB^T\\ B-tA^T & C-tC^T \end{array}\right)\\
    & = & \det(A - tB^T) \det(B-tA^T) \\
    & = & f(t)f(t^{-1})
    \end{eqnarray*}
    up to units in $\Z[t,t^{-1}]$.
  \end{description}
\end{proof}

\begin{example} \mbox{}
 \begin{itemize}
  \item The trefoil knot $3_1$ has Seifert matrix $V = \left(\begin{array}{cc} -1 & 1\\ 0 & -1\end{array}\right)$, so $V + V^T = \left(\begin{array}{cc} -2 & 1\\ 1 & -2\end{array}\right)$.  Eigenvalues are both negative, so $\sigma = -2$.  Thus the trefoil is our first example of a knot which is not slice.
  \item The Figure-8 knot $4_1$ has Seifert matrix $V = \left(\begin{array}{cc} 1 & 0\\ 1 & -1\end{array}\right)$ so $V + V^T = \left(\begin{array}{cc} 2 & 1\\ 1 & -2\end{array}\right)$.  Eigenvalues are $\pm \sqrt{5}$ so $\sigma = 0$.  However, the Alexander polynomial of $4_1$ is $t^2 -3t + 1$, which does not factorise as $f(t)f(t^{-1})$.  This is clear because $\Delta_{4_1}(-1)= 5$ which is not a square.  Thus $4_1$ is also not a slice knot.
 \end{itemize}
\end{example}

Why is the signature of the Figure-8 knot zero even though the knot isn't slice?  If we examine the knot a little more closely, we find that it is \emph{negative amphicheiral}: it is its own mirror image reverse.  This means that in the concordance group $2(4_1)= 4_1 \# -4_1 = 0$ and $4_1$ is an element of order 2.  The signature is an integer-valued additive invariant, so if $2\sigma = 0$ then $\sigma=0$.  Similarly, the non-vanishing of the trefoil signature proves that the trefoil is an element of infinite order in $\C$.

\subsection{The algebraic concordance group}
\label{subsec:algconcgroup}

In the same way that we can make the set of knots into a group by quotienting out the slice knots, we can make the set of Seifert forms into a group by quotienting out the metabolic forms.  First we need a definition of the ``set of Seifert forms'' which is independent of knots.

We have the following property that characterises Seifert matrices.

\begin{lemma}
\label{Lemma:unimodSeifert}
  If $V$ is a Seifert matrix for a knot $K$ then $V-V^T$ is unimodular; that is, $\det(V-V^T) = \pm 1$.
\end{lemma}

\begin{proof}
\begin{figure}
  \centering
  \includegraphics[width=12cm]{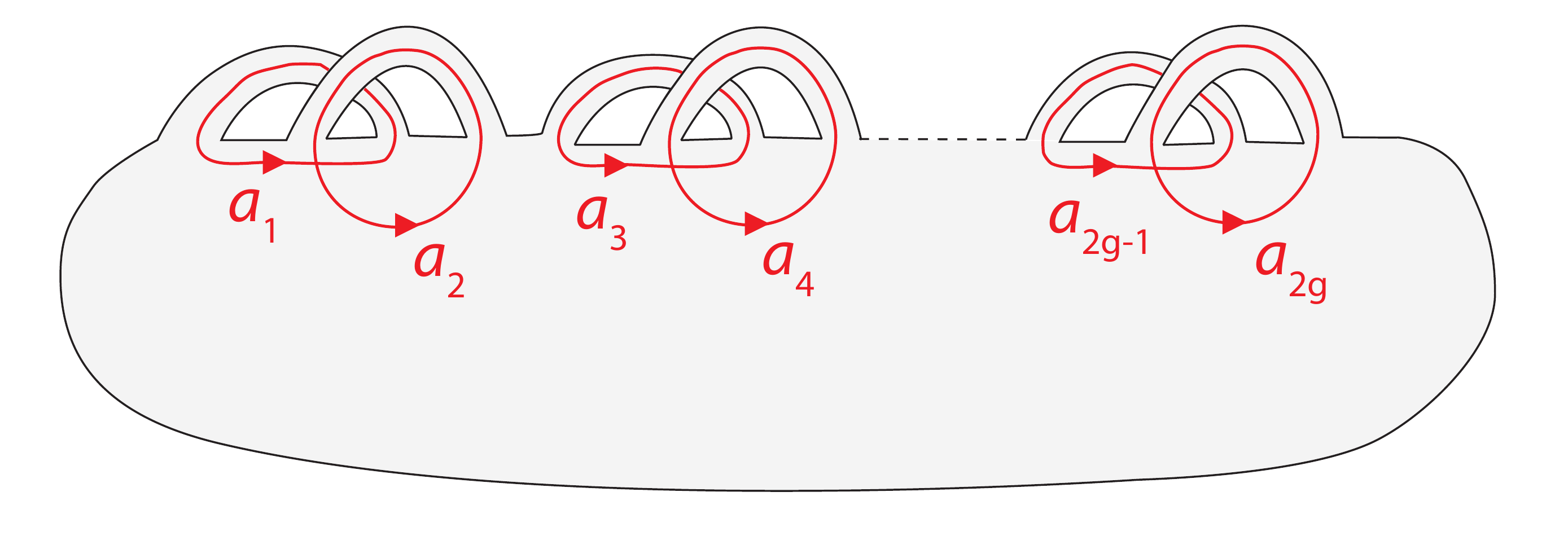}
  \caption{Every Seifert surface can be drawn as a disc with bands attached}
  \label{Fig:homologybasis}
\end{figure}
  Suppose that $V$ is a Seifert matrix obtained from a genus $g$ Seifert surface $F$ and a set of curves $\{a_1,\dots,a_{2g}\} \in H_1(F)$.  Without loss of generality, we may assume that the curves are arranged as in Figure \ref{Fig:homologybasis}, i.e. as a disc with bands attached, although the bands themselves may be twisted and knotted around each other.  The $(i,j)^{\text{th}}$ entry of $V-V^T$ is
    \[ \lk(a_i,a_j^+) - \lk(a_j,a_i^+) = \lk(a_i,a_j^+) - \lk(a_i,a_j^-) = \lk(a_i, a_j^+ - a_j^-) \text{ .}\]
  We have that $a_j^+ - a_j^-$ bounds an annulus normal to $F$, which meets $F$ in $a_j$.  The linking number of $a_j^+ - a_j^-$ with $a_i$ is then the algebraic intersection of this annulus with $a_i$, i.e. the algebraic intersection of $a_i$ with $a_j$.  We thus have that $V-V^T$ consists of
    \[ \left(\begin{array}{cc} 0 & 1 \\ -1 & 0 \end{array}\right) \]
  on the diagonal and zeros elsewhere.  It follows that this matrix is unimodular.
\end{proof}

The converse is a result of Seifert~\cite{Seifert34}; a proof may also be found in \cite[8.7]{Burde}.

\begin{lemma}
\label{Lemma:matrixIsSeifert}
  Every square matrix $V$ of even order satisfying $\det(V-V^T) = \pm 1$ is a Seifert matrix of a knot.
\end{lemma}

\begin{remark}
The proof given for Lemma \ref{Lemma:unimodSeifert} actually proves that $\det(V-V^T) = 1$, and indeed, in Lemma \ref{Lemma:matrixIsSeifert} the $\det(V-V^T)=-1$ case does not appear. This is because $V-V^T$ is skew-symmetric and every even order invertible integral skew-symmetric matrix is congruent to a block sum of the matrix $\left(\begin{array}{cc} 0 & 1 \\ -1 & 0 \end{array}\right)$. The reason we include the $-1$ case is that it streamlines theorems and arguments in higher dimensions, where the $-1$ case \emph{does} exist.
\end{remark}

Given two square matrices $V_1$ and $V_2$ we can form their block sum $V_1 \oplus V_2 = \left(\begin{array}{cc} V_1 & 0 \\ 0 & V_2 \end{array}\right)$.

\begin{definition}
  Two square matrices $V_1$ and $V_2$ are called \emph{cobordant} (or \emph{Witt equivalent}) if $V_1 \oplus (-V_2)$ is Witt trivial (metabolic).
\end{definition}

If we restrict our attention to the set of matrices which are Seifert matrices, then Levine~\cite[Lemma 1]{Levine69} shows that this is an equivalence relation.

\begin{definition}
  The \emph{algebraic concordance group} $\G$ is defined to be the set of square integral matrices $V$ satisfying $\det(V-V^T) = \pm 1$ under the operation of block sum and modulo the relation of cobordism (Witt equivalence).
\end{definition}

\begin{theorem}
  The map $\phi \colon \C \to \G$ which maps a knot to one of its Seifert matrices is an epimorphism of groups.
\end{theorem}

\begin{proof}
  If two knots $K_1$ and $K_2$ have Seifert matrices $V_1$ and $V_2$ respectively, then a Seifert matrix for $K_1 \# K_2$ is $V_1 \oplus V_2$, while a Seifert matrix for $-K_1$ is $-V_1$.  These facts, together with Theorem \ref{Thm:lagrangian} show that the map is a homomorphism, while Lemma \ref{Lemma:matrixIsSeifert} shows that it is surjective.
\end{proof}

\begin{definition}
 A knot whose image in $\G$ under the map $\phi$ is zero is called \emph{algebraically slice}.
\end{definition}

The big question of knot concordance in the 1970s was: is $\phi$ an isomorphism?  That is, are `slice' and `algebraically slice' equivalent notions? We will find the answer to this in Section \ref{Sec:CGInvariants}; for now, let us look more closely at the structure of $\G$.

\begin{theorem}[\cite{Levine69-1}]
  $\G \cong (\Z)^\infty \oplus (\Z_2)^\infty \oplus (\Z_4)^\infty \text{ .}$
\end{theorem}

Levine proved this theorem by finding a complete set of invariants for $\G$ coming from signatures and Witt groups.  We will analyse these in detail in Chapter \ref{chapter3}.

\section{Higher dimensions}
\label{Sec:higherdim}

There is an analogous notion of knot concordance in higher dimensions. The surprising result, which we will explore in this section, is that the structures of the high-dimensional concordance groups are very well understood, in stark contrast to the concordance group of $1$-dimensional knots.

\begin{definition}
  An \emph{$n$-knot} is a locally flat (or smooth) embedding of $S^n$ into $S^{n+2}$, defined up to ambient isotopy.  An $n$-knot $K$ is called \emph{slice} if it bounds a locally flat (smooth) disc $D^{n+1} \subset D^{n+3}$.  Two $n$-knots $K_1$ and $K_2$ are \emph{concordant} if $K_1 \# -K_2$ is slice, where $-K_2$ is the image of $K_2$ with reversed orientation under a reflection of $S^{n+2}$.
\end{definition}

\begin{definition}
 The $n$-dimensional concordance group $\C_n$ is the set of concordance classes of $n$-knots.
\end{definition}

As in the $1$-dimensional case, we will need Seifert surfaces as a starting point for proving that knots have slice discs.  Luckily we have the following theorem:

\begin{theorem}
  Every $n$-knot bounds some $(n+1)$-dimensional surface in $S^{n+2}$.
\end{theorem}

\begin{proof}(Sketch proof from \cite[5B1]{Rolfsen})
  Suppose $n \geq 2$ (for the $n=1$ case see Section \ref{Sec:algConcordance}).  Let $K$ be an $n$-knot and $T \cong K \times D^2$ be a tubular neighbourhood of $K$.  Define a map $f \colon T\backslash K \to S^1$, corresponding to the map $K \times (D^2 \backslash \{0\}) \to S^1$ given by $(x,y) \mapsto y/|y|$.  We would like to extend $f$ to a map of the knot exterior $F \colon X \to S^1$, where $X = S^{n+2} \backslash  \interior(T)$ and $\partial X = \partial T$.  Obstruction theory says that such an extension is possible if and only if certain elements of the cohomology groups $H^{k+1}(X,\partial X)$, with coefficients in $\pi_k(S^1)$, vanish.  If $k>1$ then the coefficient group is trivial, and if $k=1$ then we have integer coefficients and $H^2(X,\partial X) \cong H_n(X) = 0$ by Lefschetz and Alexander dualities. So all the obstructions vanish and there is a map $F \colon X \to S^1$ extending $f$.  Choose a regular point $x \in S^1$ (we may assume that we are either in the PL or smooth category); then $F^{-1}(x)$ is the required orientable surface of codimension $1$.
\end{proof}

The first thing that was proven about the higher dimensional concordance groups was that the even dimensional groups were zero.

\begin{theorem}[\cite{Kervaire65}]
 $\C_{2k} = 0$ for $k \geq 0$.
\end{theorem}

\begin{proof}
  Let $K$ be an $2k$-knot and let $F$ be a Seifert surface for $K$.  Perform ambient surgeries on $F$ below the middle dimension to turn $F$ into a slice disc $D^{2k+1} \subset D^{2k+3}$.
\end{proof}

For odd dimensions, the Seifert surfaces are even-dimensional so we have to worry about what happens in the middle dimension.  In the middle dimension is the linking form and it turns out that if this linking form has a Lagrangian then the knot is slice.  So the algebraic concordance group tells us about the geometry of slice knots.

The reason this works in dimensions above $1$ is the existence of the \emph{Whitney trick}.  When the ambient dimension is above $4$, the Whitney trick allows intersection points of opposite sign to be cancelled with each other.  If two $n$-knots have linking number zero then this means we can arrange for them to be actually disjoint.

\begin{theorem}[Whitney trick]
  Let $X^d$ be an oriented manifold of dimension $d\geq5$, and let $M^k$ and $N^{d-k}$ be oriented submanifolds of codimension at least two.  Suppose that either
  \begin{itemize}
    \item $k \geq 3$, $d-k \geq 3$, $\pi_1(X) = 0$, or
    \item $k=2$, $d-k \geq 3$, $\pi_1(X \backslash N) = 0$
  \end{itemize}
  Then pairs of intersections of $M$ and $N$ of opposite sign may be removed by isotopies of $M$ and $N$.
\end{theorem}

\begin{proof}(Sketch proof from \cite{Scorpan})
  By a general position argument, we may assume that intersections of $M$ and $N$ are transverse. Since $M$ and $N$ have complementary dimensions, they intersect in a collection of isolated points.  Each of these points has a sign coming from the orientations of $M$ and $N$.  Pick a pair of points $p$ and $q$ of opposite sign.  Draw a path linking $p$ and $q$ that lies inside $M$, and another path linking $p$ and $q$ that lies inside $N$.  Together, these two curves form a circle.  By the simple-connectedness of $X \backslash (M \cup N)$, this circle is homotopically trivial and therefore bounds an immersed disc in $X \backslash (M \cup N)$.  The weak Whitney embedding theorem tells us that immersions of discs in manifolds of dimension at least 5 can be approximated by embeddings.  Thus we have an embedded disc (called the \emph{Whitney disc}) with boundary in $M \cup N$ (Figure \ref{Fig:Whitney}, left).  Use this disc to construct an isotopy which pushes $M$ past $N$ until the intersections disappear (Figure \ref{Fig:Whitney}, right); this is possible without introducing new intersection points because of the opposite signs of $p$ and $q$.
  \begin{figure}
    \begin{center}
      \includegraphics[width=10cm]{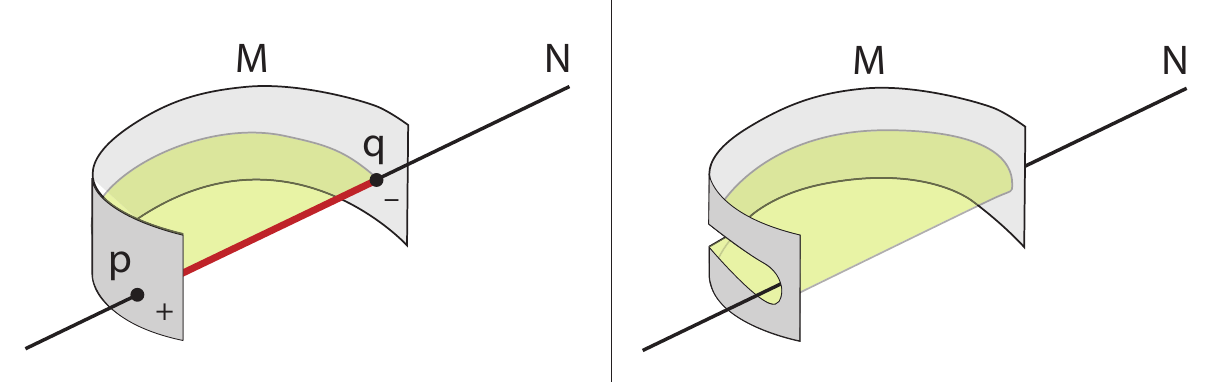}
    \end{center}
    \caption{Removing intersection points using the Whitney trick}
    \label{Fig:Whitney}
  \end{figure}
\end{proof}

It is Whitney's embedding theorem which fails in dimension 4; we cannot guarantee that the Whitney disc is an embedding rather than an immersion.

The group structure of the odd-dimensional concordance groups is $(\Z)^\infty \oplus (\Z_2)^\infty \oplus (\Z_4)^\infty$, just as for the one-dimensional algebraic concordance group.  However, the isomorphism is not to $\G$, as defined in Section \ref{subsec:algconcgroup}, but to a slight refinement of it.

\begin{definition}
  $\G_{\pm}$ is defined to be the set of matrices $V$ for which $V \pm V^T$ is unimodular, under the operation of block sum and modulo the equivalence relation of cobordism of matrices.  The group $\G_+^0$ is the subgroup of $\G_+$ of index 2, defined by matrices with property $+$ and signature $V+V^T$ a multiple of $16$.
\end{definition}

\begin{theorem}[\cite{Levine69}] \mbox{}
  \begin{itemize}
    \item $\C_{4k+1} \cong \G_{-}$ for $k \geq 1$
    \item $\C_{4k+3} \cong \G_{+}$ for $k \geq 1$
    \item $\C_3 \cong \G_+^0$.
  \end{itemize}
\end{theorem}

\section{Casson-Gordon invariants}
\label{Sec:CGInvariants}

In the 1970s Casson and Gordon devised a new invariant for the knot concordance group.  It is derived from the Atiyah-Singer index theorem and looks at the difference of two intersection forms, one of which has coefficients twisted by a representation of the knot group.  By calculating their new invariant for a family of knots known as the `$n$-twisted doubles of the unknot' (see Definition \ref{Def:twistKnot}) they proved the following theorem (which had been suspected to be true for some time).

\begin{theorem}
  The kernel of the homomorphism $\C \to \G$ is non-trivial.
\end{theorem}

In other words, Casson and Gordon were able to use their invariant to prove that some algebraically slice knots were not geometrically slice.

Here is how their invariant was constructed.  We give the construction first for a general $3$-manifold and then describe the particular manifold that will be used for obstructing the sliceness of knots. In what follows, $W(\CC(t),\I)$ is the Witt group of $\CC(t)$ consisting of (equivalence classes of) pairs $(V,\beta)$ where $V$ is a finite-dimensional vector space over $\CC(t)$ and $\beta$ is a hermitian inner product (i.e. complex-valued bilinear form) satisfying $\beta(av,bw) = a\I(b)\beta(v,w)$ and $\beta(v,w) = \I(\beta(w,v))$.  Here $\I$ is the involution that sends $\sum a_i t^i$ to $\sum \overline{a_i} t^{-i}$ with $\overline{a_i}$ being complex conjugation. (For more detail on Witt groups, see Section \ref{Subsec:background}.)

Let $(M,\rho)$ be a $3$-manifold together with a homomorphism $\rho \colon \pi_1(M) \to \Z_d \oplus \Z$, where $d$ is an odd prime power.  The $3$-dimensional bordism group $\Omega_3(\Z_d \oplus \Z)$ is finite; in fact, $d$-torsion (we will see more details in the next section).  Thus $d(M,\rho) = \partial(W,\overline{\rho})$ for some $4$-manifold $W$ and map $\overline{\rho} \colon \pi_1(W) \to \Z_d \oplus \Z$.
The manifold $W$ has a non-singular hermitian $\CC(t)$-intersection form $I(W,\overline{\rho}) \in W(\CC(t),\I)$, where the local coefficients are twisted using the map $\overline{\rho}$.  The $\Z$-action is multiplication by $t$ and the $\Z_d$-action is multiplication by $e^{2\pi i/d}$.

The manifold $W$ also has an ordinary (untwisted) intersection form $I(W) \in \text{Image } \{W(\Q) \in W(\CC(t),\I)\}$. (If this is singular, take the quotient of this form by its kernel.)  The Casson-Gordon invariant is defined to be
 \[ \tau(M,\rho) = \frac{1}{d}(I(W,\overline{\rho}) - I(W)) \in W(\CC(t),\I) \otimes \Q \]
and is independent of $W$.  Since $d$ is odd, $\tau$ actually lives in $W(\CC(t),\I) \otimes \Z_{(2)}$, where $\Z_{(2)}$ is $\Z$ localised at $2$, i.e. the set of rational numbers with odd denominator.

Let $X$ be the complement $S^3 \backslash K$ of a knot $K$.  Let $M_0$ be the $3$-manifold constructed by $0$-surgery on $X$ and $M^k$ be the $k$-fold cyclic cover of $M_0$, where $k$ is a power of a prime $p^r$.
There is a natural surjection $\eps \colon \pi_1(M^k) \to \Z$ via the covering projection of $M^k$ to $X$ and the map of $H_1(X)$ to $\Z$ determined by the orientation of $K$.  Given a surjective homomorphism $\chi \colon H_1(\Sigma_k;\Z) \to \Z_d$ (where $\Sigma_k$ is the $k$-fold branched cover of $S^3$ over $K$) we can define a representation
 \[ \rho_{\chi} \colon \pi_1(M^k) \to H_1(M^k;\Z) \to H_1(\Sigma_k;\Z) \xrightarrow{\chi} \Z_d \]
where the first map is the Hurewicz map.

The Casson-Gordon invariant which is then a slicing obstruction for $K$ is $\tau(M^{k}, \eps \times \rho_{\chi})$.  This may also be denoted $\tau(K,p^r,\chi)$.

\subsection{Casson-Gordon signatures}
\label{Subsec:CGsig}
For each class $I \in W(\CC(t),\I)$ and for each unit complex number $\omega \in S^1$, a signature $\sigma_{\omega}(I)$ is defined by evaluating a representative of that class at $\omega$ and computing the signature of the resulting Hermitian matrix.  If the representative is singular at $\omega$ then the signature is the average of the one-sided limits, i.e. $\displaystyle \lim_{x \to \omega^+}\sigma_{x}(I)$ and $\displaystyle \lim_{x \to \omega^-}\sigma_{x}(I)$.

Each $\sigma_{\omega}$ is a homomorphism that can be extended to
 \[ \sigma_{\omega} \colon W(\CC(t),\I) \otimes \Z_{(2)} \to \Z_{(2)} \]
in the obvious way.

\begin{definition}
 The \emph{Casson-Gordon signature} of a knot $K$ together with a map \newline
$\chi \colon H_1(\Sigma_k;\Z) \to \Z_d$ (where $d$ and $k$ are odd prime powers) is defined to be
   \[ \sigma_1(\tau(K,k,\chi)). \]
We shall abbreviate this to $\sigma(K,k,\chi)$.
\end{definition}

(What follows is taken from \cite{LivingstonNaik99}, Section 4.) The Casson-Gordon signature can be viewed as a map from the bordism group $\Omega_3(\Z_d \oplus \Z)$ to $\Z_{(2)}$. If two elements $(M_1,\rho_1)$ and $(M_2,\rho_2)$ represent the same element in $\Omega_3(\Z_d \oplus \Z)$ then $\tau(M_1,\rho_1)$ and $\tau(M_2,\rho_2)$ differ by an element of $W(\CC(t),\I)$. Thus the difference $\sigma(M_1,\rho_1) - \sigma(M_2,\rho_2)$ is an integer, and we can see $\sigma$ as a homomorphism from $\Omega_3(\Z_d \oplus \Z) \to \Q / \Z$. Indeed, by the definition of the Casson-Gordon invariant, $\sigma(K,k,\chi)$ takes values in $(\frac{1}{d})\Z/\Z \cong \Z_d$. Casson and Gordon \cite{CassonGordon86} show that $\sigma \colon \Omega_3(\Z_d \oplus \Z) \to \Z_d$ is actually an isomorphism.

We have further isomorphisms:
 \[ \Omega_3(\Z_d \oplus \Z) \cong \Omega_3(\Z_d) \cong \Z_d \]
where the first is given by the projection and inclusion maps, whilst the second is as follows. For a pair $(M,\rho)$ with a character $\rho \colon H_1(M) \to \Z_d$, we can always find some $x \in \text{torsion}(H_1(M))$ so that $\rho(y) = \lk(x,y)$ for all $y \in \text{torsion}(H_1(M))$. The image of $(M,\rho)$ in $\Z_d$ is then $\lk(x,x)$.

Putting these isomorphisms together, we obtain a straightforward generalisation of Theorem 2.5 of \cite{LivingstonNaik01}:
\begin{theorem}
 If $\chi \colon H_1(\Sigma_{p^r}) \to \Z_{q^s}$ is a character obtained by linking with the element $x \in H_1(\Sigma_{p^r})$, then $\sigma(K,p^r,\chi) \equiv \lk(x,x)$ mod $\Z$.
\end{theorem}

So although it is difficult to calculate Casson-Gordon signatures precisely, we can at least use this theorem to decide if the signature cannot be zero, and thus to decide if a knot cannot be slice.

\subsection{Casson-Gordon discriminants}
\label{Subsec:CGdisc}

We would like to turn the determinant of a class of $W(\CC(t),\I)$ into an invariant of that class, but to do so we will need to work modulo whatever the determinant is of the zero class.  What is the determinant of a metabolic (Witt trivial) form?

An element representing zero in $W(\CC(t),\I)$ has the form
 \[ \left(\begin{array}{cc} 0 & A \\ \I(A^T) & B \end{array}\right)\]
and if $A$ has dimension $n$ then the determinant of this is
\[ (-1)^n \det(A) \det(\I(A^T)) = (-1)^n \det(A) \I(\det(A)) \text{ .}\]

\begin{definition}
\label{Def:norm}
  A polynomial of the form $f \I(f)$ will be called a \emph{norm}.
\end{definition}

Let $N \subset \CC(t)^*$ denote the multiplicative subgroup generated by norms.

\begin{definition}
  The \emph{discriminant} of a class $I \in W(\CC(t),\I)$ is the determinant of a representative $A$ of $I$, considered modulo norms:
  \[ \disc(I) := \det(A) \in \CC(t)^* / N \text{ .}\]
\end{definition}

We would like to extend this to a map from $W(\CC(t),\I) \otimes \Z_{(2)}$.  Notice that, since matrices representing classes in $W(\CC(t),\I)$ are Hermitian, every discriminant has the property that $\det(A) = \I(\det(A))$.  This means that $\disc(cI)$ is $\disc(I)$ if $c$ is odd and is $0$ if $c$ is even.  This allows us to extend our definition of the discriminant to a function from $W(\CC(t),\I) \otimes \Z_{(2)}$ as follows:
 \[ \disc\left(I \otimes \frac{p}{q}\right) = \disc(I)^p = \disc(pI) \text{ .} \]

The Casson-Gordon discriminant of a knot, $\disc(\tau(K,p^r,\chi))$ will be the main tool of this thesis.  This is because it is equivalent to the \emph{twisted Alexander polynomial}, which we will define in Chapter \ref{chapter5} and which requires no $4$-manifold constructions.  This makes it relatively easy to compute compared with the Casson-Gordon signature and with the Casson-Gordon invariant itself. 
\chapter{\label{chapter3} Finding the algebraic concordance class of a knot}

In this chapter we delve more deeply into the algebraic concordance group $\G$ of Levine~\cite{Levine69-1, Levine69} and find out what invariants are needed to find the image of a knot in $\G$.  The structure of this chapter owes much to the excellent survey article on algebraic concordance by Charles Livingston~\cite{Livingston08}.

\section{Symmetric bilinear forms and Witt groups}
\label{Subsec:background}

We start with defining everything in full generality, so let $R$ be a commutative ring with identity.  The following definitions may be found in \cite{MilnorHusemoller} and \cite{Scharlau}.

\begin{definition}
A \emph{symmetric bilinear form} on an $R$-module $M$ is a function
\[ b \colon M \times M \to R \]
so that $b(*,y)$ and $b(x,*)$ are linear as functions of fixed $y$ and $x$ respectively, and so that $b(x,y)=b(y,x)$ for all $x,y \in M$.
\end{definition}

\begin{remark}
It is worth mentioning the cousin of the symmetric bilinear form: the quadratic form. A \emph{quadratic form} on $M$ is a function $q\colon M \to R$ such that $q(\alpha x) = \alpha^2 x$ for all $\alpha \in R$ and such that $b_q(x,y):=q(x+y)-q(x)-q(y)$ is a symmetric bilinear form. Notice that, given any symmetric bilinear form $b$, we can define the quadratic form $q_b(x):=b(x,x)$.  Conversely, given a quadratic form $q$ we can define a symmetric bilinear form $b_q(x,y):= \frac{1}{2}(q(x+y)-q(x)-q(y))$ so long as $2$ is not a zero-divisor in $R$. Thus the theory of quadratic forms is equivalent to the theory of symmetric bilinear forms over rings in which $2$ is a unit.
\end{remark}

If $M$ is a finitely generated free $R$-module with basis $e_1, \dots, e_n$, then any bilinear form $b$ on $M$ can be written as a matrix $B$, where
\[ B_{ij}=b(e_i,e_j). \]
The integer $n$ is called the \emph{dimension} of $M$, and $b$ is called \emph{nonsingular} if $\det(B) \neq 0$.

\begin{definition}
Let $M$ be a finitely generated free module over $R$ and $b$ a non-singular symmetric bilinear form on $M$.  We say that $(M,b)$ is \emph{Witt trivial} if $M \cong P \oplus N$ with
\[ P = P^\bot := \{x \in M \,\, | \,\, b(x,y) = 0 \,\, \forall y \in P \}. \] In particular, this means that $M$ has dimension $2g$ for some $g$, and $P$ has dimension $g$.\footnote{This follows from the exact sequence \[ 0 \to P^\bot \to M \to \Hom(P,R) \to 0,\] which tells us that $\dim(M) = \dim (P^\bot) + \dim(\Hom(P,R)) = 2\dim(P)$.}
$P$ is called a \emph{metaboliser} for $(M,b)$.  The forms $(M_1, b_1)$ and $(M_2,b_2)$ are called \emph{Witt equivalent} if $(M_1 \oplus M_2, b_1 \oplus -b_2)\oplus (M',b')$ is Witt trivial, where $(M',b')$ is some Witt trivial form.
\end{definition}

\begin{definition}
The Witt group $W(R)$ consists of pairs $(M,b)$ under the operation of direct sum and under the equivalence relation of Witt equivalence defined above. We may also make $W(R)$ into a commutative ring by defining multiplication as tensor product.
\end{definition}

\begin{theorem}\cite[6.4]{Scharlau}
\label{Thm:diagonalise}
If $R$ is a ring in which $2$ is a unit then any symmetric bilinear form $(M,b)$ can be diagonalised over $M$.  In other words, there is a basis $e_1, \dots, e_n$ of $M$ so that $b(e_i,e_j) = 0$ for $i \neq j$. We write $b$ in the matrix form $[d_1, \dots, d_n]$ where the $d_i$ are the diagonal entries.
\end{theorem}

\begin{remark}
\label{Rmk:diagonalSquare}
Notice that if $e_i$ is replaced with $\alpha e_i$, $\alpha \in R$, then the new diagonal form is $[d_1, \dots, \alpha^2 d_i, \dots, d_n]$.
\end{remark}

\section{The algebraic concordance group}
\label{Sec:AlgConcGroup}

The algebraic knot concordance group $\G$, as described in Chapter \ref{chapter2}, was developed by Levine~\cite{Levine69} to classify knots within the higher-dimensional geometric concordance groups $\C_n$ and to investigate the structure of $\C_1$.  We recap some of the definitions using our new terminology.

\begin{definition}
\label{Def:cobordism}
The \emph{algebraic concordance group $\G$} (or $\G^{\Z}$) consists of the set of integral square matrices $A$ satisfying $\det(A-A^T)=\pm 1$ up to Witt equivalence.
\end{definition}

\begin{definition}
The \emph{Alexander polynomial} of $A \in \G$ is $\Delta_A(t):=\det(A-tA^T)$, defined up to multiples of $\pm t$.
\end{definition}

Life becomes easier if we work with rational matrices rather than integral ones, so let us consider the group $\G^{\Q}$: square matrices $A$ with entries in $\Q$ satisfying $(A-A^T)(A+A^T)$ is non-singular, with the same equivalence relation as in $\G$.  The inclusion $\G \to \G^{\Q}$ is injective \cite[Section 3]{Levine69-1} and so we can try to find invariants in $\G^\Q$ rather than in $\G$.

Henceforth $M$ will be a finite dimensional vector space over $\Q$.

\begin{definition}
  An \emph{isometric structure} is a triple $(M,Q,T)$ where $Q$ is a nonsingular symmetric bilinear form on $M$ and $T$ is an isometry of $M$.  This means that $Q(Tx,Ty)=Q(x,y)$ for every $x,y \in M$.
\end{definition}

\begin{definition}
An isometric structure $(M,Q,T)$ is \emph{null-cobordant} if $(M,Q)$ has a metaboliser $P$ which is invariant under $T$.  Two isometric structures $(M_1, Q_1, T_1)$, $(M_2, Q_2, T_2)$ are \emph{cobordant} if $(M_1 \oplus M_2 ,Q_1 \oplus -Q_2, T_1 \oplus T_2)$ is null-cobordant.
\end{definition}

We define $\G_{\Q}$ to be the group of cobordism classes of isometric structures $(M,Q,T)$ satisfying $\Delta_T(1)\Delta_T(-1)\neq 0$, where $\Delta_T$ is the characteristic polynomial of $T$.

\begin{theorem}\cite{Levine69-1}
There is an isomorphism $\G^{\Q} \cong \G_{\Q}$ given by $A \mapsto (A+A^T, A^{-1}A^T)$.
\end{theorem}

\begin{proof}
Let $A \in \G^{\Q}$ with $A$ non-singular.  (Levine~\cite[Lemma 8]{Levine69-1} proves that every matrix in $\G^{\Q}$ is equivalent to a non-singular one.)  Define $P:=A^{-1}A^T$ and $Q:=A+A^T$.  It is easy to check that $P^TQP = Q$ and that the congruence class of $A$ determines the congruence class of $Q$ and the similarity class of $P$.  Thus $(Q,P)$ is an isometric structure, which is null cobordant whenever $A$ is Witt trivial.  We need to check that $\Delta_P(1)\Delta_P(-1)\neq 0$, but we know that
\begin{eqnarray*}
  \Delta_P(t)& = & \det(P-tI) \\
             & = & \det(A^{-1}A^T - tI) \\
             & = & \det(A^{-1})\det(A^T-tA)\\
             & = & \det(A^{-1})\det(A-tA^T)\\
             & = & c \cdot \Delta_A(t)
\end{eqnarray*}
where $c \in \Q$, and since $A \in \G^{\Q}$ we know that $(A-A^T)(A+A^T)$ is non-singular, so $\Delta_A(1)\Delta_A(-1)\neq 0$.

We now need an inverse function from $\G_{\Q}$ to $\G^{\Q}$.  Suppose we have computed the map above $A \mapsto (Q,P)$.  Then $Q=A(I+A^{-1}A^T)=A(I+P)$ so we can recover $A$ from $Q(I+P)^{-1}$.  We have
\[ A^T = Q(I+P^{-1})^{-1} =  Q(I+P)^{-1}P\]
 so
\[ A-A^T=Q(I+P)^{-1}(I-P).\]
Since $\Delta_P(1)\Delta_P(-1)\neq 0$ we have that $A-A^T$ is non-singular.
\end{proof}

We can also define cobordism classes of isometric structures over different fields.  For a field $\F$ we will use the notation $\G_\F$. The idea now is to break down the problem of finding the image of a class in $\G_{\Q}$ into the problem of finding the image of a class in $\G_{\F}$ where $\F$ is `simpler'.  And even these $\G_{\F}$ groups can be broken down further by restricting to isometric structures with a particular characteristic polynomial.  This is what motivates the next definition.



\begin{definition}
For a polynomial $f \in \F[t]$, let $\G_{\F}^f$ denote the Witt group of isometric structures $(M,Q,T)$ where $M$ is a finite-dimensional vector space over $\F$, $\Delta_T$ is a power of $f$, and $\Delta_T(1)\Delta_T(-1) \neq 0$.
\end{definition}

We will need the next lemma in order to see how to factorise characteristic polynomials.

\begin{lemma}
\label{Lemma:polyfactorise}
For any isometric structure $(M,Q,T) \in \G_{\F}$, the characteristic polynomial $\Delta_T(t)$ is symmetric, i.e. it satisfies $\Delta_T(t) = at^d \Delta_T(t^{-1})$ where $d$ is the degree of the polynomial and $a \in \F$.
\end{lemma}

\begin{proof}
Let $P$ be a matrix representative of $T$.  By definition, $Q = P^TQP$.  We have
\begin{eqnarray*}
  \Delta_T(t) & := & \det(P-tI) = \det(P^T - tI) = \det(QP^{-1}Q^{-1}-tI)\\
  & = & \det(P^{-1}-tI) = \det(-tP^{-1}(P-t^{-1}I))\\
  & = & t^d\det(-P^{-1})\Delta_T(t^{-1}).
\end{eqnarray*}
\end{proof}

In fact, it is easy to see that if $\Delta_T(1) \neq 0$ then $a$ must be equal to $1$ and if furthermore $\Delta_T(-1) \neq 0$ then $d$ must be even.

\medskip

Lemma \ref{Lemma:polyfactorise} tells us that if $\Delta_T$ factorises then it must do so as $\prod_i \delta_i^{k_i}\prod_j g_j^{l_j}$, where the $\delta_i$ are distinct irreducible symmetric polynomials and the $g_j$ are non-symmetric irreducible factors that appear in pairs $g_j(t)$ and $g_j(t^{-1})$.  The next lemma, proved by Milnor~\cite{Milnor69} and Levine~\cite{Levine69-1}, tells us that we need only worry about the symmetric factors.

\begin{theorem}
$\G_{\F} \cong \oplus_\delta \G_{\F}^{\delta}$ where the sum is over all irreducible symmetric polynomials.
\end{theorem}

\begin{proof}[Sketch proof]

Suppose $(M,Q,T)$ is an isometric structure over $\F$.  Consider the vector space on which $(M,Q)$ and $T$ are defined as $\F[t,t^{-1}]$-modules, defining the action of $t$ by $T$.  For each irreducible factor $\lambda(t)$ of $\Delta_T(t)$ define $V_{\lambda}$ to be
  \[ V_{\lambda} := \ker(\lambda(t)^N), \quad \text{for $N$ large.}\]
(More specifically, we need $N$ to be at least the multiplicity of $\lambda$ as a factor of $\Delta_T$.) Then our vector space $V$ is the direct sum of the $\{V_{\lambda}\}$.

We want to show that if $\lambda$ and $\mu$ are irreducible factors of $\Delta_T$ with $\lambda(t)$ and $\mu(t^{-1})$ relatively coprime, then $V_{\lambda}$ and $V_{\mu}$ are orthogonal.  We start with the identity
  \[ \langle m, \mu(t^{-1})^N v \rangle = \langle \mu(t)^N m, v \rangle = \langle 0,v \rangle \]
for $m \in V_{\mu}$.  This shows that $V_{\mu}$ is orthogonal to $\mu(t^{-1})^N V$.

If $\lambda(t)$ and $\mu(t^{-1})$ are coprime then the map $\phi \colon V \to V$ defined by $\phi(v) = \mu(t^{-1})^N v$ maps the subspace $V_{\lambda}$ isomorphically onto itself.  This completes the proof that $V_{\lambda}$ and $V_{\mu}$ are orthogonal.

We have already seen that $\Delta_T$ factorises as $\prod_i \delta_i^{k_i}\prod_j g_j^{l_j}$, where the $\delta_i$ are distinct irreducible symmetric polynomials and the $g_j$ are non-symmetric irreducible factors that appear in pairs $g_j(t)$ and $\overline{g_j(t)}:= g_j(t^{-1})$.  We have
\[ V = \bigoplus_i V_{\delta_i} \oplus \bigoplus_j (V_{g_j} \oplus V_{\overline{g_j}})\]
and the restriction of $(M,Q,T)$ to each of these summands gives an isometric structure.  What we have shown in the earlier part of the proof is that the factors $V_{g_j} \oplus V_{\overline{g_j}}$ are null-cobordant.  It follows that $(M,Q,T)$ is null-cobordant if and only if its restriction to each $V_{\delta_i}$ is null-cobordant.
\end{proof}

We said earlier that we would like to break down the study of a class in $\G_{\Q}$ into the study of classes in $\G_{\F}$ where $\F$ is simpler.  The following result of Levine~\cite[17]{Levine69-1} shows us which fields to consider.

\begin{theorem}
  An isometric structure over a global field $\F$ is null-cobordant if and only if the extension over every completion of $\F$ is null-cobordant.
\end{theorem}

What are the completions of the global field $\Q$?  It turns out that these are the real numbers $\R$ and the $p$-adic rationals $\Q_p$.  We will learn more about these in the next section, but first we shall bring together the theorems in this section, together with another result of Levine~\cite[16]{Levine69-1}, for a definitive guide as to when an isometric structure in $\G_{\Q}$ is trivial.

\begin{theorem}
\label{Thm:trivialclass}
\mbox{}
  \begin{itemize}
    \item A class $(M,Q,T) \in \G_{\Q}$ is trivial if and only if it is trivial in $\G_{\F}^{\delta}$ for every $\delta$ an irreducible symmetric factor of $\Delta_T$ and for $\F=\R$ and $\F=\Q_p$ for every prime $p$.
    \item A class $(M,Q,T) \in \G_{\F}^{\delta}$, where $\F=\R$ or $\F=\Q_p$ and $\delta$ is irreducible symmetric, is trivial if and only if $\Delta_T(t)$ is $\delta^e$ with $e$ even and $(M,Q)$ is trivial in the Witt group of $\F$, $W(\F)$.
  \end{itemize}
\end{theorem}

The next section will focus on understanding these Witt groups.

\section{The Witt groups $W(\R)$ and $W(\Q_p)$}
\label{Sec:WittGroups}

We begin this section with a short discussion of the $p$-adic numbers.  Given a prime $p$, any integer $n$ may be written as $n=a_0 + a_1 p + a_2 p^2 + \dots + a_k p^k$ for some $k\in \N$ and $a_i \in \F_p$, where $\F_p$ denotes the finite field of $p$ elements.  The $p$-adic integers, denoted $\Zp$, are defined to be numbers of the form
 \[ \sum_{i=0}^{\infty} a_i p^i, \quad a_i \in \F_p \]
 and the $p$-adic rationals, denoted $\Q_p$, are the field of fractions of this ring.  A $p$-adic rational may be written as
 \[ \sum_{i=-k}^{\infty} a_i p^i, \quad a_i \in \F_p \]
for some $k \in \N$.

\begin{example}
  Let us look at some elements of $\Q_5$.  We have
   \[ \frac{1}{2} = 3 + \sum_{i=1}^\infty 2(5^i) =: (3,2,2,2,\dots) \]
  because multiplying both sides by $2$ gives $(1,0,0,\dots)$ on each side.  Thus $\frac{1}{2}$ is a $5$-adic integer.

  The number $-1$ is written in $\Q_5$ as $(4,4,4,\dots)$ since this is the unique number which, when added to $1$, makes zero.  This is also a $5$-adic integer.
\end{example}

Notice that any element of $\Zp$ with $a_0=0$ cannot have a multiplicative inverse in $\Zp$.  The group of units of $\Zp$, denoted $\Zp^*$, are those elements with $a_0 \neq 0$. Every element of $\Q_p$ can be written as $p^n u$ with $u \in \Zp^*$ and $n \in \Z$.

\begin{remark}
  In the ring of integers $\Z$ there are many maximal ideals - one for each prime number $p$.  In the $p$-adic integers $\Zp$ there is precisely one non-zero maximal ideal, meaning that $\Zp$ is a \emph{discrete valuation ring} (and therefore a \emph{local ring}).  By putting together the information about all the local rings we hope to reconstruct the behaviour of the global ring.  This is the rationale behind studying $\Q$ by looking at $\Q_p$ (for all primes $p$) and $\R$ (which we can think of as $\Q_{\infty}$).
\end{remark}

Given any element in $W(\Q)$, we get elements in $W(\R)$ and $W(\Q_p)$ by extension of scalars, i.e. by tensoring over $\Q$ with $\R$ and $\Q_p$ respectively.  For this mapping to be informative, we need to know the structure of $W(\R)$ and $W(\Q_p)$.  Let us start with the easy one.

\begin{lemma}
\label{Lemma:signature}
  $W(\R) \cong \Z$.
\end{lemma}

\begin{proof}
Every real number is either a square or the negative of a square; thus every quadratic form can be diagonalised as $[1,\dots,1,-1,\dots, -1]$.  In $W(\R)$ we have $[1,-1]=0$, so every class in $W(\R)$ is determined by the sum of the signs of its diagonalisation.  This value is called the signature, denoted by $\sigma$, and is an isomorphism between $W(\R)$ and $\Z$.
\end{proof}

To start our investigation of $W(\Q_p)$, we need to understand what the squares in $\Q_p$ look like.

\begin{lemma}
\label{Lemma:Qpsquares}
If $p$ is odd, the quotient $\Q_p^*/(\Q_p^*)^2$ is isomorphic to $\Z_2 \oplus \Z_2$. The four distinct elements are $\{1,u,p,pu\}$ where $0<u<p$ is not a square modulo $p$.
\end{lemma}

\begin{proof}
We first prove that a unit $u = a_0+a_1p+a_2p^2+\dots$ in $\Zp$ is a square if and only if $a_0$ is a square in $\Z_p$.  This is due to Hensel's Lemma (see, for example, \cite[Theorem 3.7]{Eisenbud} ) which states that if $r_k$ is a solution to the congruence $f(x)\equiv 0$ mod $p^k$ for $k\geq 1$, and if $f'(r_k) \not\equiv 0$ mod $p$, then there exists a number $r_{k+1}$ which is a solution to $f(x)\equiv 0$ mod $p^{k+1}$ and $r_k\equiv r_{k-1}$ mod $p^k$.  If $a_0=b_0^2$ in $\Z_p$ we can let $f(x) = x^2-u$, so $f'(x) = 2x$ and $2(b_0) \neq 0$ in $\Z_p$.  Hensel's Lemma lets us construct the coefficients $b_1,b_2\dots$ in the $p$-adic integer which is the square root of $u$.

Up to a factor of an even power of $p$, every element of $\Q_p^*$ can be written as $u$ or $pu$ where $u$ is a unit in $\Zp$. From the first half of this proof, $u$ is a square if and only if $a_0$ is a square in $\Z_p$, and $\Z_p^* / (\Z_p^*)^2 \cong \Z_2$.  Since $p$ is not a square, the result follows.
\end{proof}

The result for $p=2$ is more complicated and a proof may be found in \cite{Scharlau}.

\begin{lemma}
  The quotient $\Q_2^*/(\Q_2^*)^2$ is isomorphic to $\Z_2 \oplus \Z_2 \oplus \Z_2$.  The eight distinct elements are the set $\{ \pm 1, \pm 2, \pm 5, \pm 10 \}$.
\end{lemma}

Now that we understand the squares of $\Q_p$, we have a way to map $W(\Q_p)$ into yet simpler Witt groups -- at least, in the case when $p$ is odd.

\begin{theorem}
\label{Thm:QpintoFp}
For $p$ odd, $W(\Q_p) \cong W(\F_p) \times W(\F_p)$.
\end{theorem}

\begin{proof}
By Theorem \ref{Thm:diagonalise}, Remark \ref{Rmk:diagonalSquare} and Lemma \ref{Lemma:Qpsquares}, any form in $W(\Q_p)$ can be diagonalised as $[u_1,\dots,u_k,pv_1,\dots,pv_j]$ where the $u_i$ and $v_i$ are units in $\Zp$.  For a unit $u=a_0 +a_1p+a_2p^2 +\dots$ write $\overline{u} = a_0\in \F_p^*$.  Then the map
\[ [u_1,\dots,u_k,pv_1,\dots,pv_j] \mapsto ([\overline{u_1},\dots,\overline{u_k}], [\overline{v_1}, \dots, \overline{v_j}])\]
is the desired isomorphism.
  For future reference, we will denote this isomorphism by $\psi_p \oplus \partial_p$.
\end{proof}

\begin{theorem}
$W(\Q_2) \cong \Z_8 \oplus \Z_2 \oplus \Z_2$.
\end{theorem}

\begin{proof}
The generators are $[1]$, $[-1,5]$ and $[-1,2]$.  For a proof, see \cite[Chapter 5, Theorem 6.6]{Scharlau}.
\end{proof}

To fully understand $W(\Q_p)$ when $p$ is odd, it thus suffices to understand $W(\F_p)$.  The following theorem deals with this question, including the case of $W(\F_2)$ for completeness.

\begin{theorem}
  \[W(\F_p) = \begin{cases} \Z_2 & \text{ if } p = 2\\
                          \Z_2 \oplus \Z_2 & \text{ if } p\equiv 1 \mod 4\\
                          \Z_4 & \text{ if } p\equiv 3 \mod 4.
  \end{cases}\]
\end{theorem}

\begin{proof}
For $p=2$ every form can be represented by a sum of the forms $[1]$ and $\left(\begin{array}{cc} 0 & 1\\ 1 & 1 \end{array}\right)$.  The first of these has order $2$ in $W(\F_2)$ and the second is Witt trivial.

For $p$ odd, the group of units $\F_p^*$ is cyclic of even order $p-1$, so $\F_p^*/(\F_p^*)^2 = \Z_2$.  Modulo squares then, every number is equivalent to $1$ or to $a$ where $a$ is not a square.  Thus every form in $W(\F_p)$ is equivalent to $[1,\dots,1,a,\dots,a]$.  If $p=1$ mod $4$ then $-1$ is a square, so any form $[b,b] = [b,-b]$, which is Witt trivial.  Hence any non-trivial form is equivalent to $[1]$, $[a]$ or $[1,a]$ and each of these is of order $2$.

If $p=3$ mod $4$, then $-1$ is not a square so we can let $a=-1$.  Since $[1,-1]$ is trivial, every form is equivalent to a multiple of $[1]$ or a multiple of $[-1]$. The form $[b,b]$ is nontrivial but $[b,b,b,b]$ is trivial, with metaboliser $<(1,0,a,c),(0,1,-c,a)>$ where $(a,c)$ satisfy $1+a^2+c^2 = 0$.  The pair $(a,c)$ exist by the Pigeonhole Principle: there are $(p+1)/2$ values for $x^2$ in $\F_p$ and also $(p+1)/2$ values for $-1-y^2$. There are only $p$ values in $\F_p$ so the equation $x^2 = -1-y^2$ must have a solution.
\end{proof}

We now understand all that we need to know about the Witt groups of the completions of $\Q$.  There is one remaining remark, which concerns the isomorphism in the proof of Theorem \ref{Thm:QpintoFp}. We denoted this isomorphism by $\psi_p \oplus \partial_p$, and it turns out that we may safely ignore $\psi_p$ and just use the map $\partial_p$.  The rest of this section gives a proof of this fact and further illuminates how we can understand the group $W(\Q)$ by looking at all the local Witt groups $W(\Q_p)$ and $W(\R)$.

  Recall that the homomorphism $\partial_p \colon W(\Q) \to W(\F_p)$ (which factors through $W(\Q_p)$) maps an element $[\alpha]$ (where $\alpha=p^n \frac{x}{y}$ with $x$ and $y$ coprime to $p$) to $\left[\frac{x}{y}\right]$ if $n$ is odd and $0$ if $n$ is even.  A quick example illustrates this.

\begin{example}
\label{Example:partialp}
Let $p=5$.
\begin{eqnarray*}
\partial_p\left(\left[\frac{13}{50},\frac{15}{2}\right]\right)
& = & \partial_p\left(\left[5^{-2}\left(\frac{13}{2}\right),\; 5\left(\frac{3}{2}\right)\right]\right)\\
& = & \left[\frac{3}{2}\right]\\
& = & \left[2^2 \frac{3}{2}\right] = \left[6\right] = \left[1\right].
\end{eqnarray*}
\end{example}

Notice that for any $\alpha \in \Q$, $\partial_p(\alpha) = 0$ for almost all $p$.  We may thus take the direct sum of all these homomorphisms to get one giant homomorphism $\partial \colon W(\Q) \to \oplus_p W(\F_p)$.

\begin{theorem}
The sequence
\[ 0 \rightarrow \Z \xrightarrow{i} W(\Q) \xrightarrow{\partial} \oplus W(\F_p) \rightarrow 0\]
is split exact, where $\partial = \oplus \partial_p$ and $i$ maps $1\in \Z$ to $[1] \in W(\Q)$.
\end{theorem}

\begin{proof}[Sketch proof]
For each $k\in \N$, let $L_k$ be the subring of $W(\Q)$ generated by $[1],[2],\dots,[k]$.  Then
\[ L_1 \subset L_2 \subset L_3 \subset \dots\]
and $\cup_i L_i = W(\Q)$.  Note that $L_1 \cong \Z$ and $L_k = L_{k-1}$ unless $k$ is prime.  Each of the homomorphisms $\partial_p$ induce an isomorphism $L_p / L_{p-1} \to W(\F_p)$. Using these isomorphisms and an argument by induction, one can show that the homomorphisms
\[ L_k \to \displaystyle \oplus_{p \leq k} W(\F_p)\]
are surjective with kernel equal to $L_1$.  Passing to the direct limit as $k \to \infty$ we get that
\[ 0 \rightarrow \Z \to W(\Q) \to \oplus W(\F_p) \rightarrow 0\]
is exact.  Using the signature homomorphism $W(\Q) \to W(\R) \to \Z$ (defined in Lemma \ref{Lemma:signature}) it follows that the sequence is actually split exact.
\end{proof}

\section{Invariants of algebraic knot concordance}
\label{Sec:classifyAlgKnots}

In the final section of this chapter we will investigate the set of algebraic concordance invariants given by Levine in \cite{Levine69-1}.  In Section \ref{Sec:AlgConcGroup} we saw that, in order to find the algebraic concordance order of a knot with image $(M,Q,T) \in \G_{\Q}$, we would have to look at the image of $(M,Q,T)$ in $W(\Q_p)$ for every prime $p$.  Since we wish to classify our knots in a finite amount of time, we need to find a way to reduce the list of primes that we have to check.

\begin{definition}
\label{Def:alginvariants}
  Let $(M,Q,T) \in \G_{\Q}$ be an isometric structure, $\lambda(t)$ be an irreducible symmetric factor of $\Delta_T(t)$ and $p$ be a prime.
  \begin{itemize}
    \item $\eps_{\lambda}(M,Q,T)$ is the exponent, modulo $2$, of $\lambda(t)$ in $\Delta_T(t)$.
    \item $\sigma_{\lambda}(M,Q,T)$ is the signature of $(M,Q) \in W(\R)$ restricted to the $\lambda(t)$-primary component.
    \item $\mu_{\lambda}^p(M,Q,T)$ is $\mu(M',Q')$ where $(M',Q')$ is the image of $(M,Q)$ in $W(\Q_p)$ restricted to the $\lambda(t)$-primary component, where $\mu$ is defined by
    \[ \mu(\alpha) = (-1,-1)^{\frac{r(r+3)}{2}} (\det(\alpha),-1)^r S(\alpha)\]
    and where $(-,-)$ is the Hilbert symbol for $\Q_p$, $S(-)$ is the Hasse invariant and $\alpha$ has rank $2r$.
  \end{itemize}
\end{definition}

This last definition of $\mu$ needs some more explanation.  We will give the definitions of Hilbert symbol and Hasse invariant, followed by a formula for the Hilbert symbol in the case of the field being $\Q_p$.

\begin{definition}
  The \emph{Hilbert symbol} of a local field $K$ is the function $(-,-) \colon K^* \times K^* \to \{-1,1\}$ defined by
  \[ (a,b) = \begin{cases} 1 & \text{if } z^2 = ax^2 + by^2 \text{ has a non-zero solution } (x,y,z) \in K^3 \\ -1 & \text{ otherwise.} \end{cases}\]
\end{definition}

\begin{definition}
  The \emph{Hasse invariant}, or \emph{Hasse symbol}, of a quadratic form $\alpha$ diagonalised as $[d_1,\dots, d_n]$ over a local field $K$ is
   \[ S(\alpha) = \prod_{i<j}(d_i,d_j) \in \{-1,1\}\]
where $(-,-)$ is the Hilbert symbol.
\end{definition}

\begin{proposition}(\cite[Chapter 3, 1.2]{Serre})
\label{Prop:Hilbertsymbol}
  If $K=\Q_p$ and if we write $a=p^\alpha u$, $b=p^\beta v$ for units $u,v \in \Zp^*$, then we have
  \[ (a,b)_p = (-1)^{\alpha \beta \epsilon(p)} \left( \frac{\overline{u}}{p}\right)^\beta \left( \frac{\overline{v}}{p}\right)^\alpha \quad \text{for } p \neq 2 \]
  \[ (a,b)_2 = (-1)^{\epsilon(u)\epsilon(v) + \alpha \omega(v) + \beta \omega(u)} \quad \text{for } p=2\]
where $\epsilon(n) = \frac{n-1}{2}$, $\omega(n) = \frac{n^2-1}{8}$ and $\left(\frac{n}{p}\right)$ is the Legendre symbol (see \cite{Serre} for more details about this).
\end{proposition}

\begin{theorem}(\cite[21]{Levine69-1})
\label{Thm:CompleteCobordismInvts}
  The functions $\{ \eps_{\lambda}, \sigma_{\lambda}, \mu_{\lambda}^p \}$ are a complete set of cobordism invariants for an isometric structure $(M,Q,T)$ in $\G_{\Q}$, when taken over every possible $\lambda$ and every prime $p$.
\end{theorem}

Of these three invariants, $\sigma$ is the only one which takes values in a non-finite group, so is the only one which can detect elements of infinite order.  The invariant $\eps$ is of order $2$ and $\mu$ is of order $4$ (as we shall see), so this proves that $1$, $2$ and $4$ are the only finite orders that a knot may have in the algebraic concordance group.

We will now look at practical ways to detect knots of order $2$, $4$ and $\infty$. Knots of order $1$, i.e. algebraically slice knots, are those for which all invariants $\{ \eps_{\lambda}, \sigma_{\lambda}, \mu_{\lambda}^p \}$ vanish, where $\lambda$ runs over all the irreducible symmetric factors of the Alexander polynomial $\Delta_K(t)$ and $p$ runs over every prime.

\subsection{Detecting infinite order elements in $\G$}
\label{Subsec:infiniteOrderElts}

If a knot has infinite order in $\G$ then it must have signature $\sigma_{\lambda} \neq 0$ for some symmetric irreducible factor $\lambda$ of $\Delta_T$.  Over the real numbers, such irreducible symmetric polynomials are of the form $t^2 + 2at + 1$, up to a unit.  The roots are
 \[ t = -a \pm \sqrt{a^2-1} \]
and because the polynomial is irreducible we must have $a^2 <1$.  Writing $a=\cos \theta$ gives the roots as $-\cos \theta \pm i \sin \theta$, which we can write as $e^{i \phi}$ for $\phi = \pm (\pi - \theta)$.

We need to know a formula for computing $\sigma_{\lambda}$ at these irreducible polynomials.

\begin{definition}
  For a knot $K$ with Seifert matrix $V$, and for a unit modulus complex number $\omega$, the \emph{$\omega$-signature} $\sigma_{\omega}$ is the signature of the Hermitian matrix
  \[ (1-\omega)V + (1-\overline{\omega})V^T \text{ .}\]
\end{definition}

The $\omega$-signature of $V$, as a map $S^1 \to \Z$, is continuous with jumps only at the unit roots of the Alexander polynomial of $K$.  More details about the $\omega$-signature and these jumps are provided in Section \ref{Sec:sigAndjump}; for now we need only the following theorem.

\begin{theorem}(\cite{Matumoto77})
  If a knot $K$ with Seifert matrix $V$ has $\delta_a = t^2+2at+1$ as a factor of its Alexander polynomial $\Delta_K(t)$, then $\sigma_{\delta_a}(V+V^T)$ equals, up to sign, the jump in the $\omega$-signature function at $e^{i\phi}$ where $\cos \phi = -a$.
\end{theorem}

If the $\omega$-signature of a knot has no jumps, i.e. is zero at every $\omega$, then the knot must be of finite order in the knot concordance group.

\subsection{Detecting order 4 elements in $\G$}

From our previous analysis of Witt groups we know there is a chance that a knot may have algebraic order $4$, since $W(\F_p) \cong \Z_4$ for $p\equiv 3$ mod $4$ and there is a factor $\Z_8 \subset W(\Q_2)$.  The following theorem was essentially worked out by Levine~\cite[22]{Levine69-1} and gives exact criteria for a knot to have algebraic order $4$.  Interestingly, there are no elements in $\G$ with image of order $4$ in $W(\Q_2)$.  The proof given here uses elements from Levine's paper as well as Livingston's paper~\cite[3.1]{Livingston08}.

\begin{theorem}
\label{Thm:order4class}
  A class $(M,Q,T) \in \G_{\Q}$ which is the image of an element of $\G$ is of order $4$ if and only if
  \begin{itemize}
	  \item $\sigma_{\lambda}(M,Q,T) = 0$ for every $\lambda$,
	  \item There exists an irreducible symmetric factor $\lambda(t)$ of $\Delta_T(t)$ with $\eps_{\lambda}(M,Q,T) = 1$ and with $\lambda(1)\lambda(-1) = p^n q$ for a prime $p \equiv 3$ mod 4, $n$ odd and $q$ relatively prime to $p$.
  \end{itemize}
\end{theorem}

\begin{proof}
  The signature function $\sigma$ is additive with values in $\Z$, so if it is not zero then $(M,Q,T)$ must have infinite order.

  Levine~\cite[19]{Levine69-1} proves that, for $\alpha \in \G_{\Q}$ and $2d=\text{degree
  } \lambda(t)$,
    \[ \mu_{\lambda}^p(2\alpha) = ((-1)^d \lambda(1)\lambda(-1),-1)^{\eps_{\lambda}(\alpha)}. \]
  In particular, if $\eps_{\lambda}(\alpha) = 0$ then this equals $1$ and $\alpha$ has order at most $2$.  So suppose there is some factor $\lambda(t)$ of $\Delta_T(t)$ with odd exponent in $\Delta_T(t)$.

  For $\alpha$ to have order $4$, we need that $\mu_{\lambda}^p(2\alpha) = -1$ for some prime $p$.  We need to consider two cases: $p$ odd and $p=2$.

  $\mathbf{p}$ \textbf{is odd:} From Proposition \ref{Prop:Hilbertsymbol} we have
  \[ (a,-1)_p = \left( \frac{-1}{p}\right)^{n} = \left((-1)^{\frac{p-1}{2}}\right)^n\]
  where $a = p^n u$ for a unit $u \in \Zp^*$.  Immediately we see that if $p\equiv 1$ mod 4 then $\mu_{\lambda}^p(2\alpha)= 1$ and $\alpha$ would be of order $2$.  So assume $p\equiv 3$ mod 4.  For $\alpha$ to be of order $4$ we also need the exponent $n$ to be odd.  In our case
  \[ a = (-1)^d \lambda(1)\lambda(-1) \]
  so this means we need $p$ to divide $\lambda(1)\lambda(-1)$ with odd exponent.

  $\mathbf{p=2}$\textbf{:} We want to show that $\mu_{\lambda}^2(2\alpha) = 1$ for any class $\alpha$.  This would mean that the prime $p=2$ cannot detect elements of order $4$.  We have
  \[ (a,-1)_2 = (-1)^{\epsilon(u)}\]
  where $a = 2^n u$ for $u$ coprime to $2$.  As before, we have  $a = (-1)^d \lambda(1)\lambda(-1)$ with degree$(\lambda) = 2d$. We will show that this is always congruent to $1$ modulo $4$, meaning that $\epsilon(u) = 0$ and $\mu_{\lambda}^2(2\alpha)$ is always equal to $1$.

 Write
 \[ \lambda(t) = a_0 + a_1 t + \dots + a_{d-1}t^{d-1} + a_d t^d + a_{d-1}t^{d+1} + \dots + a_0 t^{2d}\]
 Since $\alpha$ is the image of an element of $\G$, we have that $\Delta_T(1)$ is odd.  Thus $l:=\lambda(1)$ is odd.  We have
   \[ l = 2(a_0 + \dots + a_{d-1}) + a_d = 2a_{even} + 2a_{odd} + a_d\]
where $a_{even}$ and $a_{odd}$ are the sums of the coefficients with even or odd subscripts, respectively.  We also have
  \[ \lambda(-1) = 2a_{even} - 2a_{odd} + (-1)^d a_d = 2a_{even} - 2a_{odd} + (-1)^d(l - 2a_{even} - 2a_{odd}).\]
  If $d$ is even then $\lambda(-1) = l - 4a_{odd} \equiv l \text{ mod } 4$.  If $d$ is odd then $\lambda(-1) = -l + 4a_{even} \equiv -l \text{ mod } 4$.  In both cases we have $(-1)^d \lambda(1)\lambda(-1) \equiv l^2 \text{ mod } 4$, and since $l$ is odd this is $1$ modulo $4$.
\end{proof}

We rewrite the result of Theorem \ref{Thm:order4class} in a way that will be useful to us in Chapter 4.

\begin{corollary}
\label{Cor:AlgOrder4}
If a knot $K$ is of order 4 in the algebraic concordance group $\G$ then for some prime $p\equiv 3$ mod $4$ and some symmetric irreducible factor $g(t)$ of $\Delta_K(t)$, $p$ divides $g(-1)$ and $g$ has odd exponent in $\Delta_K$.
\end{corollary}

\subsection{Detecting order 2 elements in $\G$}

\subsubsection{Order $2$ via the exponent map}

If a knot has all its signatures $\sigma_{\lambda}$ equal to zero and is known not to be of order $4$ by Corollary \ref{Cor:AlgOrder4} but has $\eps_{\lambda}=1$ for some irreducible symmetric $\lambda$, then such a knot must be an element of algebraic order $2$.

\subsubsection{Order $2$ in $W(\F_p)$ for $p\equiv 1$ and $p\equiv 3$ mod 4}

We now suppose that we have a knot whose signatures $\sigma_{\lambda}$ are all zero, which is known not to be of order $4$ and which has $\eps_{\lambda}=0$ for all $\lambda$.  We wish to compute the image of this knot in $W(\F_p)$ for primes $p\equiv 1$ modulo $4$, and seek some way to eliminate the majority of such primes from consideration.  The hope is that, as with Corollary \ref{Cor:AlgOrder4}, we only need to consider those primes which divide $\Delta_K(-1)$.  Unfortunately the situation is somewhat more complicated.

The following is Theorem 4.1 from \cite{Livingston08}, where the notation $\Delta_V$ is equivalent to $\Delta_K$ and is the Alexander polynomial associated to the matrix $V$.

\begin{theorem}
\label{Thm:DetDisc}
Let $V$ be a non-singular Seifert matrix and suppose that each irreducible symmetric factor of $\Delta_V(t)$ has even exponent.  Then for any prime $p$ that does not divide $2\det(V)\disc(\overline{\Delta}_V(t))$ we have that $V$ represents zero in $\G_{\Q_p}$, where $\overline{\Delta}_V(t)$ denotes the product of all the distinct irreducible factors of $\Delta_V$.
\end{theorem}

The discriminant of the Alexander polynomial is closely related to the determinant of the knot.  The determinant of a knot $K$, denoted $\det(K)$, is defined to be $\det(V+V^T)=\Delta_V(1)\Delta_V(-1)$.  Notice that the primes dividing $\det(K)$ are exactly those dividing $\overline{\Delta}_V(-1)$, so we will assume for now that the Alexander polynomial has no repeated irreducible factors.

\medskip

Write the Alexander polynomial as
 \[ \Delta_V(t) = \det(V)(t^{2n} + a_{1}t^{2n-1} + \dots + a_1 t + 1).\]
Then the discriminant of $\Delta_V(t)$ is the product of the squares of the differences of the roots
\[ \disc(\Delta_V(t)) = \det(V)^{4n-2}\prod_{i < j}(\alpha_i - \alpha_j)^2 \]
where the roots $\alpha_i$ lie in some algebraic closure of $\Q$.  The Alexander polynomial is symmetric so the roots come in inverse pairs: write $\alpha_i = \frac{1}{\alpha_{n+i}}$ for $i=1, \dots, n$. Since $\pm1$ is not a root, $\alpha\neq \frac{1}{\alpha}$ and we can re-write the discriminant as
\[ \disc(\Delta_V(t)) = \det(V)^{4n-2}\prod_{i=1}^n(\alpha_i - \frac{1}{\alpha_i})^2 \prod_{\substack{\alpha_j\neq \frac{1}{\alpha_i}\\ i<j}}(\alpha_i - \alpha_j)^2. \]
Let us re-write the first product as follows:
\[ \prod_{i=1}^n(\alpha_i - \frac{1}{\alpha_i})^2 = \frac{1}{\prod_{i=1}^n \alpha_i^2} \prod_{i=1}^n (\alpha_i^2-1)^2
= \frac{1}{\prod_{i=1}^n \alpha_i^2}\prod_{i=1}^n(\alpha_i-1)^2(\alpha_i+1)^2.\]
Now let us find a formula for the determinant.  We have
\[ \Delta_V(t) = \det(V)\prod_{i=1}^n(t-\alpha_i)(t-\frac{1}{\alpha_i})\]
by definition of the $\alpha_i$. The determinant is therefore
\begin{eqnarray*}
\Delta_V(1)\Delta_V(-1) & = & \det(V)^2\prod_{i=1}^n(1-\alpha_i)(1-\frac{1}{\alpha_i})(1+\alpha_i)(1+\frac{1}{\alpha_i})\\
& = & (-1)^n \det(V)^2 \frac{1}{\prod_{i=1}^n \alpha_i^2} \prod_{i=1}^n (\alpha_i-1)^2(\alpha_i+1)^2.
\end{eqnarray*}

The result of these manipulations is the formula in the next Lemma.

\begin{lemma}
If the Alexander polynomial $\Delta_V(t)$ of a knot $K$ has no repeated irreducible factors, then the discriminant of $\Delta_V(t)$ is related to the determinant of $K$ by the following formula.
\[\disc(\Delta_V(t)) = (-1)^n\det(V)^{4n-4}\det(K)\prod_{\substack{\alpha_j\neq \frac{1}{\alpha_i}\\ i<j}}(\alpha_i - \alpha_j)^2.\]
\end{lemma}

\begin{remark}
 The discriminant of a polynomial $p(t)$ may be expressed as an integer polynomial in the coefficients of $p$ (see, for example, Proposition 7.5 in \cite{Grillet}). Since the Alexander polynomial has integer coefficients, the discriminant will therefore be an integer. The determinant of $K$ is also an integer. It therefore does not matter which algebraic closure we use to find the roots $\alpha_i$, in terms of deciding which primes divide the discriminant but not the determinant.
\end{remark}

\begin{remark}
In fact, it is fairly rare for a prime other than $2$ to divide the discriminant but not the determinant of $K$.  Of the $41$ knots forming a basis for the kernel of the signature function (see Section \ref{Sec:signature}) there are $12$ knots with this property.  This goes down to only one knot out of the $11$ knots in $B_4$ (see Section \ref{Sec:MoreWitt}), which is the set of knots that are trivial in $W(\F_p)$ for all primes $p$ dividing the determinant.
\end{remark}

\subsubsection{Order $2$ in $W(\F_p)$ for $p\equiv 3$ mod 4}

It may happen that a knot is detected to be of order $2$ in $W(\F_p)$ for $p\equiv 3$ mod $4$.  These Witt groups are isomorphic to $\Z_4$, so it is possible that an element representing $2 \in \Z_4$ arises as twice an element of order $4$.  If this is the case, then all other $\Z_2$-valued invariants will vanish. (Such examples will occur at the end of Section \ref{Sec:ElementsOrder4}.) The following theorem will help us to eliminate such a possibility.

\begin{theorem}
\label{Thm:sigjump}
  Suppose the symmetric factors of the Alexander polynomial $\Delta_K(t)$ of a knot $K$ factor (over $\CC$) as $at^k\prod(t-\omega_i)$ with all $\omega_i$ distinct unit complex numbers.  Then the jump $j_{\omega}(K)$ in the $\omega$-signature function of $K$ is $1$ or $0$ modulo $2$ depending on whether $\omega = \omega_i$ for some $i$ or not.
\end{theorem}

\begin{proof}
(Taken from Appendix A of \cite{HeddenKirkLivingston11}.) Let $V$ be a Seifert matrix for the knot $K$.
Write
\[ A(t) = (1-t)V + (1 - t^{-1})V^T = (1-t)(V-t^{-1}V^T) \in W(\CC(t),\I) \]
so the determinant of $A(t)$ is $(1-t)^{2n} \Delta_K(t^{-1})$. For a unit complex number $\omega$, the $\omega$-signature $\sigma_{\omega}(K)$ is simply the signature of $A(\omega)$, taking the average of the one-sided limits when $A(\omega)$ is singular. Since $\sigma_{\omega} \colon S^1 \to \Z$ is continuous except at these singularities, it is a constant function with jumps. Define the jump function $j_K(\omega) \colon S^1 \to \Z$ by
  \[ j_K(e^{i\theta}) = \displaystyle \frac{1}{2} \left(\lim_{x \to \theta +} \sigma_{e^{ix}}(K) - \lim_{x \to \theta -} \sigma_{e^{ix}}(K)\right). \]
  (Despite the factor of $\frac{1}{2}$, the jump function is actually integer-valued, as will be seen shortly.)

   Since we are considering $A(t)$ as an element of $W(\CC(t),\I)$, we consider its determinant as an element of $\CC(t)^*/ N$, where (as in Chapter \ref{chapter2}) $N$ is the subgroup of $\CC(t)^*$ generated by norms. Thus we may write
  \[ \det(A(t)) = at^k\prod_{i=1}^{2n}(t-\omega_i)\]
where the $\omega_i$ are distinct unit complex numbers. This is because $\det(A(t)) = \I(\det(A(t)))$ so if $(t-\omega)$ is a factor then $(t- \overline{\omega^{-1}})$ is also a factor, and if $\omega$ is not a unit complex number then these two factors form a norm and can be removed from the product. If $\omega$ is a unit complex number then any factors of $(1-\omega)^2$ can also be removed. Furthermore, the symmetry of $\det(A)$ allows us to conclude that $k=-n$ and $a^2 = \prod {\omega_i}^{-1}$.

Now, we may diagonalise any class in $W(\CC(t),\I)$, and the diagonal entries are also of the form $at^{-k} \prod_{i=1}^{2k}(t-\omega_i)$ for some distinct set of $\omega_i \in S^1$ and for $a^2 = \prod {\omega_i}^{-1}$. Some clever manipulation gives us the following lemma.

\begin{lemma}
\label{Lemma:detsin}
 If $d=at^{-k} \prod_{j=1}^{2k}(t-\omega_j)$, if $t=e^{i\theta}$ and if $\omega_j = e^{i\theta_j}$ then
  \[ d = \pm 2^{2k} \prod_{i=1}^{2k} \sin((\theta - \theta_j)/2). \]
\end{lemma}

\begin{proof}
 Write
 \begin{eqnarray*}
   \sin((\theta - \theta_j)/2) & = & \frac{1}{2i}\left( e^{i(\theta-\theta_j)/2} - e^{-i(\theta-\theta_j)/2} \right) \\
   & = & \frac{1}{2i}\left( e^{i(\theta-\theta_j)/2} - e^{-i(\theta-\theta_j)/2} \right)\left(e^{i(\theta+\theta_j)/2}/e^{i(\theta+\theta_j)/2}\right) \\
   & = & \frac{1}{2i}\left( e^{i \theta} - e^{i\theta_j} \right)/(e^{i(\theta+\theta_j)/2}) \\
   & = & \frac{1}{2i} (t-\omega_j)/(t\omega_j)^{1/2}
 \end{eqnarray*}
 So
  \[ \prod_{j=1}^{2k} \sin((\theta - \theta_j)/2) = (-1)^k 2^{-2k}at^{-k} \prod_{j=1}^{2k} (t-\omega_j).\]
\end{proof}

If $I \in W(\CC(t),\I)$ is a $1$-dimensional form with representative $d$ as given in Lemma \ref{Lemma:detsin} then it is clear that
 \begin{itemize}
   \item $\sigma_\omega(I) = 0$ or $\pm 1$ depending on whether $\omega = \omega_j$ or not.
   \item $j_I(\omega) = \pm 1$ or $0$ depending on whether $\omega = \omega_j$ or not.
 \end{itemize}
 where the second point is seen from the $\sin$ form of $d$, since this changes sign as $\theta$ goes from $\theta_j + \eps$ to $\theta_j - \eps$.

 Since the signature and jump functions are additive, this gives us the result that the jump $j_K(\omega)$ is $1$ mod $2$ if $(t-\omega)$ is a factor of $\det (A(t))$ (and therefore of the symmetric part of $\Delta_K(t)$) and $0$ mod $2$ otherwise. This completes the proof.
\end{proof}

In particular, this means that if an Alexander polynomial of a knot has any distinct complex roots of unit modulus, then the signature function \emph{will} jump at those roots, meaning that the knot must be of infinite order.

\chapter{\label{chapter4} Algebraic concordance classification of $9$-crossing knots}

In this chapter we use the theorems from Chapter \ref{chapter3} to identify the subgroup of the algebraic concordance group $\G$ generated by the prime knots of $9$ or fewer crossings.  From this `classification' we can easily decide what order a linear combination of $9$-crossing knots has in $\G$ and, in particular, which linear combinations are algebraically slice.

To begin, let us recap the definitions made in Section \ref{Subsec:9crossingClass}.

\medskip

Let $E = \left\{ 3_1,4_1,5_1,5_2,6_1,6_2,6_3,7_1, \dots,7_7,8_1,\dots,8_{21}, 9_1, \dots, 9_{49} \right\}$, where $|E|=87$ since the list includes the distinct reverses $8_{17}^r$, $9_{32}^r$ and $9_{33}^r$.

\begin{notation}
Let $\C_E$ denote the subgroup of $\C$ generated by $E$.  Denote by $\CF_E$ the free abelian group generated by $E$.
\end{notation}

There are natural maps $\CF_E \xrightarrow{\psi} \C_E \xrightarrow{\phi} \G$, and in this chapter we wish to analyse the kernel and image of $\phi \circ \psi$. (A summary of the result may be found in Section \ref{Sec:Ch4summary}.)  We will be using the invariants described in Definition \ref{Def:alginvariants}:
\begin{itemize}
  \item $\eps_{\lambda}(M,Q,T)$: the exponent, modulo $2$, of $\lambda(t)$ in $\Delta_K(t)$,
  \item $\sigma_{\lambda}(M,Q,T)$: the signature of $(M,Q) \in W(\R)$ restricted to the $\lambda(t)$-primary component,
  \item $\mu_{\lambda}^p(M,Q,T)$: $\mu(M',Q')$ where $(M',Q')$ is the image of $(M,Q)$ in $W(\Q_p)$ restricted to the $\lambda(t)$-primary component,
\end{itemize}
where $\lambda$ runs over the irreducible symmetric factors of the Alexander polynomial $\Delta_K(t)$ and $p$ is a prime. For a Seifert matrix $V$, recall that $Q=V+V^T$ and $T=V^{-1}V^T$.  However, instead of using $\mu_{\lambda}^p$ directly we will simply look at the image of knots in the Witt groups $W(\Q_p)$ restricted to each $\lambda(t)$-primary component.

If a knot has no irreducible symmetric factors, it must be algebraically slice. Similarly, if the image of the knot under all of these maps is zero, then it must be algebraically slice. That is, if the $\omega$-signature $\sigma_\omega(V)$ is zero for every unit complex number $\omega$ which is not a root of the Alexander polynomial $\Delta_K(t)$; if the exponent of every irreducible symmetric factor of $\Delta_K(t)$ is even; and if the image of $K$ in $W(\Q_p)$ is zero for every prime $p$, then $K = 0 \in \G$.

\section{Analysing the signature function}
\label{Sec:signature}
We begin by finding a basis for the knots of infinite order in $\G$ by analysing the signature function of each knot.  Given a Seifert matrix $V$ for a knot $K$, the signature function $S^1 \to \Z$ is defined for each unit complex number $\omega$ as follows
 \[ \sigma_\omega(V) := \sign((1-\omega)V + (1-\overline{\omega})V^T) \]
and is continuous except at unit roots of the Alexander polynomial $\Delta_K(t) := \det(V-tV^T)$. If $\omega$ is a unit root of $\Delta_K$, we define $\sigma_\omega(V)$ to be the average of the limit on either side.

By the analysis of Section \ref{Subsec:infiniteOrderElts} we know that if this function is non-zero for any value of $\omega \in S^1$, then the knot $K$ is of infinite order in $\G$.  Conversely, by Theorem \ref{Thm:CompleteCobordismInvts} if $\sigma_{\omega}(V)= 0$ for every $\omega \in S^1$ then $K$ is of finite order in $\G$.

Since the signature function of a knot $K$ is continuous except at the unit roots of $\Delta_K(t)$, and is integer-valued, this means that $\sigma_{\omega}(V)$ is constant between unit roots of $\Delta_K(t)$.  This simplifies our computations because it means that, to determine the full signature function of any knot, we need only take a sample of the signature values in between each pair of the unit roots of $\Delta_K(t)$.  Furthermore, we need only do this for the unit roots with positive imaginary part, because the signature function is symmetric: $\sigma_{\omega}(V) = \sigma_{\overline{\omega}}(V)$.

The set of unit roots, with positive imaginary part, of Alexander polynomials of knots in $E$ consists of 70 values, which we shall call $e^{i\theta_1}, \dots, e^{i\theta_{70}}$ ordered by argument.  We shall choose the midpoint of each consecutive pair of values:
\[  \delta_1:=1; \quad \delta_i :=\exp\left(i \frac{(\theta_{i-1} + \theta_i)}{2}\right) \text{ for } 1<i\leq 70. \]
Evaluating the signature function of each knot at each of the $\delta_i$ gives us a complete picture of the signature function of each knot.  We put the values into an $84 \times 70$ matrix and then perform column operations to get it in reduced echelon form. (See Appendix \ref{AppendixD} for this reduced-form matrix.)  From this we can read off a basis for the kernel of the signature function as well as a basis for the knots which are of infinite order.

The following is a linearly independent set of $46$ knots which form a basis in $\CF_E$ for the knots of infinite algebraic concordance order:
\begin{eqnarray*}\C^\infty & = & \{ 3_1, 5_1, 5_2, 6_2, 7_1, \dots, 7_6, 8_2, 8_4, \dots 8_7, 8_{14}, 8_{16}, 8_{19}, 9_1, 9_3, \dots, 9_7, 9_9, 9_{10}, 9_{11}, 9_{13}, 9_{15},\\
& & \quad 9_{17}, 9_{18}, 9_{20}, 9_{21}, 9_{22}, 9_{25}, 9_{26}, 9_{31}, 9_{32}, 9_{35}, 9_{36}, 9_{38}, 9_{43}, 9_{45}, 9_{47}, 9_{48}, 9_{49}\}
\end{eqnarray*}

The following $41$ knots are a basis for the kernel of the signature function:
\begin{eqnarray*}B_\sigma & = & \{ 4_1, 6_1, 6_3, 7_7, 8_1, 8_3, 8_8, 8_9, (8_{10}+3_1), (8_{11}-3_1), 8_{12}, 8_{13}, (8_{15}-7_2-3_1), 8_{17}, 8_{18}, 8_{20},\\
& & \quad (8_{21}-3_1), (9_2-7_4), (9_8-8_{14}), (9_{12}-5_2), 9_{14}, (9_{16}-7_3-3_1), 9_{19}, (9_{23}-7_4-3_1),\\
& & \quad 9_{24}, 9_{27},(9_{28}-3_1), (9_{29}+3_1), 9_{30}, 9_{33}, 9_{34}, 9_{37}, (9_{39}+7_2), (9_{40}-3_1), 9_{41},\\
& & \quad (9_{42}+8_5-3_1), 9_{44}, 9_{46}\} \\
& & \cup \,\, \{ (8_{17}-8_{17}^r), (9_{32}-9_{32}^r), (9_{33}-9_{33}^r) \}
\end{eqnarray*}
These knots must all have finite order in $\G$, and we will now analyse them further to determine what order they have in $\G$.

\section{Elements of algebraic order 4}
\label{Sec:ElementsOrder4}

In this section we begin our first analysis of the Witt groups $W(\F_p)$.  Recall Corollary \ref{Cor:AlgOrder4}, which states that

\begin{coralgorder4}
  If a knot $K$ is of order 4 in the algebraic concordance group $\G$ then for some prime $p \equiv 3$ mod 4 and some symmetric irreducible factor $g(t) \in \Z[t,t^{-1}]$ of $\Delta_K(t)$, $p$ divides $g(-1)$ and $g$ has odd exponent in $\Delta_K$.
\end{coralgorder4}

We can see immediately that if $\det(K)=\Delta_K(-1)$ factors into primes that are all congruent to $1$ mod $4$ then $K$ cannot be of order $4$. In this section we will thus consider the knots in $B_\sigma$ whose determinant contains a prime divisor that is congruent to $3$ modulo $4$. For each of these knots and for each prime congruent to $3$ mod $4$ that divides $\lambda(-1)$ for $\lambda$ a factor of $\Delta_K$ (including the factors of even exponent), we will calculate the image of $V+V^T$ in $W(\F_p)$ restricted to the $\lambda(t)$-primary component, via the map $\partial_p$ (see Theorem \ref{Thm:QpintoFp} and Example \ref{Example:partialp}).  This image is found by diagonalising $V+V^T$ to consist of square-free integers, removing the entries that are coprime to $p$ and dividing the rest by $p$.

The knots that we have to deal with are:
\begin{eqnarray*} & & 7_7, (8_{10}+3_1), (8_{11}-3_1), (8_{15}-7_2-3_1), 8_{18}, 8_{20}, (8_{21}-3_1),(9_2-7_4), (9_8-8_{14}), \\ & & (9_{12}-5_2), (9_{16}-7_3-3_1), (9_{23}-7_4-3_1), 9_{24}, (9_{28}-3_1),(9_{29}+3_1), 9_{34}, (9_{39}+7_2), \\ & & (9_{40}-3_1), (9_{42}+8_5-3_1), (9_{32} - 9_{32}^r).
\end{eqnarray*}
The last knot will clearly be zero in any Witt group $W(\F_p)$.

The results of the Witt group calculations can be found in Appendix \ref{AppendixA}.  From performing column operations on this matrix we can deduce the following:
\begin{itemize}
  \item That $\C_A^4:=\left\{ 7_7, 9_{34} \right\}$ are independent generators of $\Z_4$ summands in $\G$.
  \item That $\C^2_1:=\left\{(8_{15}-7_2-3_1), 8_{18}, (9_2-7_4), (9_{12}-5_2), (9_{42}+8_5-3_1)\right\}$ are independent generators of $\Z_2$ summands in $\G$.
  \item That $B_1:=\{(8_{21}-8_{18}-3_1), (9_{16}-8_{18}-7_3-3_1), (9_{23}-9_2-3_1), (9_{28}-8_{18}-3_1), (9_{29}-8_{18}+3_1), (9_{40}-8_{18}-3_1), (8_{10}+3_1), (8_{11}-3_1), 8_{20}, (9_8-8_{14}), 9_{24}, (9_{39}+7_2), (9_{32}-9_{32}^r) \}$ cannot be detected in the Witt groups $W(\F_p)$ for $p\equiv3$ mod $4$.
\end{itemize}

\begin{remark}
Notice that since $7_7$ and $9_{34}$ are detected at different primes (i.e. $7_7$ is non-trivial in $W(\F_7)$ while $9_{34}$ is not) it is clear that no linear combination $a7_7+b9_{34}$, $a,b \neq 0 \mod 4$, is algebraically slice. Similarly, we know that the elements of $\C^2_1$ are linearly independent since they are detected using different polynomials (even though some are detected using the same prime).
\end{remark}

It is worth further analysing the knots which have been detected as elements of order 2.  Since they have been detected using a $\Z_4$-valued function, a natural question is whether these elements of order 2 could be twice elements of algebraic order $4$.  Suppose that $K = 2L$ for some knot $L$ of order $4$.  Then the order of $K$ could not be detected by any $\Z_2$ invariants (for example, by the Witt groups $W(\F_p)$ with $p \equiv 1$ mod $4$).  Let us calculate the other Witt groups for the knots with a ``$1$ mod $4$'' prime in their determinant and see what they tell us.

\begin{center}
\begin{tabular}{|lcc|} \hline
Polynomial & $t^2-3t+1$ & $4t^2-7t+4$ \\ \hline
Prime & $5$ & $5$ \\ \hline
$\phantom{9_{12}}8_{18}$ & $(1,0)$ & $0$ \\
$(9_2-7_4)$ & $0$ & $(1,1)$ \\
$(9_{12}-5_2)$ & $(1,0)$ & $0$ \\\hline
\end{tabular}
\end{center}
Since these knots are all nontrivial in $W(\F_5)$, they cannot be twice another knot.

The knots $K_1:=(8_{15}-7_2-3_1)$ and $K_2:=(9_{42}+8_5-3_1)$ do not have any factors congruent to $1$ mod $4$ in their determinants and so we must try different tactics.  If $K_1 = 2L_1$ and $K_2 = 2L_2$ then
\[\Delta_{L_1}(t) = (1-t+t^2)(3-5t+3t^2)\]
\[\Delta_{L_2}(t)=(1-t+t^2)(1-2t+t^2-2t^3+t^4). \]
Both polynomials contain four distinct roots on the unit circle in $\mathbb{C}$.  This means that the signatures of $L_1$ and $L_2$ are non zero (by Theorem \ref{Thm:sigjump}), and so the knots cannot be of algebraic order $4$.

\section{The exponent map}

For the next section of the algebraic classification, we consider the map $\eps_{\lambda}$ which looks at the exponent of each of the symmetric irreducible factors $\lambda$ of $\Delta_K(t)$ and evaluates them mod $2$.  The knots now of concern are those in $B_1$ along with those in $B_\sigma$ whose determinants did not contain a $3$ mod $4$ prime factor.

At this point (although we could have done this at the beginning) we may isolate those knots whose Alexander polynomial does not contain any symmetric irreducible factors.  These are:
\[ B_2:= \{6_1, 8_8, 8_9, 9_{27}, 9_{41}, 9_{46} \} \]
and they are all algebraically slice.

The following knots have a unique (i.e. not found in any other knot) irreducible factor with odd exponent in their Alexander polynomial:
\[ \C_2^2 := \{ 6_3, 8_1, 8_3, 8_{12}, 8_{13}, 8_{17}, 9_{14}, 9_{19}, 9_{30}, 9_{33} \}. \]

The following knots' Alexander polynomials contain the factor $(1-3t+t^2)$ with odd exponent:
\[ 4_1, (9_{16}-8_{18}-7_3-3_1), 9_{24}, (9_{28}-8_{18}-3_1), (9_{29}-8_{18}+3_1), 9_{37}, (9_{39}+7_2), (9_{40}-8_{18}-3_1) \]
whilst those in the next list contain the factor $(1-4t+7t^2-4t^3+t^4)$:
\[ (9_{28}-8_{18}-3_1), (9_{29}-8_{18}+3_1), 9_{44}. \]
We take the first knot in each list (i.e. the one with smallest crossing number) as a generator, and differences of the others to be in the kernel of the exponent map.  Thus a basis of the kernel of the exponent map is
\begin{eqnarray*}B_3 & = & \{ (8_{10}+3_1), (8_{11}-3_1), 8_{20}, (8_{21}-8_{18}-3_1), (9_8-8_{14}),(9_{16}-8_{18}-7_3-4_1-3_1),\\
 & & \quad (9_{23}-9_2-3_1), (9_{24}-4_1), (9_{29}-9_{28}+2(3_1)), (9_{37}-4_1), (9_{39}+7_2-4_1),\\
 & & \quad (9_{40}-8_{18}-4_1-3_1), (9_{44}-9_{28}+8_{18}-4_1+3_1),\\
 & & \quad (8_{17}-8_{17}^r), (9_{32}-9_{32}^r), (9_{33}-9_{33}^r) \}.
\end{eqnarray*}

\section{More Witt group analysis}
\label{Sec:MoreWitt}

By Theorem \ref{Thm:DetDisc}, any elements of order $2$ in $B_3$ are detected only at primes that divide the determinant of the Seifert form (i.e. the leading coefficient of the Alexander polynomial), the discriminant of the (reduced) Alexander polynomial and in $W(\Q_2)$.  We need to examine the knots in $B_3$ in each of the relevant Witt groups.

The knots $\{(8_{17}-8_{17}^r), (9_{32}-9_{32}^r), (9_{33}-9_{33}^r)\}$ are clearly zero in any Witt group and are thus algebraically slice.

The only knot in $B_3$ which has an odd prime (namely, 3) as the leading coefficient of its Alexander polynomial is $9_{39}+7_2-4_1$.  However, this knot is trivial in $W(\F_3)$.  The only knot in $B_3$ whose discriminant contains an odd prime (namely, $17$) that is not a factor of the determinant is $9_{16}-8_{18}-7_3-4_1-3_1$, and this knot is also trivial in $W(\F_{17})$.

\begin{conjecture}
\label{Conj:DetPrime}
Every knot is trivial in $W(\F_p)$ for any odd prime $p$ which does not divide the determinant of the knot.
\end{conjecture}

A full analysis of the remaining knots and primes is to be found in Appendix \ref{AppendixB}.  From the table we see that
\begin{itemize}
  \item $\C_3^2:= \{ (9_{16}-8_{18}-7_3-4_1-3_1), (9_{44}-9_{28}+8_{18}-4_1+3_1)\}$ are independent generators of a $\Z_2$ summand.
  \item $B_4:=\{ (8_{21}-8_{18}-3_1), (9_{23}-9_2-3_1), (9_{24}-4_1), (9_{29}-9_{28}+2(3_1)), (9_{37}-4_1),\\ (9_{39}+7_2-4_1), (9_{40}-8_{18}-4_1-3_1)\} \cup \{ (8_{10}+3_1), (8_{11}-3_1), 8_{20}, (9_8-8_{14})\}$ cannot be detected in the Witt groups $W(\F_p)$ for $p \equiv 1$ mod $4$.
\end{itemize}

To prove that the knots in $B_4$ are algebraically slice, we must also check their image in the Witt group $W(\Q_2)$.  These all turn out to be trivial.

The knots $\{ (8_{10}+3_1), (8_{11}-3_1), 8_{20}, (9_8-8_{14}) \}$ do not have any $1$ mod $4$ primes in their determinants, but we must check their image in $W(\Q_2)$. These are all trivial.

\begin{conjecture}
\label{Conj:Witt2}
If a knot has trivial image in each $W(\F_p)$ for $p$ an odd prime, then it also has trivial image in $W(\Q_2)$.
\end{conjecture}

\section{Summary}
\label{Sec:Ch4summary}

Putting the last four sections of analysis together, we have the following theorem.

\begin{theorem}
\label{Thm:algclassification}
 A basis for the knots in $\CF_E$ of infinite algebraic concordance order is
 \begin{itemize}
   \item $\C^\infty = \{ 3_1, 5_1, 5_2, 6_2, 7_1, \dots, 7_6, 8_2, 8_4, \dots 8_7, 8_{14}, 8_{16}, 8_{19}, 9_1, 9_3, \dots, 9_7, 9_9, 9_{10}, 9_{11}, 9_{13}, \\
       \quad 9_{15}, 9_{17}, 9_{18}, 9_{20}, 9_{21}, 9_{22}, 9_{25}, 9_{26}, 9_{31}, 9_{32}, 9_{35}, 9_{36}, 9_{38}, 9_{43}, 9_{45}, 9_{47}, 9_{48}, 9_{49}\}$
 \end{itemize}
 while a basis for the kernel $\A_E$ of the map $\psi \circ \phi$, $\CF_E \xrightarrow{\psi} \C_E \xrightarrow{\phi} \G$, consists of the union of the independent sets  $4\C_A^4$, $2\C_A^2$ and $\C_A^1$ where
 \begin{itemize}
   \item $\C_A^4=\left\{ 7_7, 9_{34} \right\}$
   \item $\C_A^2= \C_1^2 \cup C_2^2 \cup \C_3^2 \cup \{4_1, (9_{28}-8_{18}-3_1)\}\\
          \phantom{\C_A^2} = \{4_1, 6_3, 8_1, 8_3, 8_{12}, 8_{13},(8_{15}-7_2-3_1), 8_{17}, 8_{18}, (9_2-7_4), (9_{12}-5_2), 9_{14},\\
          \phantom{\C_A^2 = \,\,} (9_{16}-8_{18}-7_3-4_1-3_1),  9_{19}, (9_{28}-8_{18}-3_1), 9_{30}, 9_{33},(9_{42}+8_5-3_1),\\
          \phantom{\C_A^2 = \,\,} (9_{44}-9_{28}+8_{18}-4_1+3_1)\}$
   \item $\C_A^1 = B_2 \cup B_4 \cup \{ (8_{17}-8_{17}^r), (9_{32}-9_{32}^r), (9_{33}-9_{33}^r)\}\\
         \phantom{\C_A^1} = \{6_1, 8_8, 8_9, (8_{10}+3_1), (8_{11}-3_1), (8_{17}-8_{17}^r), 8_{20}, (8_{21}-8_{18}-3_1), (9_8-8_{14}),\\
         \phantom{\C_A^1 = \,\,} (9_{23}-9_2-3_1), (9_{24}-4_1), 9_{27}, (9_{29}-9_{28}+2(3_1)), (9_{32}-9_{32}^r), (9_{33}-9_{33}^r), (9_{37}-4_1), \\
         \phantom{\C_A^1 = \,\,} (9_{39}+7_2-4_1), (9_{40}-8_{18}-4_1-3_1), 9_{41}, 9_{46}\}.$
 \end{itemize}
\end{theorem}

\begin{remark}
We may change basis in $\C_A^2$ to make things look nicer:
\begin{eqnarray*} \C_A^2 & = & \{4_1, 6_3, 8_1, 8_3, 8_{12}, 8_{13},(8_{15}-7_2-3_1), 8_{17}, 8_{18}, (9_2-7_4), (9_{12}-5_2), 9_{14}, \\
& & \quad (9_{16}-7_3-3_1), 9_{19}, (9_{28}-3_1), 9_{30}, 9_{33},(9_{42}+8_5-3_1),(9_{44}-4_1)\}
\end{eqnarray*}
\end{remark}

\begin{remark}
 The image of $\CF_E$ in the algebraic concordance group $\G$ is $\G_E \cong \Z^{46} \oplus \Z_2^{19} \oplus \Z_4^2$.
\end{remark} 
\chapter{\label{chapter5} Topological techniques}

Our ultimate aim in Chapter \ref{chapter6} will be to calculate the image of the algebraically slice $9$-crossing knots $\A_E \subset \CF_E$ in the concordance group $\C$.  There are few methods which are powerful enough to determine whether or not an algebraically slice knot is topologically slice.  One of these methods is due to Casson and Gordon \cite{CassonGordon86} (and was described in Section \ref{Sec:CGInvariants}), but calculations of their invariant are possible for only a small number of knots.  However, a determinant of the Casson-Gordon invariant turns out to be the twisted Alexander polynomial \cite{KirkLivingston99a} which has been developed by a number of people and which is relatively easy to compute.  We will concentrate on developing the use of the twisted Alexander polynomial to finish our classification of 9-crossing knots in $\C$.

To understand twisted Alexander polynomials, the following basic definition will be key.

\begin{definition}
Any finitely generated module $M$ over a principal ideal domain $R$ is the direct sum of cyclic modules
\[ F \oplus R/\langle a_1 \rangle \oplus \dots \oplus R/ \langle a_k \rangle\]
where $F$ is a free $R$-module and the $a_i \neq 1$ are defined modulo units in $R$ with $a_i$ dividing $a_{i+1}$ for all $i<k$. (See, for example, \cite[Theorem 6.3]{Grillet}.) The \emph{order} (of the torsion) of $M$ is defined to be the product of all the ideals $\langle a_i \rangle$.
\end{definition}

\begin{notation}
In this section $K$ will be an oriented knot, $\Sigma_n(K)$ will denote the $n$-fold cover of $S^3$ branched over $K$ and $X_n$ will denote the $n$-fold cyclic cover of the knot complement $X:=S^3\backslash K$. We write $\zeta_q$ for a complex number which is a $q^{\text{th}}$ root of unity and an overline, e.g. $\overline{z}$, to mean complex conjugation.
\end{notation}

\section{Twisted Alexander polynomials}

The original Alexander polynomial \cite{Alexander28} of a knot $K$ can be viewed as a description of the homology of the infinite cyclic cover $X_\infty$ of the knot complement $X = S^3 \backslash K$.  More precisely, $\Delta(t)$ is the order of the $\Z[t,t^{-1}]$-module $H_1(X_{\infty};\Z)$ (this is well-defined for knots, as the order ideal is principal).  This concept can be generalised by looking at representations of the knot group $\pi_1(X)$ and using them to twist the coefficients of $H_1(X_{\infty};\Z)$.  A detailed description of what follows may be found in \cite{KirkLivingston99a}.

\medskip

We will first work in greater generality than we need, so let $Y$ be a finite CW-complex with fundamental group $\pi := \pi_1(Y)$.  Let $\eps \colon \pi \to \Z$ be an epimorphism and $\rho \colon \pi \to \GL(V)$ be a homomorphism, where $V$ is a finite-dimensional vector space over a field $\F$.  The map $\rho$ determines a right $\Z[\pi]$-action on $V$, so we may construct the $(\F[t^{\pm 1}], \Z[\pi])$-bimodule
  \[ M := \F[t^{\pm 1}] \otimes_{\F} V \]
with right $\Z[\pi]$-action given by
  \[ (f(t) \otimes v) \cdot \gamma = t^{\eps(\gamma)} f(t) \otimes v \rho(\gamma) \,\, \text{ for } \gamma \in \pi \text{ .} \]
If $\widetilde{Y}$ is the universal cover of $Y$ then $\pi$ acts on the left of the cellular chain complex $C_*(\widetilde{Y})$ by deck transformations.  We define the twisted chain complex
  \[ C_*(Y;M) := M \otimes_{\Z[\pi]} C_*(\widetilde{Y}) \]

\begin{definition}
  The \emph{twisted Alexander polynomial} associated to $Y$, $\eps$ and $\rho$, denoted $\Delta_{Y,\eps,\rho}$, is the order of $H_1(Y;M):=H_1(C_*(Y;M))$ as a left $\F[t^{\pm 1}]$-module.  This is well-defined up to multiplication by units in $\F[t^{\pm 1}]$.
\end{definition}

\section{Twisted polynomials as slicing obstructions}
\label{Sec:twistpolyobstr}

To use twisted Alexander polynomials in a knot-slicing context, we will let $Y$ be $X_p$, the $p$-fold cyclic cover of the knot complement $X$ for a prime $p$, whilst $\F = V = \Q(\zeta_{q^r})$ for a prime $q \neq 2$.  This means that $M$ is $\Q(\zeta_{q^r})[t^{\pm1}]$.

A surjection $\eps' \colon \pi_1(X_p) \to~p\Z$ is obtained through composing the map on fundamental groups $\pi_1(X_p) \to \pi_1(X)$, induced by the covering map, with the Hurewicz homomorphism $\pi_1(X) \to H_1(X) \cong \Z$ (which is uniquely defined since our knot is oriented). Let $\eps \colon \pi_1(X_p) \to \Z$ be $\frac{1}{p}\eps'$.  We construct the representation $\rho$ in the following way, which is consistent with the representation used in the construction of Casson-Gordon invariants (see Section \ref{Sec:CGInvariants}).

Choose a character (group homomorphism) $\chi \colon H_1 (\Sigma_p;\Z) \to \Z_d$ where $d$ is a prime power $q^r$.  Precompose this map with the map on homology arising from the inclusion $X_p \hookrightarrow \Sigma_p$ and then with the Hurewicz homomorphism to get a map $\pi_1(X_p) \to \Z_{d}$.  We then have that $\Z_d$ maps into $\Q(\zeta_d)$ by
$i \mapsto \zeta^i_d$, and hence into $\Q(\zeta_d)^* = \Q(\zeta_d)\backslash \left\{0\right\}$.  All these operations give us the following composition of maps:
  \[ \pi_1(X_p) \to H_1(X_p) \to H_1(\Sigma_p) \xrightarrow{\chi} \Z_d \to \Q(\zeta_d)^*. \]
We shall denote the specific twisted Alexander polynomial obtained in this way by $\Delta_{K,\chi}$, or simply $\Delta_\chi$ where there is no confusion.

\begin{remark}
\label{Rmk:ReidemeisterTorsion}
 Milnor in \cite{Milnor62} interpreted the Alexander polynomial of a knot as the Reidemeister torsion of an associated chain complex. This allowed him to recover the result that the Alexander polynomial of a knot is symmetric, i.e. $\Delta_K(t) = \Delta_K(t^{-1})$ (up to a power of $t$). In \cite{Kitano96}, Kitano was able to show that twisted Alexander polynomials could also be interpreted as Reidemeister torsion. Kirk and Livingston~\cite{KirkLivingston99a} developed these ideas in the context above to show that, up to a factor of $rt^n$ with $r \in \Q(\zeta_d)$, the twisted Alexander polynomial $\Delta_\chi$ was the Reidemeister torsion of the complex
    \[C_*(X_p, V(t)) := (\F(t) \otimes_{\F} V) \otimes_{\rho} C_*(\widetilde{X_p})\]
 where $\F = V = \Q(\zeta_d)$ as before.
 This interpretation allowed Kirk and Livingston to conclude \cite[Corollary 5.2]{KirkLivingston99a} that twisted Alexander polynomials are symmetric in the sense that $\Delta_{\chi}(t) = \overline{\Delta_{\chi}(t^{-1})}$.
\end{remark}

The twisted Alexander polynomial defined above turns out to be equivalent to the Casson-Gordon determinant. Recall from Chapter \ref{chapter2} (Definition \ref{Def:norm}) the following notation.

 \begin{notation}
 A \emph{norm} in $\Q(\zeta_d)[t,t^{-1}]$ is a polynomial of the form $f(t)\overline{f(t^{-1})}$, where $\overline{f}$ means complex conjugation of the coefficients of $f$.
 \end{notation}

\begin{theorem}\cite[Theorem 6.5]{KirkLivingston99a}
Suppose that $p$ and $q$ are odd primes with $d=q^r$ and let $\chi$ be a character $\chi \colon H_1 (\Sigma_p;\Z) \to \Z_d$. Then
  \[ \disc(\tau(K, p, \chi)) = \Delta_{K,\chi}(t-1)^s \text{  modulo norms}\]
with $s=0$ or $s=1$ if $\chi$ is trivial or non-trivial respectively.
\end{theorem}

The work of Casson and Gordon in \cite[Lemma 4 \& corollary]{CassonGordon86} implies that if $K$ is a slice knot and if the character $\chi$ extends to the $p$-fold branched cover of the slice disc, then the Casson-Gordon invariant $\tau(K, p, \chi)$ vanishes. A particular set of characters which extend over the branched cover of the slice disc are those which vanish on a metaboliser of the linking form on $\Sigma_p$ (see \cite{Gilmer83} and \cite[Theorem 2]{CassonGordon86}). This linking form is analogous to Definition \ref{Def:linkingform} except that it is defined on the homology of the cyclic branched cover rather than on the homology of a Seifert surface for the knot.

\begin{definition}
  The \emph{linking form} \[\lk \colon H_1(\Sigma_p) \times
  H_1(\Sigma_p) \to \Q/\Z\] is a non-singular form defined as follows.  Let $x$ and $y$ be $1$-cycles in $\Sigma_p$.  Suppose
that $qx$ bounds a 2-chain $c$ for some $q \in \Z$.  Then
  \[ \lk( [x],[y]) = \frac{c \cdot y}{q} \in \Q / \Z, \]
where $c \cdot y$ is the intersection number of $c$ and $y$.  If $H_1(\Sigma_{p}) \cong \Z_q$, then $\lk$ takes values in $\Z_q$.
\end{definition}

There are various presentation matrices of $H_1(\Sigma_p)$ which we may use to calculate the linking form.
\begin{proposition}\cite[8D.5]{Rolfsen}
\label{Presentation1}
  Let $V$ be a Seifert matrix for $K$.  If $V$ is invertible, $(V^T V^{-1})^n - I$ is a presentation matrix for $H_1(\Sigma_n)$ as an abelian group.
\end{proposition}

\begin{proposition}\cite{Gilmer93}
\label{Presentation2}
  In the case that the Seifert matrix $V$ may not be invertible, let $G:= (V - V^T)^{-1}V$.  Then $G^n - (G - I)^n$ is a presentation matrix for
  $H_1(\Sigma_n)$ as an abelian group.
\end{proposition}

\begin{proposition}\cite[Theorem 9.1]{Lickorish}
  In the case of the two-fold branched cover, $V+V^T$ is a presentation matrix for $H_1(\Sigma_2)$ as an abelian group.
\end{proposition}

\begin{definition}
  An (invariant) \emph{metaboliser for} $\lk$ is a subgroup $M \subset H_1(\Sigma_p)$ (invariant under the action of the covering transformation) for which $M = M^\bot$, where
  \[ M^\bot := \left\{ x \in H_1(\Sigma_p) \: | \: \lk(x,y) = 0 \quad \forall \, y \in M \right\} \]
  It follows that $\text{order}(M)^2 = \text{order}(H_1(\Sigma_p))$.
\end{definition}

We now state the main theorem that will be used in the rest of this thesis. This theorem is a generalisation of the fact, proved by Fox and Milnor~\cite{FoxMilnor66} (and also in Corollary \ref{Cor:sliceinvts}), that Alexander polynomials for slice knots factor as $f(t)f(t^{-1})$.

\begin{theorem}\cite[Theorem 6.2]{KirkLivingston99a}
\label{Thm:sliceness}
  If $K$ is slice and $p,q$ are distinct primes with $q \neq 2$, then there is an (invariant) metaboliser of the linking form $M \subset H_1(\Sigma_p;\Z)$ with the following property.  For all characters $\chi \colon H_1(\Sigma_p;\Z) \to \Z_{q^r}$ with $\chi(M) = 0$, $\Delta_{K,\chi} (t)$ factors as
  \begin{equation}
  \label{Eq:factorisation}
    at^n \cdot f(t) \cdot \overline{f(t^{-1})} \cdot (t-1)^s,
  \end{equation}
  where $a \in \Q(\zeta_q)$, $f \in \Q(\zeta_q)[t^{\pm1}]$ and $s=0$
  or $s=1$ if $\chi$ is trivial or non-trivial, respectively.
\end{theorem}

It will often be easier to work with maps to $\Z_q$ rather than to $\Z_{q^r}$. The following corollary shows us how to do this.

\begin{corollary}\cite[Corollary 8.3]{HeraldKirkLivingston10}
\label{Thm:slicenessmodq}
  If $K$ is slice and $p,q$ are distinct primes with $q \neq 2$, then there is a subspace of the linking form $\widetilde{M} \subset H_1(\Sigma_p;\Z_q)$ with the following property.  For all characters $\chi \colon H_1(\Sigma_p;\Z_q) \to \Z_q$ with $\chi(\widetilde{M}) = 0$, $\Delta_{K,\chi} (t)$ factors as
  \[   a \cdot f(t) \cdot \overline{f(t^{-1})} \cdot (t-1)^s \]
  where $a \in \Q(\zeta_q)$, $f \in \Q(\zeta_q)[t^{\pm1}]$ and $s=0$
  or $s=1$ if $\chi$ is trivial or non-trivial, respectively.  The subspace $\widetilde{M}$ is a reduction modulo $q$ of a metaboliser for $\lk$.
\end{corollary}

To prove that a knot is not slice, it therefore suffices to show that every metaboliser $M$ of $H_1(\Sigma_p)$ (for some prime $p$) has a non-trivial map $\chi \colon H_1(\Sigma_p) \to \Z_{q^r}$ (for some prime $q$) such that $\chi(M)=0$ and such that $\Delta_\chi/(t-1)$ is not a norm.

In order to use twisted polynomials to detect torsion in $\C$, we need to know how they behave under connected sums of knots.  The following result was proved by Gilmer~\cite{Gilmer83} in the context of Casson-Gordon invariants, and by Kirk and Livingston \cite{KirkLivingston99b} for twisted Alexander polynomials.

\begin{proposition}
\label{Prop:additivity}
  Suppose $K = K_1 \# K_2$ and that $\chi$ restricts to $\chi_i$ on the factors (i.e. $\chi = \chi_1 \oplus \chi_2$).  Then
    \[ \Delta_{K,\chi}(t) = \Delta_{K_1,\chi_1}(t) \Delta_{K_2,\chi_2}(t) \mbox{ or } \Delta_{K_1,\chi_1}(t) \Delta_{K_2,\chi_2}(t) (1-t)  \]
  with the second case occurring if and only if $\chi_1$ and $\chi_2$ are non-trivial.
\end{proposition}

\section{Detecting infinite order}

Twisted Alexander polynomials are a powerful tool in proving non-sliceness of algebraically slice knots (see, e.g. \cite{KirkLivingston99b}, \cite{HeraldKirkLivingston10}).  They were used in \cite{KirkLivingston01} to prove that a specific collection of knots had infinite order and were linearly independent in $\C$.  In this section we develop techniques and theorems which will enable us to decide whether any collection of arbitrary knots is linearly independent in $\C$. For the rest of the chapter we abuse notation and write $\Delta_{\chi}$ when we mean $\Delta_{\chi}/(t-1)$.

\subsection{From metabolisers to polynomials}

Our strategy for showing that a knot has infinite order in $\C$ is to analyse metabolisers of $H_1(\Sigma_p(nK))$ for an arbitrary $n \in \N$ and to show that for each metaboliser there is a character $\chi$ vanishing on it for which $\Delta_\chi$ is not a norm.  For most of what follows, $p$ will be equal to $2$ and $q$ will be an odd prime.

\medskip

What does a metaboliser for $H := H_1(\Sigma_2(nK);\Z)$ look like?  Suppose first that $H_1(\Sigma_2(K);\Z) \cong \Z_q$.  Then $H \cong (\Z_q)^{n}$.  In order for a metaboliser to exist, we need $n$ to be even, say $n=2m$.  Let $M$ be a metaboliser.  It is spanned by $m$ vectors $v_1, \dots, v_m \in (\Z_q)^{2m}$ where $v_i \cdot v_j \equiv 0$ (mod $q$) for every $i$ and $j$. (Here we are using the standard inner product on $(\Z_q)^{2m}$.)  This is because, if we write $v_i = (v_{i,1}, \dots, v_{i,2m})$ we have
 \begin{eqnarray*}
 0 & = & \lk_{2mK}(v_i,v_j) = lk_K(v_{i,1},v_{j,1}) + \dots + lk_K(v_{i,2m},v_{j,2m})\\
   & = & \lk_K(1,1) (v_{i,1}v_{j,1} + \dots + v_{i,2m}v_{j,2m})\\
   & = & \lk_K(1,1) v_i \cdot v_j
 \end{eqnarray*}
 and since $\lk_K(1,1) \neq 0$ we have $v_i \cdot v_j \equiv 0$.

\begin{example}
For example, let $m=3$ and $q=5$.  Then the following set of vectors could be a spanning set for a metaboliser:
    \[ \begin{array}{lcl}
      v_1 & = & (1,0,0,1,2,2)\\
      v_2 & = & (0,1,0,2,1,-2)\\
      v_3 & = & (0,0,1,2,-2,1).\\
    \end{array} \]
\end{example}

Next, what characters vanish on $M$?  Let $\chi_i \colon H_1(\Sigma_2(K);\Z) \to \Z_q$ be the map $\lk(-,i)$.  Then if $w \in M$ we can write $w = (w_1,\dots,w_{2m})$ and we have that the map
  \[ \chi_w := \chi_{w_1} \oplus \dots \oplus \chi_{w_{2m}}\]
will vanish on all of $M$.  By Proposition \ref{Prop:additivity} we have that
  \begin{equation}
  \label{Eq:metabvector}
    \Delta_{\chi_w} = \Delta_{\chi_{w_1}} \dots \Delta_{\chi_{w_{2m}}}
  \end{equation}
and this should be a norm if $K$ is slice.  The following observation will help us to simplify this polynomial.

A twisted Alexander polynomial $p(t)$ has the symmetry $p(t)~=~\overline{p(t^{-1})}$ (see \ref{Rmk:ReidemeisterTorsion}).  In particular, this means that $p(t)^2$ is a norm. Thus if a vector $w \in M$ contains a repeated entry, say $w_i = w_j$, then we can simplify equation (\ref{Eq:metabvector}) by removing $\Delta_{\chi_{w_i}}(t)$ and $\Delta_{\chi_{w_j}}(t)$ from the right-hand side.  So it is only if an entry in a metabolising vector occurs an odd number of times that it contributes to the twisted Alexander polynomial.

There is a further simplification we can make when doing calculations, which is that the polynomials related to the various $\chi_i$, $i \in \Z_q$, are all related by a particular symmetry.

\begin{lemma}
  Suppose that $\chi \colon H \to \Z_q$ and $\chi' \colon H \to \Z_q$ are two characters such that $\chi' = n \chi$ for some $n \in \Z_q$.  Then $\Delta_{\chi'} = \sigma_n(\Delta_\chi)$, where $\sigma_n \colon \Q(\zeta_q) \to \Q(\zeta_q)$ is the map that takes $\zeta_q$ to $\zeta_q^n$.  We say that $\Delta_\chi$ and $\Delta_{\chi'}$ are \emph{Galois conjugates} of each other.  Notice that if $\Delta_\chi$ factorises as a norm then so does $\Delta_{\chi'}$.
\end{lemma}

So in order to stop Equation (\ref{Eq:metabvector}) from factorising as a norm we need to find a vector $w \in M$ which contains at least one entry an odd number of times. We then need to show that none of other polynomials on the right-hand side of (\ref{Eq:metabvector}) are the same as the polynomial corresponding to that `odd' entry. One way to show this would be to show that none of the Galois conjugates of $\Delta_{\chi_1}$ are equal to each other, but this is slightly overkill and we will see that we can get away with a lesser condition.
All these ideas will be made more precise in the next section.

Finally, we also need to be able to deal with the trivial twisted polynomial, $\Delta_{\chi_0}$.  This polynomial is related to the standard Alexander polynomial, but does not always behave in the same way.

\begin{lemma}
\label{Lemma:algslicenorm}
  If $K$ is algebraically slice then the trivial twisted Alexander polynomial $\Delta_{\chi_0}(t) := \text{order} (H_1(X_p,\Q[t^{\pm1}]))$ is a norm.  However the converse is not true: it may happen that $\Delta_{\chi_0}(t)$ is a norm whilst the standard Alexander polynomial $\Delta_K(t)$ does not factorise as $f(t)f(t^{-1})$.
\end{lemma}

\begin{proof}
  The trivial twisted polynomial is related to the standard Alexander polynomial by the following formula:
    \[ \Delta_{\chi_0}(t) = \prod_{i=0}^{p-1}\Delta_K(\zeta_p^i t^{\frac{1}{p}}). \]
  If $K$ is algebraically slice then $\Delta_K(t) = f(t)f(t^{-1})$ for some $f$ in $\Q[t^{\pm1}]$.  It follows trivially that $\Delta_{\chi_0}$ is a norm.

  Conversely, let $p=2$ and $\Delta_K(t)= 2-3t^2+2t^4$ (which is irreducible). Then $\Delta_{\chi_0}(t) = (2-3t+2t^2)^2$, which is a norm.
\end{proof}

\subsection{Odd vectors and the main theorem}

We will use the following lemma from \cite{KirkLivingston01} to find metabolising vectors whose entries occur an odd number of times.

\begin{lemma}
\label{Lemma:KirkLiv}
Let $E$ be a non-singular $m \times m$ matrix over $\Z_q$ for a prime $q>2$.  Suppose that every vector in the subspace of $(\Z_q)^{2m}$ spanned by the rows of the $m \times 2m$ matrix $(I \, \, E)$ has an even number of non-zero entries. Then $E$ is obtained from a diagonal
matrix by permuting the columns.
\end{lemma}

We adapt this result to our own situation.

\begin{lemma}
\label{Lemma:oddvector}
 Let $K=K_1+\dots+K_{2m}$ be a knot, where $H_1(\Sigma_2(K_i);\Z) \cong \Z_q \oplus T_{i}$ for each $i=1,\dots,2m$, with the order of the $T_{i}$ coprime to $q$.  Let $v_1, \dots, v_m$ be vectors in $(\Z_q)^{2m}$ which are a basis for a metaboliser $M \subset H_1(\Sigma_2(K);\Z_q)$ on which the linking form vanishes.  If the linking forms of $K_1,\dots,K_{2m}$ are the same (i.e. $\lk_{K_i}(1,1) = \lk_{K_j}(1,1)$ for all $i,j$) then
 \begin{itemize}
 \item If $q \equiv 1$ (mod $4$) then either there is a linear combination of the $v_i$ which contains an odd number of zero entries or there is a linear combination of the $v_i$ of the form $(1,0,\dots,0,\pm a,0,\dots,0)$ for an $a$ such that $1+a^2 \equiv 0$ (mod q).
 \item If $q \equiv 3$ (mod $4$) then $m$ is even and there must be a linear combination of the $v_i$ which contains an odd number of zero entries.
 \end{itemize}
\end{lemma}

\begin{proof}
Given any basis $v_1, \dots, v_m \in (\Z_q)^{2m}$ for a metaboliser we can perform row and column operations to put
the basis into the $m \times 2m$ matrix form $(I \, \, E)$ required by Lemma \ref{Lemma:KirkLiv}.  We need to show that $E$ is non-singular.

Different columns of $(I \, \, E)$ correspond to elements in the homology of the branched cover of different knots, and these different knots may have different linking forms.  Let $D$ be the diagonal matrix with diagonal entries $\lk_{K_1}(1,1),\dots,\lk_{K_{2m}}(1,1)$ (noting that these may be in a different order because of the aforementioned column operations). Since the rows are part of a metaboliser, we have
\[ \left(\begin{array}{cc}I & E \end{array}\right)D\left(\begin{array}{c}I \\ E^T \end{array}\right)
= \left(\begin{array}{cc}I & E \end{array}\right)\left(\begin{array}{cc}D_1 & 0 \\0 & D_2 \end{array}\right)\left(\begin{array}{c}I \\ E^T \end{array}\right)= D_1+ED_2E^T \equiv 0 \;\;(\text{mod } q).\]
Rearranging and taking the determinant of both sides gives us
\[ \det(E)^2\det(D_2) \equiv \det(-D_1) \;\;(\text{mod } q). \]
Since $\det(D_1)$ and $\det(D_2)$ are nonzero, $E$ is nonsingular too, as required.

We can conclude from Lemma \ref{Lemma:KirkLiv} that either there is a vector in the span of the $v_i$ with an odd number of non-zero entries, or that $E$ is obtained from a diagonal matrix by permuting the columns.  However, if the linking forms of the $K_i$ are identical we may deduce more.  We have $D=\alpha I$ so $EE^T \equiv -I$ (mod $q$) and $\det(E)^2 \equiv (-1)^m$ (mod $q$).  If $E$ is obtained from a diagonal matrix by permuting the columns then the diagonal entries must be $\pm a$, where $-1 \equiv a^2$ (mod $q$).  Such an element does not exist if $q \equiv 3$ (mod $4$) since $-1$ is not a square.  For the same reason, if $q \equiv 3$ (mod $4$) then $m$ must be even.
\end{proof}

By a theorem of Livingston and Naik~\cite[Theorem 1.2]{LivingstonNaik99} we know that if $H_1(\Sigma_2(K);\Z) \cong \Z_{q^n}$ for any prime $q \equiv 3$ (mod $4$) and $n$ odd, then $K$ is of infinite order in $\C$.  In this next theorem we give conditions for a knot to be of infinite order when $q \equiv 1$ (mod $4$) and $n=1$.  We will then generalise this to deal with connected sums of knots, and in the next section look at what happens when $H_1(\Sigma_2(K);\Z) \cong \Z_{q^n}$ for any value of $n$.

\begin{notation}
We use the symbol $\doteq$ to denote `modulo norms'.
\end{notation}

\begin{theorem}
\label{Thm:twistedpoly}
Suppose that we have a knot $K$ where $H_1(\Sigma_2;\Z) \cong \Z_{q} \oplus T$ for some prime $q \equiv 1$ mod $4$, where the order of $T$ is coprime to $q$.  Let $\chi_0 \colon H_1(\Sigma_2;\Z) \to \Z_q$ be the trivial map and $\chi_i \colon H_1(\Sigma_2;\Z) \to \Z_q$ be $\lk(-,i)$.  Construct $\Delta_{\chi_i}(t)$ as in Section \ref{Sec:twistpolyobstr}. Then $K$ is of infinite order if it satisfies the following conditions:
  \begin{enumerate}
    \item $\Delta_{\chi_0}(t)$ does not factorise over $\Q(\zeta_q)[t,t^{-1}]$ as $g(t)\overline{g(t^{-1})}$.
    \item There is a non-trivial irreducible factor $f(t)$ of $\Delta_{\chi_0}(t)$ for which $\overline{f(t^{-1})}$ is not a factor of $\Delta_{\chi_i}(t)$ for any $i$.
    \item $\Delta_{\chi_a}(t) \not\doteq \Delta_{\chi_1}(t)$, where $1+a^2 \equiv 0$ (mod $q$).
  \end{enumerate}
\end{theorem}

\begin{proof}
Suppose that $K$ satisfies conditions (1-3), and suppose that $(2m)K$ is slice (for some $m \in \N)$, and try to obtain a contradiction.  Since $2mK$ is slice, there is a metaboliser $M$ of rank $m$.

For each element $w=(w_1, \dots, w_{2m})$ of $M$ we have a character
  \[ \chi_w = \chi_{w_1} \oplus \dots \oplus \chi_{w_{2m}} \]
which vanishes on $M$, and a twisted Alexander polynomial
  \begin{equation}
  \label{Eqn:polyproduct}
    \Delta_{\chi_w} (t) = \Delta_{\chi_{w_1}}(t) \dots \Delta_{\chi_{w_{2m}}}(t).
  \end{equation}
By the sliceness of $2mK$, $\Delta_\chi(t)$ must factorise as a norm for every such $\chi$, up to a factor of an element of $\Q(\zeta_q)$ and a possible factor of $(t-1)$ (Theorem \ref{Thm:sliceness}).  Before embarking on the rest of the argument, we notice the trivial but important point that $\Q(\zeta_q)[t,t^{-1}]$ is a unique factorisation domain.

Recall that any twisted Alexander polynomial $x(t)$ has the property that $x(t) = \overline{x(t^{-1})}$.  Thus if $\Delta_{\chi_i}$ occurs with an even exponent on the RHS of (\ref{Eqn:polyproduct}) then we can ignore it because it factorises as a norm.  By assumption 1 and 3 (notice that assumption 3 implies that $\Delta_{\chi_i}$, $i \neq 0$, is not a norm) , the converse is also true: if a polynomial occurs with an odd exponent on the RHS of (\ref{Eqn:polyproduct}) then it cannot factorise (by itself, at least) as a norm.  We now use the result of Lemma \ref{Lemma:oddvector}, and find ourselves with two cases to consider.

In the first case there is a vector $w \in M$ which contains an odd number of zero entries.  So $\Delta_{\chi_0}(t)$ occurs with an odd exponent on the RHS of (\ref{Eqn:polyproduct}) (i.e. non-trivially). By assumption 2, we know that $\Delta_{\chi_0}(t)$ contains at least one non-trivial irreducible factor which cannot combine with any other twisted Alexander polynomial to factorise as a norm.  Thus $\Delta_{\chi}(t)$ cannot factorise as a norm, which is a contradiction.

In the second case there is a vector $w \in M$ which looks like $(1,0,\dots,0,\pm a,0,\dots,0)$ for an $a$ such that $1+a^2 \equiv 0$ (mod q).  In this case we can simplify the right-hand side of (\ref{Eqn:polyproduct}) to get
  \[ \Delta_\chi (t) \doteq \Delta_{\chi_1}(t) \Delta_{\chi_{\pm a}}(t). \]
From assumption 3 we know that $\Delta_{\chi_1}(t) \not\doteq \Delta_{\chi_ {\pm a}}(t)$, so this polynomial cannot factorise as a norm, completing the contradiction.
\end{proof}

\begin{remark}
For any $i \in \Z_q$, the twisted polynomials $\Delta_{\chi_i}$ and $\Delta_{\chi_{-i}}$ are equal because of the action of the 2-fold deck transformation on $\Sigma_2$.
\end{remark}

This theorem is quite powerful; we will see in Chapter \ref{chapter8} that it applies to 84\% of the knots of unknown concordance order with up to 12 crossings.  However, one obvious situation when it cannot be applied is when the knot is algebraically slice, since the first condition is automatically not satisfied (see Lemma \ref{Lemma:algslicenorm}).  Thankfully, there is an alternative method of checking if a knot has infinite order in $\C$ which does not rely on assuming that $\Delta_{\chi_0}$ is not a norm.

\begin{theorem}
\label{Thm:twistedpolyAlt}
Suppose that we have a knot $K$ where $H_1(\Sigma_2;\Z) \cong \Z_{q} \oplus T$ for some prime $q \equiv 1$ mod $4$, where the order of $T$ is coprime to $q$.  Let $\chi_0 \colon H_1(\Sigma_2;\Z) \to \Z_q$ be the trivial map and $\chi_i \colon H_1(\Sigma_2;\Z) \to \Z_q$ be $\lk(-,i)$.  Suppose that $\Delta_{\chi_0}$ is a norm.  Then $K$ is of infinite order if it satisfies the following condition:
  \begin{itemize}
    \item $\Delta_{\chi_i}(t)$ is coprime, up to norms in $\Q(\zeta_q)[t,t^{-1}]$, to $\Delta_{\chi_j}(t)$ for all $i \neq j$, $i,j>0$.
  \end{itemize}
\end{theorem}

\begin{proof}
The proof is identical to that of Theorem \ref{Thm:twistedpoly} except for the case where the metaboliser contains a vector with an odd number of zeros.  We now know that $\Delta_{\chi_0}$ is a norm, so we need a different argument.  A vector with an odd number of zeros also contains an odd number of non-zero entries.  This means that at least one entry, say $i$, occurs an odd number of times and so its corresponding twisted Alexander polynomial $\Delta_{\chi_i}$ occurs with odd exponent in the polynomial for that vector.  By the condition of our theorem we know that $\Delta_{\chi_i}$ is coprime to all other twisted polynomials, so there is no other factor in the product of polynomials which can combine with $\Delta_{\chi_i}$ to form a norm.  We also know that $\Delta_{\chi_i}$ is not a norm, because if it were then all of its Galois conjugates would be norms as well, and the condition would not be satisfied.  Thus the twisted Alexander polynomial corresponding to this `odd vector' is not a norm.

As in the previous proof, we also require that $\Delta_{\chi_1}(t) \Delta_{\chi_{a}}(t)$ is not a norm, for $1+a^2 \equiv 0$ (mod $q$).  This fact is taken care of by our condition too.
\end{proof}

We may generalise Theorem~\ref{Thm:twistedpoly} to concern sums of knots, though a little care must be taken to ensure that each individual knot satisfies the set-up of the theorem.  Notice that, whilst each individual knot in the sum must not be algebraically slice, there is no restriction on the algebraic sliceness of the sum.

\begin{remark}
The following theorem has an alternative version, corresponding to Theorem \ref{Thm:twistedpolyAlt}, where we don't need to assume that the component knots of the sum are not algebraically slice.  However, the following theorem suffices to prove our 9-crossing classification, so we do not write out this alternative theorem and instead leave it as an exercise for the reader.
\end{remark}

\begin{theorem}
\label{Thm:infiniteordersum}
Let $K=K_1 + \dots + K_n$ with $K_{i_1},\dots,K_{i_{n'}}$ having $H_1(\Sigma_2(K_{i_j});\Z) \cong \Z_q \oplus T_j$ with the order of the $T_j$ coprime to $q$, and the remainder of the $K_i$ having $H_1(\Sigma_2(K_i);\Z) \cong T_i$ with the order of the $T_i$ coprime to $q$.  Then $K$ has infinite order in $\C$ if the twisted Alexander polynomials (as defined in Theorem~\ref{Thm:twistedpoly}) satisfy:

\begin{enumerate}
  \item $\Delta_{\chi_0}^{K_{i_j}}(t)$ does not factorise over $\Q(\zeta_q)[t,t^{-1}]$ as a norm for any $j=1,\dots,n'$.
  \item There is a non-trivial irreducible factor $f_j(t)$ of $\Delta_{\chi_0}^{K_{i_j}}(t)$ for which $\overline{f_j(t^{-1})}$ is not a factor of $\Delta_{\chi_\alpha}^{K_{i_k}}(t)$ for all $j,k=1,\dots,n'$ and all $\alpha \neq 0$.
  \item $\Delta_{\chi_1}^{K_{i_j}}(t)\not\doteq \Delta_{\chi_\gamma}^{K_{i_k}}(t)$ where $\gamma = \sqrt{-\alpha\beta^{-1}}$ with $\alpha = \lk_{K_{i_j}}(1,1)$ and $\beta = \lk_{K_{i_k}}(1,1)$, for all $j,k$ where $\gamma$ is defined. This means that if $q\equiv 1$ mod $4$ then $\alpha$ and $\beta$ must be the same modulo squares and if $q \equiv 3$ mod 4 then $\alpha$ and $\beta$ must be different modulo squares.
\end{enumerate}
\end{theorem}

\begin{proof}
Suppose that $(2m)K$ is slice and that a basis of a metaboliser in $H_1(\Sigma_2(2mK);\Z_q)$ is
\begin{eqnarray*}
v_1 & = & (a_{1,1}, \dots, a_{1,2m},b_{1,1},\dots,b_{1,2m},\dots,0,\dots, 0)\\
v_2 & = & (a_{2,1}, \dots, a_{2,2m},b_{2,1},\dots,b_{2,2m},\dots,0,\dots, 0)\\
\vdots\\
v_{mn'} & = & (a_{mn',1}, \dots, a_{mn',2m},b_{mn',1},\dots,b_{mn',2m},\dots,0,\dots, 0)
\end{eqnarray*}
Since we are considering an even multiple of $K$, the number of zeros at the end of each vector is even and can be ignored.  Taking the $mn'\times2mn'$ matrix $V$ whose rows are the entries of the $v_i$ from the $n'$ non-trivial knots, we perform row and column operations on $V$ to turn it into $V' = (I \, \, E)$.  The proof of Lemma \ref{Lemma:oddvector} adapts to this situation so we know that $E$ is non-singular and we can therefore invoke the conclusions of Lemma~\ref{Lemma:oddvector}.  There are two cases.  (For what follows, remember that each column of $V'$ corresponds to one of the $K_{i_j}$.)

Firstly, we may be able to find a linear combination of the rows of $V'$ which contains an odd number of zeros.  An odd number of zeros must therefore come from a particular knot $K_i$.  This means that $p(t) := \Delta_{\chi_0}^{K_i}(t)$ occurs with odd exponent in the twisted Alexander polynomial $\Delta_\chi$ corresponding to that vector.  From condition (1) we know that $p(t)$ is not a norm, and from condition (2) we know that it cannot factorise as a norm as a result of being combined with non-trivial twisted polynomials from any of the other knots.  The only way that it could become a norm is to be combined with one of the other trivial twisted polynomials, say $q(t)$.  Since $q(t)$ is not a norm it must also occur with odd exponent.  But this pairing up of non-trivial twisted polynomials cannot continue indefinitely, since we know there are an odd number of zeros in our vector. There must therefore be a factor of $\Delta_\chi$ which is not a norm and which occurs with odd exponent, meaning that $\Delta_\chi$ is itself not a norm.

If there is no combination of the rows of $V'$ which contains an odd number of zeros, Lemma \ref{Lemma:KirkLiv} says that $E$ is obtained by permuting the columns of a diagonal matrix.  Let the diagonal matrix have entries $a_1, \dots, a_{mn'}$.  Every row in $V'$ then looks like
  \[ (0,\dots,0, 1, 0, \dots, 0, a_i, 0, \dots, 0).\]
In any particular row, suppose the knot corresponding to the column containing the $1$ has the linking form $\alpha$ and the knot corresponding to the column with the $a_i$ has the linking form $\beta$.  Then we would have $\alpha+\beta a_i^2 \equiv 0$ (mod $q$), which means $a_i^2 \equiv -\alpha \beta^{-1}$ (mod $q$).  If $q\equiv 1$ (mod $4$) then we need $\alpha$ and $\beta$ to be either both squares or both non-squares for there to be a solution.  If $q \equiv 3$ (mod $4$) then we need $\alpha$ and $\beta$ to be different modulo squares.  Assuming that there is a solution $a_i = \gamma$, this row then gives us the twisted polynomial
  \[ \Delta_\chi \doteq \Delta_{\chi_1}^{K_{i_j}} \Delta_{\chi_\gamma}^{K_{i_k}}\]
which by condition (3) is designed not to be a norm.
\end{proof}

\section{Prime powers}
\label{Section:powers}

Suppose $H_1(\Sigma_2(K);\Z)\cong \Z_{q^n}$ for a knot $K$, a prime $q \neq 2$ and a natural number $n$.  We want a criterion for deciding if $K$ is of infinite order.  Livingston and Naik have the following theorem in the case when $q \equiv 3$ (mod $4$) and when $n$ is odd.

\begin{theorem}[1.2, \cite{LivingstonNaik99}]
\label{Thm:LivNaik3mod4}
If $H_1(\Sigma_2(K)) = \Z_{q^n} \oplus G$ with $q$ a prime congruent to $3$ modulo $4$, $n$ odd and $q$ not dividing the order of $G$, then $K$ is of infinite order in $\C$.
\end{theorem}

But what happens when $q \equiv 1$ (mod $4$) or if $n$ is even?  The proof of Theorem \ref{Thm:LivNaik3mod4} made use of Casson-Gordon signatures, which are not powerful enough to provide an obstruction in these cases.  However, we shall see that twisted Alexander polynomials do provide computable obstructions, much as in the $n=1$ case.

\begin{theorem}
\label{Thm:twistedpolyq^n}
Suppose that we have a knot $K$ where $H_1(\Sigma_2;\Z) \cong \Z_{q^n} \oplus T$ for some prime $q$, where the order of $T$ is coprime to $q$.  Let $\chi_0 \colon H_1(\Sigma_2;\Z) \to \Z_{q}$ be the trivial map and $\chi_i \colon H_1(\Sigma_2;\Z) \to \Z_{q}$ be $\lk(-,i)$ (mod $q$).  Construct $\Delta_{\chi_i}(t)$ as in Section \ref{Sec:twistpolyobstr}. Then $K$ is of infinite order if it satisfies the following conditions:
  \begin{enumerate}
    \item If $n>1$, $\Delta_{\chi_0}(t)\Delta_{\chi_1}(t)$ is not a norm.
    \item $\Delta_{\chi_0}(t)$ is not a norm.
    \item $\Delta_{\chi_0}(t)$ is coprime, up to norms in $\Q(\zeta_q)[t,t^{-1}]$, to $\Delta_{\chi_i}(t)$ for all $i \neq 0$.
    \item If $q \equiv 1$ (mod $4$) then $\Delta_{\chi_a}(t) \not\doteq \Delta_{\chi_1}(t)$, where $1+a^2 \equiv 0$ (mod $q$).
  \end{enumerate}
\end{theorem}

\begin{remark}
Conditions (2)-(4) may be replaced by the condition that $\Delta_{\chi_i}$ should be coprime (modulo norms) to $\Delta_{\chi_j}$ for any $i \neq j$, $i,j >0$.  The proof below is easily adjusted to this other case using the same proof as for Theorem \ref{Thm:twistedpolyAlt}.
\end{remark}

\begin{proof}
We will discuss the proof of the case when $H_1(\Sigma_2;\Z) \cong \Z_{q^n}$; the case of $H_1(\Sigma_2;\Z) \cong \Z_{q^n} \oplus T$ is identical except that everywhere we would need to consider the $q$-primary component of the homology.

Suppose that $2mK$ is slice.  Then $H := H_1(\Sigma_2(2mK);\Z) \cong (\Z_{q^n})^{2m}$, so any metaboliser has order $q^{nm}$.  Let $M$ denote the matrix whose rows are the spanning set of a metaboliser.  Following the proof of \cite[1.2]{LivingstonNaik99}, we let $k_i$ denote the number of rows in $M$ divisible by $q^i$.  Since the metaboliser $M$ has the property that $H/M \cong M$, we have that $k_i = k_{n-i}$ for $i > 0$.

\begin{example}
Let $n=4$ and $m=5$, so $H = (\Z_{q^4})^{10}$.  A metaboliser needs to have order $q^{20}$.  Here is one possibility:
\[ \left(\begin{array}{cccccccccc}
1 & * & * & * & * & * & * & * & * & * \\
0 & q & 0 & 0 & * & * & * & * & * & * \\
0 & 0 & q & 0 & * & * & * & * & * & * \\
0 & 0 & 0 & q & * & * & * & * & * & * \\
0 & 0 & 0 & 0 & q^2 & 0 & * & * & * & * \\
0 & 0 & 0 & 0 & 0 & q^2 & * & * & * & * \\
0 & 0 & 0 & 0 & 0 & 0 & q^3 & 0 & 0 & * \\
0 & 0 & 0 & 0 & 0 & 0 & 0 & q^3 & 0 & * \\
0 & 0 & 0 & 0 & 0 & 0 & 0 & 0 & q^3 & * \\
\end{array}\right) \]
Here $k_0 = 1$, $k_1 = k_3 = 3$ and $k_2 = 2$.
\end{example}

As per Corollary \ref{Thm:slicenessmodq} we need to find a map $\chi \colon H_1(\Sigma_2) \to \Z_q$ which vanishes on the image modulo $q$ of a metaboliser, and for which the corresponding twisted Alexander polynomial does not factorise as a norm.  The image modulo $q$ of $M$ is clearly just the first $k_0$ rows.

\textbf{Case 1: $k_0=0$}

If all the rows are divisible by $q$, then any character from $H_1(\Sigma_2)$ to $\Z_q$ will vanish on $M$.  In particular we can take $\chi_1 \bigoplus_{2m-1} \chi_0$, giving us the condition that $\Delta_{\chi_1}\Delta_{\chi_0}$ is a norm if $2mK$ is slice.  This is obstructed by condition (1) of the theorem.

\textbf{Case 2: $k_0 = m$}

In this case, all other $k_i$ are zero and every row has order $q^n$. The image of $M$ under the map which reduces every entry modulo $q$ has rank $m$ and the form $M' :=(I \, \, E)$ for a square matrix $E$. Since the dot product of any two rows in $M$ is zero modulo $q^n$, the dot product of any two rows of $M'$ is zero modulo $q$. By Lemma \ref{Lemma:oddvector} we then know that $E$ is non-singular and that either (a) there is a vector in the span of the rows of $M'$ with an odd number of zero entries, or (b) $E$ is obtained from a diagonal matrix by permuting the columns. In case (a) the corresponding twisted polynomial is obstructed from being a norm by conditions (2) and (3) of the theorem. In case (b) the rows in $E$ have the form $(1,0,\dots,0,\pm a,0,\dots,0)$ and $1+a^2 \equiv 0$ modulo $q$ since the row has dot product zero with itself. The corresponding twisted Alexander polynomial is obstructed from being a norm by condition (4).

\textbf{Case 3: $0 < k_0 < m$}

The span of the rows must create a subspace of order $q^{nm}$. The $k_i$ rows which are a multiple of $q^i$ each have order $q^{n-i}$, telling us that
\[ n k_0 + (n-1)k_1 + \dots + 2k_{n-2} + k_{n-1} = nm.\]

Let $r:= \frac{n-1}{2}$ if $n$ is odd and $r:= \frac{n-2}{2}$ if $n$ is even. Using the fact that $k_i = k_{n-i}$ for $i>0$ we get $2\sum_{i=0}^r k_i + k_{\frac{n}{2}} = 2m$ (where the summand $k_{\frac{n}{2}}$ does not exist in the odd case). The number of rows is
\[ \sum_{i=0}^{n-1} k_i = k_0 + 2\sum_{i=1}^r k_i + k_{\frac{n}{2}}.\]
These calculations tell us that the number of rows is $2m-k_0$.

Create a new matrix $M'$ by dividing the $k^i$ rows of $M$ by $q^i$ for each $i=1,\dots,n-1$. Then perform row operations until $M'$ has the form $(I \, | \, F)$ where there is a $(2m-k_0) \times (2m-k_0)$ identity matrix on the left of the vertical line, and $F$ is a $(2m-k_0) \times k_0$ matrix. Since $k_0 < m$, we have $2m-k_0 > k_0$ and $F$ has more rows than columns.

Notice that each row of $M'$ has dot product zero modulo $q$ with the original $k_0$ rows of $M$, since the rows of $M$ have dot product zero modulo $q^n$ with the first $k_0$ rows. Therefore the rows of $M'$, denoted $v_1, \dots, v_{\alpha}$, can be thought of as maps $\chi_{v_1}, \dots, \chi_{v_{\alpha}}$ to $\Z_q$ vanishing on the image modulo $q$ of the metaboliser. We need to find a linear combination of them to give us a non-trivial twisted Alexander polynomial; we will do this by finding a row with an odd number of non-zero (and thus zero) entries in the span of the $v_i$. Conditions (2) and (3) of the theorem then obstruct the twisted Alexander polynomial from being slice.

If there are an even number of non-zero entries in any row of $F$, then the row in $M'$ has an odd number of non-zero entries and we are done.  So suppose that every row of $F$ has an odd number of non-zero entries.  Since $F$ has more rows than columns, there is a linear combination of the rows which is zero.  Say $\sum_i s_i \text{row}_i(F) \equiv 0$ (mod $q$).  If an odd number of the $s_i$ are non-zero, then $\sum_i s_i \text{row}_i(M')$ is an odd vector.

So assume that an even number of the $s_i$ are non-zero.  Choose $s_j \neq 0$, and then a non-zero $x \neq s_j \in \Z_q$.  Then
\[ \sum_i s_i \text{row}_i(M') - x\cdot \text{row}_j(M')\]
is a vector with an odd number of non-zero entries. (This is because there is an even number of non-zero values among the first $2m-k_0$ entries, and an odd number of non-zero entries coming from the last $k_0$ entries because we assumed that the rows of $F$ all had an odd number of non-zero entries.)

\end{proof}

\section{Higher order covers}
\label{Sec:HigherCovers}

In this section we analyse metabolisers of higher order branched covers of knots in order to find further criteria which indicate when a knot has infinite order in $\C$.  This is a very different situation to that of the double branched cover which we have studied up until now, because the homology of $\Sigma_p$ for a prime $p>2$ is always a direct double. That is, $H_1(\Sigma_p;\Z) \cong \Z_n \oplus \Z_n$ for some $n\in \N$. (For a proof, see \cite[Theorem 8, Section 8D]{Rolfsen}.) This means that we need different kinds of characters (i.e. not just the linking form) and a more involved analysis of the metabolisers. The arguments here owe their existence to the work of Kirk and Livingston in \cite{KirkLivingston01}; particularly their work in Chapter 5.

Let $K$ be a knot with $H_1(\Sigma_p;\Z)\cong \Z_q \oplus \Z_q \cong E_a \oplus E_b$ where $E_a$ and $E_b$ are the eigenspaces of the deck transformation $T$.  Let $e_a$ be an $a$-eigenvector (i.e. $a e_a = Te_a$) and $e_b$ be a $b$-eigenvector. Now define $\chi_a \colon H_1(\Sigma_p) \to \Z_q$ by $\chi_a(e_a) = 0$ and $\chi_a(e_b) = 1$.  Similarly, $\chi_b \colon H_1(\Sigma_p) \to \Z_q$ is defined by $\chi_b(e_a) = 1$ and $\chi_b(e_b) = 0$.

\begin{theorem}
\label{Thm:Highercover}
The knot $K$ is of infinite order in $\C$ if the following conditions on the twisted Alexander polynomial of $K$ are satisfied:
\begin{enumerate}
  \item $\Delta_{\chi_0}$ is coprime, up to norms, to both $\Delta_{\chi_a}$ and $\Delta_{\chi_b}$, and $\Delta_{\chi_0}$ is not a norm.
  \item $\Delta_{\chi_a + \chi_b} \not\doteq \Delta_{d \chi_a -d^{-1}\chi_b}$ for any $d \in \Z_q$.
\end{enumerate}
\end{theorem}

\begin{proof}
We want to show that $nK$ is not slice for any $n$, and without loss of generality we let $n=2m$.  For every metaboliser we need to find a character for which the corresponding twisted Alexander polynomial does not factorise.  Any invariant metaboliser will be spanned by eigenvectors.  Let $M_a$ be a $k \times 2m$ matrix which represents the contribution from the $a$-eigenspaces, and $M_b$ be a $(2m-k)\times 2m$ matrix representing the contribution from the $b$-eigenspaces.

\textbf{Case 1: $k>m$}

In this case we will prove that in every metaboliser there will be a vector which has an odd number of zero entries.  By Condition (1) the corresponding twisted Alexander polynomial cannot factorise.

We can write $M_a = (I \; E)$ where $E$ is a $k \times (2m-k)$ matrix.  If there are an even number of non-zero entries in any row of $E$, then the row in $M_a$ has an odd number of non-zero entries and we are done.  So suppose that every row of $E$ has an odd number of non-zero entries.  Now, $E$ has more rows than columns, so there is a linear combination of the rows which is zero.  Say $\sum_i s_i \text{row}_i(E) \equiv 0$ (mod $q$).  If an odd number of the $s_i$ are non-zero, then $\sum_i s_i \text{row}_i(M_a)$ is an odd vector.

So assume that an even number of the $s_i$ are non-zero.  Choose $s_j \neq 0$, and then a non-zero $x \neq s_j \in \Z_q$.  Then
\[ \sum_i s_i \text{row}_i(M_a) - x\cdot \text{row}_j(M_a)\]
is a vector with an odd number of non-zero entries. (This is because there is an even number of non-zero values among the first $k$ entries, and an odd number of non-zero entries coming from the last $2m-k$ entries because we assumed that the rows of $E$ all had an odd number of non-zero entries.)

\medskip

\textbf{Case 2: $k=m$}

In this case both $M_a$ and $M_b$ are $m \times 2m$ matrices.  Write $M_a = F(I \; B_a)P$ where $F$ is a change of basis matrix (i.e. row operations) and $P$ is a permutation matrix (i.e. column interchanges).  Without loss of generality, we may assume that $F$ is the identity matrix.  We may also assume that $B_a$ is non-singular, since if it weren't, we would be able to use the argument detailed in the first case to obtain an odd vector.

Now, whatever permutations we do to the matrix $M_a$, we must also do to the matrix $M_b$.  Write $M_b = (X \; Y)P$.

Notice that $\lk(e_a,e_a) = \lk(Te_a,Te_a) = \lk(ae_a,ae_a) = a^2\lk(e_a,e_a)$ so $\lk(e_a,e_a) = 0$, and similarly for $\lk(e_b,e_b)$. The linking form relating $E_a$ to $E_b$ is therefore
\[ \left(\begin{array}{cc} 0 & 1 \\ 1 & 0 \end{array}\right)\]
and because $M_a$ and $M_b$ represent elements of a metaboliser we have $M_a \cdot M_b^T \equiv 0$ (mod $q$).  This means that
\[ 0 \equiv (I \; B_a)P P^T \left(\begin{array}{c}X^T \\ Y^T \end{array}\right) =
   (I \; B_a)\left(\begin{array}{c}X^T \\ Y^T \end{array}\right) = X^T + B_a Y^T \text{.}\]
Rearranging gives us that $Y = -X(B_a^{-1})^T$.  So $M_b = (X \; -X(B_a^{-1})^T)P = X(I \; -(B_a^{-1})^T)P$.  We may assume that $X$ is non-singular, since if it weren't, we would again be in the previous case and able to find an odd vector in $M_b$.  The matrix $X$ represents a change of basis, so once again we may assume that it is the identity.

To summarise, we are in the situation where $M_a = (I \; B_a)P$ and $M_b = (I \; -(B_a^{-1})^T)P$.  This means we can apply Lemma \ref{Lemma:KirkLiv} to conclude that $B_a$ is obtained from a diagonal matrix $D$ by permuting the columns.  So write $B_a = DQ$ where $Q$ is a permutation matrix, and $R= \left(\begin{array}{cc}I & 0 \\0 & Q\end{array}\right)P$.  Then
\begin{itemize}
  \item $M_a = (I \; DQ)P = (I \; D)R$
  \item $(B_a^{-1})^T = (Q^{-1}D^{-1})^T = (Q^T D^{-1})^T = D^{-1}Q$
  \item $P = \left(\begin{array}{cc}I & 0 \\ 0 & Q^{-1}\end{array}\right)R$
  \item $M_b = (I \; -D^{-1}Q)P = (I \; -D^{-1})R$
\end{itemize}

Since the permutation matrix $R$ is common to both $M_a$ and $M_b$, we may assume $R=I$.  So $M_a = (I \; D)$ and $M_b = (I \; -D^{-1})$.

Consider the top rows of $M_a$ and $M_b$:
\[ m_a = (1, 0, \dots, 0, d, 0, \dots, 0)\]
\[ m_b = (1, 0, \dots, 0, -d^{-1},0, \dots, 0)\]
We subtract these vectors, and notice that the character $\chi_a + \chi_b$ vanishes on $e_a - e_b$, whilst $d\chi_a - d^{-1}\chi_b$ vanishes on $de_a + d^{-1}e_b$.  The twisted Alexander polynomial corresponding to $m_a - m_b$ is thus
\[ \Delta_{\chi_a + \chi_b}^K \cdot \Delta_{d\chi_a - d^{-1} \chi_b}^{K}. \]
By condition (2) of the theorem, this is assumed not to be a norm, completing the proof.
\end{proof}

To complete our classification of $9$-crossing knots, we will need an extension of this theorem to decide when a connected sum of knots is not slice. A special case of this, which we shall look at first, is deciding when a knot minus its reverse is not slice.  There do not currently exist any general methods in the literature for proving that a knot is not concordant to its reverse, let alone that the difference is of infinite order in $\C$.  Specific examples have been dealt with by Kirk, Livingston and Naik \cite{Livingston83}, \cite{Naik96}, \cite{KirkLivingston99b} using Casson-Gordon invariants and twisted Alexander polynomials, and this theorem can be seen as a generalisation of their methods.

The proof is almost identical to that of Theorem \ref{Thm:Highercover} except than an extra condition is needed at the end.

\begin{theorem}
\label{Thm:reverseInfiniteOrder}
The knot $K-K^r$ is of infinite order in $\C$ if the following conditions on the twisted Alexander polynomial of $K$ are satisfied:
\begin{enumerate}
  \item $\Delta_{\chi_0}$ is coprime, up to norms, to both $\Delta_{\chi_a}$ and $\Delta_{\chi_b}$, and none of these polynomials are themselves norms.
  \item $\Delta_{\chi_a} \not\doteq \sigma_d(\Delta_{\chi_b})$ for any $d \in \Z_q$
  \item $\Delta_{\chi_a + \chi_b} \not\doteq \Delta_{d \chi_a -d^{-1}\chi_b}$ for any $d \in \Z_q$.
\end{enumerate}
\end{theorem}

\begin{proof}
We want to show that $m(K-K^r)$ is not slice for any $m$.  This means that for every metaboliser, we need to find a character for which the corresponding twisted Alexander polynomial does not factorise.  Any invariant metaboliser will be spanned by eigenvectors.  Let $M_a$ be a $k \times 2m$ matrix which represents the contribution from the $a$-eigenspaces, and $M_b$ be a $(2m-k)\times 2m$ matrix representing the contribution from the $b$-eigenspaces.

The first part of the proof of this theorem is identical in every way to that of Theorem \ref{Thm:Highercover}, up until the following line, where it begins to differ.

We have $M_a = (I \; D)$ and $M_b = (I \; -D^{-1})$ where $D$ is a diagonal matrix.

Consider the top rows of $M_a$ and $M_b$:
\[ m_a = (1, 0, \dots, 0, d, 0, \dots, 0)\]
\[ m_b = (1, 0, \dots, 0, -d^{-1},0, \dots, 0)\]
We have two cases to consider.

In the first case, the columns corresponding to the non-zero entries are from different knots; i.e. one is from $K$ and the other is from $-K^r$.  In this case, the twisted polynomial corresponding to $m_a$ is
\[ \Delta_{\chi_a}^K \cdot \sigma_d(\Delta_{\chi_a}^{-K^r}) = \Delta_{\chi_a}^K \cdot \sigma_d(\Delta_{\chi_b}^{K}). \]
By condition (2) of the theorem, this is assumed not to be a norm.

In the second case, the two non-zero columns represent different copies of the \emph{same} knot.  In this case we subtract the two vectors, and the twisted Alexander polynomial corresponding to $m_a - m_b$ is (as explained in the previous proof)
\[ \Delta_{\chi_a + \chi_b}^K \cdot \Delta_{d\chi_a - d^{-1} \chi_b}^{K}. \]
By condition (3) of the theorem, this is also assumed not to be a norm, completing the proof.
\end{proof}

Let $K$ now be a connected sum of knots, each of which has the same condition on the $p$-fold branched cover for some prime $p$; that is, let $K = K_1 + \dots + K_n$ with $H_1(\Sigma_p(K_i);\Z) \cong \Z_q \oplus \Z_q \cong E_a \oplus E_b$ for some prime $p$, some prime $q$ and for every $i=1,\dots,n$, where $E_a$ and $E_b$ are eigenspaces of the deck transformation. For each knot we have the maps $\chi_a$ and $\chi_b$ as described earlier.

\begin{theorem}
\label{Thm:highercoverConnSum}
The knot $K=K_1 + \dots + K_n$ is of infinite order in $\C$ if the following conditions on the twisted Alexander polynomial of the $K_i$ are satisfied:
\begin{enumerate}
  \item $\Delta_{\chi_0}^{K_i}(t)$ is not a norm in $\Q(\zeta_q)[t,t^{-1}]$ for any $i=1,\dots,n$.
  \item $\Delta_{\chi_0}^{K_i}$ is coprime, up to norms, to $\Delta_{\chi_a}^{K_j}$ and $\Delta_{\chi_b}^{K_j}$ for every $i$ and $j$.
  \item $\Delta_{\chi_a}^{K_i} \not\doteq \sigma_d(\Delta_{\chi_a}^{K_j})$ (or $\Delta_{\chi_b}^{K_i} \not\doteq \sigma_d(\Delta_{\chi_b}^{K_j})$) for any $d \in \Z_q$ and any $i\neq j$.
  \item $\Delta_{\chi_a + \chi_b}^{K_i} \not\doteq \Delta_{d \chi_a -d^{-1}\chi_b}^{K_i}$ for any $d \in \Z_q$ and any $i=1,\dots,n$.
\end{enumerate}
\end{theorem}

\begin{proof}
 We want to show that $2mK$ cannot be slice for any $m$. This means that for every metaboliser, we need to find a character for which the corresponding twisted Alexander polynomial does not factorise as a norm.  Any invariant metaboliser will be spanned by eigenvectors.  Let $M_a$ be a $k \times 2mn$ matrix which represents the contribution from the $a$-eigenspaces, and $M_b$ be a $(2mn-k)\times 2mn$ matrix representing the contribution from the $b$-eigenspaces.

 \textbf{Case 1:} $k > nm$

 In this case we can follow the same proof as in Theorem \ref{Thm:Highercover} to show that in every metaboliser there will be a vector which has an odd number of zero entries. By Conditions (1) and (2) the corresponding twisted Alexander polynomial corresponding to this vector cannot be a norm.

 \medskip

 \textbf{Case 2:} $k=nm$

 We again follow the proof of Theorem \ref{Thm:Highercover} until the following line:

 We have $M_a = (I \; D)$ and $M_b = (I \; -D^{-1})$ where $D$ is a diagonal matrix. Consider the top rows of $M_a$ and $M_b$:
  \[ m_a = (1, 0, \dots, 0, d, 0, \dots, 0)\]
  \[ m_b = (1, 0, \dots, 0, -d^{-1},0, \dots, 0)\]
 We have two cases to consider.

 In the first case, the columns corresponding to the non-zero entries are from different knots; say one is from $K_i$ and the other is from $K_j$. By condition (3) the twisted Alexander polynomial corresponding to either $m_a$ or $m_b$ is assumed not to be a norm.

 In the second case, the two non-zero columns represent different copies of the same knot $K_i$. In this case we subtract the two vectors and the corresponding twisted Alexander polynomial is (as explained in the proof of Theorem \ref{Thm:Highercover})
 \[ \Delta_{\chi_a + \chi_b}^{K_i} \cdot \Delta_{d\chi_a - d^{-1} \chi_b}^{K_i}. \]
 By condition (4) of the theorem, this is assumed not to be a norm, completing the proof.
\end{proof}

\begin{remark}
 Using Corollary \ref{Thm:slicenessmodq} we can extend Theorem \ref{Thm:highercoverConnSum} to include knots $K_i$ in the connected sum for which $H_1(\Sigma_p(K_i);\Z) \cong \Z_q \oplus \Z_q \oplus T_i$ (or simply $H_1(\Sigma_p(K_i);\Z) \cong T_i$) with the order of the $T_i$ coprime to $q$.
\end{remark} 
\chapter{\label{chapter6} Geometric concordance classification of $9$-crossing knots}

In Chapter \ref{chapter4} we took the free abelian group $\CF_E$ generated by the $9$-crossing prime knots $E = \left\{ 3_1,4_1,5_1,5_2,6_1,6_2,6_3,7_1, \dots,7_7,8_1,\dots,8_{21}, 9_1, \dots, 9_{49} \right\}$ (including the distinct reverses $8_{17}^r$, $9_{32}^r$ and $9_{33}^r$) and in Theorem \ref{Thm:algclassification} found a basis for the kernel $\A_E$ of the map $\CF_E \to \C_E \to \G$, where $\C_E$ is the subgroup of $\C$ generated by $E$.

We would now like to complete our investigation into the structure of $\C_E$ by identifying a basis for the kernel of the map $\phi \colon \A_E \to \C_E$.

Let us first pick out those knots which have been shown to be slice or order $2$ by other people (see Knotinfo~\cite{ChaLivingston}). The notation $\C_T^n$ will stand for the set of elements of geometric (T$\, = \,$Topological) order $n$.
\[ \C_T^2 = \{4_1, 6_3, 8_3, 8_{12}, 8_{17}, 8_{18}\}\]
\[ \C_T^1 = \{6_1, 8_8, 8_9, (8_{10}+3_1), (8_{11}-3_1), 8_{20}, (9_{24}-4_1), 9_{27}, (9_{37}-4_1), 9_{41}, 9_{46}\} \]

It remains to look at the image under $\phi$ of the following set of knots:
\begin{eqnarray*}S & = & \{4(7_7), 4(9_{34}), 2(8_1), 2(8_{13}),2(8_{15}-7_2-3_1), 2(9_2-7_4), 2(9_{12}-5_2), 2(9_{14}), \\ & & 2(9_{16}-7_3-3_1), 2(9_{19}), 2(9_{28}-3_1), 2(9_{30}), 2(9_{33}),2(9_{42}+8_5-3_1),2(9_{44}-4_1), \\ & & (8_{17}-8_{17}^r),(8_{21}-8_{18}-3_1), (9_8-8_{14}), (9_{23}-9_2-3_1), (9_{29}-9_{28}+2(3_1)), \\ & & (9_{32}-9_{32}^r), (9_{33}-9_{33}^r), (9_{39}+7_2-4_1), (9_{40}-8_{18}-4_1-3_1) \}
\end{eqnarray*}

\begin{conjecture}
\label{Conj:geomclassification}
The image of $\A_E$ in $\C_E$ is isomorphic to $\Z^{23} \oplus \Z_2$ generated by the image under $\phi$ of the set $S$.
\end{conjecture}

The $\Z_2$ summand is generated by $(8_{17}-8_{17}^r)$, since $8_{17}$ is itself of order $2$.  For now we will remove $(8_{17}-8_{17}^r)$ from $S$, since if we can prove that no linear combination of the other knots in $S$ is slice then $(8_{17}-8_{17}^r)$ is independent of those knots. To prove the rest of the conjecture we need to show that all the remaining knots are of infinite order in $\C$ and that they are independent of one another.

\section{Knots of infinite order}
\label{Subsec:infiniteorder}

The knots in $S$ are all algebraically slice, so we will need Casson-Gordon invariants to prove that they are not geometrically slice. In particular, we will use twisted Alexander polynomials and the theorems developed in Chapter \ref{chapter5}.

The knots $4(7_7), 4(9_{34})$ and $2(8_{1})$ are known to be of infinite order from KnotInfo~\cite{ChaLivingston}.  The other singleton knots $2(8_{13}), 2(9_{14}), 2(9_{19}), 2(9_{30}), 2(9_{33})$ and $2(9_{44}-4_1) = 2(9_{44})$ are shown to be of infinite order via Theorem~\ref{Thm:twistedpoly}.

Theorem \ref{Thm:infiniteordersum} applies to the following knots to show that they are of infinite order:
\begin{eqnarray*}&& 2(8_{15}-7_2-3_1), 2(9_2-7_4), 2(9_{12}-5_2), 2(9_{16}-7_3-3_1), 2(9_{28}-3_1), 2(9_{42}+8_5-3_1),\\
                   &&(9_8-8_{14}), (9_{29}-9_{28}+2(3_1)), (9_{40}-8_{18}-4_1-3_1), (8_{21}-8_{18}-3_1), (9_{39}+7_2-4_1).
\end{eqnarray*}

The twisted Alexander polynomial calculations, performed using Maple $13$, can be found in Appendix \ref{AppendixC}. We work through a couple of examples to illustrate how the theorem works.

\begin{example}
Using Theorem \ref{Thm:infiniteordersum} we show that $8_{15}-7_2-3_1$ is of infinite order in $\C$. We have $H_1(\Sigma_2(8_{15});\Z) \cong \Z_{33}$, $H_1(\Sigma_2(7_2);\Z) \cong \Z_{11}$ and $H_1(\Sigma_2(3_1);\Z) \cong \Z_3$, so we choose to use the prime $q=11$. This means we do not have to compute any polynomials related to $3_1$.

Condition (1) of Theorem \ref{Thm:infiniteordersum} asks us to check whether any of the trivial twisted polynomials are norms. We have
\[ \Delta_{\chi_0}^{8_{15}} = (9t^2-7t+9)(t^2+t+1)\]
and
\[ \Delta_{\chi_0}^{7_2} = (9t^2-7t+9) \]
where
\[ 9t^2-7t+9 = \frac{1}{9}(9t-1+5w+5w^3+5w^4+5w^5+5w^9)(9t-1+5w^2+5w^6+5w^7+5w^8+5w^{10}) \]
for $w$ an $11^\text{th}$ root of unity. Is this a norm? Write the first factor as $9t+f(w)$. Then the inverse conjugate of this is
\[ 9t^{-1} + f(\overline{w})= 9f(\overline{w})^{-1} + t = \frac{1}{9}(81f(\overline{w})^{-1} + 9t) \]
and we are reduced to deciding whether $81f(\overline{w})^{-1} = g(w)$ where $g(w)$ is the constant term in the second factor. Calculating using a program such as Maple, we find that
\[ 81 f(\overline{w})^{-1} = -1 + 5w+5w^3+5w^4+5w^5+5w^9 = f(w) \neq g(w). \]
Thus we have shown that the trivial twisted polynomials are not norms.

Condition (2) asks us to check whether any factors of $\Delta_{\chi_0}$ for either $8_{15}$ or $7_2$ are contained in any of the non-trivial polynomials. This is clearly not the case, since $\Delta_{\chi_1}^{7_2} = 1$ and
\[ \Delta_{\chi_1}^{8_{15}} = 1+t(w^3-w^4+w^8-w^7+3w^5+3w^6-5)+t^2\]
which is irreducible over $\Q(w)[t,t^{-1}]$.

Finally, condition (3) asks us to look at the linking forms of $8_{15}$ and $7_2$. These are
 $\lk_{8_{15}}(3,3) = -5 \in \Z_{11}$ (which is not a square) and $\lk_{-7_2}(1,1) = -1 \in \Z_{11}$ (also not a square). Since the linking forms are the same modulo squares, and since $q=11 \equiv 3$ modulo $4$, there is no further work to be done.
\end{example}

\begin{example}
 As a further example to illustrate when condition (3) of Theorem \ref{Thm:infiniteordersum} is necessary, consider the knot $9_8 - 8_{14}$. For this knot we use the prime $q=31$, which is $3$ modulo $4$, and we find that $9_8$ and $-8_{14}$ have \emph{different} linking forms modulo squares. This means we have to check whether certain Galois conjugates of $\Delta_{\chi_1}^{9_8}$ and $\Delta_{\chi_1}^{8_{14}}$ are the same. The polynomials we are dealing with are of the form $(1+tf(w)+t^2)$ for $9_8$ and $(1+tg(w) + t^2)$ for $8_{14}$, where
 \begin{align*} f(w) = & -w -3w^2 -5w^3 -7w^4 -10w^5 -11w^6 -12w^7 -13w^8
-14w^9 -16w^{10}- 18w^{11}\\ & -20w^{12} -22w^{13} - 24w^{14} -24w^{15} -24w^{16}
-24w^{17} -22w^{18}-20w^{19}-18w^{20}\\ & -16w^{21}-14w^{22}-13w^{23}-12w^{24}
-11w^{25}-10w^{26}-7w^{27}-5w^{28}-3w^{29}-w^{30}
\end{align*}
and
\begin{align*} g(w) = & 32w+11w^2+20w^3+26w^4+6w^5+36w^6+5w^7+28w^8
+17w^9+13w^{10}+30w^{11}\\ & +3w^{12}+34w^{13}+9w^{14}+23w^{15}+23w^{16}
+9w^{17}+34w^{18}+3w^{19}+30w^{20}+13w^{21}\\ & +17w^{22}+28w^{23}+5w^{24}
+36w^{25}+6w^{26}+26w^{27}+20w^{28}+11w^{29}+32w^{30}.
\end{align*}
Since Galois conjugation does not change the frequency of the coefficients in $f$ or $g$, it is easy to see that no Galois conjugate of $f$ can ever equal $g$.
\end{example}

The knot $(9_{23}-9_2-3_1)$ is problematic because the twisted polynomial $\Delta_{\chi_1}$ for $9_2$ is trivial,~\footnote{The twisted Alexander polynomial of any genus $1$ knot is an element of $\Q(\zeta_d)$, which is considered trivial. This is because the degree of the twisted Alexander polynomial is a lower bound for the Thurston norm~\cite{FriedlKim08}, which is $1$ for a genus $1$ knot.} and so condition (3) of Theorem \ref{Thm:infiniteordersum} fails to hold.  The homology of the two-fold branched cover of this knot is
\[ H_1(\Sigma_2;\Z) \cong (\Z_9 \oplus \Z_5) \oplus (\Z_5 \oplus \Z_3) \oplus \Z_3 \text{ .}\]
Using the prime $q=5$, we must follow the proof of Theorem \ref{Thm:infiniteordersum} more carefully and identify whether the conclusions still hold.

\begin{proposition}
\label{Prop:9-23}
The knot $K=(9_{23}-9_2-3_1)$ is of infinite order in $\C$.
\end{proposition}

\begin{proof}
Suppose $mK$ is slice and that a basis of a metaboliser in $H_1(\Sigma_2(mK);\Z_5)$ is
\begin{eqnarray*}
v_1 & = & (a_{1,1}, \dots, a_{1,m},b_{1,1},\dots,b_{1,m},0_1,\dots, 0_m)\\
v_2 & = & (a_{2,1}, \dots, a_{2,m},b_{2,1},\dots,b_{2,m},0_1,\dots, 0_m)\\
\vdots\\
v_{m} & = & (a_{m,1}, \dots, a_{m,m},b_{m,1},\dots,b_{m,m},0_1,\dots, 0_m)
\end{eqnarray*}
where $a_i \in H_1(\Sigma_2(9_{23});\Z_5)$ and $b_i \in H_1(\Sigma_2(9_{2});\Z_5)$.

Consider the truncated $(m \times 2m)$ matrix $V$ whose rows are the $v_i$ but without the final $m$ zeros.  If a linear combination of these rows is found which contains an odd number of zeros then the polynomial $(16t^2-17t+16)$ will occur with odd degree as part of the corresponding twisted Alexander polynomial of that vector, since this polynomial is a factor of both $\Delta_{\chi_0}(9_{23})$ and $\Delta_{\chi_0}(9_{2})$ (but not $\Delta_{\chi_0}(3_1)$).  The twisted polynomial can therefore not factorise and $mK$ is not slice.  If such a linear combination cannot be found, then by Lemmas \ref{Lemma:KirkLiv} and \ref{Lemma:oddvector} we can perform row and column operations on $V$ to transform it into $(I \, \, E)$, where $E$ is obtained from a diagonal matrix whose entries are $\pm2$ (since the linking forms of $-9_2$ and $9_{23}$ are the same) by permuting the columns.  Each row of this new matrix (which we will call $V'$) gives us a twisted Alexander polynomial of the form
\[ \Delta_\chi \doteq \Delta_{\chi_1}^{K_i} \Delta_{\chi_2}^{K_j} \text{ .}\]
If $K_i \neq K_j$ or if $K_i=K_j = 9_{23}$ then this polynomial is not a norm; the difficulty comes if $K_i = K_j = 9_2$ because then this polynomial is equal to $1$.  However, it is impossible that $K_i = K_j = 9_2$ for \emph{every} row of $V'$ because that would imply that every column corresponded to the knot $9_2$, when we know that only half the columns do.  We can thus deduce that there is a row which corresponds to a twisted polynomial that is not a norm, finishing the proof that $mK$ cannot be slice.
\end{proof}

\begin{remark}
This proof can be adapted to work for any knot $K=K_1 + \dots + K_n$ which satisfies all the conditions of Theorem \ref{Thm:infiniteordersum} except that one constituent knot has a $\Delta_{\chi_1}$ which is a norm.  In fact, a more careful analysis should reveal how many of the constituent knots need not satisfy condition (3) before the theorem fails.
\end{remark}

Finally, the knots $9_{32}-9_{32}^r$ and $9_{33}-9_{33}^r$ cannot be shown to be of infinite order by the $2$-fold branched cover, so we must look to higher covers and use Theorem \ref{Thm:reverseInfiniteOrder}.  This applies to our two knots when $p=5$ (i.e. looking at $H_1(\Sigma_5;\Z)$).  For $9_{32}-9_{32}^r$ we work at the prime $q=11$ and for $9_{33}-9_{33}^r$ we use the prime $q=101$.  What are $a$ and $b$ in these cases, and why are there only two relevant eigenvalues when in principle there could be four?  The following theorem, proved by Fox~\cite[Theorem 1]{Fox70} and also restated by Hartley~\cite[1.7]{Hartley83}, answers these questions.

\begin{theorem}
Let $K$ be a knot and suppose that $p$ divides $q-1$ and that $\alpha \in \Z_q$ is an element of order $p$ in $\Z_q$.  Suppose also that $H_1(\Sigma_p(K);\Z) \cong \bigoplus_i E_i$ with the $E_i$ being different eigenspaces of the deck transformation.  Then $E_{\alpha}$ is one of the eigenspaces in this sum if and only if $q$ divides $\Delta_K(\alpha)$.
\end{theorem}

In particular, if $E_a$ is an eigenspace of the deck transformation then so is $E_{a^{-1}}$ since $\Delta(a) = \Delta(a^{-1})$ (up to a factor of $t$).  Thus $b=a^{-1}$ in $\Z_q$.

For $9_{32}-9_{32}^r$ there are four $5^\text{th}$ roots of unity in $\Z_{11}$: $3,4,5,9$.  However, only $3$ and $4$ are roots of the Alexander polynomial modulo $11$. (And indeed they are inverses of each other modulo $11$.)  For $9_{33}-9_{33}^r$ there are also four $5^\text{th}$ roots of unity in $\Z_{101}$: $36,84,87,95$.  Only $36$ and $87$ are roots of the Alexander polynomial modulo $101$.

The prime $101$ is a rather large one: for condition (3) of Theorem \ref{Thm:reverseInfiniteOrder} there are potentially $101$ twisted polynomials that need to be checked.  Thankfully, many of them turn out to be the same via the action of the deck transformation $T$.

\begin{lemma}
\label{Lemma:polynomialOrbit}
Let $(x,y) \in \Z_{q}^2$ denote the twisted Alexander polynomial $\Delta_{x\chi_a + y\chi_b}$, where $\chi_a$ and $\chi_b$ are explained at the beginning of Section \ref{Sec:HigherCovers}.  Then the polynomial $(x,y)$ is equal to $(b^n x, a^n y)$ for each $n=0,\dots,p-1$.
\end{lemma}

\begin{proof}
For any map $\chi \colon H_1(\Sigma_p;\Z) \to \Z_q$, the twisted Alexander polynomial $\Delta_{\chi}$ is equivalent to $\Delta_{\chi \circ T}$, where $T \colon H_1(\Sigma_p) \to H_1(\Sigma_p)$ is the deck transformation of $\Sigma_p$.

An element in $H_1(\Sigma_p(K))$ may be written as $\alpha e_a + \beta e_b$ for some $\alpha, \beta \in \Z_q$, where $e_a$ and $e_b$ are $a$- and $b$- eigenvectors respectively.  This means that
\begin{eqnarray*}
 (x \chi_a + y \chi_b) \circ T(\alpha e_a + \beta e_b) & = & (x \chi_a + y \chi_b)(a\alpha e_a + b\beta e_b)\\
 & = & x \chi_a(a\alpha e_a + b\beta e_b) + \chi_b(a\alpha e_a + b\beta e_b) \\
 & = & x\chi_a(b\beta e_b) + y \chi_b(a\alpha e_a) \\
 & = & b x\chi_a(\beta e_b) + a y \chi_b(\alpha e_a) \\
 & = & (b x\chi_a + a y\chi_b)(\alpha e_a + \beta e_b)
\end{eqnarray*}
because $T e_a = a e_a$ and $T e_b = b e_b$ by the definition of $e_a$ and $e_b$ being eigenvectors. Applying the deck transformation repeatedly gives the result.
\end{proof}

Continuing with the same notation as in Lemma \ref{Lemma:polynomialOrbit}, we may consider the polynomials $(x,y)$ as living in the projective space of $\Z_q \oplus \Z_q$.  This is because $i(x,y)$ is the same polynomial as $(x,y)$ up to Galois conjugation, for any $i \in \Z_q$.  This means there are $q+1$ different polynomials in our space, \emph{a priori}: $(1,0),(0,1),(1,1),(1,2), \dots, (1,q-1)$.  However, we also have an action on this group, namely, multiplication of the first entry by $b$ and multiplication of the second entry by $a$.  (Notice that since $a=b^{-1}$ we have $(b,a) = (1,a^2)$, and the action of the group is equivalent to multiplying the second entry by $a^2$.)  Orbits of this action, which are of order $p$ since $a^p = 1 = b^p$ (mod $q$), give us the same polynomial.  This means there are only $\frac{q-1}{p}$ potential different polynomials to be checked for condition (3) of Theorem \ref{Thm:reverseInfiniteOrder}.  In fact, we can cut this down by a further factor of two, since we are only interested in checking whether $(1,1) \neq (d,-d^{-1})=(1,-d^{-2})$, so we only have to check the polynomials where the second factor is a (negative) square. (The action of the group preserves squares in the second factor.)

In the case of $9_{32}-9_{32}^r$ we have $p=5$ and $q=11$, so for condition (3) there is only one polynomial to be checked against $(1,1)$ (up to Galois conjugates), namely $(1,-1)$. We must also check that $(1,0) \neq (0,1)$ up to Galois conjugates.  These conditions are satisfied, and we give the relevant polynomials in Appendix \ref{AppendixC}.  For $9_{33}-9_{33}^r$ we have $p=5$ and $q=101$, so there are $\frac{101-1}{2\times 5} = 10$ polynomials to check for condition (3).  Since $101 \equiv 1$ (mod 4), one of these polynomials is actually $(1,1)$ (when $d=10$), so we must remember to check that $\Delta_{\chi_a + \chi_b} \neq \sigma_{10}\Delta_{\chi_a + \chi_b}$. (They aren't equal!)  It would be impractical to write out all of these polynomials (each coefficient of a power of $t$ has $100$ terms!) so instead we look at the coefficient of $s$ only.  This coefficient is a polynomial in $w$, where $w$ is a $101^\text{th}$ root of unity.  In Appendix \ref{AppendixC} we give the frequency of the coefficients of the $w^i$; for example, in the polynomial $(1,6)$ we see that $-2$ occurs as a coefficient of the $w^i$ $20$ times.  Since Galois conjugation does not change these frequencies, it is possible to see that all the required polynomials are different from $(1,1)$.

\section{Linear independence}
\label{Subsec:indep}

Consider the knot
\begin{eqnarray*}K & = & 4a_1(7_7) + 4a_2(9_{34}) + 2a_3(8_1) + 2a_4(8_{13}) + 2a_5(8_{15}-7_2-3_1) + 2a_6(9_2-7_4) + \\ & & 2a_7(9_{12}-5_2) + 2a_8(9_{14}) + 2a_9(9_{16}-7_3-3_1) + 2a_{10}(9_{19}) + 2a_{11}(9_{28}-3_1) + \\ & & 2a_{12}(9_{30}) +  2a_{13}(9_{33}) + 2a_{14}(9_{42}+8_5-3_1) +  2a_{15}(9_{44}-4_1) + a_{16}(8_{21}-8_{18}-3_1) + \\ & & a_{17}(9_8-8_{14}) + a_{18}(9_{23}-9_2-3_1) + a_{19}(9_{29}-9_{28} + 2(3_1))+ a_{20}(9_{32}-9_{32}^r) + \\ & & a_{21}(9_{33}-9_{33}^r) + a_{22}(9_{39}+7_2-4_1) + a_{23}(9_{40}-8_{18}-4_1-3_1).
\end{eqnarray*}

\begin{conjecture}
$K$ is slice if and only if $a_i=0$ for all $i=1,\dots,23$.
\end{conjecture}

The method of proof will involve identifying primes in $H_1(\Sigma_2(K);\Z)$ which only occur in the homology of the branched covers of very few of the constituent knots.  The problem will then be reduced to a series of independence proofs involving only two or three knots at a time.

\begin{notation}
Let $K_i$ denote the knot whose coefficient is $a_i$ in the sum of knots making up $K$.
\end{notation}

We start by identifying primes which only occur in $H_1(\Sigma_2(K_i))$ for a single $i$.  Since we know that all the $K_i$ are of infinite order, the coefficients of these `singleton' knots must be zero.

\begin{center}
\begin{tabular}{cc}
\textbf{Prime} & \textbf{Coefficient which is zero}\\
$23$ & $a_2$ \\
$29$ & $a_4$ \\
$37$ & $a_8$ \\
$41$ & $a_{10}$ \\
$53$ & $a_{12}$ \\
$31$ & $a_{17}$ \\
($59$) & ($a_{20}$)
\end{tabular}
\end{center}

The last coefficient is in brackets because it is for the knot $9_{32}-9_{32}^r$, which cannot be shown to be of infinite order using the $2$-fold cover.  We shall ignore this knot for now, but keep it in mind to deal with later.

Here follows a table of the other coefficients with the primes we will use to attack them.

\begin{center}
\begin{tabular}{cc}
\textbf{Prime} & \textbf{Coefficients}\\
$13$ & $a_3$, $a_9$ \\
$11$ & $a_5$, $a_{22}$ \\
($61$) & ($a_{13}$, $a_{21}$) \\
$7$ & $a_1$, $a_7$, $a_{14}$ \\
$17$ & $a_{11}$, $a_{15}$, $a_{19}$ \\
\end{tabular}
\end{center}

Again, the prime which is bracketed refers to the knot $9_{33}-9_{33}^r$ which we know cannot be dealt with using the $2$-fold cover, so we will put it to one side for the moment.  For $q=13$ there is a problem with $\Delta_{\chi_1}^{8_1}$ being trivial, but a combination of Theorem \ref{Thm:infiniteordersum} and the techniques of Proposition \ref{Prop:9-23} serve to deal with it and prove that $a_3 = a_9 =0$. For $q=11,7,17$, Theorem \ref{Thm:infiniteordersum} suffices to prove that the corresponding coefficients are all zero.

The remaining coefficients are $a_6$, $a_{16}$, $a_{18}$ and $a_{23}$, each of them appearing non-trivially at the primes $3$ and $5$.  When $q=5$ there is a problem with $a_{23}$ because $H_1(\Sigma_2(9_{40});\Z_5) \cong \Z_5 \oplus \Z_5$.  When $q=3$ there is also a problem because in $H_1(\Sigma_2)$ there are summands of $\Z_3$ and $\Z_9$.  Since we cannot make further progress with the $2$-fold branched cover, we shall instead study the higher covers and attempt to make use of Theorem \ref{Thm:highercoverConnSum}.

Let
\begin{eqnarray*} K' & = & 2a_6(9_2 - 7_4) + a_{16}(8_{21}-8_{18}-3_1) + a_{18}(9_{23}-9_2-3_1) + a_{23}(9_{40}-8_{18}-4_1-3_1)\\ & & + a_{20}(9_{32}-9_{32}^r) + 2a_{13}(9_{33}) + a_{21}(9_{33}-9_{33}^r)
\end{eqnarray*}
and let us look at $H_1(\Sigma_p(K');\Z)$ for various $p$. Recall that in order to apply Theorem \ref{Thm:highercoverConnSum} we need to find primes $q$ dividing the order of $H_1(\Sigma_p)$ such that $p$ divides $q-1$.

When $p=3$, the primes $q=2$, $5$, $11$ and $23$ divide the order of $H_1(\Sigma_3(K'))$, but none of these have the property that $3$ divides $q-1$.

When $p=5$ we have more luck, since the primes dividing $H_1(\Sigma_5(K'))$ are $11$, $61$ and $101$. But $H_1(\Sigma_5(9_{40});\Z) \cong \Z_{11} \oplus \Z_{11} \oplus \Z_{11} \oplus \Z_{11}$ and this situation is not covered by Theorem \ref{Thm:highercoverConnSum}, while at the prime $61$ the twisted Alexander polynomials of $K'$ do not satisfy the theorem's requirements. In particular, the twisted polynomials for $7_4$ and $9_2$ are trivial and $\Delta_{\chi_a + \chi_b}^{9_{23}} \doteq \Delta_{\chi_a - \chi_b}^{9_{23}}$. However, the prime $101$ is a success. The twisted polynomials for $9_{33}$ satisfy the conditions of Theorem \ref{Thm:highercoverConnSum} and so we may conclude that if $K'$ is slice then $a_{13} = a_{21} = 0$.

Moving on to $p=7$, the primes dividing $H_1(\Sigma_7(K'))$ are $q=251$, $29$ and $743$. Unfortunately $7$ does not divide $250$, and again we have the problem that $H_1(\Sigma_7(9_{40});\Z) \cong \Z_{29} \oplus \Z_{29} \oplus \Z_{29} \oplus \Z_{29}$ so the prime $29$ is out of the question.  The prime $743$ is a good one because $7$ divides $742$, but the difficulty here is in the sheer amount of computation involved in checking the conditions of Theorem \ref{Thm:highercoverConnSum}. Using the analysis from Section \ref{Subsec:infiniteorder} we see that there are $\frac{742}{2 \times 7} = 53$ polynomials to be checked for condition $(4)$ of the theorem.

In fact, $9_{40}$ appears to always have a four-fold summand in the homology of its branched covers, whilst the homology of the covers of $9_{23}$ never splits into the requisite eigenspaces. It seems that we are once again stuck, so we shall go back to studying $H_1(\Sigma_2(K');\Z)$ and see if we can find a way to use twisted Alexander polynomials when there are different powers of a prime involved.

\section{Combining prime powers}
\label{Sec:combiningprimepowers}

How can we combine the theorems in Chapter \ref{chapter5} to determine the sliceness of a composite knot where different powers of a prime occur in the homologies of $\Sigma_2$ of different component knots?  For example, how can we tell if the knot $K=aK_1 + bK_2$, where $H_1(\Sigma_2(K_1)) \cong \Z_9$ and $H_1(\Sigma_2(K_2)) \cong \Z_3$, is slice?  Let us take this very simple example and investigate what a metaboliser of $K$ would look like.

Suppose that $2mK$ is slice.  Then
\[ H_1(\Sigma_2(2mK);\Z) \cong \bigoplus_{2ma}\Z_9 \oplus \bigoplus_{2mb}\Z_3\]
so a matrix $M$ representing a metaboliser will have $2m(a+b)$ columns.  The first $2ma$ columns will have entries in $\Z_9$ and the last $2mb$ columns will have entries in $\Z_3$.  Some number of the rows, say $\beta$, will be of order $9$, with the rest, $\gamma$, of order $3$.  The order of the homology is $3^{2m(2a+b)}$, so the order of the metaboliser is $3^{m(2a+b)}$.  We must therefore have the relation
\[ 2\beta + \gamma = m(2a+b) \text{.}\]
So far, we have a picture of $M$ which looks like this:

\noindent (There are $\beta$ rows above the first horizontal line and $\gamma$ rows beneath it; the entries to the left of the vertical double line are in $\Z_9$ whilst those to the right are in $\Z_3$.)
\[ \left(\begin{array}{c|c|c||c}
\begin{array}{ccc} 1 & & 0 \\ & \ddots & \\ 0 && 1 \end{array} & A_1 & A_2 & B_1 \\\hline
0 & \begin{array}{ccc} 3 & & 0 \\ & \ddots & \\ 0 && 3 \end{array} & A_3 & B_2 \\ \hline
0 & 0 & 0 & B_3
\end{array}\right)\]

By the definition of a metaboliser, the linking pairing between any pair of rows must be zero in $\Q/\Z$.  Let $x = (a_1, \dots, a_{2ma},b_1,\dots,b_{2mb}) \in M$.  Then
\[ \lk(x,x) = \lk_{K_1}(1,1) \sum_i a_i^2 + \lk_{K_2}(1,1) \sum_i b_i^2.\]
Let us assume (currently without loss of generality) that $\lk_{K_1}(1,1) = \frac{1}{9}$ and $\lk_{K_2}(1,1) = \frac{1}{3}$.  For $\lk(x,x)$ to be an integer we must have that
\begin{equation}
\label{Eqn:Z9product}
  \sum_i a_i^2 + 3\sum_i b_i^2 \equiv 0 \mod 9.
\end{equation}
It follows that $\sum_i a_i^2$ must be zero modulo $3$.  A corresponding condition holds for the dot product of any two rows of $M$.

In the $\gamma$ rows of order $3$, the dot product of the $\Z_9$-part of the rows is zero modulo $9$.  For these rows, it follows that $\sum_i b_i^2$ must be zero modulo $3$ (and similarly for the dot product between any two of the order $3$ rows).

What about the dot product between an order $9$ row and an order $3$ row?  Let us split the $\gamma$ rows of order $3$ into the $\gamma_1$ rows that have a non-trivial $\Z_9$ part and the $\gamma_2$ rows that have a trivial $\Z_9$ part.  The bottom $\gamma_2$ rows clearly need to have dot product zero (modulo $3$) with the top $\beta$ rows. In the middle $\gamma_1$ rows, the $\Z_9$-entries are all multiples of $3$. If we divide these entries by $3$, Equation \ref{Eqn:Z9product} tells us that the resulting dot product between a $\beta$ row and a modified $\gamma_1$ row is zero modulo $3$.

Here is a picture of our metaboliser $M$ where the middle $\gamma_1$ rows have had their $\Z_9$ entries divided by $3$:
\[ \left(\begin{array}{c|c|c||c|c}
I_{\beta} & A_1 & A_2 & 0 & B_1' \\\hline
0 & I_{\gamma_1} & A_3' & 0 & B_2' \\ \hline
0 & 0 & 0 & I_{\gamma_2} & B_3'
\end{array}\right)\]
To summarise the relations between rows, we have
\begin{enumerate}
  \item The top left block of $\beta$ rows have dot product zero modulo $3$ amongst themselves.
  \item The bottom right block of $\gamma_2$ rows have dot product zero modulo $3$ between themselves and any other row in the matrix. I.e. $B_3' \cdot B_1' \equiv 0$ and $B_3' \cdot B_2' \equiv 0$ modulo $3$.
  \item The middle block of $\gamma_1$ rows have dot product zero modulo $3$ with any other row in the matrix (but not necessarily themselves).
\end{enumerate}
where ``left'' and ``right'' refer to $\Z_9$ and $\Z_3$ entries respectively.

\begin{example}
An example matrix which satisfies all these conditions is (before division by 3)
\[ \left(\begin{array}{cccc||cccc}
1 & 1 & 1 & 0 & 0 & 1 & 0 & -1 \\\hline
0 & 3 & 0 & 0 & 0 & 1 & 1 & -1 \\
0 & 0 & 3 & 0 & 0 & 1 & 1 & -1 \\
0 & 0 & 0 & 3 & 0 & 0 & 0 & 0 \\\hline
0 & 0 & 0 & 0 & 1 & 1 & 0 & 1
\end{array}\right)\]
\end{example}
Here $m=2$, $a=b=1$, $\beta = 1$, $\gamma_1 = 3$, $\gamma_2 = 1$.
\medskip

Now let us try to find some `odd' vectors in $M$ so that we get a slice obstruction from the twisted Alexander polynomials.  In what follows, all entries of $M$ have been mapped into $\Z_3$.

If $\beta = 0$ then the vector $(1,0,\dots,0)$ has zero dot product with any row of $M$, giving us the twisted Alexander polynomial condition $\Delta_{\chi_1}^{K_1}\Delta_{\chi_0}^{K_1} \doteq 1$.

If $\beta \geq ma$, then condition (1) means that we can use Lemma \ref{Lemma:oddvector} to find a linear combination of rows which contains an odd number of zeros (on the `left' of $M$).  If $\Delta_{\chi_1}^{K_1}\Delta_{\chi_0}^{K_1} \not \doteq 1$ then this again obstructs $K$ from being slice.

If $\gamma_2 \geq mb$ then condition (3) means that we can use Lemma \ref{Lemma:oddvector} to find a linear combination of rows which contains an odd number of zeros (on the `right' of $M$).  If $\Delta_{\chi_1}^{K_2}\Delta_{\chi_0}^{K_2} \not \doteq 1$ then this obstructs $K$ from being slice.

Finally, we are in the situation where $0 < \beta < ma$ and $\gamma_2 < mb$.  Since $2\beta + \gamma_1 + \gamma_2 = m(2a+b)$, we have $\gamma_1 > 0$.  If we can find an odd row in the middle $\gamma_1$ rows, then we are done since these rows have zero dot product with any other row.  The problem is that we do not have a dot product condition between any of the $\gamma_1$ rows, so we cannot use any variant of Lemma \ref{Lemma:oddvector}.

At this point we are stuck and cannot complete the proof of the conjecture. In future work, perhaps Casson-Gordon signatures could be combined with twisted Alexander polynomials to eliminate the final case from consideration.

\section{Summary}

The analysis of this chapter leaves us still with the conjecture that we had at the start; namely:

\begin{conjecture}
\label{Conj:geomsummary}
A basis for the kernel of the map $\phi \colon \A_E \to \C_E$ is the union of the independent sets $2\C_A^{1'}$, $2\C_T^2$ and $\C_T^1$, where
  \begin{itemize}
   \item $\C_A^{1'} = \{ (8_{17}-8_{17}^r) \}$
   \item $\C_T^2 = \{4_1, 6_3, 8_3, 8_{12}, 8_{17}, 8_{18}\}$
   \item $\C_T^1 = \{6_1, 8_8, 8_9, (8_{10}+3_1), (8_{11}-3_1), 8_{20}, (9_{24}-4_1), 9_{27}, (9_{37}-4_1), 9_{41}, 9_{46}\}$
  \end{itemize}
\end{conjecture}

In order to prove this conjecture, it remains to show that the knots $(9_2-7_4)$, $(8_{21}-8_{18}-3_1)$, $(9_{23}-9_2-3_1)$, $(9_{40}-8_{18}-4_1-3_1)$ and $(9_{32}-9_{32}^r)$ are linearly independent in $\C$.

\begin{remark}
 If this conjecture holds, then the image of $\A_E$ in $\C_E$ is isomorphic to $\Z^{23} \oplus \Z_2$.
\end{remark}

\chapter{\label{chapter7} Examples}

In Chapters \ref{chapter4} and \ref{chapter6} we analysed the concordance relations of the knots in $E$, that is, the prime knots with up to 9 crossings.  Given that Conjecture \ref{Conj:geomsummary} is correct, we know where all the algebraic and geometrical torsion in $\C_E$ comes from and we know which knots are independent of each other in $\C$.  But what can we do with all this information?

There are two questions that our classification is designed to answer:

\begin{enumerate}
  \item Given a knot $K$ which is a linear combination of knots in $E$, what are its algebraic and topological concordance orders?
  \item Suppose that there is a mystery knot $K$ which we are told is a linear combination of knots in $E$.  How do we determine what this linear combination is (up to concordance)?
\end{enumerate}

In the first case, the final classification itself is sufficient to answer this question for any knot $K$.  For the second question, we will need not only the classification but all the calculations of invariants which were done along the way.  To illustrate the methods that one would need to answer these questions, we shall work with a specific test knot
\[ K = 2(3_1) + 5(7_2) - 3(8_{12}) - 9_8 + 9_{16} + 9_{34} + 9_{46}.\]

\section{Question 1: Finding the concordance order of $K$}

To find the concordance order of $K$, we first re-write the knot in terms of our basis.
\[ K = 3(3_1) + 5(7_2) + 7_3 - 3(8_{12}) - 8_{14} - (9_8 - 8_{14}) + (9_{16}-7_3-3_1) + 9_{34} + 9_{46}.\]
These knots are independent (in the sense described in previous chapters) and so the concordance order of $K$ is the maximum of the concordance orders of these knots. (Strictly speaking it is the least common multiple, but since the only possible orders are $1,2,4$ or $\infty$ this corresponds to the maximum.)  Since $3_1$ has infinite order, both in $\A$ and $\C$, then so does $K$.

As a further example, let us take $K' = -3(8_{12}) + (9_{34})$. We know that $8_{12}$ has both algebraic and topological order 2 (so $-3(8_{12}) = 8_{12}$) whilst $9_{34}$ is algebraically of order $4$ and topologically of infinite order.  Thus $K'$ is algebraically of order $4$ and topologically of infinite order.

We have made a website which implements this method, which can be found at

\[ \text{\url{http://www.maths.ed.ac.uk/~s0681349/Classification/bigmatrix.html}.}\]

It will output the algebraic and geometric concordance orders of any linear combination of prime 9-crossing knots, as well as giving a breakdown of the knot in terms of the basis knots from Chapter \ref{chapter6}.

\section{Question 2: Finding a decomposition of $K$ in terms of elements of $E$}

Henceforth we will assume that we do not know what $K$ is (for example, we have only some unrecognisable knot diagram).  All that we can do is ask for the values of certain invariants.  We will first try to identify the knots in the decomposition which are non-trivial in $\A$, and then to find component knots which are non-trivial in $\C$.  Notice that, in the eyes of the concordance group $\C$, our knot $K$ is equivalent to
\[ 2(3_1) + 5(7_2) + 8_{12} - 9_8 + 9_{16} + 9_{34}\]
because $8_{12}$ is amphicheiral and $9_{46}$ is slice.  We will use this version of the knot for all invariants which follow.

\subsection{Signatures}

To identify the components of $K$ which have infinite order we will use the $\omega$-signatures as detailed in Section \ref{Sec:signature}.  The $\omega$-signatures of every knot were originally evaluated at $70$ different values, but only $46$ of these values are needed to identify the infinite order knots.

If we evaluate the $\omega$-signature of $K$ at each of the relevant $46$ values of $\omega$, we find that it is non-zero at $\delta_{36}, \delta_{24}, \delta_{15}, \delta_{35} $, which correspond to the knots $3_1, 7_2, 7_3$ and $8_{14}$ respectively. (See Appendix \ref{AppendixD}.)  A closer analysis of the signature tells us the coefficients of these knots: for example, $7_2$ has signature $-2$ at $\delta_{24}$, but when we evaluate $K$ at $\delta_{24}$ we get $-10$, so the coefficient of $7_2$ must be $5$.  After this analysis we know that
\[ K = 3(3_1) + 5(7_2) + 7_3 - 8_{14} + K_1\]
where $K_1$ is a knot of finite order in $\A$.

\subsection{The Alexander Polynomial}

The Alexander polynomial of $K$ is
\begin{eqnarray*}
(1-t+t^2)^3 (3-5t+3t^2)^5 (1-7t+13t^2-7t^3+t^4) (2-8t+11t^2-8t^3+2t^4)& &\\
(2-3t+3t^2-3t^3+2t^4) (1-6t+16t^2-23t^3+16t^4-6t^5+t^6).
\end{eqnarray*}
We will denote the (unique) factors of this polynomial by the letters $A,\dots, F$, starting with $A=(1-t+t^2)$.

There are two polynomials here which are not accounted for by the knot $K-K_1$: namely $C$ and $F$.  Out of the knots in $\C_{\A}^2$ and $\C_{\A}^4$, there is exactly one knot corresponding to each polynomial.  The knot with Alexander polynomial $C$ is $8_{12}$ and the knot with Alexander polynomial $F$ is $9_{34}$.  We now know that
\[ K = 3(3_1) + 5(7_2) + 7_3 + 8_{12} - 8_{14} \pm (4a+1)(9_{34}) + K_2\]
for some $a \in \Z$, where $K_2$ is a finite-order knot whose Alexander polynomial factorises as $f(t)f(t^{-1})$.

\subsection{Witt groups for primes $p\equiv3$ mod 4}

We will examine the image of $K$ in $W(\F_p)$ for primes congruent to 3 mod 4 which divide $\Delta_K(-1)$.  These primes correspond to the polynomials $A,B,D$ and $F$.  The results are given in the following table.

\begin{center}
\begin{tabular}{lccccc}
Polynomial & $A$ & $B$ & $D$ & $F$ & $F$\\
Prime $p$ & 3 & 11 & 31 & 3 & 23\\
Image in $W(\F_p)$ & 1 & 1 & 1 & 3 & 1
\end{tabular}
\end{center}

We now apply a similar analysis to the knot $K-K_2$.
\begin{center}
\begin{tabular}{lccccc}
Polynomial & $A$ & $B$ & $D$ & $F$ & $F$\\
Prime $p$ & 3 & 11 & 31 & 3 & 23\\
Image in $W(\F_p)$ & 3 & 1 & 1 & $\pm3$ & $\pm1$
\end{tabular}
\end{center}

It is easy to see that polynomial $A$ is not giving us the right invariant.  The knot $K_2$ must therefore contain a component which has the image $2 \in W(\F_3)$ and zero in the Witt groups corresponding to the other primes. The knots in $\C_{\A}^2$ and $\C_{\A}^4$ which satisfy this condition and whose Alexander polynomial contains a factor $(1-t+t^2)$ are $8_{18}, (9_{16}-7_3-3_1)$ and $(9_{28}-3_1)$.  However, $8_{18}$ and $(9_{28}-3_1)$ both contain factors with odd exponent in their Alexander polynomials which do not divide $\Delta_K(t)$, so this rules them out.

Comparing images in $W(\F_p)$ for the polynomial $F$ gives us the information that we should take $9_{34}$ and not $-9_{34}$.  Our updated knowledge of $K$ is thus
\[ K = 3(3_1) + 5(7_2) + 7_3 + 8_{12} - 8_{14} + (2b+1)(9_{16}-7_3-3_1) + (4a+1)(9_{34}) + K_3\]
for some $a,b \in \Z$.

\subsection{Witt groups for primes $p\equiv1$ mod 4}

We will examine the image of $K$ in $W(\F_p)$ for primes congruent to 1 mod 4 which divide $\Delta_K(-1)$.  These primes correspond to the polynomials $C$ and $E$.  The results are given in the following table.

\begin{center}
\begin{tabular}{lcc}
Polynomial & $C$ & $E$\\
Prime $p$ & 29 & 13\\
Image in $W(\F_p)$ & (0,1) & (0,1)
\end{tabular}
\end{center}

The image of $K-K_3$ is
\begin{center}
\begin{tabular}{lcc}
Polynomial & $C$ & $E$\\
Prime $p$ & 29 & 13\\
Image in $W(\F_p)$ & (0,1) & (0,1)
\end{tabular}
\end{center}

Thus our current knowledge of $K$ is the best that we can do from looking at the algebraic concordance group.  It means that $K_3$ must be algebraically slice.

\subsection{Twisted Alexander polynomials}

The final tool at our disposal for finding out what knot $K_3$ is, up to concordance, is the twisted Alexander polynomial.  Let us take a quick recap of what we know about our knot so far.

Our original knot $K = 2(3_1) + 5(7_2) + 3(8_{12}) - 9_8 + 9_{16} + 9_{34} + 9_{46}$ has
 \begin{eqnarray*} H_1(\Sigma_2(K);\Z) & \cong & (\Z_3)^2 \oplus (\Z_{11})^5 \oplus (\Z_{29})^3 \oplus \Z_{31} \oplus (\Z_3 \oplus \Z_{13}) \oplus (\Z_3 \oplus \Z_{23}) \oplus \\& & \oplus (\Z_3 \oplus \Z_3)\\
 & \cong & (\Z_3)^6 \oplus (\Z_{11})^5 \oplus \Z_{13} \oplus \Z_{23} \oplus (\Z_{29})^3 \oplus \Z_{31}
 \end{eqnarray*}

Our current estimate of $K$ is
\[ K' = 3(3_1) + 5(7_2) + 7_3 + 8_{12} - 8_{14} + (2b+1)(9_{16}-7_3-3_1) + (4a+1)(9_{34})\]
whose homology is
 \begin{eqnarray*} H_1(\Sigma_2(K');\Z) & \cong & (\Z_3)^3 \oplus (\Z_{11})^5 \oplus \Z_{13} \oplus \Z_{29} \oplus \Z_{31} \oplus ((\Z_3)^2 \oplus (\Z_{13})^2)^{2b+1} \oplus\\ & & (\Z_3 \oplus \Z_{23})^{4a+1}\\
 & \cong & (\Z_3)^{6+4b+4a} \oplus (\Z_{11})^5 \oplus (\Z_{13})^{3+4b} \oplus (\Z_{23})^{4a+1} \oplus \Z_{29} \oplus \Z_{31}
 \end{eqnarray*}

Any algebraically slice knots forming $K_3$ must be detectable at the primes $3,11,13,23,29$ and $31$.

There is only one algebraically slice knot which is nontrivial at the prime $23$: this is $4(9_{34})$.  Since there is only one summand of $\Z_{23}$ in $H_1(\Sigma_2(K))$, this means that the coefficient of $9_{34}$ in $K'$ is $1$, so $a=0$.

For the prime $31$ there is only a single summand in $H_1(\Sigma_2(K))$, so there can only be a single knot to cause it.  We may calculate a non-trivial twisted Alexander polynomial corresponding to this knot and this prime.  If we do this, we will find that it gives a different, non-Galois-conjugate, polynomial to $\Delta_{\chi_1}^{K'}$.  The component $-8_{14}$ in $K'$ is therefore incorrect.  The only algebraically slice knot which is non-trivial at the prime $31$ is $9_8 - 8_{14}$.  If we let $K'' := K' - (9_8 - 8_{14})$, then we find that we indeed get the correct twisted Alexander polynomial.

Continuing this method, there is only one $\Z_{13}$ summand in $H_1(\Sigma_2(K))$ so it must be generated by only one knot.  The possible $\Z_{13}$ algebraically slice knots are $2(6_3)$, $2(8_1)$ and $2(9_{16}-7_3-3_1)$.  There is no combination of these which would produce a single $\Z_{13}$ summand, from which we conclude that $b=0$.

There are three $\Z_{29}$ summands, only one of which is accounted for by $K''$.  The algebraically slice knots which may account for this are $2(8_{12})$ and $\pm 2(8_{13})$.  The knot $8_{13}$ is of infinite order and this is detected by the twisted Alexander polynomial, whilst $8_{12}$ is amphicheiral and has twisted Alexander polynomials which are Galois conjugates of each other.  In principle it should be possible to compute the twisted Alexander polynomials for $K$ and decide which case we are in, though in practice our current Maple program gets stuck because the homology has the wrong form. That is, the homology of $H_1(\Sigma_2(2J))$ will always be a direct double for any knot $J$, precluding the use of the $2$-fold cover, whilst higher covers will have a quadruple summand.

At the prime $q=11$ there are two possible algebraically slice knots which could be part of $K_3$.  These are $2(8_{15}-7_2-3_1)$ and $9_{39}+7_2-4_1$.  However, the second of these two knots has non-trivial twisted Alexander polynomials at the prime $5$, and this prime does not divide the homology of $\Sigma_2(K)$.  Both $2(8_{15}-7_2-3_1)$ and $4(8_{15}-7_2-3_1)$ are possibilities since they preserve the $(\Z_{11})^5$ summand.  It should be possible to detect the presence (or, in this case, absence) of an $8_{15}$ component using twisted Alexander polynomials, since the $8_{15}$ twisted polynomials are (Galois-) distinct from the $7_2$ twisted polynomials.  Unfortunately, as we have already lamented, the program which currently exists cannot compute twisted polynomials for composite knots where two or more constituent knots have the same primes dividing the homology of their branched covers.

Up to this point, we have determined that
\begin{eqnarray*}
K'' & = & 3(3_1) + 5(7_2) + 7_3 + 3(8_{12}) - 9_8 + (9_{16}-7_3-3_1) + 9_{34}\\
    & = & 2(3_1) + 5(7_2) + 3(8_{12}) - 9_8 + 9_{16} + 9_{34}
\end{eqnarray*}

The only part of the homology of $K$ not accounted for is a factor $(\Z_3)^2$.  The algebraically slice knots whose $H_1(\Sigma_2)$ has order a power of $3$ are $6_1, 8_{10}+3_1, 8_{11}-3_1, 8_{20}$ and $9_{46}$.  These knots all happen to be slice, so up to concordance is it impossible to distinguish them.  However, we can say a little more than that.  The knots $6_1$ and $8_{20}$ would introduce a $\Z_9$ summand to the homology, whilst $8_{10}+3_1$ and $8_{11}-3_1$ would both introduce a $\Z_{27}$ summand.  Since the factor we are looking for is $\Z_3 \oplus \Z_3$, it must therefore come from the knot $9_{46}$.

Thus we have recovered the fact that
\[ K = 2(3_1) + 5(7_2) + 3(8_{12}) - 9_8 + 9_{16} + 9_{34} + 9_{46} \text{ .} \]

\subsection{Limitations}

Aside from the problem of computing twisted Alexander polynomials for composite knots, there are also limitations of this algorithm if we stop restricting ourselves to just the $9$-crossings knots.  For example, if a generic knot were given which claimed to be concordant to a linear combination of $9$-crossing prime knots, there is no guarantee that we could find this linear combination.  For example, suppose that a knot in the sum were concordant to $9_8$, but the homology of its 2-fold branched cover were $\Z_{31} \oplus \Z_{{31}^2}$.  The twisted Alexander polynomial theorems in Chapter \ref{chapter5} cannot deal with this sort of homology, so we would be stuck. 

\chapter{\label{chapter8} Unknown concordance orders}

The website KnotInfo (\url{http://www.indiana.edu/~knotinfo/}) tabulates all known knot invariants for prime knots of up to 12 crossings.  The algebraic concordance order is known for every such knot, but there are many unknown values for the smooth and topological concordance orders.  The smallest knot with unknown (topological) concordance order is $8_{13}$, whilst $12a_{631}$ is the only knot which has neither been proven to be slice or not slice.  The most recent list of unknown values, as of the writing of this thesis, may be found in \cite{ChaLivingston09}.

In this chapter we apply the theorems from Chapter \ref{chapter5} to each of the $325$ knots of unknown concordance order.  If the conditions of the theorems hold then these knots are of infinite order (smoothly and topologically).  If the conditions do not hold, we have investigated whether the knots have finite order and for some knots have found that this is indeed the case.  There remain only two knots for which the concordance order is unknown.  The main difficulty in this exercise is showing when a knot \emph{is} of finite order, since the only method of doing this is to find a slice movie for the knot (see Section \ref{Subsec:sliceMovie}).  When the knots involved have $11$ or more crossings, showing that they are of order $2$ means manipulating a knot diagram which has over $22$ crossings. The only shortcut available is when a knot happens to be concordant to a smaller knot of order $2$ (such as the Figure Eight knot $4_1$), but even proving this concordance is not always easy.

\section{Finding the concordance orders of prime knots up to 12 crossings}

There are $325$ knots listed as having unknown topological concordance order in Knotinfo. ($247$ of these also have unknown smooth concordance order.)  Of these, we can immediately find the concordance orders of nine of them because of previously existing theorems.  The knot $11n_{34}$ is slice because it has Alexander polynomial $\Delta_K(t) = 1$ \cite[11.7B]{FreedmanQuinn}.  The knots $12a_{48}$, $12a_{60}$, $12a_{130}$, $12a_{291}$, $12a_{303}$, $12a_{699}$, $12a_{1100}$ and $12a_{1770}$ each have a determinant which factors into primes that are congruent to $3$ modulo $4$ (with odd powers), so by Theorem \ref{Thm:LivNaik3mod4} they are of infinite order. The knot $12a_{1288}$ is known to be fully amphicheiral, so is therefore of order 2.

To do our first real batch of analysis we will use Theorem \ref{Thm:twistedpoly}, which we will restate here.

\begin{thmtwistedpoly}
Suppose that we have a knot $K$ where $H_1(\Sigma_2;\Z) \cong \Z_{q} \oplus T$ for some prime $q \equiv 1$ mod $4$, where the order of $T$ is coprime to $q$.  Let $\chi_0 \colon H_1(\Sigma_2;\Z) \to \Z_q$ be the trivial map and $\chi_i \colon H_1(\Sigma_2;\Z) \to \Z_q$ be $\lk(-,i)$.  Construct $\Delta_{\chi_1}(t)$ as in Section \ref{Sec:twistpolyobstr}. Then $K$ is of infinite order if it satisfies the following conditions:
  \begin{enumerate}
    \item $\Delta_{\chi_0}(t)$ is not a norm.
    \item There is a non-trivial irreducible factor $f(t)$ of $\Delta_{\chi_0}(t)$ for which $\overline{f(t^{-1})}$ is not a factor of $\Delta_{\chi_i}(t)$ for any $i$.
    \item $\Delta_{\chi_a}(t) \not\doteq \Delta_{\chi_1}(t)$, where $1+a^2 \equiv 0$ (mod $q$).
  \end{enumerate}
\end{thmtwistedpoly}

This theorem applies to all the remaining $316$ knots of unknown concordance order with the exception of the following:
    \begin{itemize}
      \item \textbf{Condition (1) not satisfied:} $12a_{912}$, $12n_{488}$, $12n_{499}$, $12n_{587}$, $12n_{690}$.
      \item \textbf{Condition (3) not satisfied:} $11a_{44}$, $11a_{47}$, $11a_{109}$.
      \item \textbf{Conditions (1), (2) and (3) not satisfied:} $10_{158}$, $11n_{85}$, $11n_{100}$, $12a_{309}$, $12a_{310}$, $12a_{387}$, $12a_{388}$, $12n_{286}$, $12n_{388}$.
      \item \textbf{$H_1(\Sigma_2)$ is not of the form $\Z_q \oplus T$:} $11a_{5}$, $11a_{67}$, $11a_{104}$, $11a_{112}$, $11a_{168}$, $11n_{45}$, $11n_{145}$, $12a_{169}$, $12a_{360}$, $12a_{596}$, $12a_{631}$, $12a_{836}$, $12n_{31}$, $12n_{132}$, $12n_{210}$, $12n_{221}$, $12n_{224}$, $12n_{264}$, $12n_{367}$, $12n_{480}$, $12n_{532}$, $12n_{536}$, $12n_{579}$, $12n_{631}$, $12n_{681}$, $12n_{731}$, $12n_{745}$, $12n_{760}$, $12n_{812}$, $12n_{813}$, $12n_{841}$, $12n_{846}$, $12n_{884}$.
    \end{itemize}
  This is a total of $51$ knots, meaning that the theorem has an 84\% success rate.

We will first address those knots for which condition (1) was not satisfied; namely those knots for which $\Delta_{\chi_0}$ was a norm.  This is exactly the case that Theorem \ref{Thm:twistedpolyAlt} was designed for.

\begin{thmtwistedpolyAlt}
Suppose that we have a knot $K$ where $H_1(\Sigma_2;\Z) \cong \Z_{q} \oplus T$ for some prime $q \equiv 1$ mod $4$, where the order of $T$ is coprime to $q$.  Let $\chi_0 \colon H_1(\Sigma_2;\Z) \to \Z_q$ be the trivial map and $\chi_i \colon H_1(\Sigma_2;\Z) \to \Z_q$ be $\lk(-,i)$.  Suppose that $\Delta_{\chi_0}$ is a norm.  Then $K$ is of infinite order if it satisfies the following condition:
  \begin{itemize}
    \item $\Delta_{\chi_i}(t)$ is coprime, up to norms in $\Q(\zeta_q)[t,t^{-1}]$, to $\Delta_{\chi_j}(t)$ for all $i \neq j$, $i,j>0$.
  \end{itemize}
\end{thmtwistedpolyAlt}

The knots $12a_{912}$, $12n_{488}$, $12n_{499}$, $12n_{587}$ and $12n_{690}$ do satisfy the conditions of this theorem and are therefore all of infinite order.  This brings us down to $46$ knots of unknown concordance order.

To attack the knots which have the `wrong' kind of homology, we need Theorem \ref{Thm:twistedpolyq^n}.

\begin{thmtwistedpolyq^n}
Suppose that we have a knot $K$ where $H_1(\Sigma_2;\Z) \cong \Z_{q^n} \oplus T$ for some prime $q$, where the order of $T$ is coprime to $q$.  Let $\chi_0 \colon H_1(\Sigma_2;\Z) \to \Z_{q}$ be the trivial map and $\chi_i \colon H_1(\Sigma_2;\Z) \to \Z_{q}$ be $\lk(-,i)$ (mod $q$).  Construct $\Delta_{\chi_1}(t)$ as in Section \ref{Sec:twistpolyobstr}. Then $K$ is of infinite order if it satisfies the following conditions:
  \begin{enumerate}
    \item If $n>1$, $\Delta_{\chi_0}(t)\Delta_{\chi_1}(t)$ is not a norm.
    \item $\Delta_{\chi_0}(t)$ is not a norm.
    \item $\Delta_{\chi_0}(t)$ is coprime, up to norms, to $\Delta_{\chi_i}(t)$ for all $i \neq 0$.
    \item If $q \equiv 1$ (mod $4$) then $\Delta_{\chi_a}(t) \not\doteq \Delta_{\chi_1}(t)$, where $1+a^2 \equiv 0$ (mod $q$).
  \end{enumerate}
Alternatively, if $\Delta_{\chi_0}$ is a norm, conditions (2)-(4) may be replaced by
   \begin{enumerate}
     \item[$2'$.] $\Delta_{\chi_i}$ is coprime, up to norms, to $\Delta_{\chi_j}$ for any $i \neq j$, $i,j >0$.
   \end{enumerate}
\end{thmtwistedpolyq^n}

This theorem applies to the following $19$ knots to show that they are of infinite order: $10_{158}$, $11n_{45}$, $11n_{145}$, $12a_{169}$, $12a_{360}$, $12n_{31}$, $12n_{132}$, $12n_{221}$, $12n_{224}$, $12n_{264}$, $12n_{532}$, $12n_{536}$, $12n_{579}$, $12n_{631}$, $12n_{681}$, $12n_{731}$, $12n_{812}$, $12n_{841}$ and $12n_{884}$.

Of the remaining $27$ knots, the following problems with the theorem arise:
  \begin{itemize}
    \item \textbf{Condition (1) not satisfied, so $\Delta_{\chi_0}\Delta_{\chi_1} \doteq 1$:} $11n_{85}$, $11n_{100}$, $12a_{309}$, $12a_{310}$, $12a_{387}$, $12a_{388}$, $12n_{286}$.
    \item \textbf{Condition (4) not satisfied, so $\Delta_{\chi_1}\Delta_{\chi_a} \doteq 1$:} $11a_{44}$, $11a_{47}$, $12a_{836}$.
    \item \textbf{Both $\Delta_{\chi_0}$ and $\Delta_{\chi_1}$ are norms:} $11a_{5}$, $11a_{67}$, $11a_{104}$, $11a_{109}$ $11a_{112}$, $11a_{168}$, $12a_{596}$, $12a_{631}$, $12n_{367}$, $12n_{388}$.
    \item \textbf{$H_1(\Sigma_2)$ is not of the form $\Z_{q^n} \oplus T$:} $12n_{210}$, $12n_{480}$, $12n_{745}$, $12n_{760}$, $12n_{813}$, $12n_{846}$.
  \end{itemize}

The first bullet point is interesting because here we have examples where neither $\Delta_{\chi_0}$ nor $\Delta_{\chi_1}$ are norms, but their product is a norm.  At this point it seemed worth investigating these knots to see if any of them were of finite order in $\C$.  It turned out that $11n_{85}$, $11n_{100}$ and $12n_{286}$ were all of order $2$, meaning that condition (1) of Theorem \ref{Thm:twistedpolyq^n} is a necessary one.  Upon further investigation, we also found that $11a_5$ and $12n_{388}$ are elements of order $2$; likewise the knots $11a_{104}$, $11a_{112}$ and $11a_{168}$ have been proved to be of order $2$ by Kate Kearney (in unpublished work).  None of these knots are either positive or negative amphicheiral, but they are concordant to the Figure-8 knot $4_1$.

We now have $19$ knots of unknown concordance order.  Clearly no further progress (using the existing theorems) can be made using the $2$-fold branched cover of the knots, so it is time to move on to information contained within the higher branched covers.

The setup for Theorem \ref{Thm:Highercover} is as follows.  Take a knot $K$ with $H_1(\Sigma_p;\Z)\cong \Z_q \oplus \Z_q \cong E_a \oplus E_b$, where $E_a$ and $E_b$ are the eigenspaces of the deck transformation.  Let $e_a$ be an $a$-eigenvector (i.e. $a e_a = e_a$) and $e_b$ be a $b$-eigenvector. Define $\chi_a \colon H_1(\Sigma_p) \to \Z_q$ by $\chi_a(e_a) = 0$ and $\chi_a(e_b) = 1$.  Similarly, $\chi_b \colon H_1(\Sigma_p) \to \Z_q$ is defined by $\chi_b(e_a) = 1$ and $\chi_b(e_b) = 0$.

\begin{thmHighercover}
The knot $K$ is of infinite order in $\C$ if the following conditions on the twisted Alexander polynomial of $K$ are satisfied:
\begin{enumerate}
  \item $\Delta_{\chi_0}$ is coprime, up to norms, to both $\Delta_{\chi_a}$ and $\Delta_{\chi_b}$, and $\Delta_{\chi_0}$ is not a norm.
  \item $\Delta_{\chi_a + \chi_b} \not\doteq \Delta_{d \chi_a -d^{-1}\chi_b}$ for any $d \in \Z_q$.
\end{enumerate}
\end{thmHighercover}

Let us investigate whether we can apply this theorem to any of our remaining $19$ knots.

First we will look at the information contained in the $3$-fold branched cover.  The conditions in the theorem turn out to apply to only two of the remaining knots: $11a_{67}$ and $12n_{367}$.  However, there are another three knots whose only problem is that $\Delta_{\chi_0}$ is a norm: $12a_{596}$, $12n_{210}$ and $12n_{813}$.  In the same way that Theorems \ref{Thm:twistedpoly} and \ref{Thm:twistedpolyq^n} had an alternative version in the case of $\Delta_{\chi_0}$ being a norm, so does Theorem \ref{Thm:Highercover}.  The proof still holds if we replace assumption (1) with
   \begin{enumerate}
     \item[$1'$.] $\Delta_{\chi_a}$ and $\Delta_{\chi_b}$ are not norms and are coprime over $\Q(\zeta_q)[t,t^{-1}]$ (up to norms) to $\Delta_{\chi_0}$.
   \end{enumerate}
This assumption ensures that if there is an `odd' vector in $M_a$ or $M_b$ (see the proof of Theorem \ref{Thm:Highercover} for details) then the corresponding twisted Alexander polynomial is not a norm.  The three knots $12a_{596}$, $12n_{210}$ and $12n_{813}$ do satisfy this condition, so we know they are of infinite order.

Of the remaining $14$ knots there are examples of both condition (1) (and ($1'$)) failing ($12a_{309}$, $12a_{310}$, $12a_{387}$, $12a_{388}$ and $12a_{631}$) and condition (2) failing ($11a_{44}$, $11a_{47}$, $11a_{109}$ and $12n_{745}$) as well as knots whose homology does not split into two eigenspaces ($12a_{836}$, $12n_{480}$, $12n_{760}$ and $12n_{846}$).

Moving on to yet higher order covers, the infinite order of $12n_{480}$, $12n_{745}$ and $12n_{760}$ is detected by the $7$-fold cover, whilst $12a_{836}$ is proved to be of infinite order via the $13$-fold cover.  Unfortunately the knot $12n_{846}$ has incredibly large primes in the homology of its covers: $701$ for the $5$-fold cover, $8681$ for the $7$-fold cover and $1385341$ for the $11$-fold cover.  It is not possible to compute the corresponding twisted Alexander polynomials using the computing resources that we have had access to during this thesis. It is therefore difficult to conjecture whether $12n_{846}$ will turn out to be of infinite order or finite order.

The knots $11a_{44}$, $11a_{47}$ and $11a_{109}$ all turn out to be of order $2$, concordant to the amphicheiral knot $6_3$.  The knots $12a_{309}$, $12a_{310}$, $12a_{387}$ and $12a_{388}$ are also of order $2$, concordant to the amphicheiral knot $4_1$.

The remaining knot $12a_{631}$ seems likely to be of finite order, since it fails to satisfy the conditions of each of the relevant theorems in some way. It is still unknown whether or not $12a_{631}$ is slice.

\section{Summary}

Of the $316$ prime knots of unknown concordance order (i.e. ignoring the knots listed at the beginning of the chapter which can be dealt with using existing theorems), they can all be proven to be of infinite order in $\C$ by using twisted Alexander polynomials computed from the $2$-fold branched cover, with the exception of the following:

\begin{itemize}
 \item $11a_{67}$, $12a_{596}$, $12n_{210}$, $12n_{367}$ and $12n_{813}$ are shown to be of infinite order via their $3$-fold covers.
 \item $12n_{480}$, $12n_{745}$ and $12n_{760}$ are shown to be of infinite order via their $7$-fold covers.
 \item $12a_{836}$ is shown to be of infinite order via its $13$-fold cover.
 \item $11a_5$, $11a_{104}$, $11a_{112}$, $11a_{168}$, $11n_{85}$, $11n_{100}$, $12a_{309}$, $12a_{310}$, $12a_{387}$, $12a_{388}$, $12n_{286}$ and $12n_{388}$ are all of order $2$, concordant to the Figure Eight knot $4_1$.
 \item $11a_{44}$, $11a_{47}$ and $11a_{109}$ are all of order $2$, concordant to the knot $6_3$.
 \item $12a_{631}$ remains of unknown order, but is suspected to be finite order and possibly slice.
 \item $12n_{846}$ remains of unknown order, and there are no suspicions as to whether it is of finite or infinite order in $\C$.
\end{itemize}

In summary, the theorems given in Chapter \ref{chapter5} have a 99\% success rate in detecting knots of infinite order.

Slice movies of the finite order knots can be found in Appendix \ref{AppendixE}. 
\chapter{\label{chapter10} Torus signatures}

\section{Introduction}
In this chapter we find a formula for the $L^2$ signature of a $(p,q)$ torus knot, which is the integral of the $\omega$-signatures over the unit circle.  We then apply this to a theorem of Cochran, Orr and Teichner to prove that the $n$-twisted doubles of the unknot, $n\neq 0,2$, are not slice.  This is a new proof of the result first proved by Casson and Gordon.

\begin{note}
It has been drawn to our attention that the main theorem of this chapter, Theorem \ref{Thm:L2torussig}, was first proved in $1993$ by Robion Kirby and Paul Melvin \cite{KirbyMelvin94} using essentially the same method presented here.  The theorem has also recently been reproved using different techniques by Maciej Borodzik \cite{Borodzik09}.
\end{note}

Before we recall the definition of $\omega$-signatures, let us first motivate the subject with an elementary but difficult problem in number theory.  Suppose we have two coprime integers, $p$ and $q$, together with another (positive) integer $n$ which is neither a multiple of $p$ nor of $q$.  Write
\[ n= ap+bq, \quad a,b \in \Z, \quad 0 < a < q. \]
Now we ask the question:
\begin{center}
  ``Is $b$ positive or negative?''
\end{center}

Clearly, given any \emph{particular} $p$ and $q$, the answer is easy to work out, so the question is whether there is an (explicit) formula which could anticipate the answer.  Let us define
\[ j(n) = \begin{cases}\phantom{-}1 & \text{if $b>0$,}\\ -1 & \text{if $b<0$}\end{cases} \]
and let us study the sum
\[ s(n) = \sum_{i=1}^n j(i)\]
as $n$ varies between $1$ and $pq-1$.  It would not be an unreasonable first guess to suggest that $j(n)$ is $-1$ for the first $\lfloor \frac{pq}{2} \rfloor$ values of $n$, and $+1$ for the other half of the values.  Indeed, this is true when $p=2$ (see the nice `V' shape in Figure~\ref{fig:edge-a}).  But if we investigate other values of $p$ and $q$ then strange `wiggles' in the graph of $s$ start appearing (see Figures~\ref{fig:edge-b}, \ref{fig:edge-c} and \ref{fig:edge-d}). The pattern of wiggles seems different for every $p$ and $q$: could there be an underlying principle to explain them?

\begin{figure}[htp]
  \begin{center}
    \subfigure[$p=2$, $q=19$]{\label{fig:edge-a}\includegraphics[width=7cm]{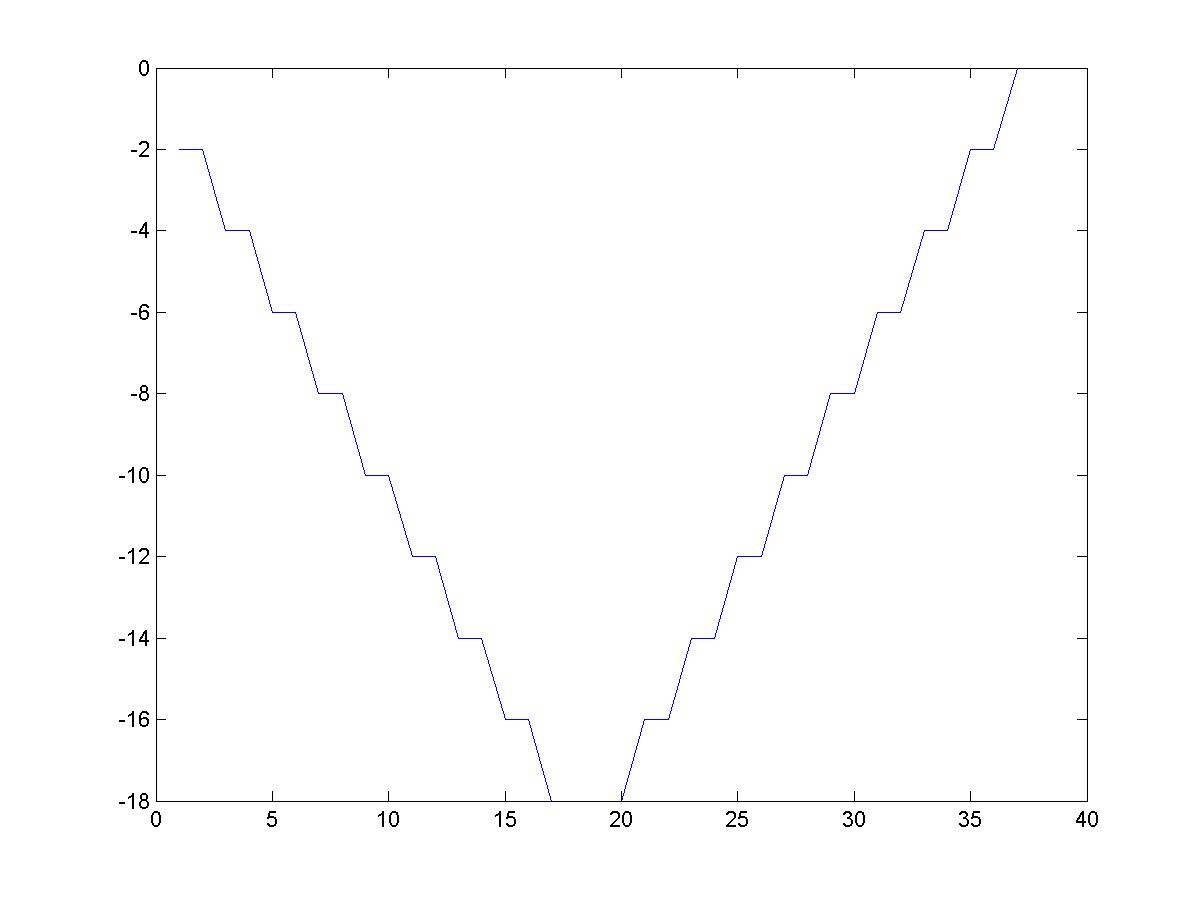}}
    \subfigure[$p=3$, $q=10$]{\label{fig:edge-b}\includegraphics[width=7cm]{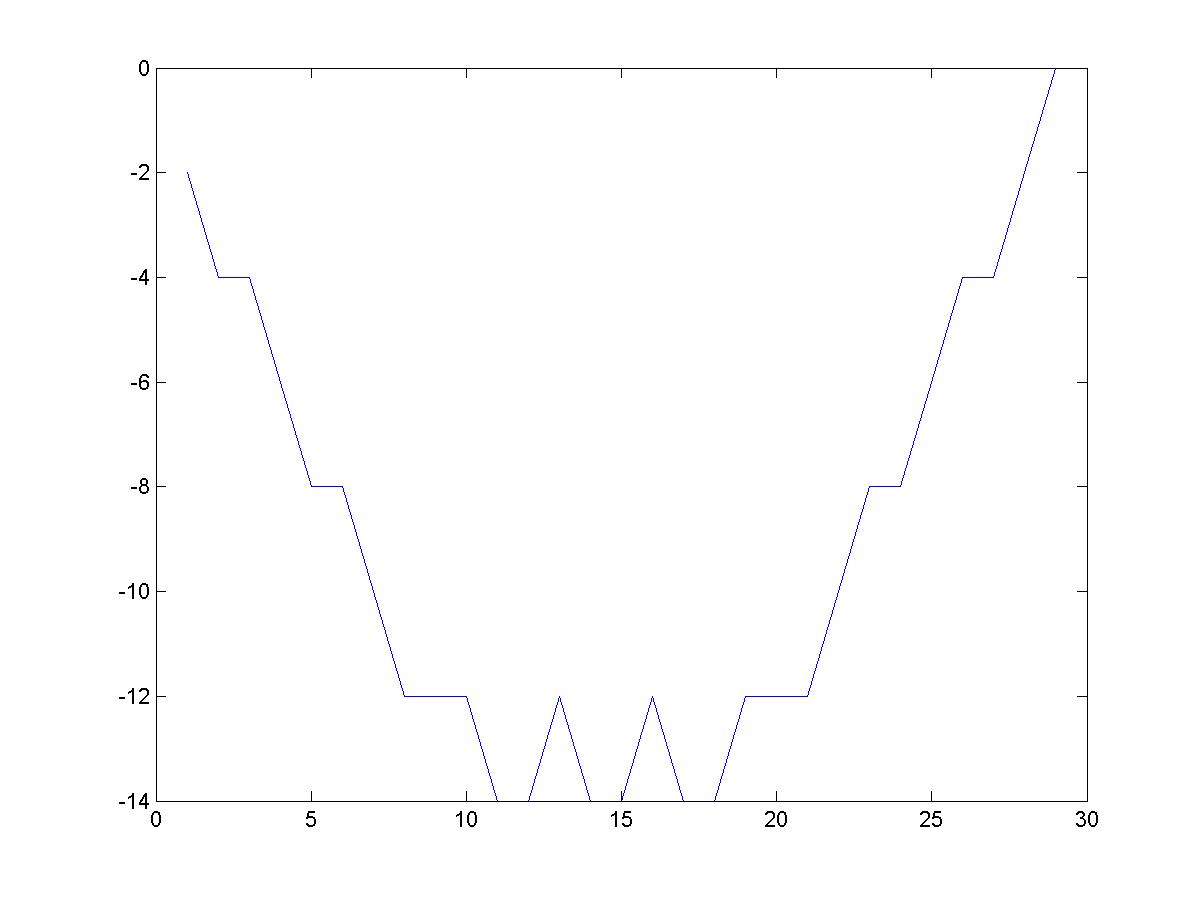}}
    \subfigure[$p=5$, $q=24$]{\label{fig:edge-c}\includegraphics[width=7cm]{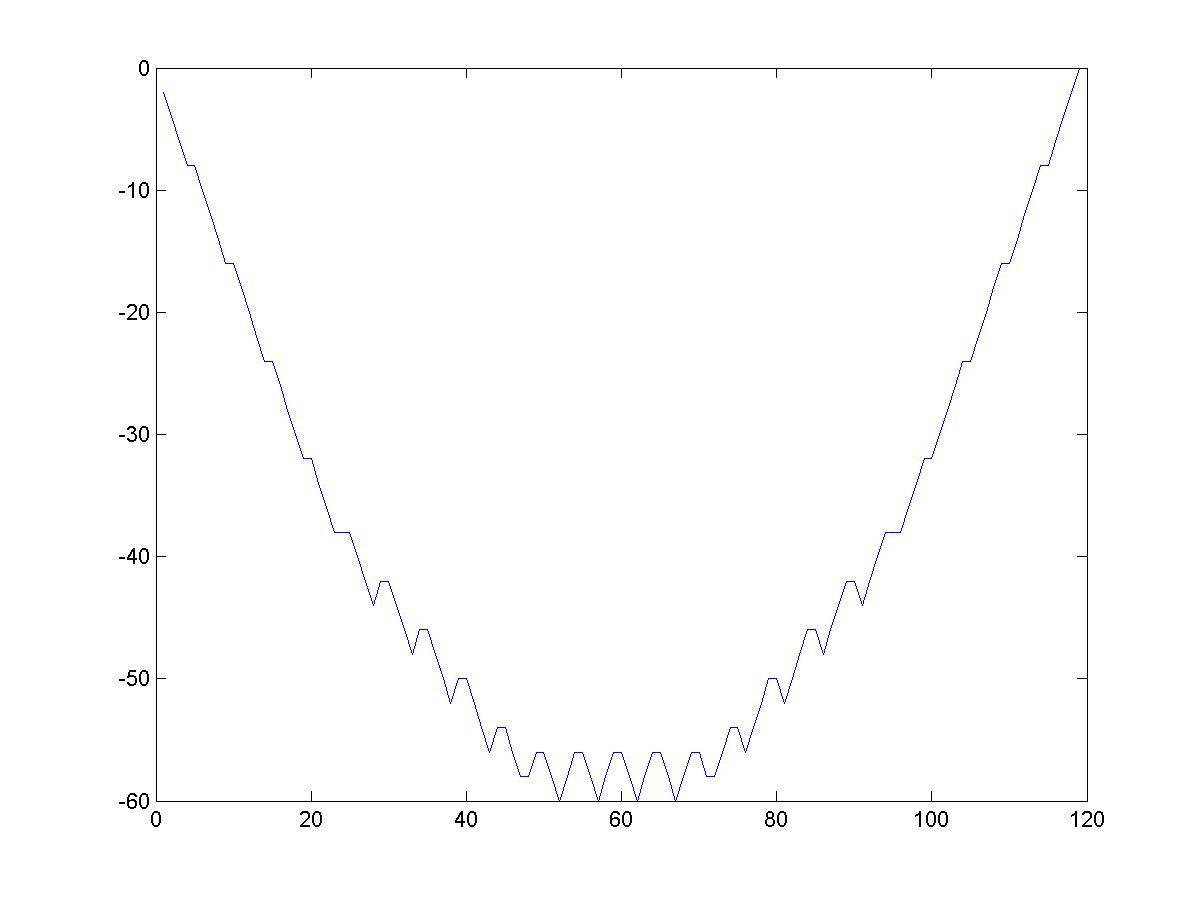}}
    \subfigure[$p=7$, $q=16$]{\label{fig:edge-d}\includegraphics[width=7cm]{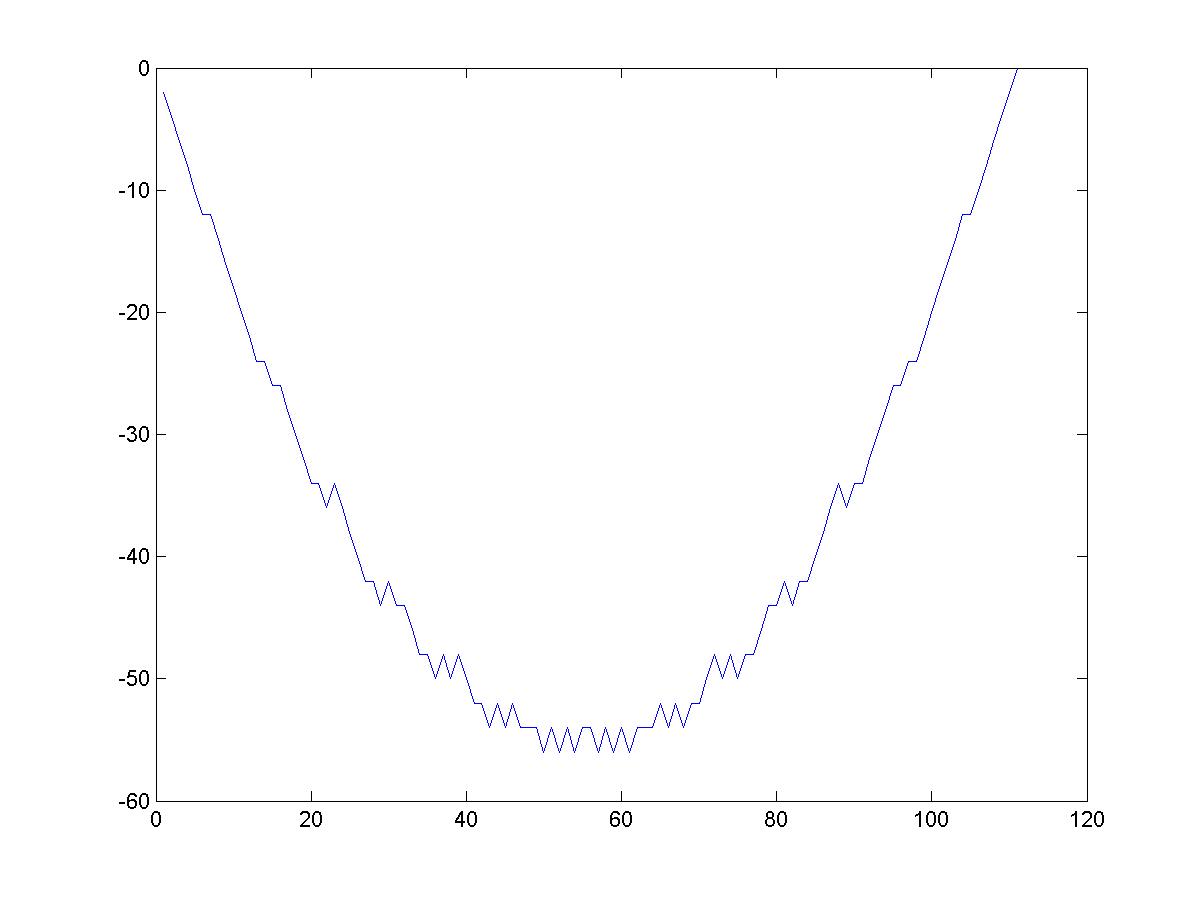}}
  \end{center}
  \caption{Graphs of $2s$ for various values of $p$ and $q$.}
  \label{graphs}
\end{figure}

It turns out that the clue to finding the pattern is to realise that the function $j$ is the jump function of the $\omega$-signature of a $(p,q)$ torus knot.  What does this mean?  A torus knot $T_{p,q}$ is a knot which lives on the boundary of a torus, wrapping $p$ times around the meridian and $q$ times around the longitude. (If $p$ and $q$ are not coprime then $T_{p,q}$ is a link rather than a knot.)  Given a non-singular Seifert matrix $V$ for $T_{p,q}$ and a unit complex number $\omega$, the \emph{$\omega$-signature} $\sigma_\omega$ is the sum of the signs of the eigenvalues of the hermitian matrix
\[(1-\omega)V + (1-\overline{\omega})V^T.\]
This is independent of the choice of Seifert matrix.  The $\omega$-signature is an integer-valued function that is continuous (and therefore constant) everywhere except at the unit roots of the Alexander polynomial $\Delta_{p,q}(t) = \det(V-tV^T)$.  At these points, the signature `jumps', with the value of the jump at $\omega = e^{2\pi i n/pq}$ being given by $2j(n)$.

The $\omega$-signature has proved to be useful in a variety of areas of mathematics; see Stoimenow's paper \cite{Stoimenow05} to get a comprehensive list, including applications to unknotting numbers of knots, Vassiliev invariants and algebraic functions on projective spaces. The jump function in particular appears to be related to the Jones polynomial~\cite{Garoufalidis03}.  But the most important use for signatures is in the study of knot concordance.

Cochran, Orr and Teichner~\cite{COT03} have recently been probing the secrets of the concordance group $\C$ using the difficult techniques of $L^2$ signatures.  Amazingly, it turns out that a special case of these $L^2$ signatures is the integral of the $\omega$-signatures over the unit circle.

Even more amazingly, we find that despite the $\omega$-signatures being fairly unpredictable for a torus knot, the \emph{integral} of the $\omega$-signatures has the following beautiful formula.

\begin{thmtorussig}
  Let $p$ and $q$ be coprime positive integers.  Then the integral of the $\omega$-signatures of the $(p,q)$ torus knot is
  \[ \displaystyle \int_{S^1} \sigma_{\omega} = -\frac{(p-1)(p+1)(q-1)(q+1)}{3pq} \text{ .} \]
\end{thmtorussig}

In this chapter we prove Theorem \ref{Thm:L2torussig}, the reason for which is that the result can be combined with a theorem in \cite{COT03} to recover the old Casson-Gordon theorem that the twist knots are not slice.  The hope is that the techniques here may prove useful in investigating signatures of other families of knots and in proving more general theorems about the structure of the concordance group.

\section{Signatures and jump functions}
\label{Sec:sigAndjump}

Let $K$ be a knot, $V$ be a Seifert matrix for $K$ of size $2g \times 2g$ and $\omega$ be a unit complex number.  We would like to define the $\omega$-signature to be the signature of $P:=(1-\omega)V + (1-\overline{\omega})V^T$.  However, notice that $P=(1-\omega)(V-\overline{\omega}V^T)$ and $\det P = (1-\omega)^{2g}\Delta_K(\overline{\omega})$ where $\Delta_K(t) := \det(V-tV^T)$ is the Alexander polynomial of $K$.  This means that $P$ becomes degenerate at the unit roots of the Alexander polynomial and we will need an alternative definition of the signature at these points.

\begin{definition}
For a unit complex number $\omega$ which is not a root of the Alexander polynomial $\Delta_K$, the \emph{$\omega$-signature} $\sigma_\omega(K)$ is the signature (i.e. the sum of the signs of the eigenvalues) of the hermitian matrix
  \[ P:=(1-\omega)V + (1-\overline{\omega})V^T \text{ .} \]
If $\omega$ is a unit root of $\Delta_K$, we define $\sigma_\omega(K)$ to be the average of the limit on either side.
\end{definition}

This concept was formulated independently by Levine~\cite{Levine69} and Tristram~\cite{Tristram69}; hence the $\omega$-signature is sometimes called the \emph{Levine-Tristram signature}.  It is a generalisation of the usual definition of a knot signature, i.e. the signature of $V+V^T$, which was developed by Trotter and Murasugi \cite{Trotter62, Murasugi65}.

For any knot $K$, the function $\sigma_\omega(K)$ is continuous as a function of $\omega$ except at roots of the Alexander polynomial.  Since $\sigma_\omega$ is integer-valued, this means that it is a step function with jumps at the roots of $\Delta_K$.

\begin{definition}
  The jump function $j_K \colon [0,1) \to \Z$ of a knot $K$ is defined by
    \[ j_K(x) = \displaystyle \frac{1}{2} \lim_{\eps \to 0} (\sigma_{\xi^+}(K) - \sigma_{\xi^-}(K)) \]
  where $\xi^+ = e^{2\pi i (x + \eps)}$ and $\xi^- = e^{2\pi i (x - \eps)}$ for $\eps>0$.
\end{definition}

\begin{lemma}
\label{Lemma:jump_properties}
  The jump function $j_K$ and the $\omega$-signature $\sigma_\omega(K)$ have the following properties.
  \begin{enumerate}
    \item $j_K(x) = 0$ if $e^{2\pi i x}$ is not a unit root of the Alexander polynomial of $K$.
    \item In particular, $j_K(0)=0$ and $\sigma_1(K) = 0$.
    \item $\sigma_\omega(K) = \sigma_{\overline{\omega}}(K)$ so $j_K(x) = -j_K(1-x)$.
    \item $\displaystyle \sigma_{e^{2\pi i x}}(K) = 2\sum_{y \in [0,x]} j_K(y)$ if $e^{2\pi i x}$ is not a root of the Alexander polynomial of $K$. (Notice that this is a finite sum because only finitely many of the jumps are non-zero.)
  \end{enumerate}
\end{lemma}

We saw in Section \ref{subsec:SeifertMatricesSlice} (specifically, Corollary \ref{Cor:sliceinvts}) that the $\omega$-signatures (excepting the jump points) vanish for slice knots, meaning they are concordance invariants.  In fact, it turns out that the integral of the $\omega$-signatures, for $\omega \in S^1$, is a special case of a more powerful invariant.

\begin{definition}
An $L^2$-signature (or $\rho$-invariant) of a knot $K$  is a number $\rho(M,\phi) \in \R$ associated to a representation $\phi\colon \pi_1(M) \to \Gamma$, where $M$ is the zero-framed surgery on $S^3$ along $K$ and $\Gamma$ is a group.
\end{definition}

The precise definition is complicated and may be found in \cite[Section 5]{COT03}.  $L^2$ signatures are, in general, very difficult to compute.  However, if we pick a nice group for $\Gamma$ then magic happens and we get an explicit formula:

\begin{lemma}\cite[Lemma 5.4]{COT03}
  When $\Gamma = \Z$ and $\phi$ is the canonical abelianisation epimorphism, we have that $\rho(M,\phi) = \int_{\omega \in S^1} \sigma_{\omega}(K)$, normalised so that the circle has total length $1$.
\end{lemma}

Henceforth we shall refer to the integral of the $\omega$-signatures as \emph{the} $L^2$ signature of the knot.  We end the section with a formula relating the $L^2$ signature of a knot to its jump function.

\begin{lemma}
\label{Lemma:L2asjump}
Suppose that the unit roots of the Alexander polynomial of a knot $K$ are $\omega_k = e^{2\pi i x_k}$ for $k=1, \dots, n$ and $x_1 < \dots < x_n$.  Then the $L^2$ signature of $K$ is
\[ \int_{\omega \in S^1} \sigma_{\omega}(K) = 2\sum_{i=1}^{n-1} \left(x_{i+1}-x_i \right) \sum_{k=1}^i j_K(x_k)\]
\end{lemma}

\begin{proof}
Let $\xi_k$ be any unit complex number between $\omega_k$ and $\omega_{k+1}$ for $k=1, \dots, n-1$. Then we have that
\[ \int_{\omega \in S^1} \sigma_{\omega}(K) = \sum_{i=1}^{n-1} \left(x_{i+1}-x_i\right) \sigma_{\xi_i}\]
where we multiply by $(x_{i+1}-x_i)$ because that is the proportion of the unit circle which has signature $\sigma_{\xi_i}$. We now use (1) and (4) of Lemma \ref{Lemma:jump_properties} to rewrite $\sigma_{\xi_i}$ in terms of the jump function:
\[ \sigma_{\xi_i} = 2\sum_{y \in [0, x_i]} j_K(y)\, = 2\sum_{k=1}^{i} j_K(x_k)\]
\end{proof}

\begin{example}
To illustrate the notation in Lemma \ref{Lemma:L2asjump} we shall calculate the $L^2$ signature of the knot $K := 5_1$, otherwise known as the cinquefoil knot or the $(2,5)$ torus knot.  The Alexander polynomial of $K$ is
\[ 1-t+t^2-t^3+t^4 = \frac{(1-t^{10})(1-t)}{(1-t^2)(1-t^5)}\]
whose roots are the $10^{\text{th}}$ roots of unity that are neither $5^{\text{th}}$ roots of unity nor $-1$.  This means that the roots are $\omega_k =  e^{2\pi i x_k}$ where $x_1 = \frac{1}{10}$, $x_2 = \frac{3}{10}$, $x_3 = \frac{7}{10}$ and $x_4 = \frac{9}{10}$.

Let $\xi_1 = e^{\frac{4}{5}\pi i}$, $\xi_2 = e^{\pi i}$ and $\xi_3 = e^{\frac{8}{5}\pi i}$.  Computing the $\omega$-signature at these points gives us $\sigma_{\xi_1} = -2$, $\sigma_{\xi_2} = -4$ and $\sigma_{\xi_3} = -2$.  We can thus draw the signature for every value on the unit circle:
\begin{figure}[h]
\begin{center}
\begin{tikzpicture}[scale=2.5,cap=round]
  \def\cosone{0.80901699}
  \def\sinone{0.58778525}
  \def\costhree{-0.30901699}
  \def\sinthree{0.951056516}
  \def\cosseven{-0.30901699}
  \def\sinseven{-0.951056516}
  \def\cosnine{0.80901699}
  \def\sinnine{-0.58778525}


   \node[above right] at (\cosone,\sinone) {$\omega_1$};
   \node[above left] at (\costhree,\sinthree) {$\omega_2$};
   \node[below left] at (\cosseven,\sinseven) {$\omega_3$};
   \node[below right] at (\cosnine,\sinnine) {$\omega_4$};

   \fill[blue!10!white] (0,0) -- (1,0) arc (0:36:1cm);
   \fill[blue!10!white] (0,0) -- (1,0) arc (0:-36:1cm);
   \fill[blue!20!white] (0,0) -- (\cosone,\sinone) arc (36:108:1cm);
   \fill[blue!20!white] (0,0) -- (\cosone,-\sinone) arc (-36:-108:1cm);
   \fill[blue!30!white] (0,0) -- (\costhree,\sinthree) arc (108:252:1cm);

   \draw (0,0) -- (\cosone,\sinone);
   \draw (0,0) -- (\costhree,\sinthree);
   \draw (0,0) -- (\cosseven,\sinseven);
   \draw (0,0) -- (\cosnine,\sinnine);

   \node at (0.5,0) {$\sigma_\omega = 0$};
   \node at (0.2, \sinone) {$\sigma_\omega = -2$};
   \node at (-0.5,0) {$\sigma_\omega = -4$};
   \node at (0.2, -\sinone) {$\sigma_\omega = -2$};

   \draw (0,0) circle (1cm);

   \draw[fill=black] (36:1cm) circle (0.01cm);
   \node[above] at (72:1cm) {$\xi_1$};
   \draw[fill=black] (72:1cm) circle (0.01cm);
   \draw[fill=black] (108:1cm) circle (0.01cm);
   \draw[fill=black] (-1,0) circle (0.01cm);
   \node[left] at (-1,0) {$\xi_2$};
   \draw[fill=black] (-36:1cm) circle (0.01cm);
   \node[below] at (-72:1cm) {$\xi_3$};
   \draw[fill=black] (-72:1cm) circle (0.01cm);
   \draw[fill=black] (-108:1cm) circle (0.01cm);
\end{tikzpicture}
\end{center}
\end{figure}

We can now compute the $L^2$ signature to be
\begin{eqnarray*}
\int_{\omega \in S^1} \sigma_\omega & = & (x_2 - x_1)\sigma_{\xi_1} + (x_3 - x_2)\sigma_{\xi_2} + (x_4 - x_3)\sigma_{\xi_3}\\
& = & \frac{2}{10}(-2) + \frac{4}{10}(-4) + \frac{2}{10}(-2)\\
& = & -\frac{12}{5}.
\end{eqnarray*}
\end{example}

\section{Torus knot signatures}

For coprime integers $p$ and $q$, the $(p,q)$ torus knot $T_{p,q}$ is the knot lying on the surface of a torus which winds $p$ times around the meridian and $q$ times around the longitude.  If $p$ and $q$ are not coprime, then $T_{p,q}$ is a link of more than one component.  The Alexander polynomial of $T_{p,q}$ is
  \[ \Delta_{p,q}(t) = \frac{(1-t^{pq})(1-t)}{(1-t^p)(1-t^q)}.\]
(A proof can be found in, for example, Lickorish~\cite[pg 119]{Lickorish}.)  The roots of this polynomial are the $pq^\text{th}$ roots of unity that are neither $p^\text{th}$ nor $q^\text{th}$ roots of unity.  This gives us $pq-p-q+1$ places at which the signature function could jump, namely $e^{2\pi i n/pq}$ for $n \in \Z$ with $0<n<pq$ such that $n$ is not divisible by $p$ or $q$.

The jump functions of torus knots have been investigated by Litherland~\cite{Litherland79}.  His result is that
    \[ j_{p,q}\left(\frac{n}{pq}\right) = |L(n)| - |L(pq+n)| \]
  where $pq>n \in \N$ and
    \[ L(n) = \left\{ (i,j) \; | \; iq+jp = n, \,\, 0 \leq i\leq p, \, 0 \leq j \leq q \right\} \text{ .} \]

Notice that if $n$ is not a multiple of $p$ or $q$ then $L(n)$ and $L(pq+n)$ cannot both be non-empty.  To see this, suppose that $(i_1, j_1) \in L(n)$ and $(i_2,j_2) \in L(pq+n)$. Then $(i_2-i_1)q + (j_2-j_1)p = pq$, and since $p$ and $q$ are coprime we must have $i_2=i_1$ (mod $p$) and $j_2=j_1$ (mod $q$).  But this forces $i_1=i_2$ and $j_1=j_2$, which is a contradiction.  A similar argument shows that neither $L(n)$ nor $L(pq+n)$ can contain more than one element.  However, at least one of the two sets is non-empty. For, we can write $n=iq+jp$ with $0<i<p$, and if $j>0$ then $(i,j) \in L(n)$ whilst if $j<0$ we have $(i,j+q) \in L(pq+n)$.

If $n$ is a multiple of $p$ or $q$ then $|L(n)| = 1 = |L(pq+n)|$.  Putting these results together gives us the following.

\begin{proposition}
\label{Prop:torusjump}
  The jump function of the $(p,q)$ torus knot is
  \[ j_{p,q}\left(\frac{n}{pq}\right) =
    \begin{cases}
      +1 & \text{ if } |L(n)|=1 \\
      -1 & \text{ if } |L(n)|=0 \\
      0  & \text{ if } n \text{ is a multiple of } p \text{ or } q
    \end{cases} \]
\end{proposition}

We need one more lemma before we are ready to find a formula for the $L^2$ signature.

\begin{lemma}
\label{Lemma:oneof}
  If $p$ and $q$ are coprime and $1 \leq n \leq pq-1$ with $n$ not a multiple of $p$ or $q$, then exactly one of $n$ and $pq-n$ can be written as $iq+jp$ for $i,j >0$.
\end{lemma}

\begin{proof}
  See, for example, \cite[Lemma 1.6]{BeckSinai}.
\end{proof}

\begin{theorem}
\label{Thm:L2torussig}
  Let $p$ and $q$ be coprime positive integers.  Then the $L^2$ signature of the $(p,q)$ torus knot is
  \[ \displaystyle \int_{S^1} \sigma_{\omega} = -\frac{(p-1)(p+1)(q-1)(q+1)}{3pq} \text{ .} \]
\end{theorem}

\begin{proof}
  Denote the jump function of the $(p,q)$ torus knot by $j_{p,q}$. The signature function at $\omega$ can be defined as the sum of the jump functions up to that point (Lemma \ref{Lemma:jump_properties}).  If $\omega_n := e^{2\pi i x}$ with $x \in (\frac{n}{pq},\frac{n+1}{pq})$ then
  \[ \sigma_{\omega_n}(T_{p,q}) = \displaystyle 2 \sum_{i=1}^n j_{p,q}\left(\frac{i}{pq}\right) \]

  We can now use Lemma \ref{Lemma:L2asjump} to find a formula for the $L^2$ signature in terms of the jump function.

  \begin{eqnarray}
  \int_{S^1} \sigma_\omega & = & \sum_{n=1}^{pq-1} \frac{1}{pq} (\sigma_{\omega_n})\\
  & = & \frac{2}{pq} \sum_{n=1}^{pq-1}\sum_{i=1}^n j_{p,q}\left(\frac{i}{pq}\right) \\
  & = & \frac{2}{pq} \left( j_{p,q}\left(\frac{1}{pq}\right) + \left(j_{p,q}\left(\frac{1}{pq}\right)+j_{p,q}\left(\frac{2}{pq}\right)\right) + \dots + \sum_{i=1}^{pq-1}j_{p,q}\left(\frac{i}{pq}\right)\right)\\
  \label{eqn:sigint} & = & \frac{2}{pq} \sum_{n=1}^{pq-1} (pq-n)\,j_{p,q}\left(\frac{n}{pq}\right)
  \end{eqnarray}

  Let $S$ be the set defined by
  \[ \bigg\{  n \in \{1,\dots, pq-1\}\; | \;  n=qx+py, \,\, 0 < x < p,\,\, 0 < y < q \bigg\} \]

  Given an integer $n \in \{1,\dots, pq-1\}$ which is not a multiple of $p$ or $q$, we can write $n=qx+py$
  with $0<x<p$. By Lemma \ref{Lemma:oneof}, either $n \in S$ or $pq-n \in S$.  Proposition \ref{Prop:torusjump}
  tells us that in the first case we have $j_{p,q}(n/pq) = 1$, whilst in the second case we have $j_{p,q}(n/pq) = -1$.
  If $n$ is a multiple of $p$ or $q$ then the jump function will be zero and so it will not contribute to the sum.

We may rewrite equation (\ref{eqn:sigint}) as

  \[ \int_{S^1} \sigma_\omega = \frac{2}{pq}\left(\sum_{n \in S}(pq-n) - \sum_{n \in S}n\right) = \frac{2}{pq}\sum_{n \in S}(pq-2n) \text{ .} \]

  There are $\frac{1}{2}(p-1)(q-1)$ points in $S$, and in the paper by Mordell \cite{Mordell51} we find the following formula
  \[ \sum_{n \in S} n = \frac{1}{3}pq(p-1)(q-1) + \frac{1}{12}(p-1)(q-1)(p+q+1) \text{ .} \]

  Putting these together gives us
  \begin{eqnarray*}
    \int_{S^1} \sigma_\omega & = & \frac{2}{pq}\left(\sum_{n \in S}pq - 2\sum_{n \in S} n\right)\\
    & = & \frac{2}{pq}\left(\frac{1}{2}(p-1)(q-1)pq - 2\left(\frac{1}{3}pq(p-1)(q-1) + \frac{1}{12}(p-1)(q-1)(p+q+1)\right) \right)\\
    & = & \frac{1}{pq}(p-1)(q-1) \left(pq - \frac{4}{3}pq - \frac{1}{3}(p+q+1)\right) \\
    & = & -\frac{1}{3pq}(p-1)(q-1)(pq+p+q+1) \\
    & = & -\frac{1}{3pq}(p-1)(q-1)(p+1)(q+1).
  \end{eqnarray*}
\end{proof}

\begin{remark}
  That there is such a neat formula for the $L^2$ signature of a torus knot is all the more surprising considering
  the absence of an explicit formula for the usual signature $\sigma_{-1}$ of a torus knot.  We have the following
  formula due to Hirzebruch~\cite{Hirzebruch67} for $p$ and $q$ odd and coprime:
    \[ \sigma_{-1}(T_{p,q}) = -\left(\frac{(p-1)(q-1)}{2} + 2(N_{p,q} + N_{q,p})\right)\]
    where
    \[ N_{p,q} = \#\left\{ (x,y) \:| \: 1 \leq x \leq \frac{p-1}{2}, \,\, 1 \leq y \leq \frac{q-1}{2}, \,\, -\frac{p}{2} < qx-py <0 \right\}.\]
    Further work was done by Brieskorn~\cite{Brieskorn66} and Gordon/Litherland/Murasugi~\cite{GordonLitherlandMurasugi81},
    without success, and in $2010$ Borodzik and Oleszkiewicz published a proof \cite{Borodzik10} that in fact there can never be a rational function $R(p,q)$ such that $R(p,q)=\sigma_{-1}(T_{p,q})$ for all odd coprime integers $p$ and $q$.
\end{remark}

\section{Twist knots}

As an important corollary, we show that the twist knots $K_n$ are not slice.  This was proved in the 1970s by Casson and Gordon~\cite{CassonGordon86} but the following proof, which uses a result of Cochran, Orr and Teichner, is much shorter and simpler\footnote{It is also an interesting historical point that Milnor used an early version of the $\omega$-signatures to show that an infinite number of the $K_n$ are independent in the concordance group $\C$~\cite{Milnor68}.}.

\begin{definition}
\label{Def:twistKnot}
The \emph{twist knots} $K_n$ are the following family of knots:
  \begin{center}
    \includegraphics[width=6cm]{chapter10/twisty}
  \end{center}
For example, $K_{-1}$ is the trefoil, $K_1$ is the figure-eight knot and $K_2$ is Stevedore's knot $6_1$.  The knot $K_n$ is sometimes called the \emph{$n$-twisted double of the unknot}.
\end{definition}

A Seifert matrix for $K_n$ is
  \[V = \left(\begin{array}{cc} -1 & 1\\ 0 & n\end{array}\right)\]
which gives the Alexander polynomial as $-nt^2+(2n+1)t-n$.  There is a class of twist knots (those for which $n=m(m+1)$ for some $m$) which are algebraically slice.  This means that there is a simple closed curve $\gamma$ on the Seifert surface $F$ such that $\gamma$ is nontrivial in $H_1(F)$ and such that $\gamma^+$, which is the curve pushed off the Seifert surface, has zero linking with $\gamma$.  The consequence of this is that all signatures and other known slice invariants vanish.  The question is then: are these knots \emph{really} slice?

The following theorem shows us that one way of finding the answer is to consider the slice properties of the curve $\gamma$ rather than those of the original knot. For those readers unfamiliar with the work of Cochran, Orr and Teichner, read ``$(1.5)$-solvable" as ``having vanishing Casson-Gordon invariant".

\begin{theorem}[\cite{COT03}]
\label{Thm:genusonesurface}
  Suppose $K$ is a $(1.5)$-solvable knot with a genus $1$ Seifert surface $F$. Suppose that the classical Alexander polynomial of $K$ is non-trivial. Then there exists a homologically essential simple closed curve $J$ on $F$, with self-linking zero, such that the integral over the circle of the $\omega$-signature function of $J$ (viewed as a knot) vanishes.
\end{theorem}

Thus if the $L^2$-signature is non-zero for any closed curve on $F$ with self-linking zero, then the knot cannot be $(1.5)$-solvable and therefore cannot be slice.

\begin{corollary}
\label{Cor:twistslice}
  The twist knots $K_n$ are not slice unless $n=0$ or $n=2$.
\end{corollary}

\begin{proof}
  The Alexander polynomial of $K_n$ is $-nt^2 + (2n+1)t - n$.  If $n<0$ then $\Delta_{K_n}$ has distinct roots on the unit circle and an easy computation shows that the signature is non-zero.  If $n>0$ then $\Delta_{K_n}$ is reducible if and only if $4n+1$ is a square.  Since the Alexander polynomial of a slice knot has the form $f(t)f(t^{-1})$ \cite{FoxMilnor66}, it follows that $K_n$ cannot be slice if $4n+1$ is not a square.

  Suppose $4n+1$ = $l^2$ with $l=2m+1$.  Then $n=m(m+1)$.  Using the obvious genus $1$ Seifert surface $F$ for $K_{m(m+1)}$ it can be seen that the only simple closed curve on $F$ with self-linking zero is the $(m,m+1)$ torus knot (see, for example, Kauffman~\cite[Chapter VIII]{Kauffman}).  Since the $L^2$ signature for any torus knot $T_{m,m+1}$ is non-zero (except for $m=0,-1,1,-2$) by Theorem \ref{Thm:L2torussig}, this means that $K_{m(m+1)}$ cannot be (1.5)-solvable and therefore not slice unless $n=0$ or $n=2$.
\end{proof}

We now apply Theorem \ref{Thm:L2torussig} to the $n$-twisted doubles of knots, an example of which can be seen in Figure \ref{fig:twisteddouble}.

\begin{corollary}
\label{Cor:twisteddouble}
Let $K$ be a knot and $D_n(K)$ the $n$-twisted double ($n \neq 0$) of $K$.
\begin{itemize}
\item[(a)] $D_n(K)$ cannot be slice unless $n=m(m+1)$ for some $m \in \Z$ and $\int_{S^1} \sigma_\omega(K) = \frac{(m-1)(m+2)}{3}$.  In particular, $D_2(K)$ can only be slice if $\int_{S^1} \sigma_\omega(K)=0$.
\item[(b)] For any given $K$ with $\int_{S^1} \sigma_\omega(K) \neq 0$, there is at most one $D_n(K)$ which can be slice.
\end{itemize}
\end{corollary}

\begin{figure}
  \centering
  \includegraphics[width=9cm]{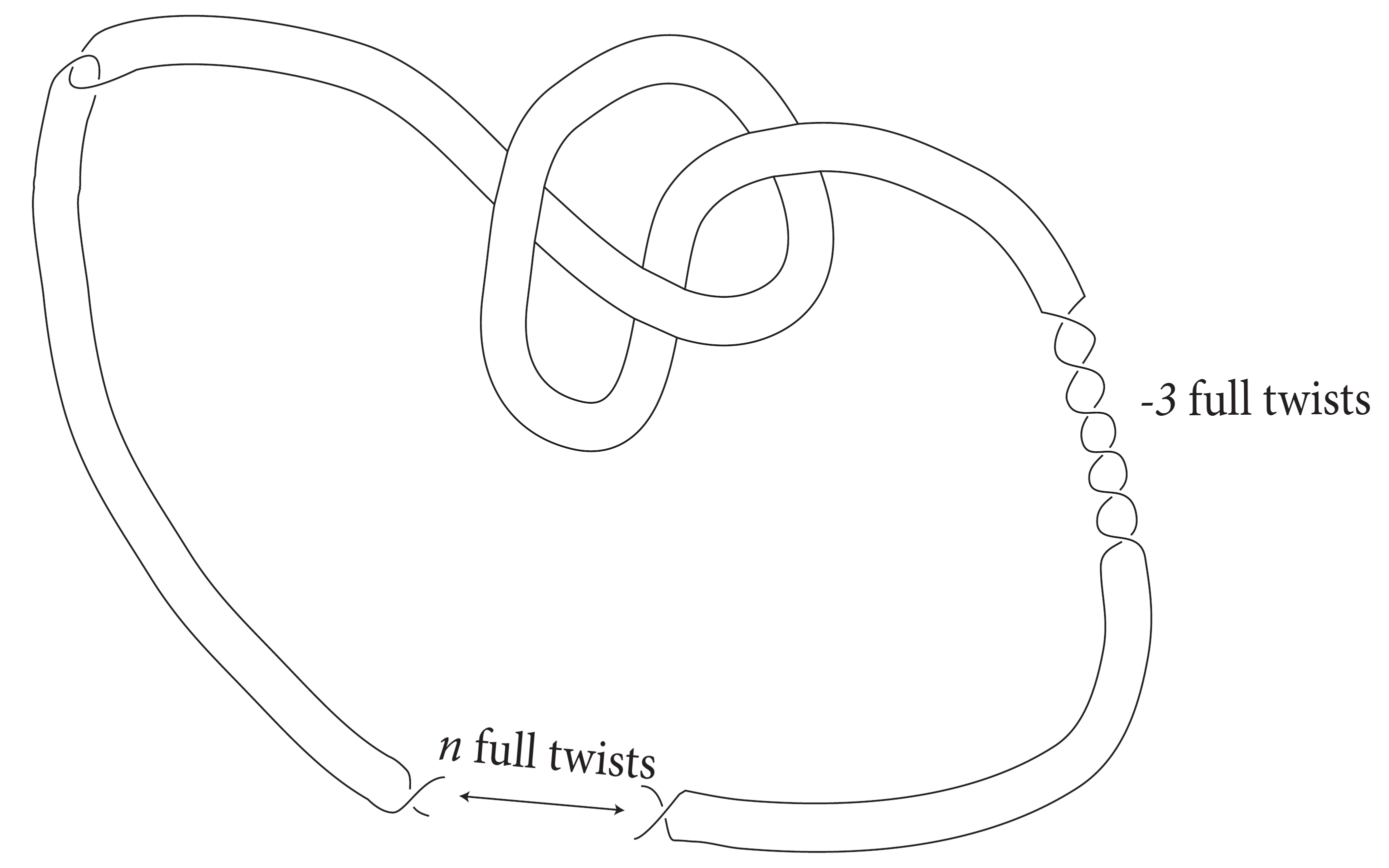}
  \caption{The $n$-twisted double of the right-handed trefoil.}
  \label{fig:twisteddouble}
\end{figure}

\begin{proof}
The Alexander polynomial of $D_n(K)$ is once again $-nt^2 + (2n+1)t -n$ and the same argument as in the proof of Corollary \ref{Cor:twistslice} shows that $D_n(K)$ is algebraically slice if and only if $n=m(m+1)$ for some integer $m$. (Notice that if $n=0$ then the Alexander polynomial is trivial and $D_0(K)$ is slice by Freedman's work \cite{FreedmanQuinn}.)  The zero-framed curve on the obvious Seifert surface is the connected sum of $K$ and the $(m,m+1)$ torus knot, $K \# T_{(m,m+1)}$.  If we denote the $L^2$ signature by $s$, we have
  \begin{eqnarray*}
    s(K \# T_{m(m+1)}) & = & s(K) + s(T_{(m,m+1)})\\
                       & = & s(K) - \frac{(m-1)(m+1)m(m+2)}{3m(m+1)}\\
                       & = & s(K) - \frac{(m-1)(m+2)}{3}
  \end{eqnarray*}
By Theorem \ref{Thm:L2torussig}, $D_n(K)$ can only be slice if $s(K) = \frac{(m-1)(m+2)}{3}$.  In particular, if $m=1$ or $m=-2$ then $T_{(m,m+1)}$ is the unknot and so $s(K)$ must be zero for $D_2(K)$ to be slice.  This proves (a).
For (b), suppose that $3s(K) = (m-1)(m+2) \neq 0$.  Rearranging, we get $m^2+m - 2-3s(K)=0$.  Suppose that $m_1$ and $m_2$ are roots.  Then $m_1+m_2 = -1$, so $m_1(m_1+1) = -(m_2+1)(-m_2) = m_2(m_2+1)$, giving only one value for $n$.
\end{proof}

Se-Goo Kim \cite{Kim05} proves that for any knot $K$, all but finitely many algebraically slice twisted doubles of $K$ are linearly independent in the concordance group $\C$.  Using our theorem we can conjecture that there is a much stronger result about the independence of the twisted doubles of $K$.

\begin{conjecture} For a fixed knot $K$, the $D_{m(m+1)}(K)$ are linearly independent in $\C$ for all but one (or two, if $\int_{S^1} \sigma_\omega(K) = 0$) values of $m(m+1)$.
\end{conjecture}

The proof of this conjecture will require Theorem \ref{Thm:genusonesurface} to be extended to connected sums of genus $1$ Seifert surfaces. The difficulty with this is analysing all the possible metabolising curves (i.e. half-bases of homology curves with self-linking zero) on a higher genus surface. The number of such curves would grow exponentially as the genus increased, and it is not possible to deduce enough information about the metabolisers of a connected sum from looking at metabolisers of the component knots. 
\chapter{\label{chapter9} Further work}

In this thesis we have developed new techniques for studying knot concordance and have determined many previously unknown concordance orders of knots. Our explorations have shed more light on the big questions of concordance and have also revealed new questions hitherto unasked. In this chapter we discuss the questions that remain open and give our thoughts on how to go about attacking them.

\section{The algebraic concordance group}

In Chapter \ref{chapter4} we make various conjectures about the classification of knots in the algebraic concordance group. To determine the order of a knot in $\G$, the current literature says that we need to look at the image of the knot in $W(\F_p)$ for primes $p$ dividing the discriminant and the leading coefficient of the Alexander polynomial.  Together, Conjectures \ref{Conj:DetPrime} and \ref{Conj:Witt2} say that all the information required to determine the algebraic order of a knot is contained in the groups $W(\F_p)$ for primes $p$ dividing the determinant of the knot.  More analysis should be done on algebraic concordance orders of knots, for example by building on the work of the $9$-crossing analysis to analyse prime knots of $10$ or $11$ crossings, to see further whether these conjectures hold.

Another interesting exercise would be to find a prime knot of algebraic order $2$ which turns out to be concordant to twice another prime knot of algebraic order $4$. Such examples have not yet been exhibited.

\section{Developing twisted Alexander polynomials}

Although the theorems of Chapter \ref{chapter5} go a long way in developing twisted Alexander polynomials as slicing obstructions, there are still cases to be looked at in order to complete the theory. These are
 \begin{itemize}
   \item where $H_1(\Sigma_2(K);\Z) \cong \Z_q \oplus \Z_q$ for a prime knot $K$, or more generally...
   \item ... where $H_1(\Sigma_2(K);\Z) \cong \Z_{q^m} \oplus \Z_{q^n}$ for a prime knot $K$ and some $m,n \in \N$;
   \item where $H_1(\Sigma_2(K_i);\Z) \cong \bigoplus \Z_{q^{n_i}}$ for at least two distinct $n_i$ and $K=K_1 + \dots + K_n$;
   \item where $H_1(\Sigma_p(K);\Z) \cong \Z_q \oplus \Z_q \oplus \Z_q \oplus \Z_q$ for a prime knot $K$ and $p>2$;
   \item most generally, where $H_1(\Sigma_p(K);\Z) \cong \bigoplus_{i} \Z_{q^{n_i}}$ for any knot $K$, any prime power $p$ and any integers $n_i$.
 \end{itemize}
 This last condition also includes the case where $H_1(\Sigma_p;\Z)$ does not split into 1-dimensional eigenspaces. For example, those cases where $H_1(\Sigma_p;\Z) \cong \Z_q \oplus \Z_q$ where $p$ does not divide $q-1$.
 
 Some work in the direction of these generalisations has been done by Livingston and Naik~\cite{LivingstonNaik09}, where they show that knots with $H_1(\Sigma_2) \cong \Z_3 \oplus \Z_{3^{2i}}$ for any $i$ are not of order $4$ in $\C$. The amount of algebra and number theory involved in even this simple case shows how difficult such general theorems would be to prove.

In addition to these generalisations, it should be possible to get tighter conditions for some of the existing theorems, particularly those relating to knots which are connected sums. The proof of Proposition \ref{Prop:9-23} shows that not all of the constituent knots of a connected sum need to satisfy the conditions of Theorem \ref{Thm:infiniteordersum}, so a next step would be to work out exactly how many of the conditions are required to prove the infinite order of a connected sum.

\section{Signatures vs twisted polynomials}

The proofs of the fact that knots with $H_1(\Sigma_2) \cong \Z_{q^n}$, $n$ odd, have infinite order in $\C$ use Casson-Gordon invariants, but in the form of signatures.  The obstructions given in this thesis to knots having finite order use Casson-Gordon invariants, but in the form of twisted Alexander polynomials.  Both of these simplifications are necessary for a computable obstruction theory, since the Casson-Gordon invariants themselves are too complicated to compute except in a few simple cases.

Twisted Alexander polynomials seem like a weaker obstruction, since they are of order $2$ (i.e. a twisted Alexander polynomial can only be an obstruction if it occurs with odd exponent).  It would be good to show that in some cases they are equally effective.  For example, in Theorem \ref{Thm:twistedpolyq^n}, it would be good to show that for $q \equiv 3$ mod $4$ and $n$ odd, that the conditions given always hold.  We would then have a second proof of the Livingston-Naik theorem, which could perhaps be generalised to other situations where the signatures are not effective.

In the paper \cite{HeddenKirkLivingston11} Hedden, Kirk and Livingston use a combination of signatures and twisted Alexander polynomials to show that a set of knots is linearly independent. Perhaps these techniques could also be used to prove Conjecture~\ref{Conj:geomclassification}, which does not appear to be tractable by using twisted Alexander polynomials alone.

\section{Structure of the concordance group $\C$}

The conjecture of the knot theoretic community is that $\C \cong (\Z)^\infty \oplus (\Z_2)^\infty$.  The main reasons behind this conjecture are that no knots of order $4$ have been found, and that there are no knots of any finite order other than $2$ and $4$ in the higher-dimensional knot concordance groups.

A first step on the road to proving this conjecture would be to prove the conjecture made by Livingston~\cite[1.4]{LivingstonNaik01} that knots of algebraic order $4$ have infinite order in $\C$.  A knot is of algebraic order $4$ if and only if it has primes congruent to $3$ mod $4$ dividing its determinant with odd exponent\cite[23(c)]{Levine69-1}.  Livingston and Naik have proved various theorems about the orders of knots with homology of this form: knots with $H_1(\Sigma_2) \cong \Z_{q^n}$ for $n$ odd are always of infinite order in $\C$ \cite{LivingstonNaik01}, and knots with $H_1(\Sigma_2) \cong \Z_3 \oplus \Z_{3^{2i}}$ for any $i$ are not of order $4$ in $\C$ \cite{LivingstonNaik09}.  These results need to be generalised to knots whose $H_1(\Sigma_2)$ is a direct sum of $\Z_{q^{n_i}}$ where $q \equiv 3$ (mod 4) and at least one $n_i$ is odd.  The difficulty seems to lie only in the linear algebra and number theory, rather than in the knot theory.

A result of this form would be surprising because it would be an abelian invariant which provides the obstruction.  Why should the shape of the homology of the $2$-fold branched cover determine the order of the knot in the concordance group?  This is certainly not true in general: for any knot which is algebraically slice, there is a slice knot which has exactly the same abelian invariants \cite{Kearton04}.  It is therefore impossible to use abelian invariants to obstruct an algebraically slice knot from being slice.

Furthermore, it is surprising to think that no more sophisticated tools than Casson-Gordon invariants will be needed to prove this conjecture.  Given the infinite filtration of the knot concordance group, why should all algebraic order $4$ knots lie in the same part of it?  A corresponding result is certainly not true for order $2$ knots: Cochran, Harvey and Leidy~\cite{CochranHarveyLeidy11} have shown that there are (infinitely many) order $2$ knots in every level $\CF_{n}/\CF_{n.5}$ of the concordance filtration.

It is fascinating that it should be some simple number theory which determines the concordance order of a knot.  The condition that the determinant of a knot has no primes congruent to 3 mod 4 with odd exponents is equivalent to the condition that the determinant is the sum of two squares.  It is known that the determinant of an amphicheiral knot is of this form \cite{Goeritz33}, but this is almost certainly a property of the concordance group, not of the symmetry type of the knot.  We conjecture that a result of the form of Hartley and Kawauchi (that the Conway polynomial $C(z)$ of an amphicheiral knot splits as $f(z)f(-z)$)\cite{HartleyKawauchi79}, or of the form of Conant's conjecture (that for an amphicheiral knot $C(z)C(iz)C(z^2)$ is a perfect square inside the ring of power series with integer coefficients)\cite{Conant06}, should be true for any knot concordant to an amphicheiral knot.

\section{Where do order 2 knots come from?}

It is an open conjecture whether the $\Z_2$ summand in $\C$ is generated by amphicheiral knots.  In other words, is there a knot of concordance order 2 which is neither amphicheiral nor concordant to an amphicheiral knot?

In Chapter \ref{chapter8} we found examples of knots of order 2 which are not negative amphicheiral but which turn out to be concordant to amphicheiral knots ($4_1$ or $6_3$).  Part of the reason why we have only found knots of this kind is that it is much easier to prove that a knot is concordant to a smaller one than to prove it is of order 2 directly.  If a knot $K$ has $12$ crossings then $K\# K$ has $24$ crossings and it is virtually impossible to see where to make a slice move. However, we have found two possible candidates for a finite order knot which is not concordant to an amphicheiral knot: namely, $12a_{631}$ and $12n_{846}$. It seems worth investigating these particular knots to see whether they are counterexamples to the conjecture; perhaps by developing computer techniques (or manual techniques) for finding slice movies. Some work has been done in this area by Ayumu Inoue~\cite{Inoue10}, who uses computer simulations to find slice movies for the twist-spun trefoils.

\section{Second-order slice obstructions}

In Chapter \ref{chapter10} we use $\omega$-signatures as a second-order obstruction to provide another proof that the twist knots are not slice.  This appears to be a powerful, yet simple, technique which could be used for other families of knots.  As mentioned at the end of the chapter, the result should also be able to prove the linear independence of families of knots, provided the Cochran-Orr-Teichner theorem can be extended to connected sums of genus $1$ knots.  It is yet a more difficult problem to extend the theorem to prime knots of higher genus, since it is then much more difficult to analyse the metabolising curves on the surface. 

\appendix
\renewcommand{\chaptername}[1]{Appendix A. }
\renewcommand{\chaptermark}[1]{\markboth{\chaptername \ #1}{}}
\chapter{Signature calculations}
\label{AppendixD}

Here we give the results of the analysis of signatures done in Section \ref{Sec:signature}.

We find the set of unit roots, with positive imaginary part, of Alexander polynomials of knots in $E$, of which there are $70$. The midpoint of each consecutive pair of values is found; these are the $\delta_i$ listed along the top of the table. 
Evaluating the signature function of each knot at each of the $\delta_i$ gives us a complete picture of the signature function of each knot.  The matrix presented here is the result of these calculations put into reduced echelon form.  It allows us to read off a basis for the kernel of the signature function as well as a basis for the knots which are of infinite order.

Non-zero values are highlighted to make reading the table easier.

\begin{sidewaysfigure}
\centering
\includegraphics[width=24cm]{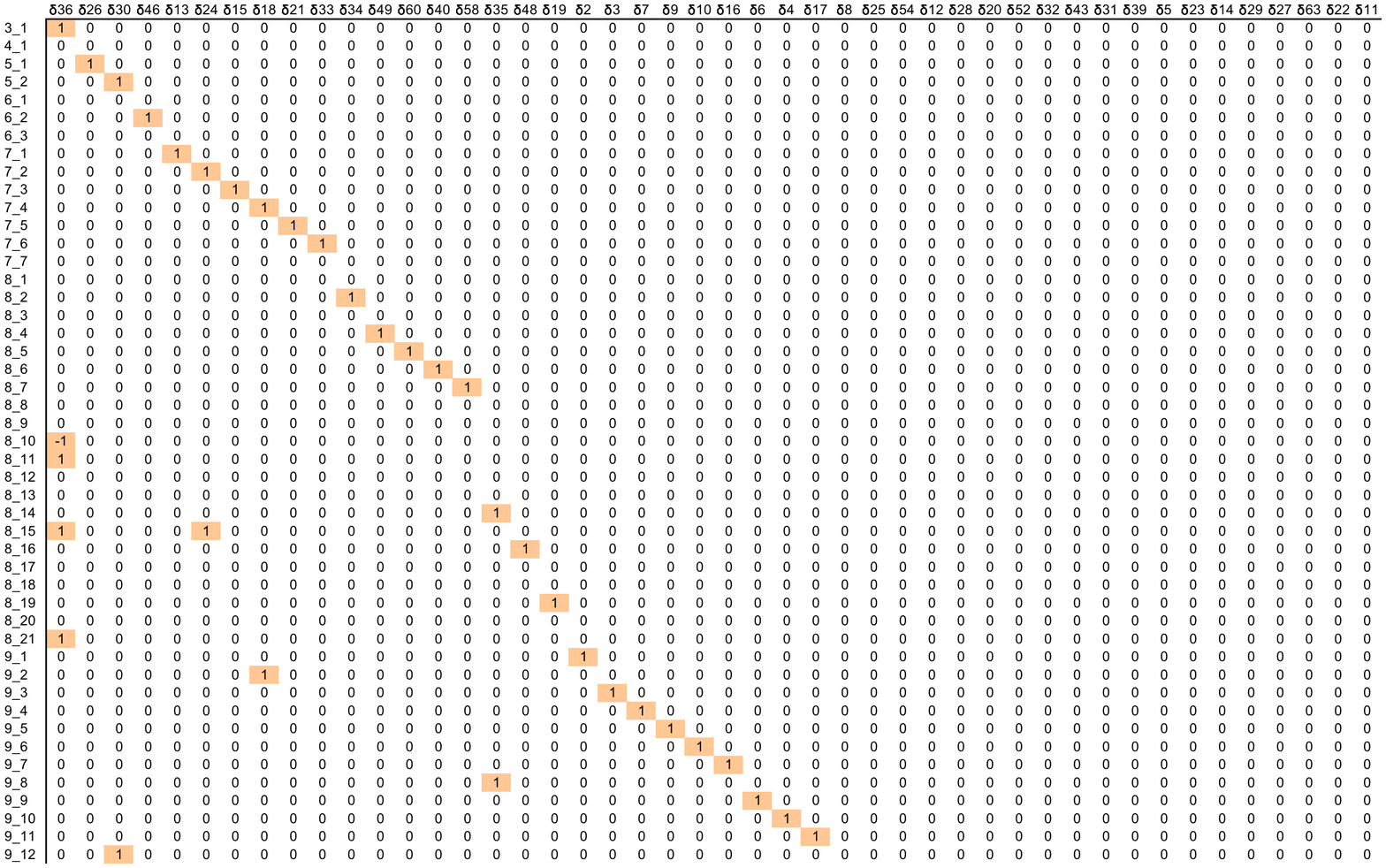}
\end{sidewaysfigure}

\newpage
\begin{sidewaysfigure}
\centering
\includegraphics[width=24cm]{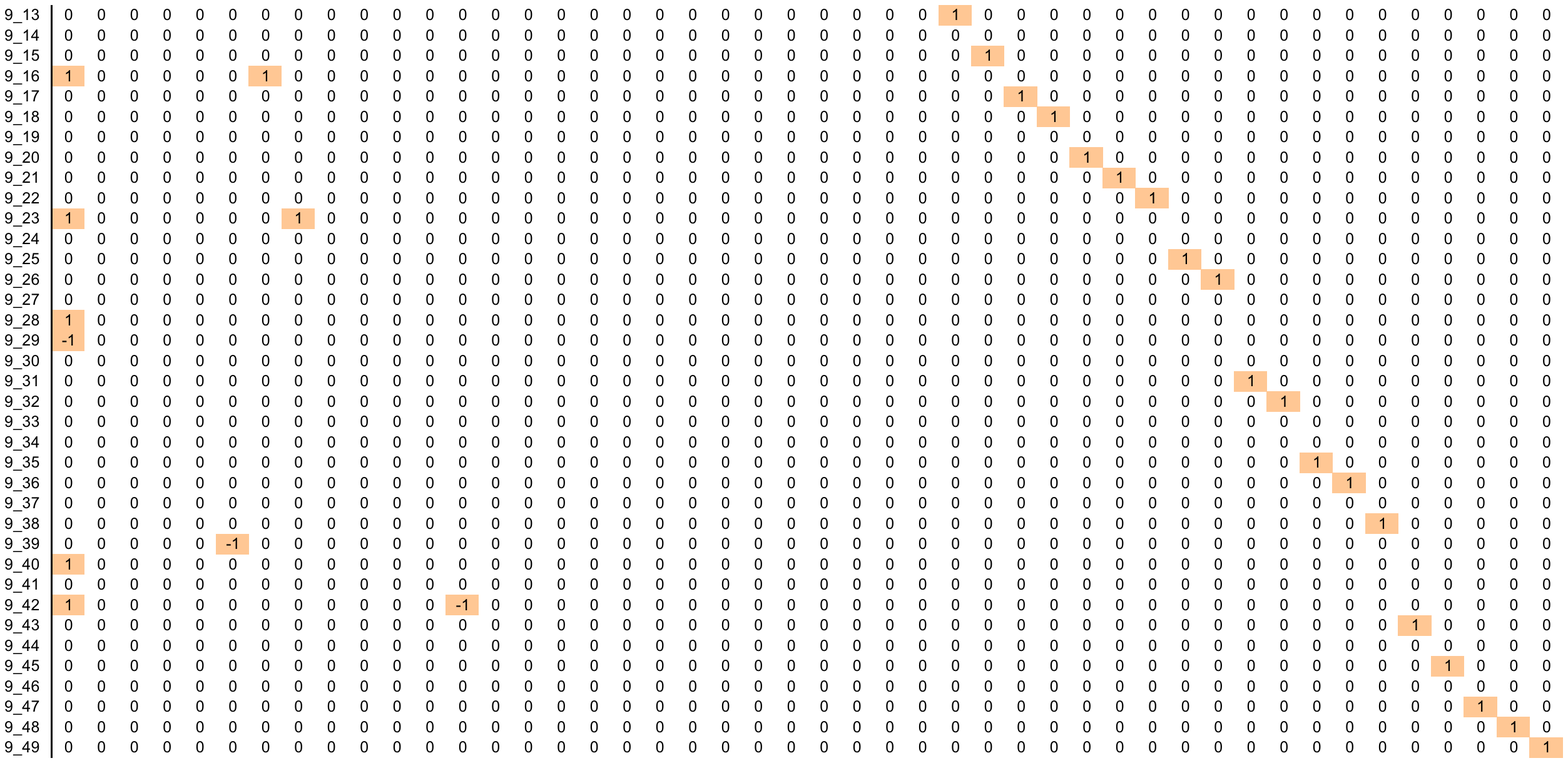}
\end{sidewaysfigure} 
\renewcommand{\chaptername}[1]{Appendix B. }
\renewcommand{\chaptermark}[1]{\markboth{\chaptername \ #1}{}}
\chapter{Factorised Alexander polynomials for knots in $E$}
\label{AppendixF}

Here we give a list of the Alexander polynomials for the knots in $B_\sigma \subset \CF_E$; that is, those which are a basis for the kernel of the signature function. The polynomials are factorised over $\Z[t,t^{-1}]$ and the determinant of the knot, which is the Alexander polynomial evaluated at $t=-1$, is given. These calculations are necessary for the analysis in Chapter \ref{chapter4}; in particular, it is the symmetric irreducible factors which are important in the algebraic concordance classification.

\newpage

\begin{tabular}{ccc}
Knot $K$ & $\Delta_K(t)$ & $\Delta_K(-1)$ \\ \hline

$4_1$ & $1-3t+t^2$ & $5$\\
$6_1$ & $(2-t)(-1+2t)$ & $3^2$\\
$6_3$ & $1-3t+5t^2-3t^3+t^4$ & $13$ \\
$7_7$ & $1-5t+9t^2-5t^3+t^4$ & $3\cdot 7$ \\
$8_1$ & $3-7t+3t^2$ & $13$ \\
$8_3$ & $4-9t+4t^2$ & $17$ \\
$8_8$ & $(1-2t+2t^2)(t^2-2t+2)$ & $5^2$\\
$8_9$ & $(t^3-2t^2+t-1)(t^3-t^2+2t-1)$ & $5^2$ \\
$(8_{10}+3_1)$ & $(1-t+t^2)^4$ & $3^4$ \\
$(8_{11}-3_1)$ & $(1-t+t^2)^2(t-2)(2t-1)$ & $3^4$ \\
$8_{12}$ & $1-7t+13t^2-7t^3+t^4$ & $29$ \\
$8_{13}$ & $2-7t+11t^2-7t^3+2t^4$ & $29$ \\
$(8_{15}-7_2-3_1)$ & $(1-t+t^2)^2(3-5t+3t^2)^2$ & $3^2 \cdot 11^2$ \\
$8_{17}$ & $1-4t+8t^2-11t^3+8t^4-4t^5+t^6$ & $37$ \\
$8_{18}$ & $(1-3t+t^2)(1-t+t^2)^2$ & $3^2 \cdot 5$ \\
$8_{20}$ & $(1-t+t^2)^2$ & $3^2$\\
$(8_{21}-3_1)$ & $(1-t+t^2)^2(1-3t+t^2)$ & $3^2 \cdot 5$ \\
$(9_2-7_4)$ & $(4-7t+4t^2)^2$ & $3^2 \cdot 5^2$ \\
$(9_8-8_{14})$ & $(2-8t+11t^2-8t^3+2t^4)^2$ & $31^2$ \\
$(9_{12}-5_2)$ & $(2-3t+2t^2)^2(1-3t+t^2)$ & $5 \cdot 7^2$ \\
$9_{14}$ & $2-9t+15t^2 - 9t^3 + 2t^4$ & $37$ \\
$(9_{16}-7_3-3_1)$ & $(1-t+t^2)^2(2-3t+3t^2-3t^3+2t^4)^2$ & $3^3 \cdot 13^2$ \\
$9_{19}$ & $2-10t+17t^2-10t^3+2t^4$ & $41$ \\
$(9_{23}-7_4-3_1)$ & $(1-t+t^2)^2(4-7t+4t^2)^2$ & $3^4 \cdot 5^2$ \\
$9_{24}$ & $(1-t+t^2)^2(1-3t+t^2)$ & $3^2 \cdot 5$ \\
$9_{27}$ & $(t^3-3t^2+2t-1)(t^3-2t^2+3t-1)$ & $7^2$ \\
$(9_{28}-3_1)$ & $(1-t+t^2)^2(1-4t+7t^2-4t^3+t^4)$ & $3^2 \cdot 17$ \\
$(9_{29}+3_1)$ & $(1-t+t^2)^2(1-4t+7t^2-4t^3+t^4)$ & $3^2 \cdot 17$ \\
$9_{30}$ & $1-5t+12t^2-17t^3+12t^4-5t^5+t^6$ & $53$ \\
$9_{33}$ & $1-6t+14t^2-19t^3+14t^4-6t^5+t^6$ & $61$ \\
$9_{34}$ & $1-6t+16t^2-23t^3+16t^4-6t^5+t^6$ & $3 \cdot 23$ \\
$9_{37}$ & $(1-3t+t^2)(t-2)(2t-1)$ & $3^2 \cdot 5$ \\
$(9_{39}+7_2)$ & $(1-3t+t^2)(3-5t+3t^2)^2$ & $5 \cdot 11^2$ \\
$(9_{40}-3_1)$ & $(1-t+t^2)^2(1-3t+t^2)^2$ & $3^2 \cdot 5^2$ \\
$9_{41}$ & $(t^2-3t+3)(3t^2-3t+1)$ & $7^2$ \\
$(9_{42}+8_5-3_1)$ & $(1-t+t^2)^2(1-2t+t^2-2t^3+t^4)^2$ & $3^2 \cdot 7^2$ \\
$9_{44}$ & $1-4t+7t^2-4t^3+t^4$ & $17$ \\
$9_{46}$ & $(t-2)(2t-1)$ & $3^2$ \\
$(8_{17}-8_{17}^r)$ & $(1-4t+8t^2-11t^3+8t^4-4t^5+t^6)^2$ & $37^2$ \\
$(9_{32}-9_{32}^r)$ & $(1-6t+14t^2-17t^3+14t^4-6t^5+t^6)^2$ & $59^2$ \\
$(9_{33}-9_{33}^r)$ & $(1-6t+14t^2-19t^3+14t^4-6t^5+t^6)^2$ & $61^2$
\end{tabular} 
\renewcommand{\chaptername}[1]{Appendix C. }
\renewcommand{\chaptermark}[1]{\markboth{\chaptername \ #1}{}}
\chapter{Witt group calculations for $p\equiv 3$ mod 4}
\label{AppendixA}

Here we analyse the image of the knots in $B_{\sigma}$ in the Witt groups $W(\F_p)$ for $p\equiv 3$ mod $4$. More precisely, for each knot we compute a Seifert matrix $V$ and the corresponding isometric structure $(V+V^T, V^{-1}V^T)$, then look at the image of this form in $W(\F_p)$, restricted to the $\lambda(t)$-primary component for each $\lambda(t)$ an irreducible symmetric factor of the Alexander polynomial $\Delta_V(t)$. The method of finding this image is detailed in Section \ref{Sec:ElementsOrder4}. The entries of the following table are elements in $\Z_4$.

\newpage

\begin{tabular}{lcccccccccc}
Polynomial & \begin{sideways}$1-t+t^2$\end{sideways} & \begin{sideways}$1-5t+9t^2-5t^3+t^4$\end{sideways} & \begin{sideways}$1-5t+9t^2-5t^3+t^4$\end{sideways} & \begin{sideways}$3-5t+3t^2$\end{sideways} & \begin{sideways}$4-7t+4t^2$\end{sideways} & \begin{sideways}$2-8t+11t^2-8t^3+2t^4$\end{sideways} & \begin{sideways}$2-3t+2t^2$\end{sideways} & \begin{sideways}$1-6t+16t^2-23t^3+16t^4-6t^5+t^6$\end{sideways} & \begin{sideways}$1-6t+16t^2-23t^3+16t^4-6t^5+t^6$\end{sideways} & \begin{sideways}$1-2t+t^2-2t^3+t^4$\end{sideways}\\ \hline
Prime & 3 & 3 & 7 & 11 & 3 & 31 & 7 & 3 & 23 & 7 \\ \hline
$7_7$ & 0 & 3 & 3 & 0 & 0 & 0 & 0 & 0 & 0 & 0 \\
$(8_{10}+3_1)$ & 0 & 0 & 0 & 0 & 0 & 0 & 0 & 0 & 0 & 0 \\
$(8_{11}-3_1)$ & 0 & 0 & 0 & 0 & 0 & 0 & 0 & 0 & 0 & 0 \\
$(8_{15}-7_2-3_1)$ & 2 & 0 & 0 & 2 & 0 & 0 & 0 & 0 & 0 & 0\\
$8_{18}$ & 2 & 0 & 0 & 0 & 0 & 0 & 0 & 0 & 0 & 0 \\
$8_{20}$ & 0 & 0 & 0 & 0 & 0 & 0 & 0 & 0 & 0 & 0 \\
$(8_{21}-3_1)$ & 2 & 0 & 0 & 0 & 0 & 0 & 0 & 0 & 0 & 0\\
$(9_2-7_4)$ & 0 & 0 & 0 & 0 & 2 & 0 & 0 & 0 & 0 & 0\\
$(9_8-8_{14})$ & 0 & 0 & 0 & 0 & 0 & 0 & 0 & 0 & 0 & 0\\
$(9_{12}-5_2)$ & 0 & 0 & 0 & 0 & 0 & 0 & 2 & 0 & 0 & 0\\
$(9_{16}-7_3-3_1)$ & 2 & 0 & 0 & 0 & 0 & 0 & 0 & 0 & 0 & 0\\
$(9_{23}-7_4-3_1)$& 0 & 0 & 0 & 0 & 2 & 0 & 0 & 0 & 0 & 0 \\
$9_{24}$ & 0 & 0 & 0 & 0 & 0 & 0 & 0 & 0 & 0 & 0\\
$(9_{28}-3_1)$ & 2 & 0 & 0 & 0 & 0 & 0 & 0 & 0 & 0 & 0\\
$(9_{29}+3_1)$ & 2 & 0 & 0 & 0 & 0 & 0 & 0 & 0 & 0 & 0\\
$9_{34}$ & 0 & 0 & 0 & 0 & 0 & 0 & 0 & 3 & 1 & 0\\
$(9_{39}+7_2)$ & 0 & 0 & 0 & 0 & 0 & 0 & 0 & 0 & 0 & 0\\
$(9_{40}-3_1)$ & 2 & 0 & 0 & 0 & 0 & 0 & 0 & 0 & 0 & 0\\
$(9_{42}+8_5-3_1)$ & 2 & 0 & 0 & 0 & 0 & 0 & 0 & 0 & 0 & 2
\end{tabular} 

\renewcommand{\chaptername}[1]{Appendix D. }
\renewcommand{\chaptermark}[1]{\markboth{\chaptername \ #1}{}}
\chapter{Witt group calculations for $p\equiv 1$ mod 4}
\label{AppendixB}

Here we analyse the image of the knots in $B_{3}$ in the Witt groups $W(\F_p)$ for $p\equiv 1$ mod $4$. As in Appendix \ref{AppendixA}, for each knot we compute a Seifert matrix $V$ and the corresponding isometric structure $(V+V^T, V^{-1}V^T)$, then look at the image of this form in $W(\F_p)$, restricted to the $\lambda(t)$-primary component for each $\lambda(t)$ an irreducible symmetric factor of the Alexander polynomial $\Delta_V(t)$. The entries of the following table are elements in $\Z_2 \oplus \Z_2$.

\medskip

\begin{tabular}{lcccc}
Polynomial & \begin{sideways}$1-3t+t^2$\end{sideways} & \begin{sideways}$2-3t+3t^2-3t^3+2t^4$\end{sideways} & \begin{sideways}$4-7t+4t^2$\end{sideways} & \begin{sideways}$1-4t+7t^2-4t^3+t^4$\end{sideways} \\ \hline
Prime & 5 & 13 & 5 & 17  \\ \hline
$(8_{21}-8_{18}-3_1)$ & 0 & 0 & 0 & 0 \\
$(9_{16}-8_{18}-7_3-4_1-3_1)$ & (1,1) & (1,1) & 0 & 0  \\
$(9_{23}-9_2-3_1)$ & 0 & 0 & 0 & 0  \\
$(9_{24}-4_1)$ & 0 & 0 & 0 & 0 \\
$(9_{29}-9_{28}+2(3_1))$ & 0 & 0 & 0 & 0  \\
$(9_{37}-4_1)$ & 0 & 0 & 0 & 0  \\
$(9_{39}+7_2-4_1)$ & 0 & 0 & 0 & 0 \\
$(9_{40}-8_{18}-4_1-3_1)$ & 0 & 0 & 0 & 0 \\
$(9_{44}-9_{28}+8_{18}-4_1+3_1)$ & (1,1) & 0 & 0 & (1,1) \\
\end{tabular} 

\renewcommand{\chaptername}[1]{Appendix E. }
\renewcommand{\chaptermark}[1]{\markboth{\chaptername \ #1}{}}
\chapter{Twisted Alexander polynomial calculations}
\label{AppendixC}

Here we list the knots analysed in Section \ref{Subsec:infiniteorder}, giving the trivial ($\Delta_{\chi_0}$) and non-trivial ($\Delta_{\chi}/(t-1)$) twisted Alexander polynomials factorised into irreducibles over $\Q(\zeta_q)[t,t^{-1}]$ for each knot.  The prime $q$ being used is given, and $\zeta_q$ is abbreviated to $w$ in each case. Polynomials are written modulo norms; thus, for example, in the first calculation the trivial twisted polynomial corresponding to $7_2$ is \emph{not} a norm.

The program used for these computations (and for those in Chapter \ref{chapter8}) was adapted from one written by Herald, Kirk and Livingston, described fully in \cite{HeraldKirkLivingston10} and run in Maple 13. This algorithm is slightly different from the one described in Section \ref{Sec:twistpolyobstr} in that it takes advantage of a relationship between twisted polynomials associated to representations of $\pi_1(X_p)$ and twisted polynomials associated to representations of $\pi_1(X)$ in order to simplify calculations. Group presentations of $\pi_1(X_p)$ become complicated very quickly as $p$ increases, so this algorithm makes computation much faster for the polynomials associated to the higher branched covers.

In addition to the twisted Alexander polynomials, we also give the relevant linking forms for each knot (since these are necessary for the use of Theorem \ref{Thm:infiniteordersum}) and say whether each one is a square in $\Z_q$. The linking form for $\Sigma_2$ is found by calculating $(V+V^T)^{-1}$ for a Seifert matrix $V$ and taking one of the diagonal entries.

\medskip

$\mathbf{K=8_{15}-7_2-3_1}$, $\mathbf{H_1(\Sigma_2;\Z) = \Z_{33} \oplus \Z_{11} \oplus \Z_3}$, $\mathbf{q=11}$, $\lk_{8_{15}}(3,3) = \frac{-5}{11}$ (not square), $\lk_{-7_2}(1,1) = \frac{-1}{11}$ (not square);

\smallskip

\begin{tabular}{cccl}
$8_{15}$ & $\Delta_{\chi_0}$ & $=$ & $(9t^2-7t+9)(t^2+t+1)$\\
& $\Delta_{\chi_1}$ & $=$ & $1+t(w^3-w^4+w^8-w^7+3w^5+3w^6-5)+t^2$\\
$7_2$ & $\Delta_{\chi_0}$ & $=$ & $9t^2-7t+9$\\
& $\Delta_{\chi_1}$ & $=$ & $1$\\
$3_1$ & $\Delta_{\chi_0}$ & $=$ & $t^2+t+1$
\end{tabular}

\medskip

where
\[ 9t^2-7t+9 = \frac{1}{9}(9t-1+5w+5w^3+5w^4+5w^5+5w^9)(9t-1+5w^2+5w^6+5w^7+5w^8+5w^{10}) \]

\medskip

$\mathbf{K=9_2-7_4}$, $\mathbf{H_1(\Sigma_2;\Z)= \Z_{15} \oplus \Z_{15}}$, $\mathbf{q=3}$, $\lk_{9_2}(5,5) = \frac{-1}{3}$ (not square), $\lk_{-7_4}(5,5) = \frac{-1}{3}$ (not square);

\smallskip

\begin{tabular}{cccl}
$9_{2}$ & $\Delta_{\chi_0}$ & $=$ & $16t^2-17t+16$\\
& $\Delta_{\chi_1}$ & $=$ & $1$\\
$7_4$ & $\Delta_{\chi_0}$ & $=$ & $16t^2-17t+16$\\
& $\Delta_{\chi_1}$ & $=$ & $1$\\
\end{tabular}

\bigskip

$\mathbf{K=9_{12}-5_2}$, $\mathbf{H_1(\Sigma_2;\Z) = \Z_{35} \oplus \Z_7}$, $\mathbf{q=7}$, $\lk_{9_{12}}(5,5)= \frac{2}{7}$ (square), $\lk_{-5_2}(1,1) = \frac{4}{7}$ (square);

\smallskip

\begin{tabular}{cccl}
$9_{12}$ & $\Delta_{\chi_0}$ & $=$ & $(4t^2-t+4)(t^2-7t+1)$\\
& $\Delta_{\chi_1}$ & $=$ & $t^2+t(w^2-w^3-w^4+w^5)+1$\\
$5_2$ & $\Delta_{\chi_0}$ & $=$ & $4t^2-t+4$\\
& $\Delta_{\chi_1}$ & $=$ & $1$\\
\end{tabular}

\medskip

where
\[ 4t^2-t+4 = \frac{1}{4}(4t-2-3w-3w^2-3w^4)(4t+1+3w+3w^2+3w^4) \]

\medskip

$\mathbf{K=9_{16}-7_3-3_1}$, $\mathbf{H_1(\Sigma_2;\Z) = \Z_{39} \oplus \Z_{13} \oplus \Z_3}$, $\mathbf{q=3}$, $\lk_{9_{16}}(13,13) = \frac{-1}{3}$ (not square), $\lk_{-3_1}(1,1) = \frac{-1}{3}$ (not square);

\smallskip

\begin{tabular}{cccl}
$9_{16}$ & $\Delta_{\chi_0}$ & $=$ & $(t-w)(t-w^2)(4t^4+3t^3-t^2+3t+4)$\\
& $\Delta_{\chi_1}$ & $=$ & $4t^4-3t^3+2t^2-3t+4$\\
$3_1$ & $\Delta_{\chi_0}$ & $=$ & $(t-w)(t-w^2)$\\
& $\Delta_{\chi_1}$ & $=$ & $1$\\
\end{tabular}

\bigskip

$\mathbf{K=9_{28}-3_1}$, $\mathbf{H_1(\Sigma_2;\Z) =\Z_{51} \oplus \Z_3}$, $\mathbf{q=3}$, $\lk_{9_{28}}(17,17) = \frac{-1}{3}$ (not square), $\lk_{-3_1}(1,1) = \frac{-1}{3}$ (not square);

\smallskip

\begin{tabular}{cccl}
$9_{28}$ & $\Delta_{\chi_0}$ & $=$ & $(t-w)(t-w^2)(t^4-2t^3+19t^2-2t+1)$\\
& $\Delta_{\chi_1}$ & $=$ & $t^4-5t^3+4t^2-5t+1$\\
$3_1$ & $\Delta_{\chi_0}$ & $=$ & $(t-w)(t-w^2)$\\
& $\Delta_{\chi_1}$ & $=$ & $1$\\
\end{tabular}

\bigskip

$\mathbf{K=9_{42}+8_5-3_1}$, $\mathbf{H_1(\Sigma_2;\Z) = \Z_7 \oplus \Z_{21} \oplus \Z_3}$, $\mathbf{q=3}$, $\lk_{8_5}(7,7) = \frac{-1}{3}$ (not square), $\lk_{-3_1}(1,1) = \frac{-1}{3}$ (not square);

\smallskip

\begin{tabular}{cccl}
$8_{5}$ & $\Delta_{\chi_0}$ & $=$ & $(t-w)(t-w^2)(t^4-2t^3-5t^2-2t+1)$\\
& $\Delta_{\chi_1}$ & $=$ & $t^4-5t^3+4t^2-5t+1$\\
$3_1$ & $\Delta_{\chi_0}$ & $=$ & $(t-w)(t-w^2)$\\
& $\Delta_{\chi_1}$ & $=$ & $1$\\
\end{tabular}

\bigskip

$\mathbf{K=9_{29}-9_{28}+2(3_1)}$, $\mathbf{H_1(\Sigma_2;\Z) = \Z_{51} \oplus \Z_{51} \oplus \Z_3 \oplus \Z_3}$, $\mathbf{q=3}$, $\lk_{9_{29}}(17,17) = \frac{1}{3}$ (square), $\lk_{-9_{28}}(17,17) = \frac{1}{3}$ (square), $\lk_{3_1}(1,1) = \frac{1}{3}$ (square);

\smallskip

\begin{tabular}{cccl}
$9_{29}$ & $\Delta_{\chi_0}$ & $=$ & $(t-w)(t-w^2)(t^4-2t^3+19t^2-2t+1)$\\
& $\Delta_{\chi_1}$ & $=$ & $t^4+t^3+t^2+t+1$\\
$9_{28}$ & $\Delta_{\chi_0}$ & $=$ & $(t-w)(t-w^2)(t^4-2t^3+19t^2-2t+1)$\\
& $\Delta_{\chi_1}$ & $=$ & $t^4-5t^3+4t^2-5t+1$\\
$3_1$ & $\Delta_{\chi_0}$ & $=$ & $(t-w)(t-w^2)$\\
& $\Delta_{\chi_1}$ & $=$ & $1$\\
\end{tabular}

\bigskip

$\mathbf{K=9_{40}-8_{18}-4_1-3_1}$ (but consider $\mathbf{2K=2(9_{40}-3_1)}$),\\ $\mathbf{H_1(\Sigma_2(2K);\Z) = 2((\Z_{5} \oplus \Z_5 \oplus \Z_3) \oplus \Z_3)}$, $\mathbf{q=3}$, $\lk_{9_{40}}(25,25) = \frac{-1}{3}$ (not square), $\lk_{-3_1}(1,1) = \frac{-1}{3}$ (not square);

\smallskip

\begin{tabular}{cccl}
$9_{40}$ & $\Delta_{\chi_0}$ & $=$ & $(t-w)(t-w^2)(t^4-2t^3-5t^2-2t+1)$\\
& $\Delta_{\chi_1}$ & $=$ & $t^4-5t^3+4t^2-5t+1$\\
$3_1$ & $\Delta_{\chi_0}$ & $=$ & $(t-w)(t-w^2)$\\
& $\Delta_{\chi_1}$ & $=$ & $1$\\
\end{tabular}

\bigskip

$\mathbf{K=(8_{21}-8_{18}-3_1)}$ (but consider $\mathbf{2K=2(8_{21}-3_1)}$), $\mathbf{H_1(\Sigma_2(2K);\Z) = 2(\Z_{15} \oplus \Z_3)}$, $\mathbf{q=5}$;
\smallskip

\begin{tabular}{cccl}
$8_{21}$ & $\Delta_{\chi_0}$ & $=$ & $(t^2+t+1)$\\
& $\Delta_{\chi_1}$ & $=$ & $t^2 - 3t(w^3+w^2+1) +1$\\
\end{tabular}

\bigskip

$\mathbf{K=(9_{39}+7_2-4_1)}$ (but consider $\mathbf{2K=2(9_{39}+7_2)}$), $\mathbf{H_1(\Sigma_2(2K);\Z) = 2(\Z_{55} \oplus \Z_{11})}$, $\mathbf{q=5}$

\smallskip

\begin{tabular}{cccl}
$9_{39}$ & $\Delta_{\chi_0}$ & $=$ & $(9t^2-7t+9)$\\
& $\Delta_{\chi_1}$ & $=$ & $t^2+t(5w^2+5w^3-3)+1$\\
\end{tabular}

\bigskip

$\mathbf{K=(9_{23}-9_2-3_1)}$, $\mathbf{H_1(\Sigma_2;\Z) = (\Z_9 \oplus \Z_5) \oplus (\Z_5 \oplus \Z_3) \oplus \Z_3}$, $\mathbf{q=5}$, $\lk_{9_{23}}(9,9) = \frac{1}{5}$ (square), $\lk_{-9_2}(3,3) = \frac{-1}{5}$ (square);

\smallskip

\begin{tabular}{cccl}
$9_{23}$ & $\Delta_{\chi_0}$ & $=$ & $(t^2+t+1)(16t^2-17t+16)$\\
& $\Delta_{\chi_1}$ & $=$ & $t^2+t(9w^2+9w^3-7)+1$\\
$9_{2}$ & $\Delta_{\chi_0}$ & $=$ & $16t^2-17t+16$\\
& $\Delta_{\chi_1}$ & $=$ & $1$\\
\end{tabular}

\bigskip

$\mathbf{K=(9_8-8_{14})}$, $\mathbf{H_1(\Sigma_2;\Z) = \Z_{31} \oplus \Z_{31}}$, $\mathbf{q=31}$, $\lk_{9_8}(1,1) = \frac{11}{31}$ (not square), $\lk_{-8_{14}}(1,1) = \frac{-12}{31}$ (square);

\smallskip

\begin{tabular}{cccl}
$9_{8}$ & $\Delta_{\chi_0}$ & $=$ & $4t^4-20t^3+t^2-20t+4$\\
& $\Delta_{\chi_1}$ & $=$ & $1+t^2 + t(-w -3w^2 -5w^3 -7w^4 -10w^5 -11w^6 -12w^7 -13w^8$\\
& & & $-14w^9 -16w^{10}- 18w^{11} -20w^{12} -22w^{13} - 24w^{14} -24w^{15} -24w^{16}$\\
& & & $-24w^{17} -22w^{18}-20w^{19}-18w^{20}-16w^{21}-14w^{22}-13w^{23}-12w^{24}$\\
& & & $-11w^{25}-10w^{26}-7w^{27}-5w^{28}-3w^{29}-w^{30})$\\

$8_{14}$ & $\Delta_{\chi_0}$ & $=$ & $4t^4-20t^3+t^2-20t+4$\\ 
& $\Delta_{\chi_1}$ & $=$ & $1+t^2+t(32w+11w^2+20w^3+26w^4+6w^5+36w^6+5w^7+28w^8$\\
& & & $+17w^9+13w^{10}+30w^{11}+3w^{12}+34w^{13}+9w^{14}+23w^{15}+23w^{16}$\\
& & & $+9w^{17}+34w^{18}+3w^{19}+30w^{20}+13w^{21}+17w^{22}+28w^{23}+5w^{24}$\\
& & & $+36w^{25}+6w^{26}+26w^{27}+20w^{28}+11w^{29}+32w^{30})$
\end{tabular}

\bigskip

$\mathbf{K=(9_{32}-9_{32}^r)}$, $\mathbf{H_1(\Sigma_5;\Z) = \Z_{11} \oplus \Z_{11}} = E_3 \oplus E_4$, $\mathbf{q=11}$

\smallskip

\begin{tabular}{ccl}
$\Delta_{\chi_0}$ & $=$ & $1-16t + 5354t^2 - 10557t^3 + 5354t^4 - 16t^5 + t^6$\\
$\Delta_{\chi_3}$ & $=$ & $1+t(7w+10w^2+7w^3+7w^4+7w^5+10w^6+10w^7+10w^8+7w^9$\\
& & $+10w^{10}) + 23t^2 + t^3(10w+7w^2+10w^3+10w^4+10w^5+7w^6+7w^7$ \\
& & $+7w^8+10w^9+7w^{10}) + t^4$\\
$\Delta_{\chi_4}$ & $=$ & $1-t(15w+12w^2+15w^3+15w^4+15w^5+12w^6+12w^7+12w^8+15w^9$\\
& & $+12w^{10})+133t^2 - t^3(12w+15w^2+12w^3+12w^4+12w^5+15w^6+$\\
& & $15w^7+15w^8+12w^9+15w^{10})+ t^4$ \\
$\Delta_{\chi_3 + \chi_4}$ & $=$ & $1 + t(-4w+w^2+2w^3-3w^4-2w^5-4w^6-2w^7-3w^8+7w^9+5w^{10})$\\
& & $+ t^2(21w+20w^2-7w^3-15w^4+9w^5+9w^6-15w^7-7w^8+20w^9$\\
& & $+21w^{10}) + t^3(5w+7w^2-3w^3-2w^4-4w^5-2w^6-3w^7+2w^8+w^9$\\
& & $-4w^{10}) + t^4$\\
$\Delta_{\chi_3 - \chi_4}$ & $=$ & $1 + t(-3w-2w^2+3w^3+5w^4-7w^5-w^6-3w^7+2w^8-5w^9-3w^{10})$\\
& & $+ t^2(36w+10w^2+13w^3+23w^4+25w^5+25w^6+32w^7+13w^8+10w^9$\\
& & $+36w^{10})+ t^3(-3w-5w^2+2w^3-3w^4-w^5-7w^6+5w^7+3w^8-2w^9$\\
& & $-3w^{10})+ t^4$\\
\end{tabular}

\bigskip

$\mathbf{K=(9_{33}-9_{33}^r)}$, $\mathbf{H_1(\Sigma_5;\Z) = \Z_{101} \oplus \Z_{101}} = E_{36} \oplus E_{87}$, $\mathbf{q=101}$

\smallskip

\begin{tabular}{ccl}
$\Delta_{\chi_0}$ & $=$ & $1-236t - 706t^2 - 8319t^3 - 706t^4 - 236t^5 + t^6$
\end{tabular}

\medskip

\begin{tabular}{ccccccccccccccc}
& \multicolumn{14}{c}{Frequency of coefficients of $w^i$ in coefficient of $t$} \\
Polynomial & 4 & 3 & 2 & 1  & -1 & -2 & -3 & -4 & -5 & -6 & -7 & -8 & -9 & -10\\ \hline
(1,0)      & 0 & 0 & 0 & 10 & 15 & 55 &  5 & 10 &  0 &  0 &  0 &  0 &  0 &   0\\
(0,1)      & 0 & 0 & 0 & 0  &  0 &  0 &  5 &  0 &  0 & 50 &  5 & 35 &  5 &   0\\
(1,1)      & 0 & 0 & 1 & 3  &  9 & 18 & 20 & 25 &  9 &  3 &  0 &  1 &  0 &   0\\
(1,4)      & 0 & 0 & 1 & 2  & 15 & 10 & 22 & 26 &  9 &  3 &  1 &  0 &  0 &   0\\
(1,5)      & 0 & 0 & 0 & 0  &  0 &  2 &  0 &  5 &  4 &  9 &  7 & 13 & 16 &  15\\
(1,6)      & 0 & 0 & 1 & 2  & 18 & 20 & 24 & 19 &  9 &  4 &  0 &  0 &  0 &   0\\
(1,9)      & 0 & 0 & 0 & 0  &  0 &  0 &  1 &  5 &  4 &  8 &  6 & 15 & 22 &  14\\
(1,13)     & 1 & 1 & 5 & 15 & 24 & 16 & 11 &  3 &  0 &  0 &  0 &  0 &  0 &   0\\
(1,19)     & 0 & 0 & 0 & 0  &  2 &  6 & 11 & 25 & 27 & 19 &  6 &  2 &  1 &   1\\
(1,22)     & 0 & 0 & 2 & 6  & 23 & 21 & 21 &  8 &  4 &  0 &  0 &  0 &  0 &   0\\
(1,24)     & 0 & 0 & 0 & 0  &  1 &  2 &  6 & 14 & 22 & 22 & 15 & 15 &  2 &   1\\
(1,25)     & 0 & 0 & 0 & 0  &  0 &  2 &  0 &  2 &  1 &  0 &  6 &  1 &  3 &  11

\end{tabular} 
\renewcommand{\chaptername}[1]{Appendix F. }
\renewcommand{\chaptermark}[1]{\markboth{\chaptername \ #1}{}}
\chapter{Slice movies}
\label{AppendixE}

Here we give details of the slice movies used in Chapter \ref{chapter8} to prove that certain knots were of order $2$. For each of the knot diagrams we show where to make a saddle move (see Section \ref{Subsec:sliceMovie} for details on this procedure) in order to obtain two unlinked knots. In each case, one of the unlinked knots is the unknot and the other is either $4_1$ or $6_3$. Since both $4_1$ and $6_3$ are of order $2$ in $\C$, this means that the original knot must also be of order $2$.

\begin{figure}
  \begin{center}
    \includegraphics[width=15cm]{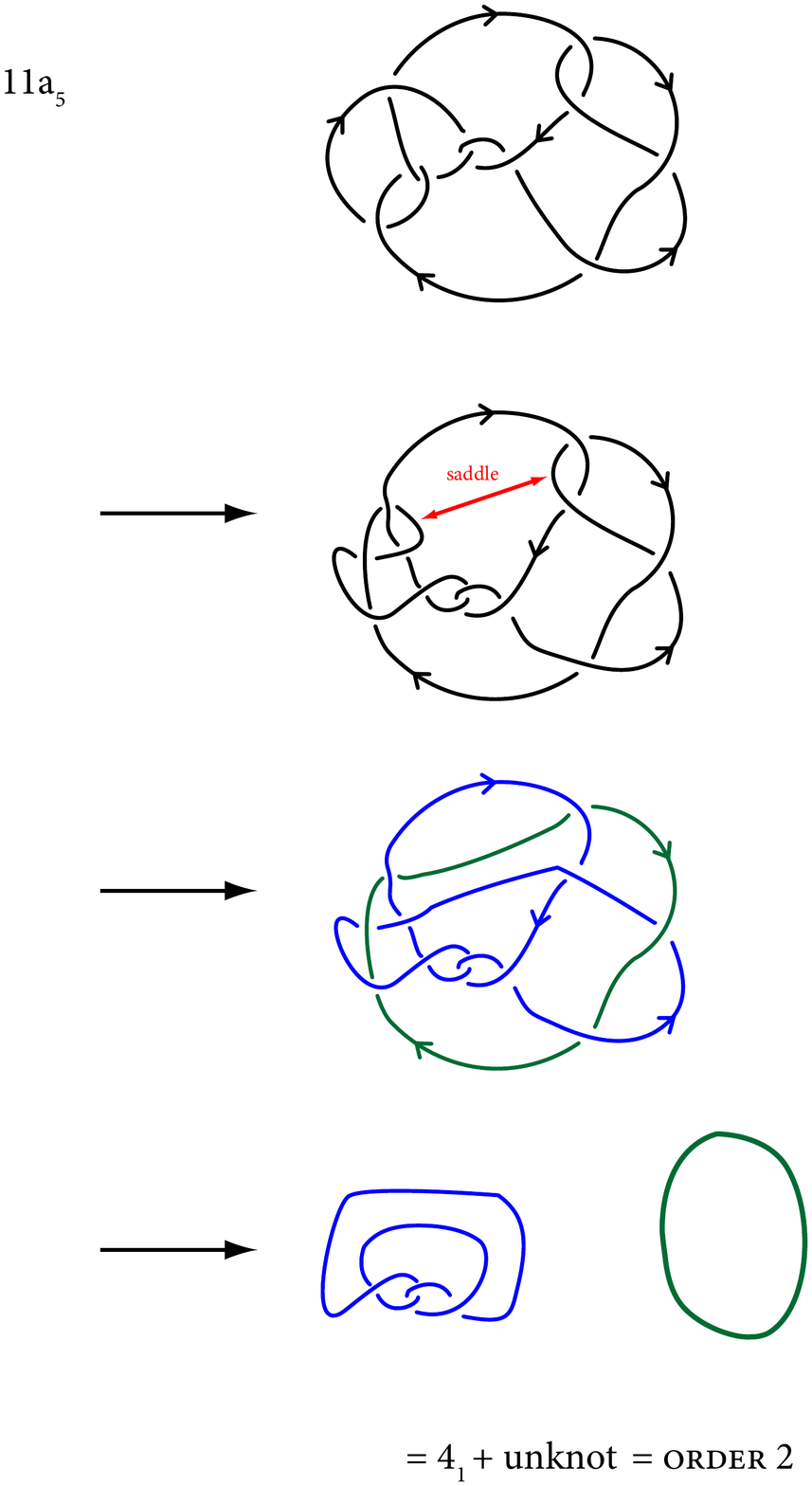}
  \end{center}
  \label{Fig:11a5}
\end{figure}

\begin{figure}
  \begin{center}
    \includegraphics[width=15cm]{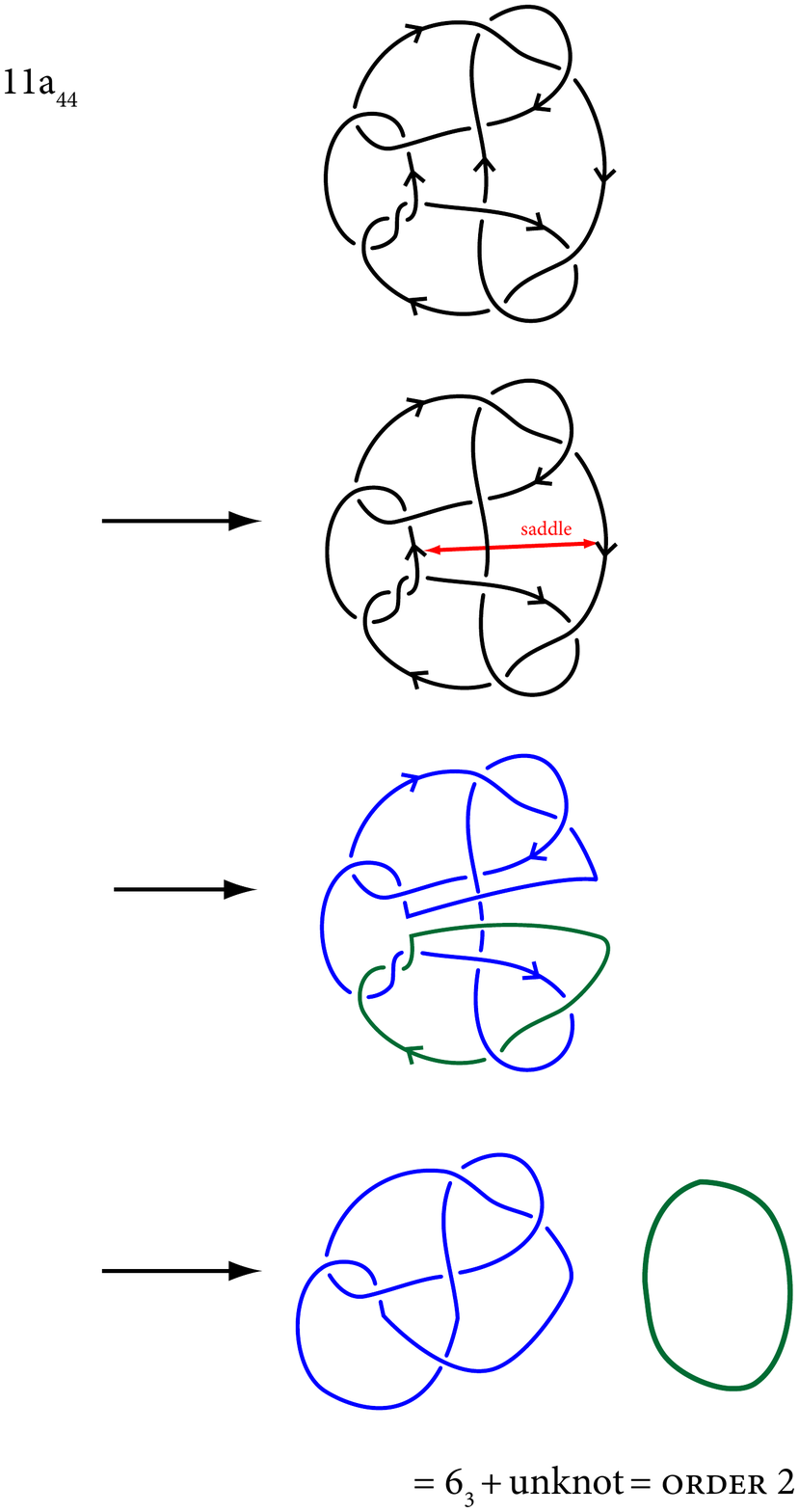}
  \end{center}
  \label{Fig:11a44}
\end{figure}

\begin{figure}
  \begin{center}
    \includegraphics[width=15cm]{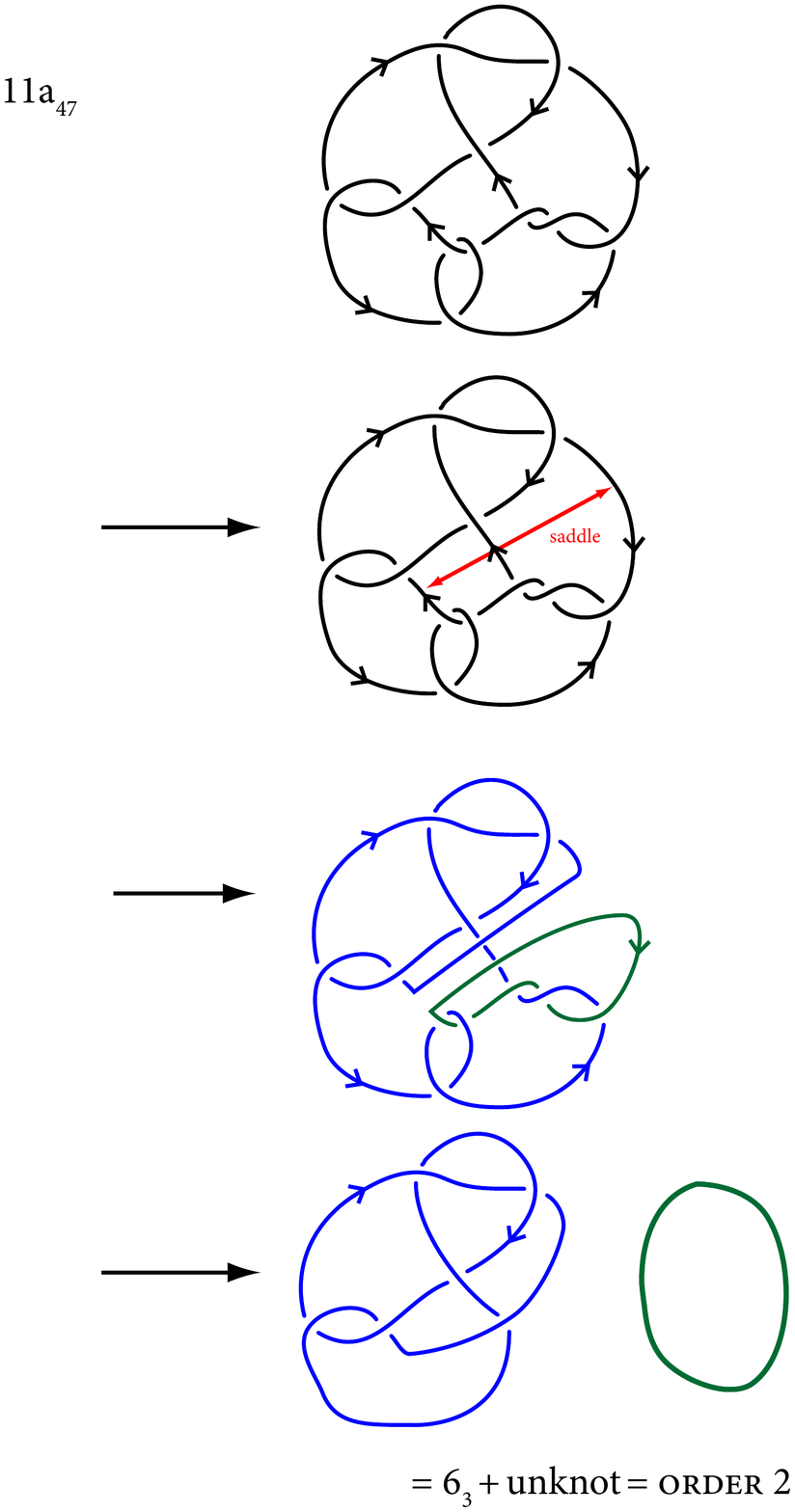}
  \end{center}
  \label{Fig:11a47}
\end{figure}

\begin{figure}
  \begin{center}
    \includegraphics[width=15cm]{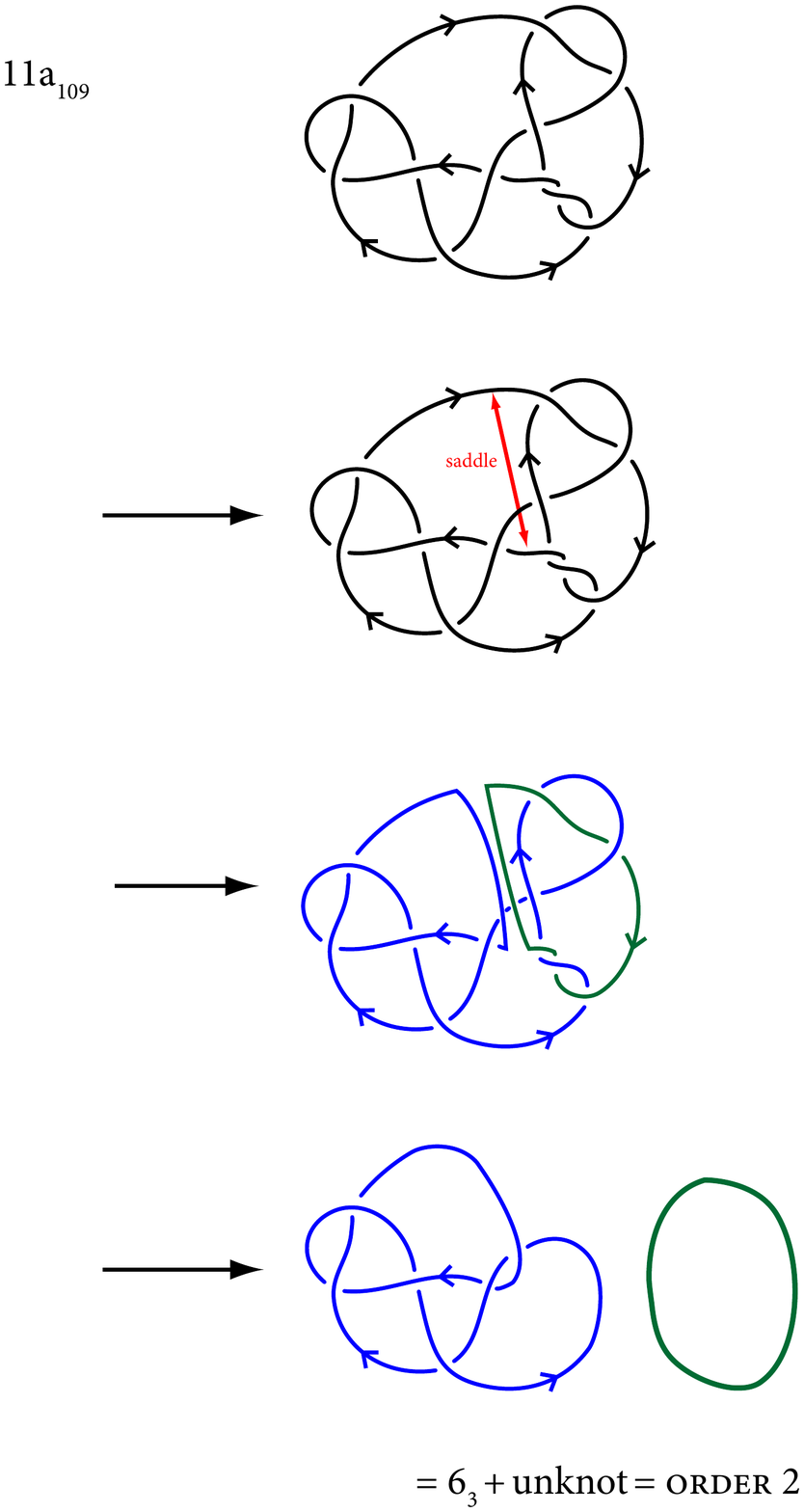}
  \end{center}
  \label{Fig:11a109}
\end{figure}

\begin{figure}
  \begin{center}
    \includegraphics[width=15cm]{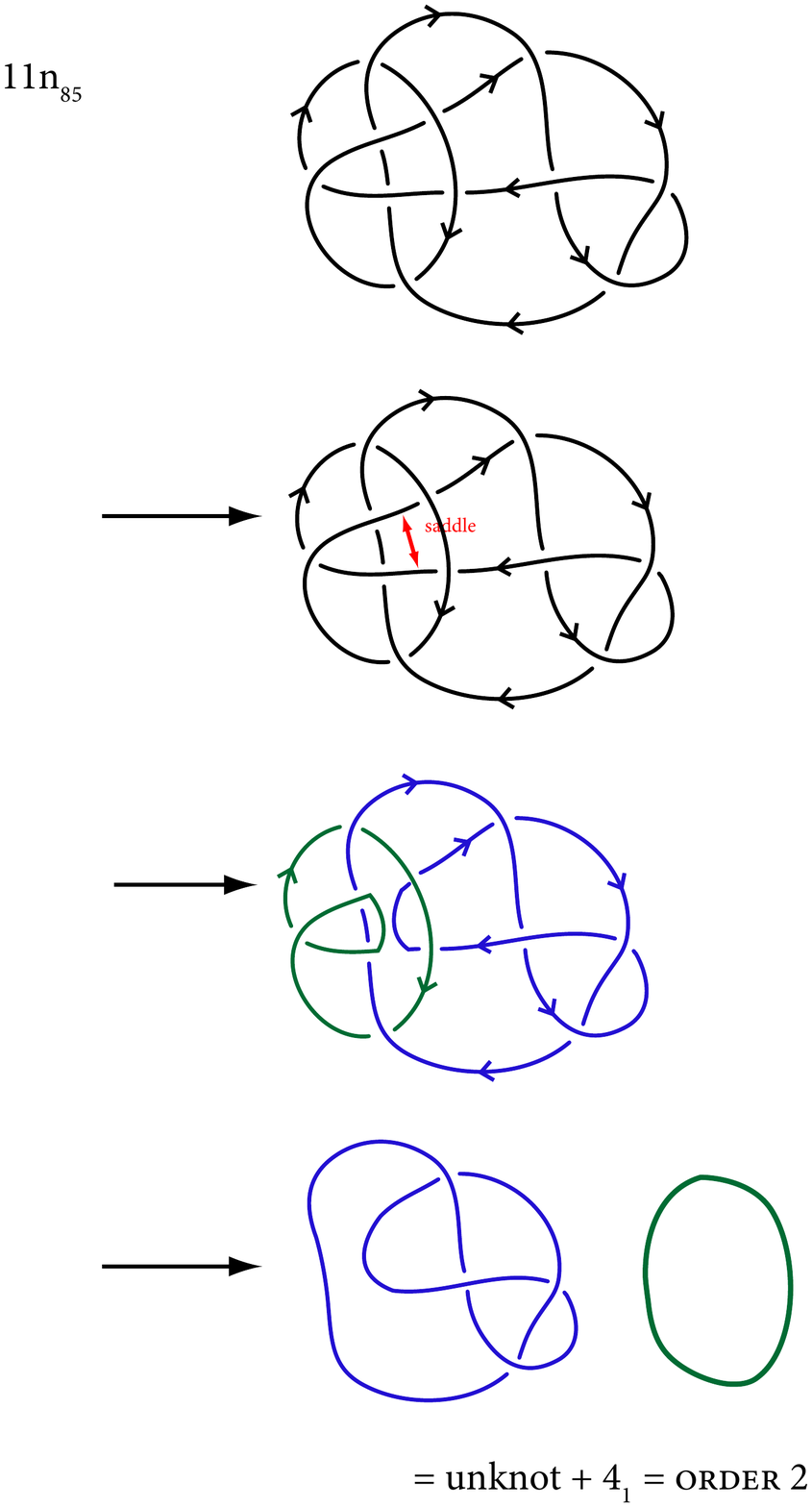}
  \end{center}
  \label{Fig:11n85}
\end{figure}

\begin{figure}
  \begin{center}
    \includegraphics[width=15cm]{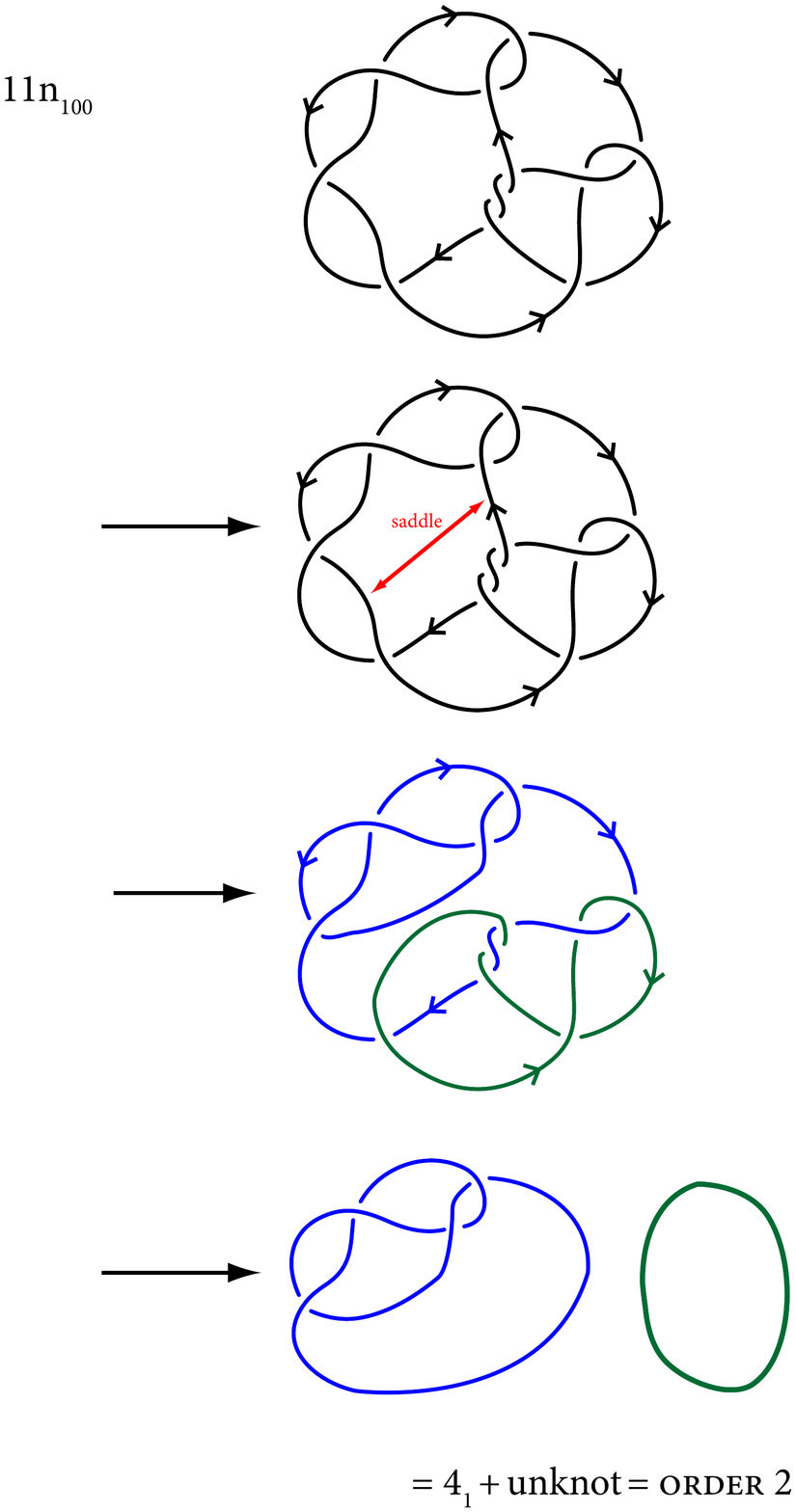}
  \end{center}
  \label{Fig:11n100}
\end{figure}

\begin{figure}
  \begin{center}
    \includegraphics[width=15cm]{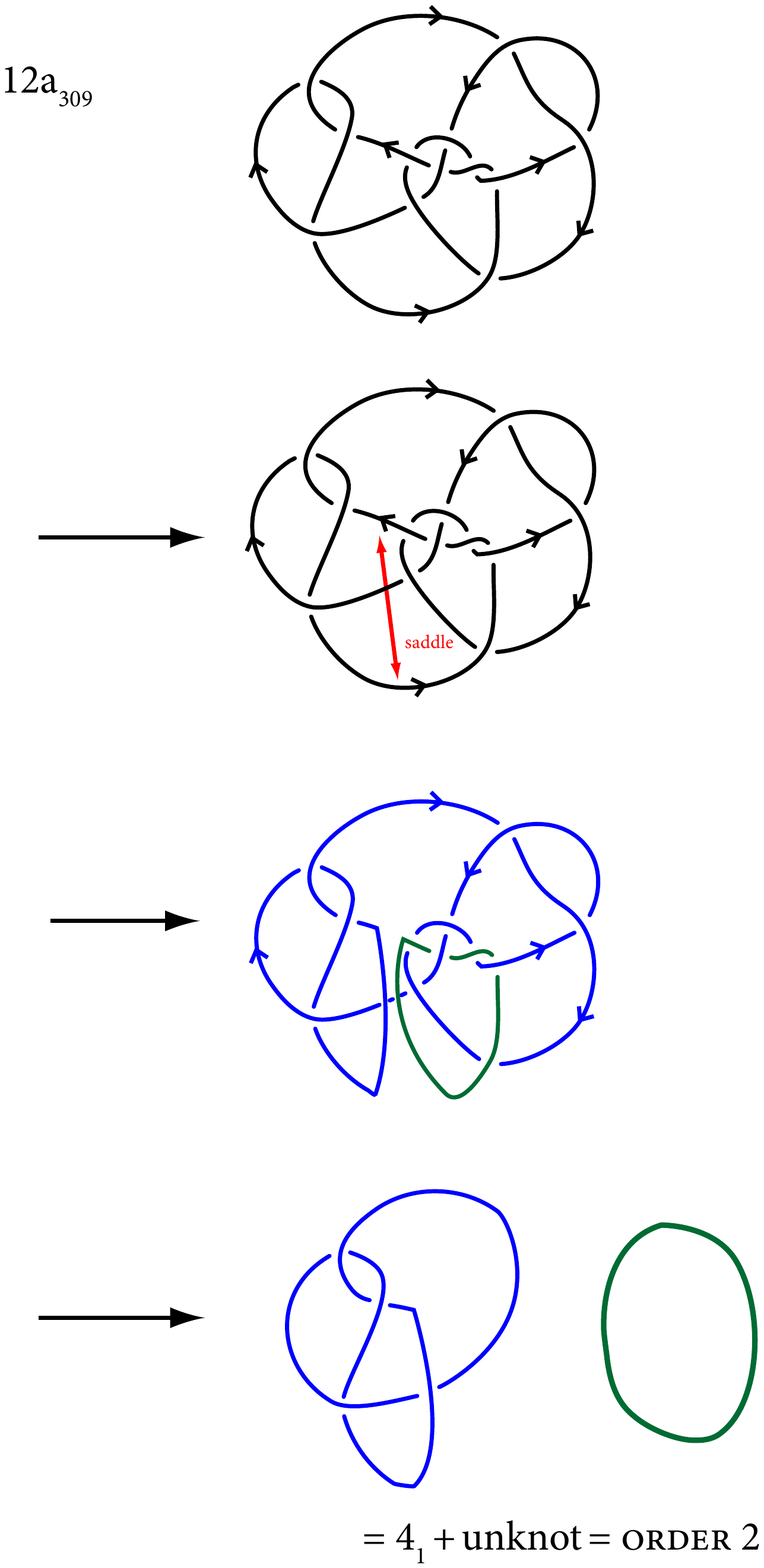}
  \end{center}
  \label{Fig:12a309}
\end{figure}

\begin{figure}
  \begin{center}
    \includegraphics[width=15cm]{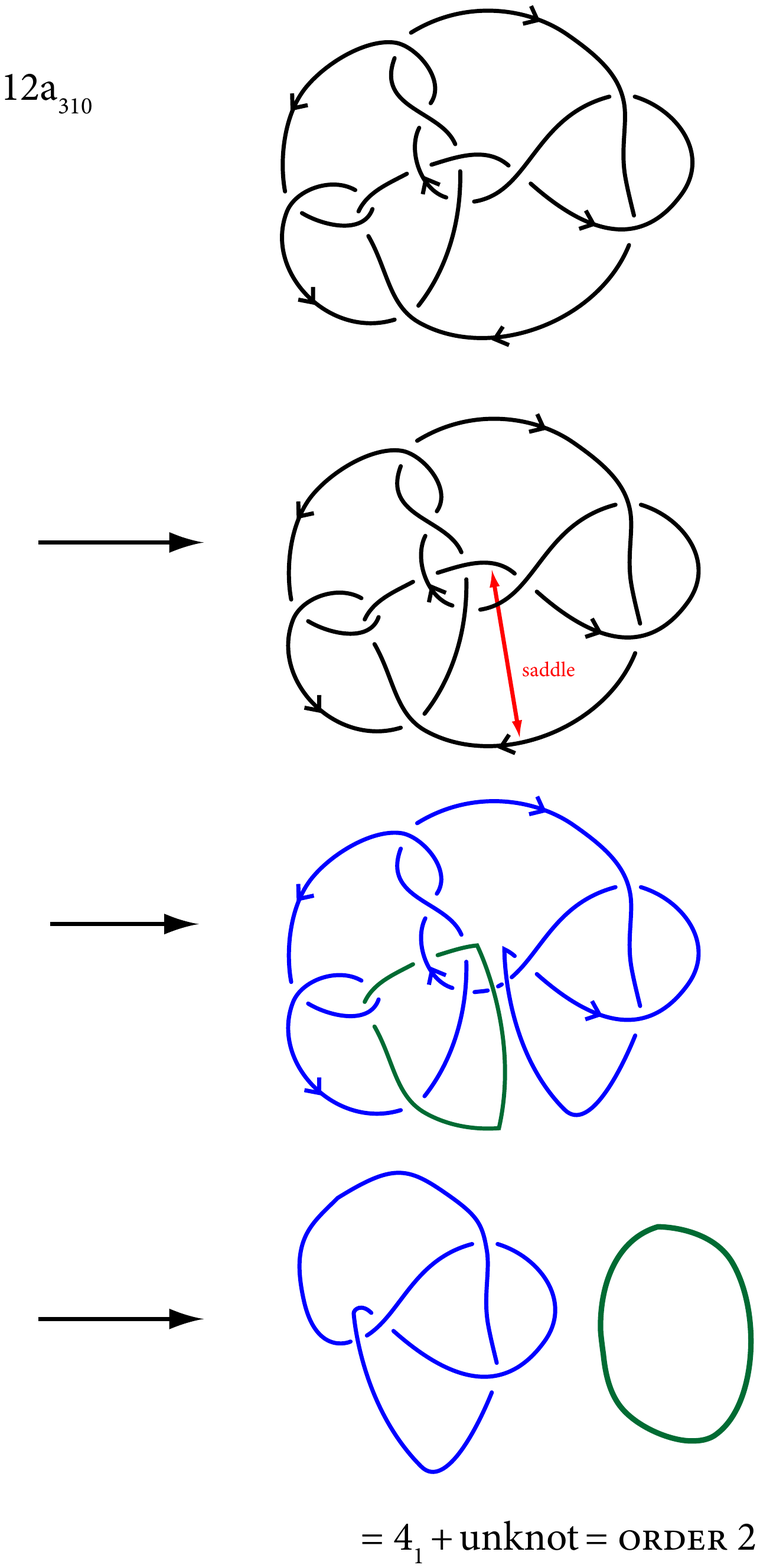}
  \end{center}
  \label{Fig:12a310}
\end{figure}

\begin{figure}
  \begin{center}
    \includegraphics[width=15cm]{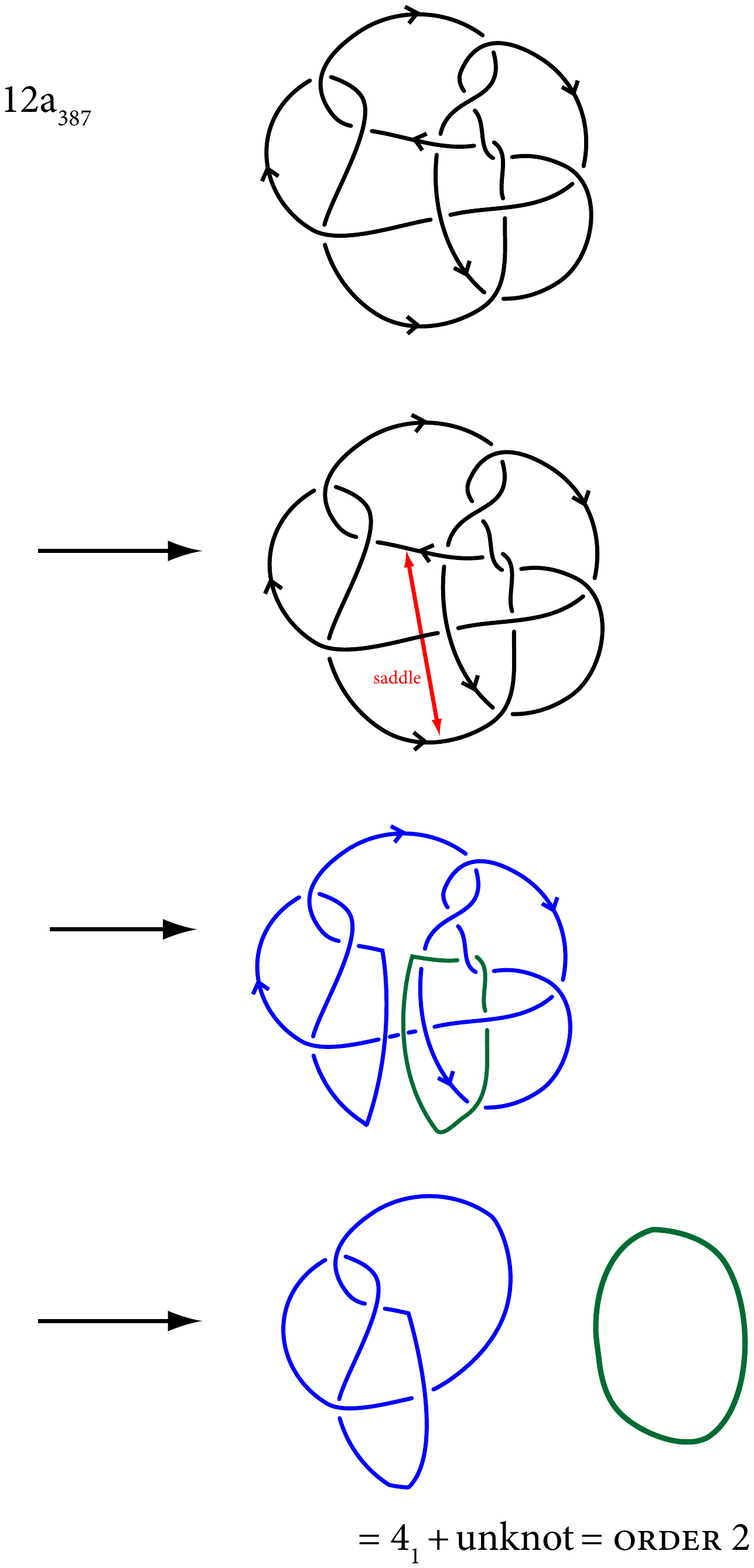}
  \end{center}
  \label{Fig:12a387}
\end{figure}

\begin{figure}
  \begin{center}
    \includegraphics[width=15cm]{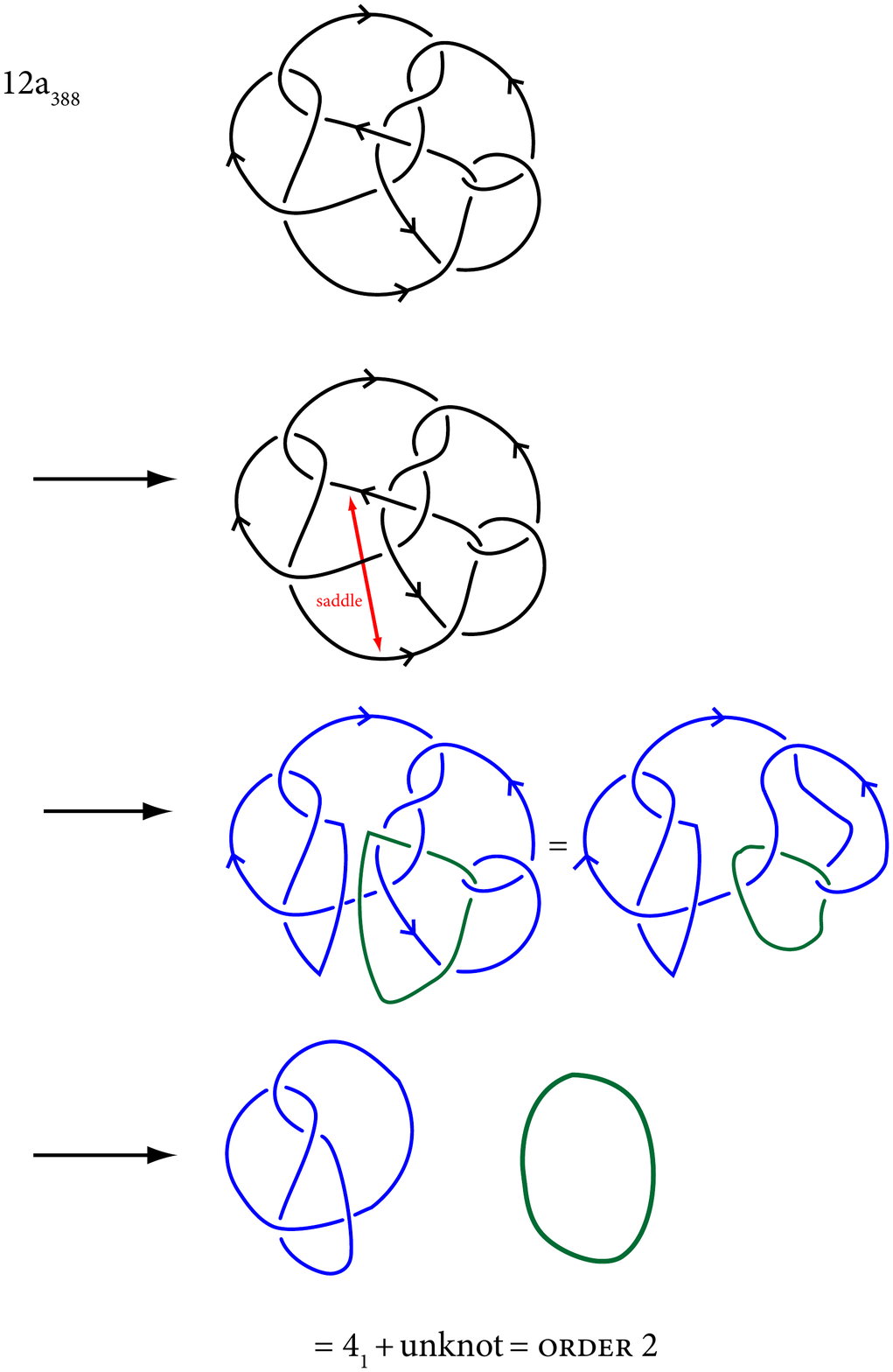}
  \end{center}
  \label{Fig:12a388}
\end{figure}

\begin{figure}
  \begin{center}
    \includegraphics[width=15cm]{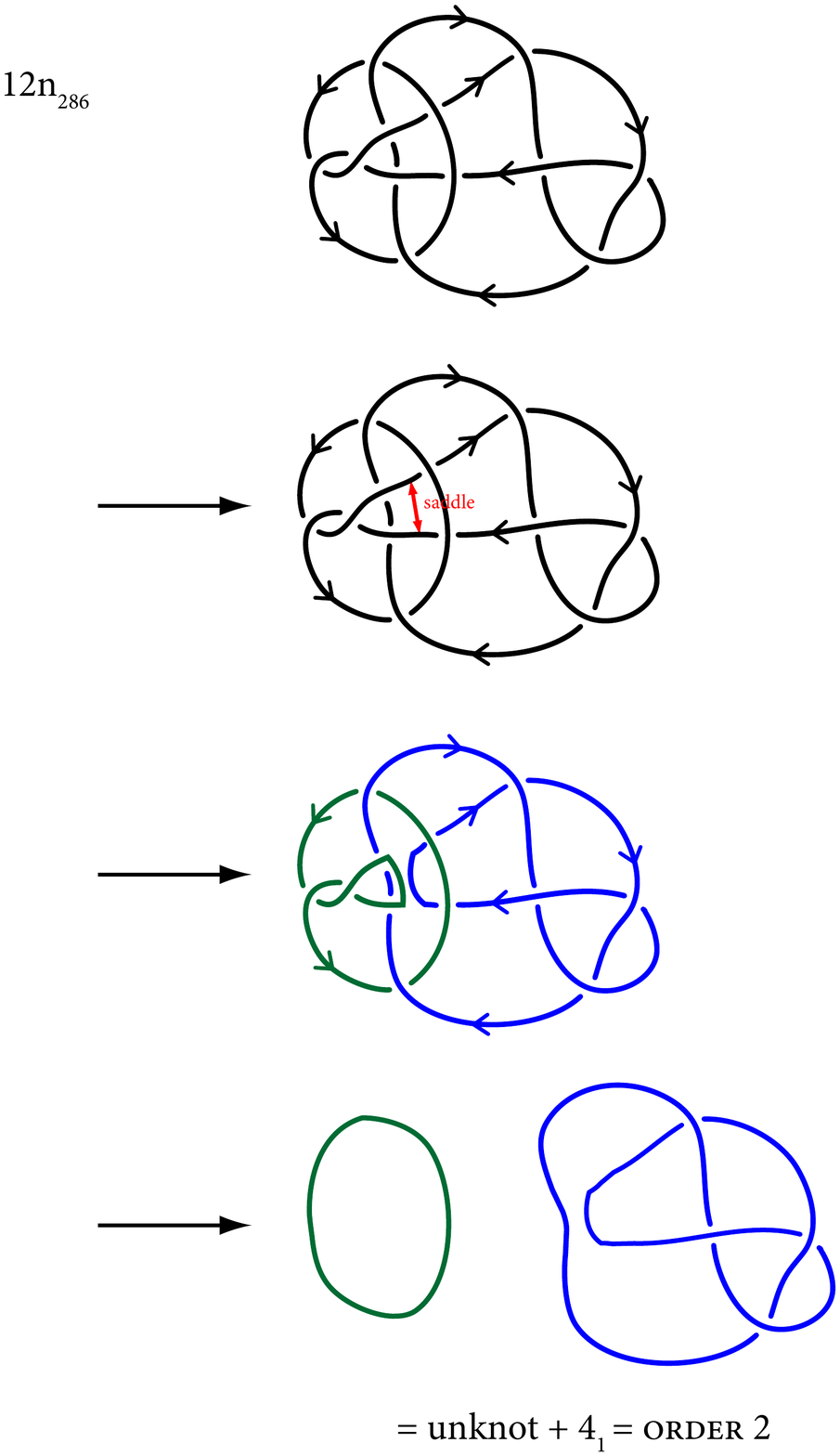}
  \end{center}
  \label{Fig:12n286}
\end{figure}

\begin{figure}
  \begin{center}
    \includegraphics[width=15cm]{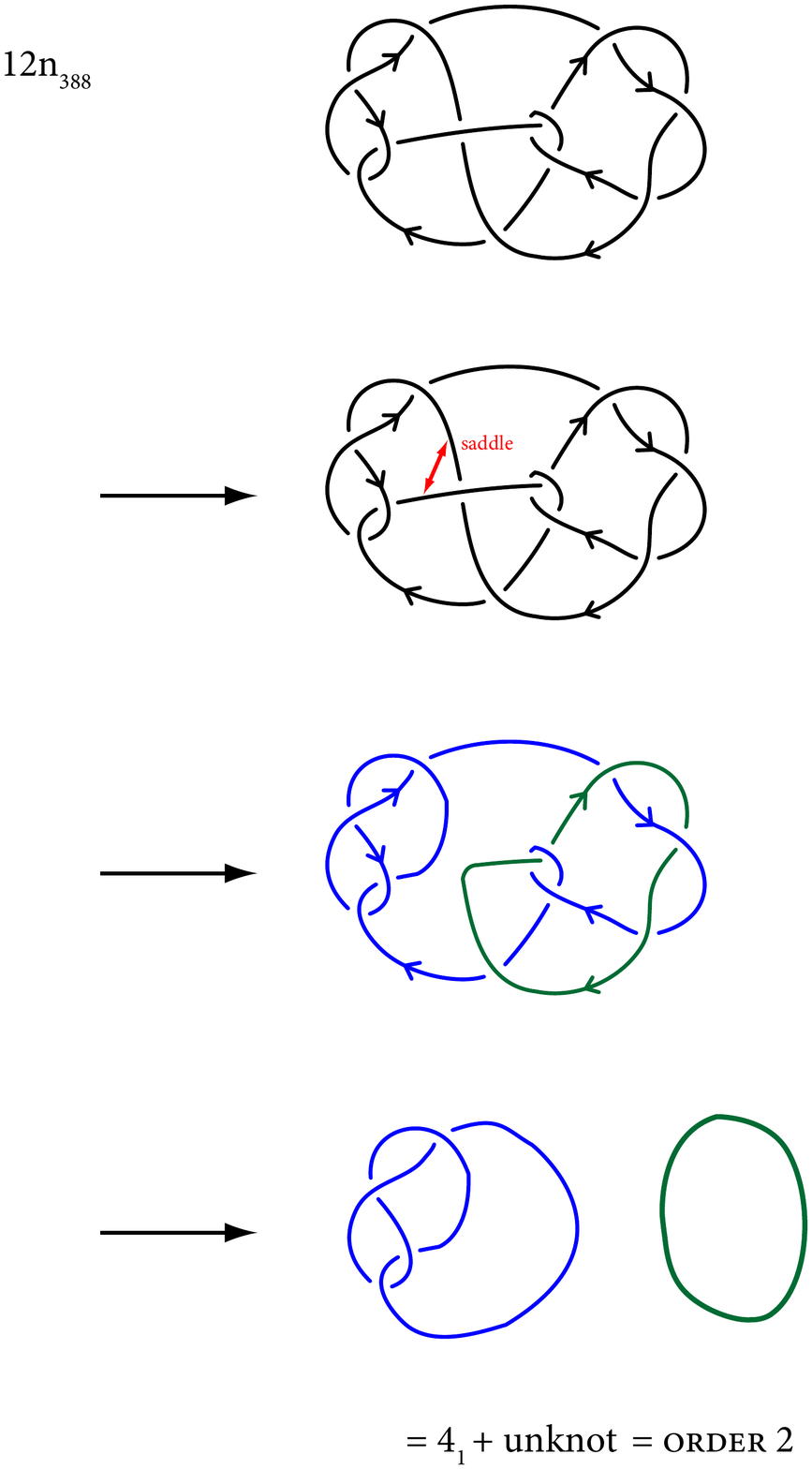}
  \end{center}
  \label{Fig:12n388}
\end{figure}

\small{
\singlespacing

\addcontentsline{toc}{chapter}{Bibliography}
\bibliographystyle{amsalpha}
\bibliography{bibfile}

\providecommand{\bysame}{\leavevmode\hbox to3em{\hrulefill}\thinspace}
\providecommand{\MR}{\relax\ifhmode\unskip\space\fi MR }
\providecommand{\MRhref}[2]{%
  \href{http://www.ams.org/mathscinet-getitem?mr=#1}{#2}
}
\providecommand{\href}[2]{#2}
\begin{thebibliography}{GLM81}

\bibitem[Ale28]{Alexander28}
J.~W. Alexander, \emph{Topological invariants of knots and links}, Transactions
  of A.M.S. \textbf{30} (1928), 275--306.

\bibitem[Art25]{Artin25}
Emil Artin, \emph{Zur {I}sotopie zweidimensionaler {F}l\"{a}chen im ${R}^4$},
  Abh. Math. Seminar Univ. Hamburg \textbf{4} (1925), 174--177.

\bibitem[BO10]{Borodzik10}
Maciej Borodzik and Krzysztof Oleszkiewicz, \emph{On the signatures of torus
  knots}, arXiv:math.AT/1002.4500 (2010).

\bibitem[Bor09]{Borodzik09}
Maciej Borodzik, \emph{A rho-invariant of iterated torus knots},
  arXiv:math.AT/0906.3660 (2009).

\bibitem[BR07]{BeckSinai}
Matthias Beck and Sinai Robins, \emph{Computing the continuous discretely},
  Undergraduate Texts in Mathematics, Springer, New York, 2007.

\bibitem[Bri66]{Brieskorn66}
Egbert Brieskorn, \emph{Beispiele zur {D}ifferentialtopologie von
  {S}ingularit\"aten}, Invent. Math. \textbf{2} (1966), 1--14.

\bibitem[BZ02]{Burde}
G.~Burde and H.~Zieshang, \emph{Knots}, De Gruyter Studies in Mathematics, 5,
  Walter de Gruyter, 2002.

\bibitem[CG86]{CassonGordon86}
A.~J. Casson and C.~McA. Gordon, \emph{Cobordism of classical knots}, \`A la
  recherche de la topologie perdue, Progr. Math., vol.~62, Birkh\"auser Boston,
  1986, With an appendix by P. M. Gilmer, pp.~181--199.

\bibitem[CHL09]{CochranHarveyLeidy09}
Tim~D. Cochran, Shelly Harvey, and Constance Leidy, \emph{Knot concordance and
  higher-order {B}lanchfield duality}, Geom. Topol. \textbf{13} (2009), no.~3,
  1419--1482.

\bibitem[CHL11]{CochranHarveyLeidy11}
\bysame, \emph{2-torsion in the $n$-solvable filtration of the knot concordance
  group}, Proceedings of the London Mathematical Society. Third Series
  \textbf{102} (2011), no.~2, 257–290.

\bibitem[CL09]{ChaLivingston09}
Jae~Choon Cha and Charles Livingston, \emph{Unknown values in the table of
  knots}, arXiv:math/0503125v17 (2009).

\bibitem[CL10]{ChaLivingston}
\bysame, \emph{Knotinfo: Table of knot invariants},
  {\texttt{http://www.indiana.edu/~knotinfo}}, February 2010.

\bibitem[Con06]{Conant06}
James Conant, \emph{Chirality and the {C}onway polynomial}, Topology Proc.
  \textbf{30} (2006), no.~1, 153--162, Spring Topology and Dynamical Systems
  Conference.

\bibitem[COT03]{COT03}
Tim~D. Cochran, Kent~E. Orr, and Peter Teichner, \emph{Knot concordance,
  {W}hitney towers and {$L\sp 2$}-signatures}, Ann. of Math. (2) \textbf{157}
  (2003), no.~2, 433--519.

\bibitem[COT04]{COT04}
\bysame, \emph{Structure in the classical knot concordance group}, Comment.
  Math. Helv. \textbf{79} (2004), no.~1, 105--123.

\bibitem[CT07]{CochranTeichner07}
Tim~D. Cochran and Peter Teichner, \emph{Knot concordance and von {N}eumann
  {$\rho$}-invariants}, Duke Math. J. \textbf{137} (2007), no.~2, 337--379.

\bibitem[Don87]{Donaldson87}
S.~K. Donaldson, \emph{Irrationality and the {$h$}-cobordism conjecture}, J.
  Differential Geom. \textbf{26} (1987), no.~1, 141--168.

\bibitem[Eis95]{Eisenbud}
David Eisenbud, \emph{Commutative algebra}, Graduate Texts in Mathematics, vol.
  150, Springer-Verlag, New York, 1995, With a view toward algebraic geometry.

\bibitem[FK08]{FriedlKim08}
Stefan Friedl and Taehee Kim, \emph{Twisted {A}lexander norms give lower bounds
  on the {T}hurston norm}, Trans. Amer. Math. Soc. \textbf{360} (2008), no.~9,
  4597--4618.

\bibitem[FM66]{FoxMilnor66}
R.~H. Fox and John~W. Milnor, \emph{Singularities of {$2$}-spheres in
  {$4$}-space and cobordism of knots}, Osaka J. Math. \textbf{3} (1966),
  257--267.

\bibitem[Fox62]{Fox62}
R.~H. Fox, \emph{A quick trip through knot theory}, Topology of 3-manifolds and
  related topics (Proc. The Univ. of Georgia Institute, 1961), Prentice-Hall,
  1962, pp.~120--167.

\bibitem[Fox70]{Fox70}
\bysame, \emph{Metacyclic invariants of knots and links}, Canad. J. Math.
  \textbf{22} (1970), 193--201.

\bibitem[FQ90]{FreedmanQuinn}
Michael~H. Freedman and Frank Quinn, \emph{Topology of 4-manifolds}, Princeton
  Mathematical Series, vol.~39, Princeton University Press, Princeton, NJ,
  1990.

\bibitem[Gar03]{Garoufalidis03}
Stavros Garoufalidis, \emph{Does the {J}ones polynomial determine the signature
  of a knot?}, arXiv:math/0310203 (2003).

\bibitem[Gil83]{Gilmer83}
Patrick~M. Gilmer, \emph{Slice knots in {$S\sp{3}$}}, Quart. J. Math. Oxford
  Ser. (2) \textbf{34} (1983), no.~135, 305--322.

\bibitem[Gil93]{Gilmer93}
\bysame, \emph{Classical knot and link concordance}, Comment. Math. Helv.
  \textbf{68} (1993), no.~1, 1--19.

\bibitem[GLM81]{GordonLitherlandMurasugi81}
C.~McA. Gordon, R.~A. Litherland, and Kunio Murasugi, \emph{Signatures of
  covering links}, Canad. J. Math. \textbf{33} (1981), no.~2, 381--394.

\bibitem[Goe33]{Goeritz33}
Lebrecht Goeritz, \emph{Knoten und quadratische {F}ormen}, Math. Z. \textbf{36}
  (1933), no.~1, 647--654.

\bibitem[Gri07]{Grillet}
Pierre~Antoine Grillet, \emph{Abstract algebra}, second ed., Graduate Texts in
  Mathematics, vol. 242, Springer, New York, 2007. \MR{2330890 (2008c:20001)}

\bibitem[Har83]{Hartley83}
Richard Hartley, \emph{Identifying noninvertible knots}, Topology \textbf{22}
  (1983), no.~2, 137--145.

\bibitem[Hir95]{Hirzebruch67}
Friedrich Hirzebruch, \emph{Singularities and exotic spheres}, S\'eminaire
  {B}ourbaki 1966/67, {V}ol.\ 10, Soc. Math. France, Paris, 1995, pp.~Exp.\
  No.\ 314, 13--32.

\bibitem[HK79]{HartleyKawauchi79}
Richard Hartley and Akio Kawauchi, \emph{Polynomials of amphicheiral knots},
  Math. Ann. \textbf{243} (1979), no.~1, 63--70.

\bibitem[HKL10]{HeraldKirkLivingston10}
Chris Herald, Paul Kirk, and Charles Livingston, \emph{Metabelian
  representations, twisted {A}lexander polynomials, knot slicing, and
  mutation}, Math. Z. \textbf{265} (2010), no.~4, 925--949.

\bibitem[HKLar]{HeddenKirkLivingston11}
Matthew Hedden, Paul Kirk, and Charles Livingston, \emph{Non-slice linear
  combinations of algebraic knots}, Journal of the European Mathematical
  Society (JEMS) (to appear).

\bibitem[Ino10]{Inoue10}
Aymum Inoue, \emph{A symmetric motion picture of the twist-spun trefoil},
  arXiv:1008.2819v1 (2010).

\bibitem[Jab07]{Jabuka07}
Stanislav Jabuka, \emph{Heegaard {F}loer homology and knot concordance: a
  survey of recent advances}, Glas. Mat. Ser. III \textbf{42(62)} (2007),
  no.~1, 237--256.

\bibitem[Jia81]{Jiang81}
Bo~Ju Jiang, \emph{A simple proof that the concordance group of algebraically
  slice knots is infinitely generated}, Proc. Amer. Math. Soc. \textbf{83}
  (1981), no.~1, 189--192.

\bibitem[JN07]{JabukaNaik07}
Stanislav Jabuka and Swatee Naik, \emph{Order in the concordance group and
  {H}eegaard {F}loer homology}, Geom. Topol. \textbf{11} (2007), 979--994.

\bibitem[Kau87]{Kauffman}
Louis Kauffman, \emph{On knots}, Princeton University Press, 1987.

\bibitem[Kea04]{Kearton04}
C.~Kearton, \emph{{$S$}-equivalence of knots}, J. Knot Theory Ramifications
  \textbf{13} (2004), no.~6, 709--717.

\bibitem[Ker65]{Kervaire65}
M.~Kervaire, \emph{Les n\oe uds de dimensions sup\'erieures}, Bull. Soc. Math.
  France \textbf{93} (1965), 225--271.

\bibitem[Kim05]{Kim05}
Se-Goo Kim, \emph{Polynomial splittings of {C}asson-{G}ordon invariants}, Math.
  Proc. Cambridge Philos. Soc. \textbf{138} (2005), no.~1, 59--78.

\bibitem[Kit96]{Kitano96}
Teruaki Kitano, \emph{Twisted {A}lexander polynomial and {R}eidemeister
  torsion}, Pacific J. Math. \textbf{174} (1996), no.~2, 431--442.

\bibitem[KL99a]{KirkLivingston99a}
Paul Kirk and Charles Livingston, \emph{Twisted {A}lexander invariants,
  {R}eidemeister torsion, and {C}asson-{G}ordon invariants}, Topology
  \textbf{38} (1999), no.~3, 635--661.

\bibitem[KL99b]{KirkLivingston99b}
\bysame, \emph{Twisted knot polynomials: inversion, mutation and concordance},
  Topology \textbf{38} (1999), no.~3, 663--671.

\bibitem[KL01]{KirkLivingston01}
\bysame, \emph{Concordance and mutation}, Geom. Topol. \textbf{5} (2001),
  831--883.

\bibitem[KM94]{KirbyMelvin94}
Robion~C. Kirby and Paul Melvin, \emph{Dedekind sums, {$\mu$}-invariants and
  the signature cocycle}, Math. Ann. \textbf{299} (1994), no.~2, 231--267.

\bibitem[Lev69a]{Levine69-1}
J.~Levine, \emph{Invariants of knot cobordism}, Invent. Math. 8 (1969),
  98--110; addendum, ibid. \textbf{8} (1969), 355.

\bibitem[Lev69b]{Levine69}
\bysame, \emph{Knot cobordism groups in codimension two}, Comment. Math. Helv.
  \textbf{44} (1969), 229--244.

\bibitem[Lev10]{ALevine10}
Adam~Simon Levine, \emph{Applications of {H}eegaard {F}loer homology to knot
  and link concordance},
  {\texttt{http://www.math.columbia.edu/~thaddeus/theses/2010/levine.pdf}},
  2010.

\bibitem[Lic97]{Lickorish}
W.~B.~Raymond Lickorish, \emph{An introduction to knot theory}, Springer
  Graduate Texts in Mathematics, 1997.

\bibitem[Lit79]{Litherland79}
R.~A. Litherland, \emph{Signatures of iterated torus knots}, Topology of
  low-dimensional manifolds ({P}roc. {S}econd {S}ussex {C}onf., {C}helwood
  {G}ate, 1977), Lecture Notes in Math., vol. 722, Springer, 1979, pp.~71--84.

\bibitem[Liv83]{Livingston83}
Charles Livingston, \emph{Knots which are not concordant to their reverses},
  Quart. J. Math. Oxford Ser. (2) \textbf{34} (1983), no.~135, 323--328.

\bibitem[Liv99]{Livingston98}
\bysame, \emph{Order 2 algebraically slice knots}, Proceedings of the
  {K}irbyfest ({B}erkeley, {CA}, 1998), Geom. Topol. Monogr., vol.~2, Geom.
  Topol. Publ., Coventry, 1999, pp.~335--342 (electronic).

\bibitem[Liv01]{Livingston01}
\bysame, \emph{Infinite order amphicheiral knots}, Algebr. Geom. Topol.
  \textbf{1} (2001), 231--241 (electronic).

\bibitem[Liv05]{Livingston05}
\bysame, \emph{A survey of classical knot concordance}, Handbook of knot
  theory, Elsevier B. V., Amsterdam, 2005, pp.~319--347.

\bibitem[Liv08]{Livingston08}
\bysame, \emph{The algebraic concordance order of a knot},
  {\texttt{arXiv:0806.3068v25 [math.GT]}}, 2008.

\bibitem[LN99]{LivingstonNaik99}
Charles Livingston and Swatee Naik, \emph{Obstructing four-torsion in the
  classical knot concordance group}, J. Differential Geom. \textbf{51} (1999),
  no.~1, 1--12.

\bibitem[LN01]{LivingstonNaik01}
\bysame, \emph{Knot concordance and torsion}, Asian J. Math. \textbf{5} (2001),
  no.~1, 161--167.

\bibitem[LN09]{LivingstonNaik09}
Charles Livingston and Swatee Naik, \emph{Stabilizing four-torsion in classical
  knot concordance}, \texttt{arXiv:0908.0005 [math.GT]}, 2009.

\bibitem[Mat77]{Matumoto77}
Takao Matumoto, \emph{On the signature invariants of a non-singular complex
  sesquilinear form}, J. Math. Soc. Japan \textbf{29} (1977), no.~1, 67--71.

\bibitem[MH73]{MilnorHusemoller}
John~W. Milnor and Dale Husemoller, \emph{Symmetric bilinear forms},
  Springer-Verlag, New York, 1973, Ergebnisse der Mathematik und ihrer
  Grenzgebiete, Band 73.

\bibitem[Mil62]{Milnor62}
John~W. Milnor, \emph{A duality theorem for {R}eidemeister torsion}, Ann. of
  Math. (2) \textbf{76} (1962), 137--147.

\bibitem[Mil68]{Milnor68}
\bysame, \emph{Infinite cyclic coverings}, Conference on the Topology of
  Manifolds (Michigan State Univ., E. Lansing, Mich., 1967), Prindle, Weber \&
  Schmidt, Boston, Mass., 1968, pp.~115--133.

\bibitem[Mil69]{Milnor69}
\bysame, \emph{On isometries of inner product spaces}, Invent. Math. \textbf{8}
  (1969), 83--97.

\bibitem[MO07]{ManolescuOwens07}
Ciprian Manolescu and Brendan Owens, \emph{A concordance invariant from the
  {F}loer homology of double branched covers}, Int. Math. Res. Not. IMRN
  (2007), no.~20, Art. ID rnm077, 21.

\bibitem[Mor51]{Mordell51}
L.~J. Mordell, \emph{The reciprocity formula for {D}edekind sums}, Amer. J.
  Math. \textbf{73} (1951), 593--598.

\bibitem[Mur65]{Murasugi65}
Kunio Murasugi, \emph{On a certain numerical invariant of link types}, Trans.
  Amer. Math. Soc. \textbf{117} (1965), 387--422.

\bibitem[Nai96]{Naik96}
Swatee Naik, \emph{Casson-{G}ordon invariants of genus one knots and
  concordance to reverses}, J. Knot Theory Ramifications \textbf{5} (1996),
  no.~5, 661--677.

\bibitem[OS03]{OzsvathSzabo03}
Peter Ozsv{\'a}th and Zolt{\'a}n Szab{\'o}, \emph{Knot {F}loer homology and the
  four-ball genus}, Geom. Topol. \textbf{7} (2003), 615--639.

\bibitem[OS04]{OszvathSzabo04}
\bysame, \emph{Heegaard diagrams and holomorphic disks}, Different faces of
  geometry, Int. Math. Ser. (N. Y.), vol.~3, Kluwer/Plenum, New York, 2004,
  pp.~301--348.

\bibitem[Ras04]{Rasmussen04}
Jacob Rasmussen, \emph{Khovanov homology and the slice genus},
  arXiv:math.GT/0402131. (2004).

\bibitem[Rol03]{Rolfsen}
Dale Rolfsen, \emph{Knots and links}, AMS Chelsea Publishing, 2003.

\bibitem[Rud93]{Rudolph93}
Lee Rudolph, \emph{Quasipositivity as an obstruction to sliceness}, Bull. Amer.
  Math. Soc. (N.S.) \textbf{29} (1993), no.~1, 51--59.

\bibitem[Sch85]{Scharlau}
Winfried Scharlau, \emph{Quadratic and {H}ermitian forms}, Grundlehren der
  Mathematischen Wissenschaften [Fundamental Principles of Mathematical
  Sciences], vol. 270, Springer-Verlag, 1985.

\bibitem[Sco05]{Scorpan}
Alexandru Scorpan, \emph{The wild world of 4-manifolds}, American Mathematical
  Society, Providence, RI, 2005.

\bibitem[Sei35]{Seifert34}
H.~Seifert, \emph{{\"{U}ber das Geschlecht von Knoten}}, Math. Ann.
  \textbf{110} (1935), 571--592.

\bibitem[Ser73]{Serre}
Jean-Pierre Serre, \emph{A course in arithmetic}, Springer-Verlag, New York,
  1973, Translated from the French, Graduate Texts in Mathematics, No. 7.

\bibitem[Sma61]{Smale61}
Stephen Smale, \emph{Generalized {P}oincar\'e's conjecture in dimensions
  greater than four}, Ann. of Math. (2) \textbf{74} (1961), 391--406.

\bibitem[Sto05]{Stoimenow05}
A.~Stoimenow, \emph{Some applications of {T}ristram-{L}evine signatures and
  relation to {V}assiliev invariants}, Adv. Math. \textbf{194} (2005), no.~2,
  463--484.

\bibitem[Tam02]{Tamulis02}
Andrius Tamulis, \emph{Knots of ten or fewer crossings of algebraic order 2},
  J. Knot Theory Ramifications \textbf{11} (2002), no.~2, 211--222.

\bibitem[Tei01]{Teichner01}
Peter Teichner, \emph{Slice knots: Knot theory in the 4th dimension},
  {\texttt{http://www.maths.ed.ac.uk/$\sim$s0783888/sliceknots2.pdf}}, 2001,
  Edited by Julia Collins and Mark Powell.

\bibitem[Tei04]{Teichner04}
Peter Teichner, \emph{{What is ... a grope?}}, Notices A.M.S. \textbf{51}
  (2004), 894--895.

\bibitem[Tri69]{Tristram69}
A.~G. Tristram, \emph{Some cobordism invariants for links}, Proc. Cambridge
  Philos. Soc. \textbf{66} (1969), 251--264.

\bibitem[Tro62]{Trotter62}
H.~F. Trotter, \emph{Homology of group systems with applications to knot
  theory}, Ann. of Math. (2) \textbf{76} (1962), 464--498.

\bibitem[vWC06]{Wijk06}
Jarke~J. van Wijk and Arjeh~M. Cohen, \emph{Visualization of {S}eifert
  surfaces}, IEEE Trans. on visualization and computer graphics \textbf{1}
  (2006), 1--13.

\end{thebibliography}


}
\end{document}